%% file: differential-equations-final.tex
\documentclass{amsart}

\usepackage{newtxtext, newtxmath}

\usepackage{geometry}
\usepackage{latexsym, bm}
\usepackage{amsmath, amssymb, amsfonts, amsthm, amscd}
\usepackage{cases}
\usepackage{enumerate, paralist}
\usepackage{appendix}

\usepackage{graphicx}
\usepackage{subfigure}

\usepackage[colorlinks=true]{hyperref}
\hypersetup{
	linkcolor=blue,
	anchorcolor=blue,
	citecolor=blue,
	filecolor=blue,
	urlcolor=blue,
	menucolor=blue,
}

\usepackage[msc-links, alphabetic, abbrev, nobysame]{amsrefs}

\usepackage{titletoc}

\input{ThmAndAutoref}

\DeclareMathOperator*{\lip}{Lip}

\newcommand{\graph}{\mathrm{Graph}}
\newcommand{\id}{\mathrm{id}}

\allowdisplaybreaks

\numberwithin{equation}{section}

\begin{document}

\title[Exponential dichotomy and Invariant manifolds]{The exponential dichotomy and invariant manifolds for some classes of differential equations}

\author{DeLiang Chen}
\address{Department of Mathematics, Shanghai Jiao Tong University, Shanghai 200240, People's Republic of China}
\email{chernde@sjtu.edu.cn}

\thanks{The author is very indebted to Prof. Shigui Ruan and Prof. Dongmei Xiao for their useful discussions. The author would also like to thank Lianwang Deng for his useful discussions and particularly drawing the author's attention to the references \cite{MY90, vdMee08}.}



\begin{abstract}
We study some classes of semi-linear differential equations including both well-posed and ill-posed cases that can generate cocycles (or cocycle correspondences with generating cocycles).
Under exponential dichotomy condition with other mild assumptions, we investigate the existence, persistence and regularity of different types of invariant manifolds for these differential equations based on our previous works about invariant manifold theory for abstract `generalized dynamical systems': invariant graphs (global version) and normally hyperbolic invariant manifolds (local version); brief summaries of those works are also given.
\end{abstract}

\maketitle

\titlecontents{section}[0pt]{\vspace{0\baselineskip}\bfseries}
{\thecontentslabel\quad}{}
{\hspace{.5em}\titlerule*[10pt]{$\cdot$}\contentspage}

\titlecontents{subsection}[10pt]{\vspace{0\baselineskip}}
{\thecontentslabel\quad}{}
{\hspace{.5em}\titlerule*[10pt]{$\cdot$}\contentspage}

\setcounter{tocdepth}{2}

\input{sect0.tex}
\input{sect1.tex}
\input{sect2.tex}

\input{sect3.tex}

\begin{appendices}
	\setcounter{equation}{0}
	\renewcommand{\theequation}{\Alph{section}.\arabic{equation}}
	\input{app.tex}

\end{appendices}

\bibliographystyle{amsalpha}

\input{ref.bbl}

\end{document}

%% file: sect0.tex
\section{Introduction}

\subsection{motivation}

In this paper, a sequel to our previous papers \cite{Che18a, Che18b}, we try to apply our new developed invariant manifold theory to some classes of differential equations; that is the existence and regularity of invariant graphs for cocycle correspondences with (relatively) partial hyperbolicity \cite{Che18a} and approximately normal hyperbolicity theory \cite{Che18b}, see \autoref{continuous case} and \autoref{normalH} for brief summaries. This is a partial program of giving a \emph{general} procedure to deal with invariant manifold theory for both \emph{well-posed} and \emph{ill-posed} differential equations in Banach spaces.

Invariant manifold theory provides an extremely useful tool to understand the dynamics of the nonlinear differential equations. As an illustration, we list some applications of this theory.
\begin{asparaenum}[(i)]
	\item It gives a finite-dimensional reduction for differential equations. This is a way to address one of the central topics in the theory of dissipative dynamical systems generated by PDEs, namely, whether or not the underlying dynamics can be described by finite-dimensional systems (i.e. ODEs). The notion of \emph{inertial manifold} (or equivalently \emph{pseudo-unstable manifold}) gives a perfect description of finite-dimensionality of dissipative differential equations; see \cite{MS88, Zel14, Tem97} for more details. Similarly, the center manifold provides a local reduction principle, namely, the differential equation can be restricted in a `low' dimensional space, i.e. a center manifold, such that it may be simpler than the original one but also can reflect some properties of itself in the whole space.
	\item The (center-) (un)stable manifolds with strong (un)stable foliations characterize very clearly the asymptotic behaviors of a dynamic around its invariant manifold and so the stability of the invariant manifold; see \cite{BLZ00, BLZ08, KNS15, NS12, NS11}. Furthermore, through using different types of invariant foliations and invariant manifolds, one could decouple the system into a simple form (see e.g. \cite[Corollary 4.19]{Che18a} and \cite{Lu91, PS70}).
	\item As is well known, invariant manifold theory with other tools is powerful for finding special interesting orbits such as periodic orbit, homoclinic orbit and heteroclinic orbit. For example, if a periodic orbit of a dynamic is normally hyperbolic, then the small $ C^1 $ perturbed dynamic persists a periodic orbit; if a center manifold of an equilibrium for a differential equation, no matter if it is well-posed or ill-posed, is finite-dimensional, then the classical Hopf bifurcation theorem can be applied under some other conditions to obtain periodic orbits. See also \cite{SS99, MR09a, LMSW96, Zen00}.
	\item Invariant manifold provides some true solutions of ill-posed differential equations. In general, an ill-posed differential equation might not exist a solution for given an initial value. However, such equation will be well-posed in the so called center manifold, meaning that there always exist a solution in this manifold for all time $ t \in \mathbb{R} $ if initial value belongs to this manifold. The so named center-(un)stable manifolds play similar roles. See \cite{EW91,Gal93,dlLla09,ElB12}.
\end{asparaenum}

Invariant manifold theory has being extensively developed in infinite-dimensional dynamical systems. \textbf{(I)} In \cite{Rue82, Man83, LL10}, the authors investigated the invariant manifolds in the non-uniform hyperbolicity case which can be applied to random dynamical systems in Banach spaces. Also, in \cite{CL97}, Chicone and Latushkin studied existence of the Lipschitz center manifold for a semi-linear cocycle (i.e. skew-product flow).

\textbf{(II)} There are many authors devoted to develop the theory of invariant manifold around an equilibrium for different types of well-posed differential equations such as the semi-linear and quasi-linear parabolic or hyperbolic PDEs, see e.g. \cite{Hen81, CL88, MS88, BJ89, DPL88, Tem97, MR09a, Zel14} in abstract settings, where those papers also contained some examples with detailed analysis, too numerous to list here. In \cite{CL88}, Chow and Lu considered the case when the linear operator is densely-defined sectorial operator with an unbounded nonlinear perturbation (see also \cite{Hen81}). The restriction `densely-defined' of the linear operator was removed in Da Prato and Lunardi's work \cite{DPL88} (but with an additional unnecessary `compact' assumption for the linear operator). In \cite{BJ89}, Bates and Jones studied the case when the linear operator is a generator of $ C_0 $ group with additional restriction that the stable and unstable subspaces are finite dimensional. Note that this restriction can also be satisfied by many Hamiltonian PDEs (see e.g. \cite{LZ17}). More recently, by introducing more general operators in \cite{MR07} than Hille-Yosida operators, Magal and Ruan in \cite{MR09a} investigated the smooth center manifolds of more general semi-linear differential equations (see also \autoref{HYccc}) under an unnecessary `compact' assumption for the linear operator which can be replaced by the \emph{uniform trichotomy condition} (see \autoref{def:ud+} and \autopageref{def:ut}) which was characterized detailedly in \cite{CL99}. The list we give is by no means exhaustive, and we refer to the introduction of \cite{BLZ98} for more details.

Among previous results on invariant manifolds, the existence of exponential dichotomy or sometimes the exponential trichotomy, is important for it provides a framework to analyze the local nonlinear dynamics, is technically assumed to be hold. To verify this, one usually hopes that the spectral mapping theorem holds: $ \exp(\sigma(A))  = \sigma (\exp A) \backslash \{0\} $ for the linear operator $ A $; or the weak form $ \overline{\exp (\sigma(A))} = \sigma (\exp A) $.
So some additional compactness condition on the semigroup was assumed to be hold; see e.g. \cite{DPL88,MR09a}. For more details about this issue see \cite[Chapter IV]{EN00} and \cite{NP00}. See also \cite{LZ17} for a beautiful characterization of the exponential trichotomy for some particular operators which does not induce from spectral mapping theorem. However, beyond the spectral gap condition, invariant manifolds might also exist by using other conditions. This was done in \cite{MS88} where the authors introduced a condition which we call \emph{(A) (B) condition} in the non-linear version (see \autoref{defAB}); for relevant results by using (A) (B) condition to obtain the invariant manifolds, see also \cite{Zel14, LYZ13} and our work about invariant manifold theory. In the present paper, we only focus on the uniform dichotomy condition plus a `small' Lipschitz perturbation to verify the (A) (B) condition, but referring the reader to see \cite{MS88, Zel14} where the perturbation can be more general than `small' Lipschitz (namely \emph{spatial averaging}) which makes (A) (B) condition also hold.

\textbf{(III)} The invariant manifolds also exist around an equilibrium for some ill-posed differential equations which even can not generate semiflows such as the good Boussinesq equation, the elliptic problem on the cylinder, the \emph{spatial} dynamics induced by the reaction-diffusion equations, etc; see e.g. the works of Eckmann and Wayne \cite{EW91}, Gally \cite{Gal93}, de la Llave \cite{dlLla09}, and ElBialy \cite{ElB12} but the results are further less. Such equations have a common feature: for most initial conditions, there does not exist a local solution. The setting for establishing the (stable, center-stable, pseudo-stable, center, etc) invariant manifolds in \cite{EW91,Gal93,ElB12} are essentially the same which we now use the notion of a generator of a \emph{bi-semigroup} under uniform dichotomy condition based on the work of Latushkin and Pogan \cite{LP08} (or uniform trichotomy condition for establishing the center manifolds). See the very interesting proof in \cite{ElB12} beyond the Lyapunov-Perron method or Hadamard graph transform method. In \cite{SS99}, Sandstede and Scheel also obtained invariant manifolds both for equilibrium and periodic orbit of the spatial dynamic generated by \autoref{ex:spatial}.

In \cite{Che18a}, we gave a unified study on the global version of invariant manifolds, i.e. the invariant graphs for bundles or \emph{bundle correspondences with generating bundle maps} (see \autoref{basicCorr}) in non-trivial bundles, which can be applied to different settings; see also \autoref{continuous case} and \autoref{generalA}.

\textbf{(IV)} Beside the previous results on invariant manifolds of an equilibrium, the existence and persistence of invariant manifolds around an invariant manifold are also evidently important where the invariant manifold are usually taken as equilibriums, (a-)periodic orbit, several orbits with their closure (including e.g. homoclinic orbits, heteroclinic orbits, etc), or the global compact attractor. The notion of \emph{normal hyperbolicity} plays a crucial role as the hyperbolicity of equilibrium which is the right condition for persistence. For a normally hyperbolic invariant manifold, loosely speaking, it means the linearized dynamic along this manifold contracts or expands along the normal direction and does so to a greater degree than it does along the tangential direction. The theory of normal hyperbolicity was investigated at length by Hirsch, Pugh and Shub \cite{HPS77} and Fenichel \cite{Fen72, Fen74, Fen77}. Further developments were given by Li and Wiggins \cite{LW97} and Pliss and Sell \cite{PS01} with an aim to make it applicable to partial differential equations in Banach spaces. A more significant generalization was made by Bates, Lu and Zeng \cite{BLZ98, BLZ99, BLZ00, BLZ08} where the authors extended the classical theory to the abstract infinite-dimensional dynamical systems with allowing the invariant manifold to be immersed and non-compact. In \cite{Che18b}, we expanded the scope of normal hyperbolicity theory to more general settings (than \cite{BLZ08}) in order to deal with non-smooth and non-Lipschitz dynamical systems and ill-posed (as well as well-posed) differential equations; see also \autoref{normalH}.

\textbf{(V)} In some cases, the invariant manifold might be not normally hyperbolic but there may exist center-(un)stable manifolds around this manifold which also give some detailed characterizations of the dynamical behaviors; for example for a periodic orbit with period $ T $, if the associated time-$ T $ solution map of its linearized dynamic along this orbit has a non-simple spectrum $ 1 $, then this periodic orbit is not normally hyperbolic. A notion of \emph{partially normal hyperbolicity} which we used in \cite{Che18} can give a way to deal with this situation. In finite-dimensional dynamical systems, this was also settled in \cite{CLY00, CLY00a, BC16}. The corresponding results for abstract dynamical systems in infinite-dimension was addressed in \cite{Che18}. See also \cite{NS12, KNS15, JLZ17} as well as \cite{HVL08, SS99} where the results were obtained for some concrete PDEs. In the present paper, we do not consider this situation, referring to see \cite{Che18}.

\subsection{nontechnical overviews of main results}\label{overview}

We only focus on some classes of \emph{abstract} differential equations. Roughly, the differential equations are written as two parts: the linear part with a `small Lipschitz' non-linear perturbation; the linear part usually is written as a closed linear operator plus a linear perturbation in a `cocycle' form. To be more precisely, consider
\[
\dot{z}(t) = \mathcal{C}(t\omega)z(t) + f(t\omega)z(t), \tag{$ \clubsuit $}
\]
where $ \mathcal{C}(\omega): D(\mathcal{C}(\omega)) \subset Z \to Z $, $ \omega \in M $, are closed linear operators, $ M $ is a topology space, $ Z $ is a Banach space, $ t: M \to M $ is a $ C_0 $ semiflow, and $ f: M \times Z \to Z $ is a nonlinear operator.
We discuss the following three situations about $ \{ \mathcal{C}(\omega) \} $ and $ f $.
\begin{enumerate}
	\item[(type I)] $ \mathcal{C}(\omega) = C + L(\omega), ~L: M \to L(Z,Z), ~f: M \times Z \to Z, $

	\noindent where $ C $ is a generator of a \emph{$ C_0 $ bi-semigroup} (see e.g. \cite{vdMee08, LP08} and \autoref{illustration}) or a \emph{$ C_0 $ semigroup} (see e.g. \cite{EN00,Paz83}) in $ Z $.

	\item[(type II)] $ \mathcal{C}(\omega) = C + L(\omega), ~L: M \to L(\overline{D(A)},Z), ~f: M \times \overline{D(A)} \to Z, $

	\noindent where $ A: D(A) \subset Z \to Z $ is a \emph{Hille-Yosida operator} (see \cite{DPS87, EN00, ABHN11} and \autoref{operator}), or more generally an \emph{MR operator} (see the assumption (MR) in \autopageref{MR} which was studied in \cite{MR07, MR09, MR09a}). This class of $ A $ includes many concrete different equations (see e.g. \autoref{examples} and the references we list before).

	\item[(type II$ _1 $)] $ \mathcal{C}(\omega) = C + L(\omega), ~L: M \to L(\overline{D(A)},D(A^{-\alpha})), ~f: M \times \overline{D(A)} \to D(A^{-\alpha}), $

	\noindent where $ A: D(A) \subset Z \to Z $ is a Hille-Yosida operator (or MR operator) with additional assumption that $ A_{\overline{D(A)}} $, the part of $ A $ in $ \overline{D(A)} $ (see \autoref{operator} for a definition), is a densely-defined \emph{sectorial} operator for some suitable $ \alpha > 0 $. For example $ A $ is a sectorial operator, i.e. a generator of a holomorphic semigroup (see e.g. \cite{Ama95,Hen81}) and $ 0 < \alpha < 1 $. Here we assume without loss of generality spectral bound $ s(A) < 0 $. \emph{This case is very similar as (type II), so the details are omitted.}

	\item[(type III)] Equation ($ \clubsuit $) generates a $ C_0 $ cocycle or a $ C_0 $ cocycle correspondence (see \autoref{def:diffAndCocycle}) in $ Z $ and $ f: M \times Z \to Z $; in fact, in this case, we consider the integral equation \eqref{equ:IE}. This case is essentially the same as (type I).
\end{enumerate}

Note that (type I) and (type II) (or (type II$ _1 $)) are the most important cases to study the following autonomous different equation around some invariant set $ M $,
\[
\dot{z}(t) = Az(t) + g(z(t)), \tag{$ \spadesuit $}
\]
where $ g \in C^{1}(Z_0, Z_{-1}) $, and $ A: D(A) \subset Z \to Z $ and $ Z_0, Z_{-1} $ are one of the following cases.
\begin{enumerate}[({type} $ \bullet $a)]
	\item $ A $ is a generator of a $ C_0 $ semigroup (\cite{EN00}) or a $ C_0 $ bi-semigroup (\cite{vdMee08, LP08}); $ Z_0 = Z_{-1} = Z $.
	\item $ A $ is a sectorial operator (\cite{ABHN11}) or more generally $ A $ is a \emph{MR operator} (see the assumption (MR) in \autopageref{MR}) with additional assumption that $ A_{\overline{D(A)}} $ is a densely-defined \emph{sectorial} operator; $ Z_0 = \overline{D(A)} $ and $ Z_{-1} = D(A^{-\alpha}) $ for some suitable $ \alpha > 0 $. Here we assume without loss of generality the spectral bound $ s(A) < 0 $.
	\item $ A $ is a Hille-Yosida operator (\cite{ABHN11}) or more generally an \emph{MR operator} (see the assumption (MR) in \autopageref{MR}); $ Z_0 = \overline{D(A)} $ and $ Z_{-1} = Z $. Note that sectorial operators are Hille-Yosida operators.
\end{enumerate}
For some concrete examples of ($ \spadesuit $), see \autoref{examples}.

\vspace{.5em}
\noindent{\textbf{dichotomy and (A) (B) condition.}}
In order to deal with the ill-posed differential equations like (type I) and (type III), in \autoref{basicCorr}, a class of `generalized hyperbolic dynamical systems' are introduced, named cocycle correspondence over a semiflow and continuous correspondence by using the notion of \emph{correspondence} originally due to \cite{Cha08}. This is necessary since the ill-posed differential equations in general can not generate semiflows or cocycles but induce cocycle correspondences (see \autoref{diffCocycle}). In addition, by using the notion of \emph{dual correspondence} (see \autoref{dual}), one can give a unified approach to obtain the `stable results' and `unstable results' for cocycles, which is different with classical methods e.g. \cite{BLZ98}. Unlike the classical way, we adopt the notion of \emph{(A) (B) condition} to describe the hyperbolicity motivated by \cite{MS88, LYZ13, Zel14} which is close to invariant cone condition but in the non-linearity version; see \autoref{defAB}. There is another purpose we introduce these conceptions, that is, by doing so, our results can be applied to non-smooth and non-Lipschitz dynamical systems. We refer the readers to see \cite{Che18a, Che18d, Che18b, Che18} for more results about this `generalized dynamical system' with some hyperbolicity described by (A) (B) condition.

To apply our existence and regularity results in \autoref{continuous case} and \autoref{normalH} (or see \cite{Che18a, Che18b} in detail) to semi-linear abstract differential equations ($ \clubsuit $) or ($ \spadesuit $), from the abstract view, one needs to show \emph{the differential equations can generate cocycles (for the well-posed case) or cocycle correspondences with generating cocycles (for the ill-posed case) satisfying (A) (B) condition}. In \autoref{diffCocycle}, we will deal with with relationship between the \emph{dichotomy} (or more precisely the exponential dichotomy of (linear) differential equations) and (A) (B) condition. See the main results \autoref{thm:biAB}, \autoref{thm:gencocycleAB} and \autoref{them:spec} in \autoref{diffCocycle}. For a comprehensive study of uniform dichotomy for $ C_0 $ linear cocycles, see \cite{CL99} and the references therein.

A first goal in \autoref{diffCocycle} is to rewrite some equations satisfied by the \emph{(mild) solutions} of the different equations in different forms, i.e. the different `variant of constant formulas' satisfied by the solutions. And then in some appropriate forms, one can verify the (A) (B) condition under uniform dichotomy condition (see \autoref{def:ud+}). As an illustration, see \autoref{illustration} (particularly \autoref{lem:simpleAB}); for a more special case, see also \cite[Lemma 4]{LYZ13}. For other conditions verifying the (A) (B) condition which are far away with `small Lipschitz' perturbations, see \cite{MS88} and \cite[Section 2.8]{Zel14}.

In \autoref{HYccc}, we consider the well-posed case for (type II). Due to the \emph{difficulty} that the linear part of the differential equations does not generate $ C_0 $ cocycle in the whole space but the ranges of the non-linear perturbations are taken in the whole space, the link of the different `variant of constant formulas' is not so clear. There are some classical ways e.g. Yosida approximations, extrapolation spaces to deal with this difficulty under some special contexts. In the present paper, we use a very effective tool, namely the \emph{integrated semigroup} theory (see \cite{ABHN11}), to handle the general case; see \autoref{lem:vcf1} (and \autoref{lem:coycle0} \eqref{integral}).

A very analogous argument which we do not give details in this paper, can be applied to settle equations (type II$ _1 $), i.e. the linear operator is a Hille-Yosida operator (or MR operator) with some analytic properties of its `$ C_0 $ semigroup' in the closure of its domain, and the linear and non-linear perturbations are allowed to be in some sense `unbounded' (see also \cite{Ama95}).

In \autoref{bi-semigroup}, we study the ill-posed case (type I). The closed linear operator is assumed to be a generator of a \emph{$ C_0 $ bi-semigroup} (or an \emph{exponentially dichotomous operator} `in some particular situation' which was studied comprehensively in \cite{vdMee08}) and the perturbations are required to be bounded. The uniform dichotomy condition for this case is taken from \cite{LP08} where the authors first studied this exponential dichotomy for the ill-posed differential equations in an abstract way (but in a special setting); see also \cite{SS01}. We mention that the spectral theory for the ill-posed differential equation is not well developed yet. In \autoref{genccc}, we also give a sketch discussion about a light general case that the linear part of the differential equation generates a $ C_0 $ cocycle or $ C_0 $ cocycle correspondence on the \emph{whole} space. 

There is an interesting thing that by our argument we also obtain the sharpness of the spectral gap for the $ C_0 $ bi-semigroup case (but not the Hille-Yosida operators or general MR operators) case (and also in the `cocycle' case) in the spirit of \cite{Mcc91, Rom93} (see also \cite{Zel14}); see \autoref{lem:simpleAB} (\autoref{rmk:sharp}) and \autoref{thm:gencocycleAB}. 

\vspace{.5em}
\noindent{\textbf{invariant manifold.}}
In \cite{Che18a} and \cite{Che18b}, we investigated extensively about the existence and regularity of the invariant graphs (for bundles or bundle correspondences) and normally hyperbolic manifolds (for maps or correspondences) in the discrete context, respectively; see \autoref{continuous case} and \autoref{normalH} for brief summaries in the corresponding \emph{continuous} circumstance. A number of applications of the main results in the \cite{Che18a} like decoupling theorem, different types of invariant foliations (laminations) including strong stable laminations and fake invariant foliations, and holonomies for cocycles, which can be used to derive more properties of ($ \clubsuit $), were also given in that paper but not included in this paper. As a simple application of the results in \autoref{continuous case}, we give a global invariant manifold result concerning equation ($ \clubsuit $), which, heuristically, can be summarized as follows.
\begin{thmA}\label{A}
	Assume for all $ \omega \in M $, ($ \circ $1) $ f(\omega)(0) = 0 $, or ($ \circ $2) $ \sup_{t \geq 0} \sup_{z} |f(t\omega)(z)| < \infty $.
	Under uniform dichotomy condition (so there are two bundles $ X, Y $ over $ M $ such that $ M \times Z = X \times Y $), Lipschitz continuity of $ f(\omega)(\cdot) $ with `suitable' Lipschitz constant, and certain spectral gap condition according to case ($ \circ $1) or ($ \circ $2), there is a set $ \mathcal{M} = \bigcup_{\omega \in M} (\omega, \mathcal{M}_{\omega}) \subset M \times Z $ such that
	\begin{enumerate}[(1)]
		\item $ \mathcal{M}_{\omega} = \graph \Psi_{\omega} $, a Lipschitz graph of $ \Psi_{\omega}: X_{\omega} \to Y_{\omega} $;
		\item if ($ \circ $1) holds, then $ 0 \in \mathcal{M}_{\omega} $; if ($ \circ $2) holds, then $ \sup_{t \geq 0} |\Psi_{t\omega}(0)| < \infty $;
		\item $ \mathcal{M} $ is positively invariant under equation ($ \clubsuit $), meaning for each $ (\omega,z) \in \mathcal{M} $, there is a (mild) solution $ u(t) $ ($ t\geq 0 $) of equation ($ \clubsuit $) with $ u(0) = z $ such that $ u(t) \in \mathcal{M}_{t\omega} $ for all $ t \geq 0 $.
		\item If $ z \mapsto f(\omega)z $ is $ C^1 $ for each $ \omega \in M $, so is $ x \mapsto \Psi_{\omega}(x) $.
	\end{enumerate}
\end{thmA}
See \autoref{thm:equI} and \autoref{thm:equII} for detailed statements. This theorem gives different types of invariant foliations of autonomous different equation ($ \spadesuit $) around an equilibrium, different types of invariant manifolds of equation ($ \spadesuit $) around an equilibrium (when $ M $ reduces as a one point set), or the strong (un)stable lamination of equation ($ \spadesuit $) around the invariant set $ M $. So \autoref{A} as well as the results in \autoref{continuous case} recover many classical results about the existence and regularity of invariant manifolds and invariant foliations obtained in e.g. \cite{CL88, MS88, DPL88, BJ89, CY94, CHT97, MR09a, LYZ13} (for the well-posed case), \cite{EW91,Gal93,SS99,ElB12} (for the ill-posed case), \cite{CL97} (for the $ C_0 $ cocycle case), and \cite{CLL91}; and in some cases they are even new, for instance, (i) the invariant foliations of equation ($ \spadesuit $) for the case that $ A $ is a Hille-Yosida operator (or MR operator) which can be seen as a supplement of \cite{MR09a}, (ii) the invariant manifolds of equation ($ \clubsuit $) in (type II), and (iii) the more precise spectral gap condition when $ A $ is a Hille-Yosida operator or for the equation ($ \clubsuit $) in (type III), etc. Also, there are many other results in \cite{Che18a} can be applied to equation ($ \clubsuit $) or ($ \spadesuit $) with the help of \autoref{thm:biAB}, \autoref{thm:gencocycleAB} and \autoref{them:spec} in \autoref{diffCocycle}.

Turn to consider the normal hyperbolicity case.

\begin{thmA}\label{B}
	Let $ M $ be a uniformly Lipschitz immersed submanifold of $ Z_0 $ (assumption (B1) in \autoref{normalH}). Assume $ M $ is invariant and normally hyperbolic with respect to equation ($ \spadesuit $) and the amplitude of $ g|_{\mathbb{B}_{\epsilon}(M)} $ is small as $ \epsilon \to 0 $, then the following hold for some $ r > 0 $.
	\begin{enumerate}[(1)]
		\item (Center-(un)stable manifolds) There are center-stable and center-unstable manifolds $ W^{cs}_{loc}(M) $, $ W^{cu}_{loc}(M) $ of $ M $ in $ \mathbb{B}_{r}(M) $, which are $ C^1 $ immersed submanifolds of $ Z_0 $ and $ W^{cs}_{loc}(M) \cap W^{cu}_{loc}(M) = M $. There is a positive constant $ r' < r $ such that for any $ z_0 \in W^{cs}_{loc}(M) \cap \mathbb{B}_{r'}(M) $ (resp. $ z_0 \in W^{cu}_{loc}(M) \cap \mathbb{B}_{r'}(M) $), there is a mild solution $ \{ u(t) \}_{t \geq 0} \subset W^{cs}_{loc}(M) $ (resp. $ \{ u(t) \}_{t \leq 0} \subset W^{cu}_{loc}(M) $) of equation ($ \spadesuit $) with $ u(0) = z_0 $.

		\item (Exponential tracking) If a mild solution $ \{u(t)\}_{t \geq 0} $ (resp. $ \{u(t)\}_{t \leq 0} $) of equation ($ \spadesuit $) always `stays' in $ \mathbb{B}_{r}(M) $, then it must belong to $ W^{cs}_{loc}(M) $ (resp. $ W^{cu}_{loc}(M) $), and there is certain $ \omega \in M $ such that $ |u(t) - t\omega| \to 0 $ (resp. $ |u(-t) - (-t)(\omega)| \to 0 $) exponentially as $ t \to \infty $.

		\item (Strong (un)stable foliations) $ W^{c\kappa}_{loc}(M) $ is foliated by $ \mathcal{W}^{\kappa\kappa} $ with leaves $ W^{\kappa\kappa}(\omega) $, $ \omega \in M $, $ \kappa = s, u $. Each leaf $ W^{\kappa\kappa}(\omega) $ is a Lipschitz immersed submanifold of $ X $. In fact, $ \mathcal{W}^{ss} $, $ \mathcal{W}^{uu} $ are H\"older bundles over $ M $.
		The foliations $ \mathcal{W}^{\kappa\kappa} $ are invariant with respect to equation ($ \spadesuit $), i.e. if $ z \in {W}^{ss}(\omega) \cap \mathbb{B}_{r'}(M) $, then there is a mild solution $ \{ u(t) \}_{t \geq 0} $ (resp. $ \{ u(t) \}_{t \leq 0} $) of equation ($ \spadesuit $) with $ z(0) = z $ satisfying $ u(t) \in {W}^{ss}(t\omega) $ for all $ t \geq 0 $ (resp. $ t \leq 0 $).

		\item (Smoothness) (i) Under the smooth condition (assumption (B4)), $ W^{cs}_{loc}(M) $, $ W^{cu}_{loc}(M) $, $ M $, and $ W^{ss}(\omega) $, $ W^{uu}(\omega) $, $ \omega \in \Sigma $, are all $ C^1 $ immersed submanifolds. So particularly, the two immersed submanifolds $ W^{cs}_{loc}(\Sigma) $, $ W^{cu}_{loc}(\Sigma) $ are transverse. Moreover, under higher smooth condition and spectral gap condition, these immersed submanifolds would be higher smooth.

		\noindent (ii) Under more restrictive smooth conditions and center bunching conditions (see \autoref{cor:tri} \eqref{regF}), $ \mathcal{W}^{ss} $, $ \mathcal{W}^{uu} $ are $ C^1 $ (in some cases even $ C^{1,\zeta} $) foliations.

		\item (Persistence) The above results are persistent under small $ C^1 $ perturbation of equation ($ \spadesuit $). Moreover, there is a true center manifold $ \widetilde{M} $ which is $ C^1 $ immersed in $ \mathbb{B}_{r}(M) $, homeomorphic (in fact $ C^1 $ diffeomorphic) to $ M $ and invariant with respect to the perturbed equation of equation ($ \spadesuit $) (i.e. $ g $ is replaced by $ \widetilde{g} $ in ($ \spadesuit $) with $ |\widetilde{g} - g|_{ C^1(\mathbb{B}_r(M)) } $ being small when $ r $ is small); also $ W^{cs}_{loc}(\widetilde{M}) \cap W^{cu}_{loc}(\widetilde{M}) = \widetilde{M} $, and $ \widetilde{M} \to M $, $ T\widetilde{M} \to TM $ as $ |\widetilde{g} - g|_{ C^1(\mathbb{B}_r(M)) } $ and the amplitude of $ \widetilde{g}|_{\mathbb{B}_{r}(M)} $ (with $ r $) approach $ 0 $. Here $ \widetilde{g} $ can be some `large' perturbations (see \autoref{rmk:largeP}).
	\end{enumerate}
\end{thmA}

See \autoref{generalB} and \autoref{normalH} for precise statements and more general results.
In \cite{LW97, PS01}, the authors considered the above corresponding results for the special (type $ \bullet $b) of PDEs with $ M $ being $ C^2 $ compact embedding submanifold. In a series of papers \cite{BLZ98, BLZ99, BLZ00, BLZ08}, the authors also obtained the theory of the normal hyperbolicity for abstract infinite-dimensional dynamical systems with $ M $ unnecessarily being compact or embedding. Our setting for the submanifold $ M $ (see \autoref{normalH}) is essentially the same as \cite{BLZ08} where the only difference is that $ M $ is not assumed to be $ C^1 $. If we further assume $ g \in \lip(\mathbb{B}_{r}(M), Z_{-1}) $ and equation ($ \spadesuit $) is well-posed, then one can apply the results in \cite{BLZ08} to obtain \autoref{B} as well. For this case, the almost uniform Lipschitz condition on the semiflow $ t: M \to M $ (see Settings B (BII) in \autoref{generalB}) can be removed which was implied by the Lipschitz continuity of $ g $; but also note that $ t $ being $ C^0 $ in the immersed topology of $ M $ is essential. The smoothness of strong (un-)stable foliations was not discussed in \cite{BLZ08}, which is almost the consequence of \cite{Che18a}. However, it is obvious that the results in \cite{BLZ08} can not be applied to the ill-posed differential equation ($ \spadesuit $) when $ A $ is a generator of a $ C_0 $ bi-semigroup since this equation in general does not generate a semiflow. \autoref{B} as well as the results in \autoref{generalB} and \autoref{normalH} are the first time to address the problem of the existence and persistence of the normally hyperbolic invariant manifolds for ill-posed differential equations. Also, our results in \autoref{normalH} can be applied to the non-Lipschitz and non-smooth dynamical systems. New ideas and techniques should be developed to tackle the difficulties arising in our general settings, although some basic methods are due to \cite{HPS77, Fen72, BLZ08}; for detailed proofs of the results in \autoref{normalH}, see \cite{Che18b}.

This is a paper that aims to give an application of our abstract results in \cite{Che18a, Che18b} to both well-posed and ill-posed abstract differential equations like ($ \clubsuit $) or ($ \spadesuit $), but not to give a detailed analysis of some concrete differential equations. Also, it is not a purpose of this paper to develop a unified spectral theory for the well-posed and ill-posed linear differential equations.

\subsection{structure of this paper}

\autoref{correspondence} contains some basic notions we will use throughout this paper. The relation between the dichotomy and (A) (B) condition for some classes of differential equation is given in \autoref{diffCocycle}. A quick review of the main results about invariant manifold theory in \cite{Che18a, Che18b} with an application to different equations is contained in \autoref{ManifoldApp}.

\vspace{.5em}
\noindent{\textbf{Guide to Notation:}}
\begin{enumerate}[$ \bullet $]
	\item $\lip f$: the Lipschitz constant of $f$.

	\item $\mathbb{R}_{+} \triangleq \{x\in \mathbb{R}: x \geq 0\}$.

	\item $X(r) \triangleq \mathbb{B}_r = \{ x \in X: |x| <  r\}$, if $X$ is a Banach space.

	\item For a correspondence $ H: X \to Y $ (defined in \autoref{basicCorr}),
	\begin{asparaitem}
		\item $ H(x) \triangleq \{ y: \exists (x, y) \in \graph H \} $,
		\item $A \subset H^{-1}(B)$, if $ \forall x \in A $, $ \exists y \in B $ such that $ y \in H(x) $,
		\item $ \graph H $, the graph of the correspondence.
		\item $ H^{-1}: Y \to X $, the inversion of $ H $ defined by $ (y,x) \in \graph H^{-1} \Leftrightarrow (x, y) \in \graph H $.
	\end{asparaitem}

	\item $f(A) \triangleq \{ f(x): x \in A \}$, if $f$ is a map.

	\item $\graph f \triangleq \{ (x, f(x)): x \in X \}$, the graph of the map $f: X \rightarrow Y$.

	\item $ Df_m(x) = D_x f_m(x) $: the derivative of $ f_m(x) $ with respect to $ x $; $ D_1F_m(x,y) = D_x F_m(x,y) $, $ D_2F_m(x,y) = D_y F_m(x,y) $: the derivatives of $ F_m(x,y) $ with respect to $ x $, $ y $, respectively.

	\item $\widetilde{d} (A, z) \triangleq \sup_{\tilde{z} \in A} d(\tilde{z}, z) $, if $A$ is a subset of a metric space, defined in \autopageref{notationDD}.

	\item $ a_n \lesssim b_n $, $ n \to \infty $ ($ a_n \geq 0, b_n > 0 $) means that $ \sup_{n\geq 0}b^{-1}_n a_n <\infty $, defined in \autopageref{notation1}.

	\item $ (T * g)(t) \triangleq \int_{0}^{t} T(t-s)g(s) ~\mathrm{ d } s $: the convolution of $ T $ and $ g $ (see \eqref{convol}).

	\item $ (S\Diamond f)(t) \triangleq \frac{\mathrm{d}}{\mathrm{d} t} \int_{0}^{t} S(t-s)f(s) ~\mathrm{d} s $ defined in \autoref{preHYccc}.
	
	\item $ (S_0 \Diamond f)(\omega)(t) = \frac{\mathrm{d}}{\mathrm{d} t} \int_{0}^{t} S_0(t-s, s\omega) f(s) ~\mathrm{d} s $ defined in \autoref{ss11}.

	\item $ |f|_{[0,t]} \triangleq \sup_{s \in [0,t]} |f(s)| $, if $ f \in C([0,t], X) $ defined in \autoref{preHYccc}.

	\item $ D(A) $: the domain of a linear operator $ A $.

	\item $ A_Y $: the part of linear operator $ A $ in $ Y $ (see \autoref{operator}).

	\item $ \mathcal{E}_1(t) $: defined in \autopageref{eee}.

\end{enumerate}

%% file: sect1.tex
\section{Correspondence with generating map and (A) (B) condition}\label{correspondence}

In this section, we list some notions in order to deal with the differential equations in Banach spaces for both \emph{ill-posed} and well-posed case.
All the mathematical materials appeared in this section are taken from \cite{Che18a}, where the readers can find more details in that paper.

\subsection{some notions about bundle}

$ (X, M, \pi_1) $ (or for short $ X $) is called a (set) bundle over $ M $ if $ \pi_1: M \to X $ is a surjection. Call $ X_{m} = \pi_1^{-1}(m) $, $ m \in M $, the fibers of $ X $, $ M $ the base space of $ X $ and $ \pi_1 $ the projection.
The elements of $ X $ are usually written as $ (m,x) $ where $ x \in X_{m} $, $ m \in M $.
If $ X $ and $ Y $ are bundles over $ M $, the Whitney sum $ X \times Y $ of $ X, Y $ is defined by
\[
X \times Y = \{ (m,x,y): x \in X_{m}, y \in Y_{m},m \in M \}.
\]

Let $ (X, M_1, \pi_1), (Y, M_2, \pi_2) $ be two bundles and $ u: M_1 \to M_2 $ a map. We say a map $ f: X \to Y $ is a bundle map over $ u $ if $ f(X_{m}) \subset Y_{u(m)} $ for all $ m \in M_1 $; in this case, we write $ f(m,x) = (u(m), f_{m}(x)) $ and call $ f_{m}: X_{m} \to Y_{u(m)} $ a fiber map of $ f $.

\subsection{correspondence with generating map} \label{basicCorr}

Let $X, Y$ be sets. $H: X \rightarrow Y$ is said to be a \textbf{correspondence} (see \cite{Cha08}), if there is a non-empty subset of $X \times Y$ called the graph of $ H $ and denoted by $\graph H$.
There are some standard operations between the correspondences.
\begin{enumerate}[$ \bullet $]
	\item (inversion) For a correspondence $ H: X \rightarrow Y $, define its inversion $ H^{-1} : Y \rightarrow X $ by $ (y,x) \in \graph H^{-1} $ if only if $ (x,y) \in \graph H $.
	\item (composition) For two correspondences $H_1: X \rightarrow Y$, $H_2: Y \rightarrow Z$, define $H_2 \circ H_1: X \rightarrow Z$ by 
	\[
	\graph H_2 \circ H_1 = \{ (x,z): \exists y \in Y, ~\text{such that}~ (x,y) \in \graph H_1, (y,z) \in \graph H_2 \}.
	\]
	\item (linear operation) Let $X, Y$ be vector spaces. For correspondences $H_1,~ H_2: X \rightarrow Y $, $H_1 - H_2: X \rightarrow Y$ is defined by 
	\[
	\graph (H_1 - H_2) = \{ (x,y): \exists (x, y_i) \in \graph H_i, ~\text{such that}~ y = y_1-y_2 \}.
	\]
	In particular, if $H: X \rightarrow Y$ is a correspondence, then $H_m \triangleq H(m+\cdot) - \widehat{m}: X \rightarrow Y$ means $\graph H_m = \{ (x, y - \widehat{m}) : \exists (x+m, y) \in \graph H \}$.
\end{enumerate}
The following notations for a correspondence $ H: X \to Y $ will be used frequently: for $ x \in X $, $ A \subset X $,
\[
H(x) \triangleq \{ y \in Y: \exists (x, y) \in \graph H \},~H(A) \triangleq \bigcup_{x \in A} H(x);
\]
allow $ H(x) = \emptyset $; if $ H(x) = \{y\} $, write $ H(x) = y $.
So by $ A \subset H^{-1}(B) $ we mean $ \forall x \in A $, $ \exists y \in B $ such that $ y \in H(x) $ (i.e. $ x \in H^{-1}(y) $). If $ X = Y $, we say $ A \subset X $ is \emph{invariant} under $ H $ if $ A \subset H^{-1}(A) $.
Evidently, $ x \mapsto H(x) $ can be regarded as a `multiple-valued map', but it is useless from our purpose as we only concern the description of $ \graph H $.

We say a correspondence $H: X_1 \times Y_1 \rightarrow X_2 \times Y_2$ has a \textbf{generating map} $(F,G)$, denoted by $H\sim(F,G)$, if there are maps $F: X_1 \times Y_2 \rightarrow X_2$, $G: X_1 \times Y_2 \rightarrow Y_1$, such that 
\[
(x_2,y_2) \in H(x_1,y_1) \Leftrightarrow y_1 = G(x_1,y_2), ~x_2 = F(x_1, y_2).
\]

Let $X, Y$ be sets. $H:\mathbb{R}_+ \times X \rightarrow X$ is called a \textbf{continuous (semi-)correspondence}, if
\begin{enumerate}[(a)]
	\item $\forall t \in \mathbb{R}_+$, $H(t): X \rightarrow X$ is a correspondence;
	\item $H(0) = \id$, $H(t+s) = H(t)\circ H(s)$, $\forall t, s \in \mathbb{R}_+$.
\end{enumerate}
Furthermore, a continuous semi-correspondence $H:\mathbb{R}_+ \times X \times Y \rightarrow X \times Y$ is said has a generating map $(F,G)$, if every $H(t) \sim (F_t, G_t)$. If $X,Y$ are topology spaces, and $(t, x, y) \mapsto F_{t}(x,y), ~(t, x, y) \mapsto G_t(x,y)$ are continuous, we say $H$ has a \emph{continuous generating map}. In analogy, if $X,Y$ are Banach spaces, or more generally, Banach manifolds, and $(t,x,y) \mapsto D^{i}F_t(x,y) $, $(t,x,y) \mapsto D^{i}G_t(x,y)$, $ i = 0, 1, \ldots,k $, are continuous, we say $H$ has a \emph{$ C^k $ smooth generating map}.

Let $(X, M, \pi_1), (Y, M, \pi_2)$ be bundles, and $t: M \rightarrow M,~ \omega \mapsto t\omega$ a semiflow. $H:\mathbb{R}_+ \times X \rightarrow X$ is called a \textbf{(semi-)cocycle correspondence over} $t$, if
\begin{enumerate}[(a)]
	\item $\forall t \in \mathbb{R}_+$, $H(t): X \rightarrow X$ is a bundle correspondence over $t$;
	\item $H(0, \omega) = \id$, $H(t+s, \omega) = H(t, s\omega)\circ H(s,\omega)$, $\forall t, s \in \mathbb{R}_+$, where $H(t,\omega) \triangleq H(t)(\omega, \cdot): X_{\omega} \rightarrow X_{t\omega}$ is a correspondence.
\end{enumerate}
In this case, $ H^{-1} $, the inversion of $ H $, means $ H^{-1}(t,\omega) \triangleq H(t,\omega)^{-1}: X_{t\omega} \to X_{\omega} $.
Furthermore, a cocycle correspondence $H:\mathbb{R}_+ \times X \times Y \rightarrow X \times Y$ is said has a \textbf{generating cocycle} $(F,G)$, if every $H(t, \omega) \sim (F_{t,\omega}, G_{t,\omega})$, where $F_{t,\omega}: X_{\omega} \times Y_{t\omega} \rightarrow X_{t\omega}, ~G_{t,\omega}: X_{\omega} \times Y_{t\omega} \rightarrow Y_{\omega}$ are maps.
See examples in \autoref{diffCocycle} (abstract differential equations) and \autoref{examples} (concrete differential equations).

Let $u: M \rightarrow N$ be a map. Suppose $H_m: X_m \times Y_m \rightarrow X_{u(m)} \times Y_{u(m)}$ is a correspondence for every $m \in M$. Using $H_m$, one can determine a correspondence $H: X \times Y \rightarrow X \times Y$, by $\graph H \triangleq \bigcup_{m \in M}(m, \graph H_m$), i.e. $(u(m), x_{u(m)}, y_{u(m)}) \in H(m, x_m, y_m) \Leftrightarrow (x_{u(m)}, y_{u(m)}) \in H_m(x_m, y_m)$. We call $H$ a \textbf{bundle correspondence over a map $u$}.
If $H_m$ has a generating map $(F_m, G_m)$ for every $m \in M$, where $F_m: X_m \times Y_{u(m)} \rightarrow X_{u(m)}, ~G_m: X_m \times Y_{u(m)} \rightarrow Y_m$ are maps, then we say $H$ has a \textbf{generating bundle map} $(F,G)$ over $u$, which is denoted by $H \sim (F,G)$.

\subsection{dual correspondence}\label{dual}

Let $ H: X_1 \times Y_1 \to X_2 \times Y_2 $ be a correspondence with a generating map $ (F, G) $. The \textbf{dual correspondence} $ \widetilde{H} $ of $ H $ is defined by the following. Set $ \widetilde{X}_1 = Y_2 $, $ \widetilde{X}_2 = Y_1 $, $ \widetilde{Y}_1 = X_2 $, $ \widetilde{Y}_2 = X_1 $ and
\[
\widetilde{F}(\widetilde{x}_1, \widetilde{y}_2) = G(\widetilde{y}_2, \widetilde{x}_1), ~ \widetilde{G}(\widetilde{x}_1, \widetilde{y}_2) = F(\widetilde{y}_2, \widetilde{x}_1).
\]
Now $ \widetilde{H} \sim (\widetilde{F}, \widetilde{G}) : \widetilde{X}_1 \times \widetilde{Y}_1 \to \widetilde{X}_2 \times \widetilde{Y}_2 $, i.e.
\[
(\widetilde{x}_2,\widetilde{y}_2) \in \widetilde{H}(\widetilde{x}_1,\widetilde{y}_1) \Leftrightarrow \widetilde{y}_1 = \widetilde{G}(\widetilde{x}_1,\widetilde{y}_2), ~\widetilde{x}_2 = \widetilde{F}(\widetilde{x}_1, \widetilde{y}_2).
\]
One can similarly define the \textbf{dual bundle correspondence} $ \widetilde{H} $ of bundle correspondence $ H $ over $ u $ if $ u $ is invertible; $ \widetilde{H} $ now is over $ u^{-1} $. Also, the \textbf{dual cocycle correspondence} of cocycle correspondence over $ t $ can be defined analogously if $ t $ is a flow.

$ \widetilde{H} $ and $ H $ have some duality in the sense that $ \widetilde{H} $ can reflect some properties of  `$ H^{-1} $'. For instance, if $ H $ satisfies (A)$(\alpha; \alpha', \lambda_u)$ condition (see \autoref{AB} below), then $ \widetilde{H} $ satisfies (B)$(\alpha; \alpha', \lambda_u)$ condition. So one can get the `unstable results' of $ H $ through the `stable results' of $ \widetilde{H} $. This approach, which we learned from \cite{Cha08}, is important when we deal with invariant manifold theory for non-invertible dynamics.

\subsection{(A) (B) condition} \label{defAB}

Let $X_i, Y_i, ~ i=1,2$ be metric spaces. \emph{For the convenience of writing, we write the metrics $d(x,y) \triangleq |x-y|$}.

\begin{defi}\label{AB}
	We say a correspondence $H: X_1 \times Y_1 \rightarrow X_2 \times Y_2$ satisfies \textbf{(A) (B) condition, or (A)$(\bm{\alpha; \alpha', \lambda_u})$ (B)$(\bm{\beta; \beta', \lambda_s})$ condition}, if the following conditions hold.
	$\forall ~(x_1, y_1) \times (x_2, y_2),~ (x'_1, y'_1) \times (x'_2, y'_2) \in \graph H$,
	\begin{enumerate}[(A)]
		\item (A1) if $|x_1 - x'_1| \leq \alpha |y_1 - y'_1|$, then $|x_2 - x'_2| \leq \alpha' |y_2 - y'_2|$;

		\noindent(A2) if $|x_1 - x'_1| \leq \alpha |y_1 - y'_1|$, then $ |y_1 - y'_1| \leq \lambda_u |y_2 - y'_2|$;

		\item (B1) if $|y_2 - y'_2| \leq \beta |x_2 - x'_2| $, then $ |y_1 - y'_1| \leq \beta' |x_1 - x'_1|$;

		\noindent(B2) if $|y_2 - y'_2| \leq \beta |x_2 - x'_2| $, then $ |x_2 - x'_2| \leq \lambda_s |x_1 - x'_1|$.
	\end{enumerate}
	If $\alpha = \alpha', ~\beta = \beta'$, we also use notation \textbf{(A)$(\bm{\alpha, \lambda_u})$ (B)$(\bm{\beta, \lambda_s})$ condition}.
\end{defi}

In particular, if $H \sim (F,G)$, then the maps $F,G$ satisfy the following Lipschitz conditions.
\begin{enumerate}[(A$'$)]
	\item (A1$'$) $\sup_{x}\lip F(x,\cdot) \leq \alpha'$,  (A2$'$) $\sup_{x}\lip G(x,\cdot) \leq \lambda_u$.
	\item (B1$'$) $\sup_{y}\lip G(\cdot,y) \leq \beta'$,  (B2$'$) $\sup_{y}\lip F(\cdot,y) \leq \lambda_s$.
\end{enumerate}
If $F,G$ satisfy the above Lipschitz conditions, then we say $H$ satisfies \textbf{(A$\bm{'}$)($\bm{\alpha', \lambda_u}$) (B$\bm{'}$)($\bm{\beta', \lambda_s}$) condition}, or \textbf{(A$\bm{'}$) (B$\bm{'}$) condition}. Similarly, we can define (A$'$) (B) condition, or (A) (B$'$) condition; or (A) condition, (A$'$) condition, etc, if $H$ only satisfies (A), (A$'$), etc, respectively.

Our definition of (A)(B) condition is associated with the hyperbolicity. Roughly, the numbers $ \lambda_{s}, \lambda_{u} $ are related with the \emph{Lyapunov numbers}, the spaces $ X_i, Y_i $, $ i = 1,2 $, play a similar role of \emph{spectral spaces}, and the numbers $ \alpha, \alpha' $, $ \beta, \beta' $ describe how the spaces $ X_i, Y_i $ ($ i = 1,2 $) are approximately invariant.
It might be intuitive to see this in \autoref{thm:biAB} or \autoref{thm:gencocycleAB}, a relation between the (exponential) dichotomy and (A) (B) condition, which is a main issue of this paper addressed in \autoref{diffCocycle}. We refer the readers to see \cite[Section 3.2 and 3.3]{Che18a} as well as \cite{MS88, Zel14} for more results about the verification of (A) (B) condition.

%% file: sect2.tex
\section{Relation between dichotomy and (A) (B) condition} \label{diffCocycle}

In this section, we will give some classes of abstract differential equations that generate cocycle correspondences with generating cocycles, including both well-posed and ill-posed case. For some concrete examples, see \autoref{examples}. We focus on the relationship between the dichotomy and (A) (B) condition, which is important for us to apply our results in \autoref{continuous case} and \autoref{normalH} (as well as \cite{Che18a, Che18}) to some differential equations.

The dichotomy or more precisely the exponential dichotomy of differential equations is related with spectral theory, which is well developed for $ C_0 $ cocycle (especially the well-posed linear differential equations); see e.g. \cite{CL99} and the literatures therein for a comprehensive study, as well as \cite{LZ17,LL10} for further developments.
It is worth pointing out that the existing spectral theory for the ill-posed differential equations is not so well developed even in `equilibrium' case, not mention that in general `cocycle' case. No attempt has been made here to develop such theory. We refer to \cite{LP08} (and also \cite{SS01}) and the references therein for some general results in this direction and detailed spectral analysis of some particular concrete differential equations.

Throughout this section, we make the following settings.

\begin{enumerate}[$ \bullet $]
	\item  $ M $ is a Hausdorff topology space. $ t: M \to M $ is a continuous semiflow, i.e. $ \mathbb{R}_+ \times M \to M $, $ (t, \omega) \mapsto t\omega $ is continuous and $ 0\omega = \omega $, $ (t+s)\omega = t(s\omega) $ for all $ t, s \in \mathbb{R}_+ $, $ \omega \in M $.
	\item Let $ Z $ be a Banach space. Assume $ \mathcal{C}(\omega): D(\mathcal{C}(\omega)) \subset Z \to Z $, $ \omega \in M $, are closed linear operators.
\end{enumerate}

In the following, we will consider the following two differential equations in different settings:
\begin{equation}\label{equ:linear}
\dot{z}(t) = \mathcal{C}(t\omega)z(t),
\end{equation}
and
\begin{equation}\label{equ:main}
\dot{z}(t) = \mathcal{C}(t\omega)z(t) + f(t\omega)z(t),
\end{equation}
where $ f: M \times Z \to Z $ is continuous and for every $ \omega \in M $, $ \sup_{t\geq 0 } \lip f(t\omega)(\cdot) = \varepsilon(\omega) < \infty $.

In \autoref{bi-semigroup} and \ref{HYccc}, we concentrate on the following the special form of $ \{\mathcal{C}(\omega)\} $:
\[\label{equ:cc}\tag{$ \odot $}
\mathcal{C}(\omega) = A + L(\omega), ~\omega \in M,
\]
where $ A: D(A) \subset Z \to Z $ is a closed linear operator and $ L: M \to L(\overline{D(A)}, Z) $ satisfies the following assumption. Note that $ D(\mathcal{C}(\omega)) = D(A) $.

\begin{enumerate}[(D1)] \label{d1L}
	\item  Suppose $ L $ is strongly continuous, i.e. $ (\omega, z) \mapsto L(\omega)z $ is continuous. Moreover, assume that (i) for every $ \omega \in M $, $ \sup_{t \geq 0}|L(t\omega)| = \tau(\omega) < \infty $; and (ii) $ \omega \mapsto \tau(\omega) $ is locally bounded.
\end{enumerate}

\begin{defi}[mild solution]\label{def:mild}
	Let $ \{\mathcal{C}(\omega)\} $ be as \eqref{equ:cc}.
	A function $ u \in C([a,b], Z) $ is called a \textbf{(mild) solution} of \eqref{equ:main} if it satisfies (i) $ \int_{a}^{t} u(s) ~\mathrm{ d } s \in D(A) $ for all $ t \in [a,b] $ and (ii) the following
	\[
	u(t) = u(a) + A \int_{a}^{t} u(s) ~\mathrm{ d } s + \int_{a}^{t} (L(s\omega)u(s) + f(s\omega)u(s)) ~\mathrm{ d } s, ~t \in [a,b].
	\]
	Similarly, a function $ u \in C([a,b), Z) $ (resp. $ u \in C((a,b], Z) $) is called a (mild) solution of \eqref{equ:main} if for any $ r \in (a, b) $, $ u|_{[a, r]} $ (resp. $ u|_{[r, b]} $) is a mild solution of \eqref{equ:main}.
\end{defi}
\begin{defi}
	Let $ \{\mathcal{C}(\omega)\} $ be as \eqref{equ:cc}. We say equation \eqref{equ:main} is \textbf{well-posed}, if for every $ \omega \in M $ and every $ x \in \overline{D(\mathcal{C}(\omega))} = \overline{D(A)} $, equation \eqref{equ:main} has a mild solution $ u \in C([0,\chi(\omega,x)], Z) $ with $ u(0) = x $, where $ \chi(\omega,x) > 0 $ depending on choice of $ \omega, x $; otherwise, we say equation \eqref{equ:main} is \textbf{ill-posed}.
\end{defi}
That the differential equation is well-posed or ill-posed depends on how we define the solution of the equation. In this paper, we only focus on the \emph{mild} solutions. 

We will consider the three types of equations \eqref{equ:main}, i.e. (type I) $ \sim $ (type III) listed in \autoref{overview}, which are important for applications.
We will show that equation \eqref{equ:main} gives a cocycle correspondence $ H $ with generating cocycle through the mild solutions under additional mild conditions. The cocycle correspondence $ H $ will satisfy (A) (B) condition, roughly speaking, if some uniform dichotomy of \eqref{equ:linear} is assumed and the Lipschitz constants of $ f(\omega)(\cdot) $ (i.e. $ \varepsilon(\cdot) $) are `small'; see \autoref{thm:biAB}, \autoref{thm:gencocycleAB} and \autoref{them:spec}. So our results in \autoref{continuous case} and \autoref{normalH} can be applied to give some dynamical results of the equation \eqref{equ:main}, as well as the results in \cite{Che18a, Che18}.

See \cite{EN00,ABHN11,vdMee08} for some basic backgrounds from operator semigroup theory. In \autoref{operator}, we give some basic definitions and notations taken from operator semigroup theory for readers' convenience. We deal with (type I) in \autoref{bi-semigroup} and (type II) in \autoref{HYccc}. A light more general case (type III) is also discussed in \autoref{genccc}.

\subsection{$ C_0 $ (bi-)semigroup case}\label{bi-semigroup}

\subsubsection{an illustration: autonomous system case}\label{illustration}

Let $ X, Y $ be two Banach spaces. Assume $ T(t): X \to X $, $ S(-t): Y \to Y $, $ t \geq 0 $, are $ C_0 $ semigroups, and
\begin{equation}\label{tsspectral}
|T(t)| \leq e^{\mu_s t}, ~|S(-t)| \leq e^{-\mu_u t}, ~ t \geq 0.
\end{equation}

\begin{rmk}
	In general, for a $ C_0 $ semigroup $ T $, it must have $ |T(t)| \leq Ce^{\mu t} $ for some $ C \geq 1 $, $ \mu $. The constant $ C $ might not equal $ 1 $. But we can always choose an equivalent norm $ \| \cdot \| $ such that $ \|T(t)\| \leq e^{\mu t} $; see \cite{EN00}. That $ C = 1 $ is a key in our argument.
\end{rmk}

Consider
\begin{equation}\label{*1}
\begin{cases}
	x(t) = T(t - t_1)x_1 + \int_{t_1}^{t} T(t-s)B_1(x(s),y(s)) ~\mathrm{d} s, \\
	y(t) = S(t - t_2)y_2 - \int_{t}^{t_2} S(t-s)B_2(x(s),y(s)) ~\mathrm{d} s,
\end{cases}
t_1 \leq t \leq t_2,
\end{equation}
where $ B_1 : X \times Y \to X $, $ B_2 : X \times Y \to Y $ are Lipschitz with $ \lip B_i \leq \varepsilon $.

Denote by $ A, -B $ the generators of $ T, S $, respectively. Then \eqref{*1} can be considered as the mild solutions of the following differential equation,
\begin{equation}\label{bi-diff0}
\begin{cases}
\dot{x} = A x + B_1 (x, y), \\
\dot{y} = B y + B_2 (x, y),
\end{cases}
\end{equation}
i.e. $ z(\cdot) = (x(\cdot),y(\cdot)) $ satisfies $ \int_{t_1}^{t} x(s) \mathrm{d} s \in D(A) $, $ \int_{t_1}^{t} y(s) \mathrm{d} s \in D(B) $, and
\begin{equation}\label{bi-mild0}
\begin{cases}
x(t) = x_1 + A \int_{t_1}^{t} x(s) ~\mathrm{d} s + \int_{t_1}^{t} B_1(x(s), y(s)) ~\mathrm{d} s ,\\
y(t) = y_1 + B \int_{t_1}^{t} y(s) ~\mathrm{d} s + \int_{t_1}^{t} B_2(x(s), y(s)) ~\mathrm{d} s,
\end{cases}
\end{equation}
for all $ t_1 \leq t \leq t_2 $, where $ (x_1, y_1) \in X \times Y $, $ y_1 = y(t_1) $. As usual, \eqref{*1} is called a \textbf{variant of constant formula} of \eqref{bi-diff0} (or \eqref{bi-mild0}).

Set
\[
C = \left(\begin{matrix}
A	&  \\
& B
\end{matrix} \right) : D(A) \times D(B) \subset X \times Y \to X \times Y.
\]
$ C $ is called a generator of a \textbf{$ C_0 $ bi-semigroup}. See \cite{vdMee08} for more characterizations and \autoref{operator}.

Note that the existence of the solutions of \eqref{*1} is a standard application of Banach Fixed Point Theorem; the detail is omitted here (see also \cite{ElB12}). Any solution of \eqref{*1} satisfies \eqref{bi-mild0}, and if \eqref{bi-mild0} exists a solution with $ (x(t_1), y(t_1)) = (x_1, y_1) $, then it must satisfy \eqref{*1} with $ y(t_2) = y_2 $; this is a standard consequence of linear $ C_0 $ semigroup theory (see e.g. \cite{EN00,ABHN11} for details) by setting $ f_i(s) = B_i(x(s),y(s)) $, $ i = 1,2 $. Using the parameter-dependent fixed point theorem (see e.g. \cite[Appendix D.1]{Che18a}), one can easily show the continuous and smooth dependence of the solution of \eqref{*1} about $ (x_1, y_2) $ when $ B_i $, $ i = 1,2 $, have higher regularity. We emphasis that \eqref{bi-mild0}, unlike the classical case, might not have a solution for $ (x(t_1), y(t_1)) = (x_1, y_1) $, i.e. \eqref{bi-diff0} is ill-posed. In contrast, \eqref{*1} always has a (unique) solution for $ (x(t_1), y(t_2)) = (x_1, y_2) $.

Define a correspondence $ H(s) : X \times Y \to X \times Y $ as follows. Let $ t_1 = 0, t_2 = s $. $ (x_2, y_2) \in H(s)(x_1, y_1) $ if and only if there is a continuous $ (x(t), y(t)) $, $ 0 \leq t \leq s $, satisfying \eqref{*1} with $ (x(t_i), y(t_i)) = (x_i, y_i) $, $ i = 1,2 $. $ H(s) $ has a natural generating map $ (F_{s}, G_{s}) $, which is defined by $ F_{s}(x_1, y_2) = x(s) $, $ G_{s}(x_1, y_2) = y(0) $, where $ (x(t), y(t)) $ satisfies \eqref{*1} with $ x(0) = x_1, y(s) = y_2 $.
By verifying directly, we have $ H(t+s) = H(t) \circ H(s) $, i.e. $ H(\cdot) $ is a \emph{continuous correspondence} defined in \autoref{basicCorr}. Note that in general, $ H $ is not a flow or semiflow. In \cite{ElB12}, $ H $ was also called the \emph{dichotomous flow} induced by \eqref{*1} (or \eqref{bi-diff0} \eqref{bi-mild0}).

\begin{lem}\label{lem:simpleAB}
	Let $ T(\cdot), S(-\cdot) $ be two $ C_0 $ semigroups satisfying \eqref{tsspectral}. Assume $ \lip B_i \leq \varepsilon $, where $ X \times Y $ equips with the max norm defined by $ |(x,y)| = \max\{ |x|, |y| \} $, $ (x, y) \in X \times Y $. Let $ H $ be the continuous correspondence induced by \eqref{*1}. Assume $ \mu_u - \mu_s - 2\varepsilon > 0 $. Take $ \alpha, \beta $ such that $ \frac{\varepsilon}{\mu_u - \mu_s - \varepsilon} \leq \alpha, \beta < 1 $, and $ \lambda_u = e^{-\mu_u + \varepsilon} $, $ \lambda_s = e^{\mu_s + \varepsilon} $. Then $ H(t) $ satisfies \textnormal{(A)}$(\alpha, \lambda^t_u)$ \textnormal{(B)}$(\beta, \lambda^t_s)$ condition.

	In fact, if $ \alpha, \beta \in (\frac{\varepsilon}{\mu_{u} - \mu_s - \varepsilon}, 1) $ and $ t \geq \epsilon_1 > 0 $, then $ H(t) $ satisfies (A)($ \alpha; k_{\alpha}\alpha, \lambda^t_{u} $) (B)($ \beta; k_{\beta}\beta, \lambda^t_{s} $), where
	\[
	k_{h} = \frac{(\mu_{u} - \mu_s - \varepsilon - \frac{\varepsilon}{h})e^{-(\mu_{u} - \mu_s - \varepsilon)\epsilon_1} + \frac{\varepsilon}{h}}{\mu_{u} - \mu_s - \varepsilon} < 1, ~ h = \alpha,\beta.
	\]
\end{lem}

\begin{rmk}\label{rmk:sharp}
	\begin{enumerate}[(a)]
		\item The condition $ \mu_u - \mu_s - 2\varepsilon > 0 $ in some sense is sharp, which has been obtained independently in \cite{Mcc91} and \cite{Rom93}. See also \cite{Zel14} and the references therein more details. The proof given here is quite different from previous literatures we list.
		\item Note that the Lipschitz constants of $ B_i $ are computed with respect to the max norm of $ X \times Y $.
		If we employ the $ p $-norm in $ X \times Y $ ($ 1 \leq p < \infty $), i.e. $ |(x, y)|_{p} = \{ |x|^{p} + |y|^{p} \}^{1/p} $, we have another estimate (which in some cases is useful). Assume $ \mu_u - \mu_s - 4\varepsilon > 0 $. Take $ \alpha, \beta $ such that $ \Delta_1 \leq \alpha, \beta < 1 $, where $ \Delta_1 \triangleq \frac{2\varepsilon}{\mu_u - \mu_s - 2\varepsilon + \sqrt{\Delta}} $, $ \Delta = (2\varepsilon - (\mu_u - \mu_s))^2 - 4 \varepsilon^2 > 0 $. Then $ H(t) $ also satisfies \textnormal{(A)}$(\alpha, \lambda^t_u)$ \textnormal{(B)}$(\beta, \lambda^t_s)$ condition and $ \alpha\beta<1 $, $ \lambda_s \lambda_u < 1 $. The proof is essentially the same as using the max norm, so we leave it to readers.
		\item There is a special case for $ T, S $. Let $ \widehat{T}: X \times Y \to X \times Y $ be a $ C_0 $ semigroup. Suppose $ \widehat{T}(t) X \subset X $, $ \widehat{T}(t) Y \subset Y $, for $ t \geq 0 $, and $ \widehat{T}|_{Y} $ is a $ C_0 $ group. Now take $ T(t) = \widehat{T}(t)|_{X} $, $ S(-t) = (\widehat{T}(t)|_{Y})^{-1} $. For this case, the result is more or less classical. See also \cite[Lamma 4]{LYZ13} for essentially the same result where the estimate thereof is not optimal.
	\end{enumerate}
\end{rmk}

\begin{rmk}\label{rmk:detail}
	If we distinguish different Lipschitz constants of $ \lip B_1 $, $ \lip B_2 $, then the result can be a little bit more detailed. Let $ \lip B_1 \leq \varepsilon_{s} $, $ \lip B_2 \leq \varepsilon_{u} $. Assume $ \mu_u - \mu_s - \varepsilon_{s} - \varepsilon_{u} > 0 $. Then we can take $ \alpha \in [\frac{\varepsilon_s}{\mu_u - \mu_s - \varepsilon_{u}} , 1) $, $ \beta \in [\frac{\varepsilon_u}{\mu_u - \mu_s - \varepsilon_{s}} , 1) $, and $ \lambda_u = e^{-\mu_u + \varepsilon_{u}} $, $ \lambda_s = e^{\mu_s + \varepsilon_{s}} $; particularly if $ \varepsilon_{s} \to 0 $, then we can take $ \alpha \to 0 $. If $ \alpha \in (\frac{\varepsilon_s}{\mu_u - \mu_s - \varepsilon_{u}} , 1) $, $ \beta \in (\frac{\varepsilon_u}{\mu_u - \mu_s - \varepsilon_{s}} , 1) $ and $ t \geq \epsilon_1 > 0 $, then we can take
	\begin{gather*}
	k_{\alpha} = \frac{(\mu_{u} - \mu_s - \varepsilon_u - \frac{\varepsilon_s}{\alpha})e^{-(\mu_{u} - \mu_s - \varepsilon_u)\epsilon_1} + \frac{\varepsilon_s}{\alpha}}{\mu_{u} - \mu_s - \varepsilon_{u}} < 1,\\
	k_{\beta} = \frac{(\mu_{u} - \mu_s - \varepsilon_s - \frac{\varepsilon_u}{\beta})e^{-(\mu_{u} - \mu_s - \varepsilon_s)\epsilon_1} + \frac{\varepsilon_u}{\beta}}{\mu_{u} - \mu_s - \varepsilon_{s}} < 1.
	\end{gather*}
\end{rmk}

\begin{proof}[Proof of \autoref{lem:simpleAB}]
	Let $ (x(t),y(t)) $, $ (x'(t),y'(t)) $, $ t_1 \leq t \leq t_2 $, satisfy \eqref{*1} with $ (x_1,y_2) $ being equal to $ (x(t_1),y(t_2)) $, $ (x'(t_1),y'(t_2)) $, respectively. Set $ \hat{x}(t) = x(t) - x'(t) $, $ \hat{y}(t) = y(t) - y'(t) $. It suffices to show if $ |\hat{x}(t_1)| \leq \alpha |\hat{y}(t_1)| $, then $ |\hat{x}(t_2)| \leq \alpha |\hat{y}(t_2)| $  and $ |\hat{y}(t_1)| \leq \lambda^{t_2 - t_1}_u |\hat{y}(t_2)| $. (B) condition can be proved similarly.

	By \eqref{*1}, \eqref{tsspectral}, and $ \lip B_i \leq \varepsilon $, we have
	\begin{gather*}
	|\hat{x}(t)| \leq e^{\mu_s (t - t_1)}|\hat{x}(t_1)| + \varepsilon \int_{t_1}^{t} e^{\mu_s (t -s)} |(\hat{x}(s), \hat{y}(s))| ~\mathrm{d} s, \\
	|\hat{y}(t)| \leq  e^{\mu_u (t - t_2)}|\hat{y}(t_2)| + \varepsilon \int_{t}^{t_2} e^{\mu_u (t -s)} |(\hat{x}(s), \hat{y}(s))| ~\mathrm{d} s,
	\end{gather*}
	for all $ t_1 \leq t \leq t_2 $.
	\begin{slem} \label{slem000}
		For any $ \alpha \in (\frac{\varepsilon}{\mu_u - \mu_s - \varepsilon}, 1) $, if $ |\hat{x}(t)| \leq \alpha |\hat{y}(t)| $ for $ t \in [t_1, t'_2] $, $ t_1 < t'_2 \leq t_2 $, then $ |\hat{x}(t'_2)| <  \alpha |\hat{y}(t'_2)| $ and $ |\hat{y}(t_1)| \leq \lambda_u^{t'_2 - t_1} |\hat{y}(t'_2)| $.
	\end{slem}
	\begin{proof}
		Since $ |\hat{x}(t)| \leq \alpha |\hat{y}(t)| $ for $ t \in [t_1, t'_2] $, and $ \alpha < 1 $, we have
		\begin{align*}
		|\hat{x}(t)| & \leq e^{\mu_s (t - t_1)}|\hat{x}(t_1)| + \varepsilon \int_{t_1}^{t} e^{\mu_s (t -s)} |\hat{y}(s)| ~\mathrm{d} s, \\
		|\hat{y}(t)| & \leq e^{\mu_u (t - t'_2)}|\hat{y}(t'_2)| + \varepsilon \int_{t}^{t'_2} e^{\mu_u (t -s)} |\hat{y}(s)| ~\mathrm{d} s.
		\end{align*}
		By Gronwall inequality, $ |\hat{y}(t)| \leq e^{ (-\mu_u + \varepsilon)(t'_2 - t) } |\hat{y}(t'_2)| = \lambda_u^{t'_2 - t} |\hat{y}(t'_2)| $. So
		\begin{align*}
		|\hat{x}(t'_2)| & \leq \alpha \left\{  e^{(\mu_s - \mu_u + \varepsilon)(t'_2 - t_1)} + \frac{\varepsilon}{\alpha} \int_{t_1}^{t'_2} e^{(\mu_s - \mu_u + \varepsilon)(t'_2 - s)} ~\mathrm{d} s \right\} |\hat{y}(t'_2)| \\
		& = \alpha \left\{ e^{(\mu_s - \mu_u + \varepsilon)(t'_2 - t_1)} + \frac{\varepsilon}{\alpha} \frac{1-e^{(\mu_s - \mu_u + \varepsilon)(t'_2 - t_1)}}{\mu_u - \mu_s - \varepsilon} \right\} |\hat{y}(t'_2)|.
		\end{align*}
		Since $ \mu_u - \mu_s - \varepsilon > \mu_u - \mu_s - 2\varepsilon > \mu_u - \mu_s - \varepsilon - \frac{\varepsilon}{\alpha} > 0 $ and $ \frac{\varepsilon}{\mu_u - \mu_s - \varepsilon} < \alpha $, we have
		\[
		e^{(\mu_s - \mu_u + \varepsilon)(t'_2 - t_1)} + \frac{\varepsilon}{\alpha} \frac{1-e^{(\mu_s - \mu_u + \varepsilon)(t'_2 - t_1)}}{\mu_u - \mu_s - \varepsilon} = \frac{ (\mu_u - \mu_s - \varepsilon - \frac{\varepsilon}{\alpha})e^{-(\mu_u - \mu_s - \varepsilon)(t'_2 - t_1)} + \frac{\varepsilon}{\alpha} }{\mu_u - \mu_s - \varepsilon} < 1,
		\]
		completing the proof.
	\end{proof}
	\begin{slem}\label{slem111}
		Let $ \alpha \in [\frac{\varepsilon}{\mu_u - \mu_s - \varepsilon}, 1) $. If $ |\hat{x}(t_1)| \leq \alpha |\hat{y}(t_1)| $, then $ |\hat{x}(t)| \leq \alpha |\hat{y}(t)| $ for all $ t > t_1 $.
	\end{slem}
	\begin{proof}
		Take any $ \alpha' $ such that $ \alpha < \alpha' < 1 $. We show $ |\hat{x}(t)| < \alpha' |\hat{y}(t)| $ for all $ t > t_1 $. Consider
		\[
		D = \{ t'_2 \in [t_1, t_2]: |\hat{x}(t)| \leq \alpha' |\hat{y}(t)|, \forall t \in [t_1, t'_2]  \}.
		\]
		Let $ t_0 = \sup D $. Note that $ t_0 \in D $ and $ t_0 > t_1 $ (since $ |\hat{x}(t_1)| \leq \alpha |\hat{y}(t_1)| < \alpha' |\hat{y}(t_1)| $).
		If $ t_0 < t_2 $, then $ |\hat{x}(t_0)| = \alpha' |\hat{y}(t_0)| $ and $ |\hat{x}(t)| \leq \alpha' |\hat{y}(t)| $ for all $ t_1 \leq t \leq t_0 $. By the above sublemma, we know $ |\hat{x}(t_0)| < \alpha' |\hat{y}(t_0)| $, which yields a contradiction. So $ t_0 = t_2 $, i.e.  $ |\hat{x}(t)| \leq \alpha' |\hat{y}(t)| $ for all $ t_1 \leq t \leq t_2 $. Finally, let $ \alpha' \to \alpha $, then $ |\hat{x}(t)| \leq \alpha |\hat{y}(t)| $.
	\end{proof}

	Now, combine the above two sublemmas to complete the proof of the first conclusion. For the second conclusion, this in fact has been proved in \autoref{slem000}.
\end{proof}

\subsubsection{uniform dichotomy on $ \mathbb{R}_+ $}\label{bundleCorres}

\begin{defi}[uniform dichotomy]\label{def:ud+}
	We say a $ C_0 $ cocycle correspondence $ H_1 $ (or $ \{H_1(t,\omega)\} $, i.e. $ H_1(t,\omega) \sim (T_1(t,\omega), S_1(-t,t\omega)): X_{\omega} \oplus Y_{\omega} \to X_{t\omega} \oplus Y_{t\omega} $) on $ M \times Z $ satisfies \textbf{uniform dichotomy on $ \mathbb{R}_+ $} if the following hold.
	\begin{enumerate}[(a)]
		\item Assume $ Z = X_{\omega} \oplus Y_{\omega} $, $ \omega \in M $, associated with projections $ P_{\omega} $, $ P^c_{\omega} = I - P_{\omega} $ such that $ R(P_{\omega}) = X_{\omega} $, $ R(P^c_{\omega}) = Y_{\omega} $. $ (\omega, z) \mapsto P_{\omega}z $ is continuous.
		Usually, we call $ X_{\omega}, Y_{\omega} $, $ \omega \in M $, the spectral spaces, $ P_{\omega}, P^{c}_{\omega} $, $ \omega \in M $, the spectral projections, and also $ \bigsqcup_{\omega \in M} X_{\omega}, \bigsqcup_{\omega \in M} Y_{\omega} $ , the spectral subbundles.

		\item \label{udd} There are two $ C_0 $ linear cocycles $ T_1 $, $ S_1 $ such that $ T_1(t,\omega): X_{\omega} \to X_{t\omega} $, $ S_1(-t,t\omega): Y_{t\omega} \to Y_{\omega} $, for all $ (t, \omega) \in \mathbb{R}_+ \times M $. 

		\item There is a constant $ C_1 > 0 $ such that $ \sup_{\omega}|P_{\omega}| \leq C_1 $, $ \sup_{\omega}|P^c_{\omega}| \leq C_1 $.

		\item There are functions $ \mu_s, \mu_u $ of $ M \to \mathbb{R} $, such that
		\begin{equation*}
		|T_1(t,r\omega)| \leq e^{\mu_s(\omega) t},~|S_1(-t,t(r\omega))| \leq e^{-\mu_u(\omega) t},
		\end{equation*}
		for all $ t,r \geq 0 $ and $ \omega \in M $. See also \autoref{rmk:asymptotic} for a reason why we only consider the such estimates about $ T_1, S_1 $.
	\end{enumerate}
\end{defi}

\begin{rmk}[$ C_0 $ cocycle]\label{rmk:cocycle}
	We say $ U $ or $ \{ U(t,\omega) \} $ is a \emph{$ C_{0} $ cocycle} over $ t $ on a bundle $ \widehat{X} $ which is a topology space and over $ M $, if the following properties hold.
	\begin{enumerate}[(i)]
		\item $ U(t,\omega): \widehat{X}_{\omega} \to \widehat{X}_{t\omega} $ for each $ t \geq 0 $, $ \omega \in M $, i.e. $ U(t,\cdot)(\cdot) $ can be considered as a bundle map over a fixed map $ t $;
		\item ($ C_0 $ property) $ (t, \omega, x) \mapsto U(t,\omega)x : \mathbb{R}_+ \times \widehat{X} \to \widehat{X} $ is continuous;
		\item (cocycle property) $ U(0, \omega) = \id $, $ U(t+s,\omega) = U(t,s\omega)U(s,\omega) $ for all $ t,s \geq 0 $, $ \omega \in M $.
	\end{enumerate}
	When each fiber $ \widehat{X}_{\omega} $ of $ \widehat{X} $ is a normed space, we say $ U $ is a $ C_0 $ linear cocycle if $ U(t,\omega) \in L(\widehat{X}_{\omega}, \widehat{X}_{t\omega}) $ for each $ t \geq 0 $, $ \omega \in M $; in this case, sometimes we also say $ U $ is a strongly continuous (linear) cocycle. In \autoref{def:ud+} \eqref{udd}, $ T_1 $ being $ C^0 $ is in this sense when $ \bigsqcup_{\omega \in M}X_{\omega} $ is endowed with sub-topology of $ M \times Z $; note that $ \bigsqcup_{\omega \in M}X_{\omega} $ in general is not a $ C^0 $ vector bundle unless $ \omega \mapsto P_{\omega} \in L(Z,Z) $ is continuous (not just strongly continuous).

	However, as $ t $ might not be a flow, we need to explain more about $ S_1 $ (in \autoref{def:ud+} \eqref{udd}). \textbf{\emph{First}}, we mention that $ S_1(-t,t\omega) $ should be written as $ S'_1(-t,\omega): Y_{t\omega} \to Y_{\omega} $ in a more strict sense; that is the second variable $ t\omega $ in $ S_1(-t,t\omega) $ only means $ \omega $, and so $ S_1(t-s,s\omega) $ $ (=S_1(t-s,(s-t)(t\omega))) $ ($ t \leq s $) means $ S'_1(t-s,t\omega) $. We write it in this form only for an intuitive sense when $ t $ indeed is a flow. \textbf{\emph{Second}}, except we can not say $ S_1 $ (or $ S'_1 $) is over $ -t $, properties (ii) (iii) can make sense when $ U(t,\omega) = S_1(-t,t\omega) = S'_1(-t,\omega) $; this is what we mean for $ S_1 $ being a $ C_0 $ linear cocycle.
	The continuity of $ T_1, S_1 $ is to give the continuity of $ z(\cdot) $ (in \autoref{def:ud+} \eqref{udd}) and to make sense of the following `variant of constant formulas'.
\end{rmk}

Set $ Z \triangleq X \times Y $, $ P_X: (x,y) \mapsto x $, $ P_Y: (x,y) \mapsto y $. Consider the following differential equation
\begin{equation}\label{*02}
\dot{z}(t) = Cz(t) + L(t\omega)z(t),
\end{equation}
or its variant of constant formula
\begin{equation}\label{*03}
\begin{cases}
x(t) = T(t - t_1)x_1 + \int_{t_1}^{t} T(t-s)P_XL(s\omega)(x(s),y(s)) ~\mathrm{d} s, \\
y(t) = S(t - t_2)y_2 - \int_{t}^{t_2} S(t-s)P_YL(s\omega)(x(s),y(s)) ~\mathrm{d} s,
\end{cases}
~t_1 \leq t \leq t_2,
\end{equation}
where $ C $ is a generator of a $ C_0 $ bi-semigroup and $ L: M \to L(Z,Z) $ satisfies assumption (D1) in \autopageref{d1L}.

\begin{enumerate}\label{UD+}
	\item [$ (\mathrm{UD}_+) $]
	Let \eqref{*02} satisfy \textbf{uniform dichotomy on $ \mathbb{R}_+ $} (see \cite{LP08} in a special setting). That is, there is a $ C_0 $ cocycle correspondence $ H_1 $ (i.e. $ H_1(t,\omega) \sim (T_1(t,\omega), S_1(-t,t\omega)): X_{\omega} \oplus Y_{\omega} \to X_{t\omega} \oplus Y_{t\omega} $) on $ M \times Z $ satisfies satisfies \emph{uniform dichotomy on $ \mathbb{R}_+ $} (see \autoref{def:ud+}); moreover, if $ z(t) \triangleq ( T_1(t-t_1, t_1 \omega)x_1, S_1(t-t_2,t_2 \omega)y_2 ) \in X_{t\omega} \oplus Y_{t\omega} $, $ t_1 \leq t \leq t_2 $, then $ z(\cdot) $ is the mild solution of \eqref{equ:linear} with $ P_{t_1\omega} z(t_1) = x_1 $ and $ P^c_{t_2\omega} z(t_2) = y_2 $. 
\end{enumerate}

\begin{rmk}
	The existing literatures on the characterization of uniform dichotomy on $ \mathbb{R}_+ $ (or $ \mathbb{R} $) in the case of $ C_0 $ bi-semigroup are far less. For a theoretical result see \cite{LP08} (for the case when $ Z $ is a Hilbert space, $ M = \mathbb{R} $, and $ t(s) = t + s $), where the notion of the uniform dichotomy on $ \mathbb{R}_+ $ is taken from that paper. Others are about special differential equations, see the references in \cite{LP08} and \cite{SS01}. A more systematical theory should be established, which is not included in this paper.
\end{rmk}

Consider the following nonlinear differential equation,
\begin{equation}\label{*04}
\dot{z}(t) = Cz(t) + L(t\omega)z(t) + f(t\omega)z(t),
\end{equation}
where $ f: M \times Z \to Z $ satisfies the following assumption.
\begin{enumerate}[(D2)]
	\item \label{d2fff} $ f $ is continuous. For every $ \omega \in M $, $ \sup_{t \geq 0} \lip f(t\omega)(\cdot) = \varepsilon(\omega) < \infty $, and $ \omega \to \varepsilon(\omega) $ is locally bounded.
\end{enumerate}

The following result is important for it tells us how \eqref{*04} gives the cocycle correspondence under the uniform dichotomy condition \textnormal{($ \mathrm{UD}_+ $)}.

\begin{lem}\label{lem:dich00}
	A continuous function $ z(t) = (x_0(t), y_0(t)) \in X \times Y $, $ t_1 \leq t \leq t_2 $, is a mild solution of \eqref{*04}, i.e. it satisfies the following with $ x_0(t_1) = x'_1 $, $ y_0(t_2) = y'_2 $,
	\begin{equation}\label{*05}
	\begin{cases}
	x_0(t) = T(t - t_1)x'_1 + \int_{t_1}^{t} T(t-s) P_X ( L(s\omega)z(s) + f(s\omega)z(s) ) ~\mathrm{d} s, \\
	y_0(t) = S(t - t_2)y'_2 - \int_{t}^{t_2} S(t-s) P_Y( L(s\omega)z(s) + f(s\omega)z(s) ) ~\mathrm{d} s,
	\end{cases}
	t_1 \leq t \leq t_2,
	\end{equation}
	if and only if $ z(t) = (x(t), y(t)) \in X_{t\omega} \oplus Y_{t\omega} $ satisfies
	\begin{equation}\label{*06}
	\begin{cases}
	x(t) = T_1(t - t_1, t_1\omega)x(t_1) + \int_{t_1}^{t} T_1(t-s, s\omega) P_{s\omega} f(s\omega)z(s) ~\mathrm{d} s, \\
	y(t) = S_1(t - t_2, t_2\omega)y(t_2) - \int_{t}^{t_2} S_1(t-s, s\omega) P^c_{s\omega} f(s\omega)z(s) ~\mathrm{d} s,
	\end{cases}
	t_1 \leq t \leq t_2,
	\end{equation}
	where $ x(t) = P_{t\omega}z(t) $, $ y(t) = P^c_{t\omega}z(t) $.
\end{lem}

\begin{proof}
	What we need here are the uniform dichotomy condition \textnormal{($ \mathrm{UD}_+ $)} in \textnormal{(b)}, i.e. $ T_1, S_1 $ satisfy the following:
	\begin{equation}\label{PT1x}
	P_X T_1(t, \omega) = T (t) P_X + (T * ( P_X L(\cdot \omega) T_1(\cdot,\omega) ))(t), ~ \forall t \geq 0,
	\end{equation}
	and
	\begin{equation}\label{PS1x}
	P_X S_1(t-s, s\omega) = T (t) P_X S_1(-s, s\omega) + (T * ( P_X L(\cdot \omega) S_1(\cdot - s, s\omega) ))(t),~ \forall 0 \leq t \leq s,
	\end{equation}
	where $ T * g $ means the \emph{convolution} of $ T $ and $ g $, i.e.
	\begin{equation}\label{convol}
	(T * g)(t) \triangleq \int_{0}^{t} T(t-s)g(s) ~\mathrm{ d } s.
	\end{equation}

	The `if part' and the `only if part' are dual, so we only consider the `if part'. Fix $ \omega $. Let $ z(t) $, $ t_1 \leq t \leq t_2 $, satisfy \eqref{*06} and be fixed. Set
	\[
	g(s) = f(s\omega)z(s).
	\]
	Now the equations become `non-homogeneous linear equations'. By the condition on $ T_1, S_1 $, the solutions of `homogeneous parts' of \eqref{*05} and \eqref{*06} are equal. So it suffices to consider the `non-homogeneous parts', i.e. let $ (x(t_1), y(t_2)) = (0, 0) $. Set
	\begin{align*}
	u^{t_1}(t,\omega) = & (T_1 * P_{(\cdot + t_1)\omega} g_{t_1}(\cdot))(t - t_1) \triangleq \int_{t_1}^{t} T_1(t-s, s\omega) P_{s\omega} g(s) ~\mathrm{d} s, \\
	v^{t_2}(t, \omega) = & - \int_{t}^{t_2} S_1(t-s, s\omega) P^c_{s\omega} g(s) ~\mathrm{d} s,
	\end{align*}
	$ t_1 \leq t \leq t_2 $. We need to show under the projection $ P_X $, they satisfy
	\begin{align}
	P_X u^{t_1}(t,\omega) = & \left\{ T * (  P_X P_{(\cdot + t_1)\omega} g_{t_1}(\cdot) + P_X L_{t_1}(\cdot \omega) u^{t_1}_{t_1}(\cdot,\omega)  ) \right\}  (t - t_1)  \label{pux}\\
	\triangleq & \int_{t_1}^{t} T(t-s) (  P_X P_{s\omega} g(s) + P_X L(s\omega)u^{t_1}(s,\omega)  ) ~\mathrm{d} s, \notag
	\end{align}
	and
	\begin{align}\label{pvx}
	P_X v^{t_2} (t, \omega) = T(t - t_1) P_X v^{t_2} (t_1, \omega) + \left\{  T * (  P_X P^{c}_{(\cdot + t_1)\omega}g_{t_1}(\cdot) + P_X L_{t_1}(\cdot\omega) v^{t_2}_{t_1}(\cdot, \omega)  ) \right\} (t - t_1),
	\end{align}
	$ t_1 \leq t \leq t_2 $, where we use the notation $ f_t(s) = f(t + s) $. In particular, $ x_0(t) \triangleq P_X z(t) = P_X (u^{t_1}(t, \omega) + v^{t_2}(t, \omega) ) $ satisfies the first equation in \eqref{*05} with $ x'_1 = P_X v^{t_2} (t_1, \omega) $. Using the similar equations satisfied by $ P_Y u^{t_1} $, $ P_Y v^{t_2} $, one can show $ y_0 (t) \triangleq P_Y z(t) $ satisfies the second equation in \eqref{*05}, which yields $ z(\cdot) $ satisfies \eqref{*05}.

	That $ P_X u^{t_1}(t,\omega) $ satisfies \eqref{pux} is an easy consequence of the fact that $ T_1 $ satisfies \eqref{PT1x}. More especially, take convolution on both sides of \eqref{PT1x} by $ P_{(\cdot + t_1)\omega} g_{t_1}(\cdot) $.

	Next we will show \eqref{pvx} holds. Multiply (right) on both sides of \eqref{PS1x} by $ P^c_{s\omega}g(s) $, and then integrate them from $ t $ to $ t_2 $ with respect to $ s $, yielding
	\begin{align}
	-P_X \int_{t}^{t_2} y(t,s) ~\mathrm{d} s = & -T(t)P_X \int_{t}^{t_2} y(0, s) ~\mathrm{d} s - \int_{t}^{t_2} \int_{0}^{t} T(t-r)P_XL(r\omega)y(r,s) ~\mathrm{d} r ~\mathrm{d} s \notag \\
	= &  -T(t)P_X \int_{t}^{t_2} y(0, s) ~\mathrm{d} s + \int_{0}^{t} T(t - r)P_XL(r\omega)v^{t_2}(r,\omega) ~\mathrm{d}r \notag \\
	& + \int_{0}^{t} T(t - r) P_X L(r \omega) \int_{r}^{t} y(r, s) ~\mathrm{d} s ~\mathrm{d} r, \label{equ:mm0}
	\end{align}
	where $ y(t,s) = S_1(t-s,s\omega) P^c_{s \omega}g(s) $.

	Also, multiplying (right) on both sides of \eqref{PS1x} by $ P^c_{s\omega}g(s) $ and letting $ t = s $, we get
	\[
	P_X P^c_{s \omega} g(s) = T(s)P_X y(0,s) + \left\{   T * (  P_X L(\cdot \omega) S_1 (\cdot - s, s\omega)P^c_{s \omega} g(s)  ) \right\} (s).
	\]
	Multiply (left) on both sides of the above equality by $ T(t-s) $ and then integrate them with respect to $ s $, yielding
	\begin{equation}\label{equ:mm1}
	T(t)P_X \int_{0}^{t} y(0,s) ~\mathrm{d} s - \int_{0}^{t} T(t-s) P_X P^{c}_{s \omega} g(s) ~\mathrm{d} s + \int_{0}^{t} T(t - r) P_X L(r \omega) \int_{r}^{t} y(r, s) ~\mathrm{d} s ~\mathrm{d} r = 0.
	\end{equation}

	By virtue of \eqref{equ:mm0} \eqref{equ:mm1}, we have shown that \eqref{pvx} holds for $ t_1 = 0 $, i.e.
	\[
	P_X v^{t_2} (t, \omega) = T(t) P_X v^{t_2} (0, \omega) + \left\{  T * (  P_X P^{c}_{(\cdot \omega)}g(\cdot) + P_X L(\cdot\omega) v^{t_2}(\cdot, \omega)  ) \right\} (t).
	\]
	Since
	\begin{align*}
	& \left\{  T * (  P_X P^{c}_{(\cdot \omega)}g(\cdot) + P_X L(\cdot\omega) v^{t_2}(\cdot, \omega)  ) \right\} (t) \\
	= & \left\{  T * (  P_X P^{c}_{(\cdot + t_1)\omega}g_{t_1}(\cdot) + P_X L_{t_1}(\cdot\omega) v^{t_2}_{t_1}(\cdot, \omega)  ) \right\} (t - t_1) \\
	& + T(t - t_1) \left\{  T * (  P_X P^{c}_{(\cdot \omega)}g(\cdot) + P_X L(\cdot\omega) v^{t_2}(\cdot, \omega)  ) \right\} (t_1) \\
	= & \left\{  T * (  P_X P^{c}_{(\cdot + t_1)\omega}g_{t_1}(\cdot) + P_X L_{t_1}(\cdot\omega) v^{t_2}_{t_1}(\cdot, \omega)  ) \right\} (t - t_1) \\
	& + T(t - t_1) \left\{  P_X v^{t_2} (t_1, \omega) - T(t_1) P_X v^{t_2} (0, \omega)  \right\},
	\end{align*}
	now \eqref{pvx} follows. This gives the proof.
\end{proof}

The existence and uniqueness of the solution $ z(t) = (x(t), y(t)) \in X_{t\omega} \oplus Y_{t\omega} $ ($ t_1 \leq t \leq t_2 $) of \eqref{*06} with $ x(t_1) = x_1, y(t_2) = y_2 $ are a standard application of Banach Fixed Point Theorem, as well as the continuous and smooth dependence of the `initial values', i.e. the following \autoref{lem:00} holds (by using the parameter-dependent fixed point theorem (see e.g. \cite[Appendix D.1]{Che18a})).

Using \eqref{*06}, we can define a unique \emph{cocycle correspondence} $ H(s, \omega): X_{\omega} \oplus Y_{\omega} \to X_{s\omega} \oplus Y_{s\omega} $ satisfying the following. Let $ t_1 = 0 $, $ t_2 = s $ in \eqref{*06}. Let $ z(t) = (x(t), y(t)) \in X_{t\omega} \oplus Y_{t\omega} $, $ 0 \leq t \leq s $, be the unique solution of \eqref{*06} with $ x(0) = x_1, y(s) = y_2 $. Define
\begin{gather*}
F_{s, \omega}: X_{\omega} \oplus Y_{s\omega} \to X_{s\omega}, (x_1, y_2) \mapsto x(s),\\
G_{s, \omega}: X_{\omega} \oplus Y_{s\omega} \to Y_{\omega}, (x_1, y_2) \mapsto y(0).
\end{gather*}
Now we have a unique correspondence $ H(s, \omega): X_{\omega} \oplus Y_{\omega} \to X_{s\omega} \oplus Y_{s\omega} $ with generating map $ (F_{s, \omega}, G_{s, \omega}) $. By the cocycle property of $ T_1, S_1 $ and the uniqueness of the solutions of \eqref{*06}, one can easily verify that $ H(t+s, \omega) = H(t,s\omega) \circ H(s, \omega) $ for all $ t,s \geq 0 $, $ \omega \in M $.

\begin{lem}\label{lem:00}
	Under \textnormal{(D1) (D2) ($ \mathrm{UD}_+ $)}, $ (s,\omega,z_1,z_2) \mapsto F_{s, \omega}(P_{\omega}z_1,P^c_{s\omega}z_2), G_{s, \omega}(P_{\omega}z_1,P^c_{s\omega}z_2) $ are continuous. Moreover, if $ f(\omega)(\cdot) \in C^r $ for all $ \omega \in M $ and $ (\omega, z) \mapsto D^r_zf(\omega)z $ is continuous, so are $ (s,\omega,z_1,z_2) \mapsto D^{r}_{(z_1,z_2)}F_{s, \omega}(P_{\omega}z_1,P^c_{s\omega}z_2), D^{r}_{(z_1,z_2)}G_{s, \omega}(P_{\omega}z_1,P^c_{s\omega}z_2) $.
\end{lem}

\begin{thm}\label{thm:biAB}
	Let $ C $ be a generator of $ C_0 $ bi-semigroup, and \textnormal{(D1) (D2) ($ \mathrm{UD}_+ $)} hold. Let $ \{H(t, \omega)\} $ be the cocycle correspondence induced by \eqref{*06}.
	Assume $  { \mu_u(\omega) - \mu_s(\omega) - 2 \varepsilon'(\omega) }  > 0 $, where $ \varepsilon'(\omega) = 2C_1\varepsilon(\omega) $. Take $ \alpha, \beta, \lambda_u, \lambda_s $ such that
	\[
	\frac{\varepsilon'(\omega)}{\mu_u(\omega) - \mu_s(\omega) -  \varepsilon'(\omega)} \leq \alpha(\omega), \beta(\omega) < 1,~
	\lambda_u(\omega) = e^{-\mu_u(\omega) + \varepsilon'(\omega) },~
	\lambda_s(\omega) = e^{\mu_s(\omega) + \varepsilon'(\omega) }.
	\]
	Then $ H(t, s\omega) $ satisfies \textnormal{(A)}$(\alpha(\omega), \lambda^t_u(\omega))$ \textnormal{(B)}$(\beta(\omega), \lambda^t_s(\omega))$ condition for all $ t,s \geq 0 $ and $ \omega \in M $.
	In fact, if $ \alpha(\omega), \beta(\omega) \in (\frac{\varepsilon'(\omega)}{\mu_{u}(\omega) - \mu_s(\omega) - \varepsilon'(\omega)}, 1) $ and $ t \geq \epsilon_1 > 0 $, then $ H(t,s\omega) $ satisfies (A)($ \alpha(\omega) $; $ k_{\alpha}(\omega)\alpha(\omega) $, $ \lambda^t_{u}(\omega) $) (B)($ \beta(\omega) $; $ k_{\beta}(\omega)\beta(\omega) $, $ \lambda^t_{s}(\omega) $) condition, where
	\[
	\frac{( \sigma(\omega)- \frac{\varepsilon'(\omega)}{h(\omega)})e^{-\sigma(\omega)\epsilon_1} + \frac{\varepsilon'(\omega)}{h(\omega)}}{\sigma(\omega)} \leq k_{h}(\omega) < 1,~\sigma(\omega) = \mu_{u}(\omega) - \mu_s(\omega) - \varepsilon'(\omega),
	~ h = \alpha,\beta.
	\]
	In particular, if there is a constant $ c > 1 $ such that
	\[
	\inf_{\omega}\{ \mu_u(\omega) - \mu_s(\omega) - (1+c) \varepsilon'(\omega) \} > 0,
	\]
	then $ \alpha, \beta, k_{\alpha},k_{\beta} $ can be taken constants less than $ 1 $ and $ \sup_{\omega}\lambda_s(\omega) \lambda_u(\omega) < 1 $. In fact, $ \sup_{\omega} \{\alpha(\omega)\} $, $ \sup_{\omega}\{\beta(\omega)\} \to 0 $ if $ \sup_{\omega}\varepsilon'(\omega) \to 0 $.
\end{thm}

\begin{proof}
	The proof is essentially the same as proving \autoref{lem:simpleAB}. What we need is the equation \eqref{*06}. We give a sketch. Let any given $ t_1 < t_2 $ and $ \omega $ be fixed.	Let $ (x(t),y(t)) $, $ (x'(t),y'(t)) $, $ t_1 \leq t \leq t_2 $, satisfy \eqref{*06}; $ \hat{x}(t) = x(t) - x'(t) $, $ \hat{y}(t) = y(t) - y'(t) $. We need to show if $ |\hat{x}(t_1)| \leq \alpha(\omega) |\hat{y}(t_1)| $, then $ |\hat{x}(t_2)| \leq \alpha(\omega) |\hat{y}(t_2)| $  and $ |\hat{y}(t_1)| \leq \lambda^{t_2 - t_1}_u(\omega) |\hat{y}(t_2)| $.

	Let
	$ \alpha(\omega) \in [ \frac{\varepsilon'(\omega)}{\mu_u(\omega) - \mu_s(\omega) -  \varepsilon'(\omega)}, 1 ) $. Take any $ \alpha'(\cdot) $ such that $ \alpha(\omega) < \alpha'(\omega) < 1 $ and $ \alpha'(t\omega) \leq \alpha'(\omega) $ for all $ t \geq 0 $.
	The first step is to show if $ |\hat{x}(t)| \leq \alpha'(\omega) |\hat{y}(t)| $ for $ t \in [t_1, t'_2] $, $ t_1 < t'_2 \leq t_2 $, then $ |\hat{x}(t'_2)| < \alpha'(\omega) |\hat{y}(t'_2)| $ and $ |\hat{y}(t_1)| \leq \lambda_u^{t'_2 - t_1}(\omega) |\hat{y}(t'_2)| $. Only the Gronwall inequality is used.
	The next step is to show if $ |\hat{x}(t_1)| \leq \alpha'(\omega) |\hat{y}(t_1)| $, then $ |\hat{x}(t)| < \alpha'(\omega) |\hat{y}(t)| $ for all $ t > t_1 $ by using the first step with the same argument in the proof of \autoref{slem111}. Now the result follows.
\end{proof}

\subsection{Hille-Yosida (or MR) case} \label{HYccc}

\subsubsection{preliminaries}\label{preHYccc}

Let us introduce a more general class of operators than \textbf{Hille-Yosida} operators (see \autoref{operator} for a definition). Through this subsection, we assume the following (MR) holds.
\begin{enumerate}[(MR)]\label{MR}
	\item  Let the operator $ A : D(A) \subset Z \to Z $ (with $ \rho(A) \neq \emptyset $) satisfy the following properties, which we call an \textbf{MR operator}.
	\begin{enumerate}[(a)]
		\item $ A_0 \triangleq A_{\overline{D(A)}} $ (see \autoref{operator}) generates a $ C_0 $ semigroup $ T_0(\cdot) $ in $ \overline{D(A)} $ such that
		\[
		|T_0(t)| \leq e^{\mu t}, ~\forall t \geq 0.
		\]
		\item Then $ A $ generates a once \emph{integrated semigroup} (see \autoref{operator} for a definition) $ S $ in $ X $. Suppose
		\begin{equation}\label{sft}
		|(S\Diamond f)|_{[0,t]} \leq \delta(t) |f|_{[0,t]}, ~\forall t > 0,
		\end{equation}
		for all $ f \in C^{1}([0,t], X) $, and $ \delta(\cdot) $ is increase satisfying $ \delta(t) \to 0 $ as $ t \to 0 $, where
		\begin{equation*}
		(S\Diamond f)(t) \triangleq \frac{\mathrm{d}}{\mathrm{d} t} \int_{0}^{t} S(t-s)f(s) ~\mathrm{d} s .
		\end{equation*}
		Here $ |\cdot|_{[0,t]} $ is defined by $ |f|_{[0,t]} \triangleq \sup_{s \in [0,t]} |f(s)| $, if $ f \in C([0,t], Z) $. Note that in this case, \eqref{sft} also holds for all $ f \in C([0,t], Z) $ by standard density argument. Also, note that $ (S\Diamond f)(t) \in \overline{D(A)} $ for all $ t \geq 0 $.
	\end{enumerate}
\end{enumerate}

This class of operators was studied extensively in \cite{MR07,MR09} where the authors introduced them in order to deal with some class of non-linear differential equations, for instance the age structured models in $ L^p $ ($ p > 1 $) (see \autoref{ex:nondiss} \eqref{asm}). However, see also \autoref{ex:classical} \eqref{rde} for the case $ X = C^{0,\gamma}(\Omega) $ and \autoref{ex:nondiss} \eqref{ade}.
We are more interested in the case when $ A $ is a Hille-Yosida operator, which is a special case of MR operators. The Hille-Yosida operators were investigated comprehensively in \cite{DPS87}; see also \cite{EN00,ABHN11}. Since with a little more efforts, our argument also works for general case, so we also give the corresponding results for MR operators. We refer the reader to see \cite{ABHN11} for more details about integrated semigroup theory; we use this tool in order to give a representation of `variant of constant formula' under this general context.

What we really need is that if $ A $ is a Hille-Yosida operator, then it satisfies (a) (b). In fact, in this case, \eqref{sft} also holds for $ f \in L^{1}([0,t], X) $ and $ \delta(t) /t \to 1 $ as $ t \to 0^+ $ (note that we have assumed $ |T_0(t)| \leq e^{\mu t} $); this follows from the fact that in this case the integrated semigroup $ S $ is locally Lipschitz; for details, see \cite[Section 3.5]{ABHN11}.
Particularly, when $ A $ is a generator of a $ C_0 $ semigroup $ T_0 $, we have $ \delta(t)/t \to 1 $ as $ t \to 0^+ $ and $ (S\Diamond f)(t) = (T_0 * f)(t) $.

\subsubsection*{A key lemma: an illustration}
\begin{lem}[A pre-lemma] \label{lem:prelem}
	Let $ n \in \mathbb{N} $. If $ x(\cdot) \in C([0,n\varepsilon_1 ], Z) $ satisfies $ |x(t)| \leq e^{\hat{\mu} t}b $ for all $ t \in [0,n\varepsilon_1 ] $, then
	\begin{equation}\label{sx0}
	|(S\Diamond x)(n\varepsilon_1 )| \leq b \max\{ 1, e^{\hat{\mu}\varepsilon_1} \} \delta(\varepsilon_1) e^{\mu (n-1) \varepsilon_1} \frac{e^{(\hat{\mu} - \mu)n\varepsilon_1} - 1}{e^{ (\hat{\mu} - \mu)\varepsilon_1 } - 1} .
	\end{equation}
\end{lem}
\begin{proof}
	Set $ K_m = (S\Diamond x)(m\varepsilon_1) $. First note that $ K_m $, $ 1 \leq m \leq n $, satisfy
	\[
	K_m = T_0(\varepsilon_1)K_{m-1} + (S\Diamond x(\cdot + (m-1)\varepsilon_1 ))(\varepsilon_1),
	\]
	which follows directly from the definition of $ S\Diamond x $ (see also \cite{MR07}). Since
	\begin{align*}
	|K_m| \leq & e^{\mu \varepsilon_1}|K_{m-1}| + \delta(\varepsilon_1) \sup_{t \in [0,\varepsilon_1]}|x( t + (m-1)\varepsilon_1 )| \\
	\leq & e^{\mu \varepsilon_1} |K_{m-1}| + b \max\{ 1, e^{\hat{\mu}\varepsilon_1} \} \delta(\varepsilon_1)e^{\hat{\mu}(m-1)\varepsilon_1}, \\
	|K_0| \leq &  b \max\{ 1, e^{\hat{\mu}\varepsilon_1} \} \delta(\varepsilon_1),
	\end{align*}
	by induction, we know $ |K_n| = |(S\Diamond x)(n\varepsilon_1)| $ satisfies the conclusion.
\end{proof}

\begin{lem}[A key lemma]\label{lem:keylem}
	If $ x(\cdot), y(\cdot) \in C([0,N], Z) $ ($ N $ might be $ \infty $) satisfy
	\[
	|x(t)| \leq \beta |y(t)|, ~ |u(t)| \leq e^{\mu t}a, ~|y(t)| \leq |u(t)| + |(S\Diamond x)(t)|,
	\]
	for all $ t \in [0, N] $, then $ |y(t)| \leq e^{(\mu + \lambda)t} k a $. The constants $ \lambda > 0 $, $ k \geq 1 $ are chosen according to the following.
	\begin{enumerate}[(a)]
		\item If $ \delta(\epsilon)/{\epsilon} \to 1 $ as $ \epsilon \to 0^+ $, then we can choose $ \lambda = \beta $ and $ k = 1 $, e.g. when $ A $ is a Hille-Yosida operator or a generator of a $ C_0 $ semigroup.
		\item In general, let $ \sigma $ be any but fixed such that $ 0 < \sigma \leq 1/2 $. Since $ \delta(t) \to 0 $ as $ t \to 0^+ $, we can take $ \hat{\epsilon} > 0 $ small satisfying for all $ \epsilon \in [\hat{\epsilon}, 2\hat{\epsilon}] $,
		\begin{equation}\label{lambda}
		\sigma > K(\epsilon) \triangleq e^{\sigma} \delta(\epsilon)\beta\max\{ e^{-\mu \epsilon}, 1 \}, ~\beta \delta(\epsilon) < 1.
		\end{equation}
		Take $ \lambda = \sup_{\epsilon \in [\hat{\epsilon}, 2\hat{\epsilon}]} \{ K(\hat{\epsilon}) / {\hat{\epsilon}} \} $, $ k = (1 - \beta\delta(\hat{\epsilon}))^{-1} \max\{1,e^{-\mu \hat{\epsilon}}\} $. (Only need that \eqref{lambda} holds for $ \epsilon = 2\hat{\epsilon} $.)
	\end{enumerate}
\end{lem}

\begin{rmk}
	This lemma given here is to overcome the failure of the Gronwall inequality in our general setting.
	There is also a well known and similar result in the analytic semigroup setting, i.e. the singular Gronwall inequality (see e.g. \cite[Chapter II section 3.3]{Ama95} and \cite{Hen81}). In this case, $ \delta(\epsilon) = \epsilon^{1/p} $, $ p > 1 $.
	In their paper \cite{MR09}, Magal and Ruan indeed had already known in the spirit of this type result (see \cite[Proposition 2.14]{MR09}), in order to prove the existence and some stability results of some class of semi-linear differential equations (see \cite{MR09}) and the existence of the center manifolds in `equilibrium' case (see \cite{MR09a}).
\end{rmk}

\begin{proof}[Proof of \autoref{lem:keylem}]
	First, we \emph{claim} that if $ \lambda > 0 $ satisfies
	\begin{equation}\label{lambda11}
	\max\{ 1, e^{(\mu + \lambda)\epsilon} \} \delta(\epsilon) \beta \leq \lambda \epsilon e^{\mu \epsilon} < (e^{\lambda \epsilon} - 1) e^{\mu \epsilon},
	\end{equation}
	for $ \epsilon \in [\epsilon_0, 2\epsilon_0] $, and $ \beta \delta(\epsilon_0) < 1 $, then $ |y(t)| \leq e^{(\mu + \lambda )t}ka $ for all $ t \in [0, N] $, where
	$ k = (1 - \beta\delta(\epsilon_0))^{-1} \max\{1,e^{-\mu \epsilon_0}\} $.
	The same strategy in the proof of \autoref{lem:simpleAB} will be used here.

	\begin{slem}
		Let $ t_0 \geq \epsilon_0 $, $ a' \geq a $. If $ |y(t)| \leq e^{(\mu + \lambda)t} a' $ for all $ t \in [0, t_0] $, then $ |y(t_0)| < e^{(\mu + \lambda)t_0} a' $.
	\end{slem}
	\begin{proof}
		Let $ t_0 = n \varepsilon_1 $, where $ n \in \mathbb{N} $ and $ \varepsilon_1 \in [\epsilon_0, 2\epsilon_0] $. Note that $ n \geq 1 $. Now by \eqref{sx0} and \eqref{lambda11},
		\begin{align*}
		|y(t_0)| \leq & e^{\mu n \varepsilon_1} a + |(S\Diamond x)(n\varepsilon_1 )| \\
		\leq & e^{\mu n \varepsilon_1}a' +  \beta \max\{ 1, e^{(\mu + \lambda)\varepsilon_1} \} \delta(\varepsilon_1) e^{\mu (n-1) \varepsilon_1} \frac{e^{\lambda n\varepsilon_1} - 1}{e^{ \lambda\varepsilon_1 } - 1} a'\\
		< & e^{(\mu + \lambda)n\varepsilon_1} a',
		\end{align*}
		giving the proof.
	\end{proof}
	\begin{slem}
		Let $ a' \geq a $. If $ |y(t)| \leq e^{(\mu + \lambda)t} a' $ for all $ t \in [0, \epsilon_0] $, then $ |y(t)| \leq e^{(\mu + \lambda)t} a' $ for all $ t \in [0, N] $.
	\end{slem}
	\begin{proof}
		Only need proof when $ \epsilon_0 < N $. First by above sublemma, $ |y(\epsilon_0)| < e^{(\mu + \lambda)\epsilon_0} a' $. Set
		\[
		D_1 = \{  t_1 \in [0,N]: |y(t)|\leq e^{(\mu + \lambda)t}a',~\forall t \in [0, t_1] \},
		\]
		and $ t_0 = \sup D_1 $. If $ t_0 < N $, then $ t_0 \in D_1 $, $ t_0 > \epsilon_0 $ and $ |y(t_0)| = e^{(\mu + \lambda)t_0}a' $. Again by above sublemma, this is a contradiction. So $ t_0 = N $.
	\end{proof}

	If $ t \in [0, \min\{ \epsilon_0, N \}] $, then by \eqref{sft}, we have
	\[
	|y|_{[0,t]} \leq \max\{ 1, e^{\mu t} \}a + \beta \delta(t) |y|_{[0,t]} \leq \max\{ 1, e^{\mu t} \} a + \beta \delta(\epsilon_0) |y|_{[0,t]},
	\]
	yielding $ |y(t)| \leq e^{(\mu + \lambda )t}ka $. Now using the above sublemma, we prove the claim.

	Let $ \epsilon = 2\hat{\epsilon} $ satisfy \eqref{lambda}, then \eqref{lambda} holds for all $ \epsilon \in [\hat{\epsilon}, 2\hat{\epsilon}] $. Take $ \lambda(\epsilon) = K(\epsilon)/\epsilon $. Then it satisfies
	\[
	\mathrm{(a')}~ \lambda(\epsilon) >( \delta(\epsilon) / \epsilon ) \beta e^{-\mu \epsilon}, ~~
	\mathrm{(b')} ~ \lambda(\epsilon) = e^\sigma ( \delta(\epsilon)  / \epsilon )\beta > e^{K(\epsilon)} (\delta(\epsilon) / \epsilon)  \beta = e^{\lambda(\epsilon) \epsilon}  ( \delta(\epsilon) / \epsilon ) \beta,
	\]
	for all $ \epsilon \in [\hat{\epsilon}, 2\hat{\epsilon}] $.
	Let $ \lambda = \sup_{\epsilon \in [\hat{\epsilon}, 2\hat{\epsilon}]} \{ K(\epsilon) / {\epsilon} \} $ and $ \epsilon \in [\hat{\epsilon}, 2\hat{\epsilon}] $. Then $ \lambda \epsilon \leq 1 $ since $ \sigma \leq 1/2 $, which yields $ \lambda > e^{\lambda \epsilon}  ( \delta(\epsilon) / \epsilon ) \beta $. (Note that $ \lambda e^{-\lambda \epsilon} \geq \lambda(\epsilon) e^{-\lambda(\epsilon) \epsilon} $.)
	That is $ \lambda $ satisfies \eqref{lambda11} for $ \epsilon \in [\hat{\epsilon}, 2\hat{\epsilon}]  $. Therefore, by applying the above claim, we finish the proof of the general case (b).

	Let us consider the case (a), i.e. $ \delta(t) / t \to 1 $ as $ t \to 0^+ $. Let $ \lambda > \beta $ be fixed. Then there is an $ \hat{\epsilon} > 0 $ such that if $ 0 < \epsilon_0 < \hat{\epsilon} $, then \eqref{lambda11} holds for $ \epsilon = \epsilon_0 $ and $ \beta \delta(\hat{\epsilon}) < 1 $. So by the claim, we see
	\[
	|y(t)| \leq e^{(\mu + \lambda)t} \frac{\max\{1,e^{-\mu \epsilon_0}\}}{1 - \beta\delta(\epsilon_0)} a, ~\forall t \in [0, N].
	\]
	Let $ \epsilon_0 \to 0^+ $ and then $ \lambda \to \beta^+ $, which yields the result.	The proof is complete.
\end{proof}

\subsubsection{the study of $ \dot{z}(t) = Az(t) + L(t\omega)z(t) $} \label{ss11}

\paragraph*{\textbf{Notation}} \label{eee}
For the convenience of writing, if a function $ \mu : M \to \mathbb{R} $ satisfying $ \mu(t \omega) \leq \mu(\omega) $ for all $ t \geq 0 $, $ \omega \in M $, we will write $ \mu \in \mathcal{E}_1(t) $. For example, $ \tau \in \mathcal{E}_1(t) $ (in assumption (D1) in \autopageref{d1L}).

Consider the following differential equation
\begin{equation}\label{*2}
\dot{z}(t) = Az(t) + L(t\omega)z(t),
\end{equation}
where $ L: M \to L(\overline{D(A)},Z) $ satisfies assumption (D1) in \autopageref{d1L}.
$ z(\cdot) $ is called a \emph{classical solution} of \eqref{*2}, if $  z(\cdot) \in C^1([0,\infty), X)  $ and it point-wisely satisfies \eqref{*2}.

Now by the perturbation results from \cite{MR07}, we know $ A(\omega) \triangleq A + L(\omega) $ is also an MR operator (if $ A $ is Hille-Yosida so is $ A(\omega) $). In fact, we have the following.

\begin{lem}\label{lem:coycle0}
	Let \textnormal{(MR) (D1)} hold. Then we have the following about $ A(\omega) \triangleq A + L(\omega) $.
	\begin{enumerate}[(1)]
		\item $ A_0(\omega) \triangleq A(\omega)|_{\overline{D(A)}}: D(A_0(\omega)) \subset \overline{D(A)} \to \overline{D(A)} $, $ \omega \in M $, generate a $ C_0 $ linear cocycle $ \{T_0(t, \omega)\} $ satisfying the following.
		\begin{enumerate}[(a)]
			\item $ T_0(t, \omega) = T_0(t) + [ S\Diamond (L(\cdot \omega) T_0(\cdot, \omega)) ](t) $.
			\item $ T_0(t+s, \omega) = T_0(t, s\omega)T_0(s, \omega) $ for all $ t, s \in \mathbb{R}_+ $, $ \omega \in M $.
			\item $ \overline{D(A_0(\omega))} = \overline{D(A)} $. If $ \frac{\mathrm{d}}{\mathrm{d} t}|_{t=0} T_0(t, \omega)x $ exists, then $ \frac{\mathrm{d}}{\mathrm{d} t}|_{t=0} T_0(t, \omega)x = A_0(\omega) x $.
			Moreover, under $ \delta(t) = O(t) $ as $ t \to 0^+ $, e.g. $ A $ is a Hille-Yosida operator or a generator of a $ C_0 $ semigroup, then $ \frac{\mathrm{d}}{\mathrm{d} t}|_{t=0} T_0(t, \omega)x $ exists if and only if $ x \in D(A_0(\omega)) $.

		\end{enumerate}
		\item \label{integral} There is a strongly continuous function $ S_0(\cdot, \cdot): \mathbb{R}_+ \times M \to L(Z, Z) $ satisfying the following.
		\begin{enumerate}[(a)]
			\item $ S_0(t, \omega) = S(t) + [ S\Diamond (L(\cdot \omega) S_0(\cdot, \omega)) ](t) $.
			\item For any $ f \in C([0,t], X) $, any $ t > 0 $, set
			\[
			(S_0 \Diamond f)(\omega)(t) = \frac{\mathrm{d}}{\mathrm{d} t} \int_{0}^{t} S_0(t-s, s\omega) f(s) ~\mathrm{d} s,
			\]
			which exists and is continuous. Moreover, it satisfies
			\begin{equation}\label{s01}
			|(S_0 \Diamond f)(\omega)(\cdot)|_{[0,t]} \leq \delta_1(t,\omega) |f(\cdot)|_{[0,t]},
			\end{equation}
			and $ \delta_1(t,\omega) / \delta(t) \to 1 $ as $ t \to 0^+ $. In fact, $ \delta_1(t,\omega) = \frac{\delta(t)}{1 - \tau(\omega)\delta(t)} $ when $ \tau(\omega)\delta(t) < 1 $, $ \delta_1(t,\cdot) \in \mathcal{E}_1(t) $ and $ \delta_1(\cdot,\omega) $ is increased.
			\item (\textbf{key}) $ S_0 \Diamond f $ satisfies the following equality:
			\begin{equation}\label{sf1}
			(S_0 \Diamond f)(\omega)(t+s) = T_0(t,s\omega)(S_0 \Diamond f)(\omega)(s) + (S_0 \Diamond f_s)(s\omega)(t),
			\end{equation}
			where $ f_s(t) \triangleq f(s + t) $.
		\end{enumerate}
	 \item
	 $ z(\cdot) $ is a mild solution of \eqref{*2} if and only if $ z(0) \in \overline{D(A)} $ and $ z(t) = T_0(t, \omega)z(0) $.
	 If $ \delta(t) = O(t) $ as $ t \to 0^+ $, then $ z(\cdot) $ is a classical solution of \eqref{*2} if and only if $ z(0) \in D(A_0(\omega)) $ and $ z(t) = T(t, \omega)z(0) $.
	\end{enumerate}
\end{lem}

\begin{rmk}
	We can not give any result like $ \frac{\mathrm{d}}{\mathrm{d} t}|_{t=0} T_0(t, \omega)x $ exists if and only if $ x \in D(A_0(\omega)) $ when $ \delta(t) = O(t) $ as $ t \to 0^+ $ does not satisfy, unless more regularity of $ L(\cdot \omega) $ is supposed; see \autoref{lem:diff1} for such a result.
	In many situations, that $ A $ is a Hille-Yosida operator is sufficient for us. See \cite{Are04, Ama95} and the references therein for more details about the maximal regularity with respect to time variable.
\end{rmk}

\begin{proof}[Proof of \autoref{lem:coycle0}]
	\textbf{(I)}. The proof of the following sublemma is standard by applying Banach Fixed Point Theorem. Here we give a sketch for the convenience of readers.
	\begin{slem}\label{slem123}
		For any strongly continuous $ V: \mathbb{R}_+ \times M  \to L(Z, Z) $, there is a unique strongly continuous $ W: \mathbb{R}_+ \times M  \to L(Z, Z) $ satisfying
		\[
		W(t, \omega) = V(t, \omega) + [ S\Diamond (L(\cdot \omega) W(\cdot, \omega)) ](t), ~t \geq 0. \tag{$ \odot $}
		\]
		Moreover, if $ V(t,\omega) \in \overline{D(A)} $ for all $ \omega \in M, t > 0 $, then so is $ W(t, \omega) $.
	\end{slem}
	\begin{proof}
		Fix $ \omega $. Since $ \delta(t) \to 0 $ as $ t \to 0^+ $, we can choose $ \vartheta(\omega) > 0 $ such that $ \delta(\vartheta(\omega)) < 1 $. Now we can uniquely define $ W(t, \omega) $ for $ t \in [0, \vartheta(\omega)] $ satisfying above equation by using Banach Fixed Point Theorem. Then observing the following elementary fact about $ S \Diamond f $ that
		\[
		(S \Diamond f) (t + s) = T_0(t) (S \Diamond f)(s) + (S \Diamond f_s)(t),
		\]
		one can give $ W(t, \omega) $ for $ t \in [\vartheta(\omega), 2\vartheta(\omega)] $. Or more specifically, there is a unique $ \widetilde{W}(t,\omega) $ satisfying
		\[
		\widetilde{W}(t,\omega) =  V(t + \vartheta(\omega), \omega) + T_0(t) [ S\Diamond (L(\cdot \omega) W(\cdot, \omega)) ](\vartheta(\omega)) + [ S\Diamond (L((\cdot+ \vartheta(\omega)) \omega) \widetilde{W}(\cdot, \omega)) ](t),
		\]
		for $ t \in [0, \vartheta(\omega)] $. (Note that $ \tau(\vartheta(\omega)\omega) \leq \tau(\omega) $.) Set $ W(t, \omega) = \widetilde{W}(t-\vartheta(\omega),\omega) $ for $ t \in [\vartheta(\omega), 2\vartheta(\omega)] $. Then it also satisfies ($ \odot $) in $ [\vartheta(\omega), 2\vartheta(\omega)] $. Continue this way to construct $ W $. The strong continuity of $ W $ is obvious since $ \omega \mapsto \tau(\omega) $ is locally bounded.
	\end{proof}

	Using above sublemma by setting $ V(t, \omega) = T_0(t) $, we have $ T_0(t, \omega) $ satisfying (1) (a). Let us show (1) (b). Observe that
	\begin{align*}
	T_0(t+s,\omega) = & T_0(t+s) +  (S\Diamond ( L(\cdot \omega)T_0(\cdot,\omega) ))(t+s) \\
	= & T_0(t + s) + T_0(t) (S\Diamond ( L(\cdot \omega)T_0(\cdot,\omega) ))(s) \\
	& + (S\Diamond ( L((\cdot + s) \omega)T_0(\cdot + s,\omega) ))(t), \\
	T_0(t,s\omega)T_0(s, \omega) = & T_0(t)T(s,\omega) + (  S\Diamond ( L((\cdot + s) \omega)T_0(\cdot ,s\omega) ) T_0(s, \omega)  )(t) \\
	= & T_0 (t + s) + T_0(t)  (S\Diamond ( L(\cdot \omega)T_0(\cdot,\omega) )) (s) \\
	& + (  S\Diamond ( L((\cdot + s) \omega)T_0(\cdot ,s\omega) ) T_0(s, \omega)  )(t).
	\end{align*}
	Let $ g^1(t) = T_0(t, s\omega)T_0(s, \omega) $, $ g^2(t) = T_0(t+s, \omega) $. We see that they all satisfy
	\[
	g^i (t) = T_0 (t + s) + T_0(t)  (S\Diamond ( L(\cdot \omega)T_0(\cdot,\omega) )) (s) +  (  S\Diamond ( L((\cdot + s) \omega) g^i(\cdot) )  ) (t).
	\]
	By the uniqueness result in \autoref{slem123}, we know $ g^1 = g^2 $.

	\textbf{(II)}. By letting $ V(t,\omega) = S(t) $ in \autoref{slem123}, we have $ S_0(t,\omega) $ satisfying (2) (a). For any $ f \in C^1([0,t], Z) $, since (by exchanging the order of integration)
	\[
	([ S\Diamond (L(\cdot \omega) S_0(\cdot, \omega)) ] \Diamond f) (t) = [ S\Diamond (L(\cdot \omega) (S_0 \Diamond f )(\omega)(\cdot)  ) ] (t),
	\]
	which exists and is continuous, we conclude that
	\[
	|(S_0 \Diamond f)(\omega)(\cdot)|_{[0,t]} \leq \frac{\delta(t)}{1 - \tau(\omega) \delta(t)} |f|_{[0,t]},
	\]
	provided $ \tau(\omega)\delta(t) < 1 $. By the standard density argument, we know that the above also holds for $ f \in C([0,t], Z) $. We have proven (2) (a) (b). Let us consider (2) (c).

	\begin{slem}\label{slem:aa}
		Let
		\[
		W(t, \omega)x =T_0(t, \omega)x + ( S_0 \Diamond f )(\omega) (t), ~t \geq 0, \omega \in M.
		\]
		Then it satisfies
		\[
		W(t, \omega)x = T_0(t)x + [ S \Diamond ( L(\cdot \omega) W(\cdot,\omega)x + f(\cdot) ) ] (t),
		\]
		and vice versa.
	\end{slem}
	\begin{proof}
		This follows from
		\begin{align*}
		& T_0(t)x + [ S \Diamond ( L(\cdot \omega) W(\cdot,\omega)x + f(\cdot) ) ] (t) \\
		= & T_0(t)x + [ S \Diamond ( L(\cdot \omega) T_0( \cdot, \omega )x ) ] (t) \\
		& + [ S \Diamond \{ L (\cdot \omega) ( W(\cdot,\omega)x - T_0(\cdot, \omega)x ) \} ] (t) + ( S \Diamond  f ) (t) \\
		= & T_0(t,\omega)x + (S_0 \Diamond f)(\omega)(t) \\
		& + [ S \Diamond \{ L (\cdot \omega) ( W(\cdot,\omega)x - T_0(\cdot, \omega)x ) \} ] (t) - [S \Diamond ( L(\cdot \omega) (S_0 \Diamond f)(\omega)(\cdot) )](t) \\
		= & W(t, \omega)x + 0,
		\end{align*}
		which gives the proof.
	\end{proof}

	Let $ u(t) = (S_0 \Diamond f)(\omega)(t) $, then it follows that $ u(t) $ satisfies
	\[
	u(t) = [ S \Diamond ( L(\cdot \omega) u(\cdot) + f(\cdot) ) ] (t).
	\]
	In particular, $ u_s(t) \triangleq u(t + s) $ satisfies
	\[
	u_s(t) = T_0(t)u(s) + [ S \Diamond ( L((\cdot + s )\omega) u_s(\cdot) + f_s(\cdot) ) ] (t).
	\]
	So by above sublemma,
	\[
	u_s(t) = T_0(t,s\omega)u(s) + (S_0 \Diamond f_s )(s \omega)(t),
	\]
	i.e. (2) (c) holds.

	\textbf{(III)}. The first part of (3) is a consequence of \autoref{slem:aa} for $ f \equiv 0 $.
	For the proof of (1) (c), and the second part of (3),
	we need the following.
	\begin{slem}\label{slem:diff}
		Consider the mild solution of the following equation.
		\[
		\begin{cases}
		\dot{u}(t) = Au(t) + h(t), \\
		u(0) = x,
		\end{cases}
		\]
		where $ h $ is continuous. (1) If $ \dot{u}(0) $ exists, then $ x \in D(A) $, $ Ax + h(0) \in \overline{D(A)} $ and $ \dot{u}(0) = Ax + h(0) $; (2) vice versa if $ \delta(t) = O(t) $ as $ t \to 0^+ $.
	\end{slem}
	\begin{proof}
		(1) If $ \dot{u}(0) $ exists, which yields $ A\frac{1}{t}\int_{0}^{t} u(s) ~\mathrm{d}s =- \frac{1}{t} \int_{0}^{t} h(s) ~\mathrm{d}s + \frac{u(t) - x}{t} $ exists as $ t \to 0^+ $, then by the closeness of $ A $, $ x = \lim_{t \to 0} \frac{1}{t}\int_{0}^{t} u(s) ~\mathrm{d}s \in D(A) $ and $ Ax = -h(0) + \dot{u}(0) \in \overline{D(A)} $. As $ u(t) \in \overline{D(A)} $, we see $ \dot{u}(0) \in \overline{D(A)} $.

		Let us consider the part of `vice versa'. Since $ x \in D(A) $, we have $ (S(t)x)' = x + S(t)Ax $. Thus, by \cite[Lemma 3.2.9]{ABHN11}, we see
		\[
		u(t) = x + S(t)Ax + (S \Diamond h)(t).
		\]
		Therefore,
		\begin{equation}\label{equ00}
		u(t) - x = S(t)(Ax+h(0)) + [S \Diamond (h(\cdot) - h(0) ) ](t).
		\end{equation}
		Since $ Ax + h(0) \in \overline{D(A)} $,  we have
		\[
		\frac{\mathrm{d}}{\mathrm{d} t} S(t)(Ax + h(0)) = T_0(t)(Ax + h(0)).
		\]
		Furthermore, by $ \delta(t) = O(t) $ as $ t \to 0^+ $, we get
		\[
		|\frac{1}{t}[S \Diamond (h(\cdot) - h(0) ) ](t)| \leq \frac{\delta(t)}{t} \sup_{s \in [0,t]} |h(s) - h(0)| \to 0,~\text{as}~ t \to 0^+,
		\]
		which yields that $ \dot{u}(0) = Ax + h(0) $.
	\end{proof}

	Now, using above sublemma by letting $ h(t) = L(t\omega)z(t) $ which is continuous, we know that (1) (c) holds. (Note that $ x \in D(A_0(\omega)) $ means that, by definition, $ x \in D(A) $ and $ Ax + L(\omega)x \in \overline{D(A)} $). The second part of (3) is a consequence of (1) (c). The proof is complete.
\end{proof}

\subsubsection{the study of $ \dot{z}(t) = Az(t) + L(t\omega)z(t) + f(t\omega)z(t) $}

Consider the following nonlinear differential equation,
\begin{equation}\label{*3}
\dot{z}(t) = Az(t) + L(t\omega)z(t) + f(t\omega)z(t),
\end{equation}
where $ f: M \times \overline{D(A)} \to Z $ satisfies assumption (D2) in \autopageref{d2fff}.

\begin{lem}\label{lem:vcf1}
	Let \textnormal{(MR) (D1) (D2)} hold. Then there is a unique $ C_0 $ (non-linear) cocycle $ U: \mathbb{R}_+ \times M \times Z \to Z $ over $ t $ such that the following hold.
	\begin{enumerate}[(1)]
		\item  $ U(\cdot, \omega)x $, $ x \in X $, are the mild solutions of \eqref{*3}, and $ U $ satisfies the following \textbf{variant of constant formula} (through $ T_0(\cdot) $),
		\begin{equation}\label{Uvcf0}
		U(t, \omega)x = T_0(t)x + [S \Diamond (( L(\cdot\omega) + f(\cdot \omega) )U(\cdot, \omega)x)](t),
		\end{equation}
		and another \textbf{variant of constant formula} (through $ T_0(\cdot, \cdot) $),
		\begin{equation}\label{Uvcf}
		U(t, \omega)x = T_0(t, \omega)x + [S_0 \Diamond ( f(\cdot \omega) U(\cdot, \omega)x)](\omega)(t).
		\end{equation}
		\item If $ f(\omega)(\cdot) \in C^r $, then $ U(t, \omega)(\cdot) \in C^r $. Moreover, if $ (\omega,x) \mapsto D^r_xf(\omega)(x) $ is continuous, so is $ (t,\omega, x) \mapsto D^r_xU(t,\omega)(x) $.
	\end{enumerate}
\end{lem}

The results might be well known at least in the case when $ M = \mathbb{R} $, $ t $ is the translation dynamic system in $ \mathbb{R} $, i.e. $ t(s) = t+s $, and $ A $ is a Hille-Yosida operator (or more particularly a sectorial operator). See also \cite{Hen81,Ama95}.

\begin{proof}[Proof of \autoref{lem:vcf1}]
	We can define a unique $ U $ by using \eqref{Uvcf0} (see also \autoref{slem123} and \cite{MR09}). Moreover, a continuous function $ z(\cdot) $ is a mild solution of \eqref{*3} if and only if $ z(t) = U(t,\omega)x $ satisfying \eqref{Uvcf0}.
	The strong continuity of $ U $ and (2) are obvious by noting that $ \omega \mapsto \varepsilon(\omega) $ is locally bounded.
	The cocycle property of $ U $ can be proved the same as that of $ T_0(\cdot,\cdot) $ in \autoref{lem:coycle0} (1) (b). That $ U $ also satisfies \eqref{Uvcf} follows from \autoref{slem:aa} by setting $ f(t) = f(t\omega)U(t,\omega)x $. The proof is complete.
\end{proof}

\subsubsection{uniform dichotomy on $ \mathbb{R}_+ $: $ C_0 $ cocycle case}

\begin{enumerate}[(UD)] \label{UD}
	\item Assume that the $ C_0 $ cocycle $ T_0(\cdot, \cdot)(\cdot): \mathbb{R}_+ \times M \times \overline{D(A)} \to \overline{D(A)} $ (obtained in \autoref{lem:coycle0} (1)) satisfies the following property, which is called the \textbf{uniform dichotomy on $ \mathbb{R}_+ $}; see \cite{CL99} for details and characterizations.

	\begin{enumerate}[(a)]
		\item (spectral spaces) The space $ \overline{D(A)} $ can be decomposed as $ \overline{D(A)} = X_{\omega} \oplus Y_{\omega} $, $ \omega \in M $, with associated projections $ P_\omega $, $ \omega \in M $ such that $ R(P_\omega) = X_{\omega} $, $ \ker P_\omega = Y_{\omega} $ and $ \omega \to P_{\omega} $ are strongly continuous. $ P^c_\omega = I - P_\omega $. Moreover,
		\[
		T_0(t,\omega) X_{\omega} \subset X_{t \omega}, ~ T_0(t,\omega) Y_{\omega} \subset Y_{t \omega},
		\]
		and $ T_0(t,\omega): Y_{\omega} \to Y_{t \omega} $ is invertible. Set
		\[
		\widehat{T}(t,\omega) = T_0(t, \omega)P_{\omega}: X_{\omega} \to X_{t \omega},~ \widehat{S}(-t,t\omega) = (T_0(t, \omega))^{-1}P^c_{t\omega}: Y_{t\omega} \to Y_{ \omega}.
		\]
		\item (angle condition) There is a constant $ C_1 > 0 $ such that $ \sup_{\omega}|P_{\omega}| \leq C_1 $, $ \sup_{\omega}|P^c_{\omega}| \leq C_1 $.
		\item (spectra) There are functions $ \mu_s, \mu_u $ of $ M \to \mathbb{R} $, such that
		\begin{equation}\label{spectra}
		|\widehat{T}(t,r\omega)| \leq e^{\mu_s(\omega) t},~|\widehat{S}(-t,t(r\omega))| \leq e^{-\mu_u(\omega) t},
		\end{equation}
		for all $ t, r \geq 0 $ and $ \omega \in M $.
	\end{enumerate}
\end{enumerate}
Here $ \widehat{S}(-t,t\omega) $ should be written as $ \widehat{S}'(-t,\omega) $ in a more rigid sense; we write this form only for giving an intuitive sense when $ t $ indeed is a flow; see also \autoref{rmk:cocycle}.

\begin{rmk}[about the asymptotic behavior of $ T_0(\cdot,\cdot) $]\label{rmk:asymptotic}
	The general results about the estimates \eqref{spectra} are
	\[
	|\widehat{T}(t,r\omega)| \leq M_0(\omega)e^{\mu_s(\omega) t}, ~|\widehat{S}(-t,t(r\omega))| \leq M_0(\omega)e^{-\mu_u(\omega) t},
	\]
	where $ M_0 : M \to \mathbb{R}_+ $.
	Consider a new equivalent norm $ |\cdot|_{\omega} $ on $ X_{\omega} $ defined by
	\[
	|x|_{\omega} = \inf\{  \sup_{t  \geq 0}|e^{-\mu_s(\omega_0) t}T_0(t,\omega)x| : \exists s_0 \geq 0 \text{ such that } s_0\omega_0 = \omega \} , ~x \in X_{\omega}.
	\]
	Then $ |x| \leq |x|_{\omega} \leq M_0(\omega) |x| $ if $ x \in X_{\omega} $. Now if $ x \in X_{\omega} $, $ s_0\omega_0 = \omega $, where $ s_0 \geq 0 $, then
	\begin{align*}
	|\widehat{T}(t,\omega)x|_{t\omega} = & \inf\{ \sup_{r \geq 0}|e^{-\mu_s(\omega'_0) r}T_0(r,t\omega)T_0(t,\omega)x|: \exists s'_0 \geq 0 \text{ such that } s'_0\omega'_0 = t\omega \}\\
	= & \inf\{ \sup_{r \geq 0}|e^{-\mu_s(\omega'_0) r}T_0(t+r,\omega)x|: \exists s'_0 \geq 0 \text{ such that } s'_0\omega'_0 = t\omega \}   \\
	= & \inf\{ e^{\mu_s(\omega'_0) t} \sup_{r \geq 0}|e^{-\mu_s(\omega'_0) (t+r)}T_0(t+r,\omega)x|: \exists s'_0 \geq 0 \text{ such that } s'_0\omega'_0 = t\omega \}\\
	\leq & \inf\{ e^{\mu_s(\omega''_0) t} \sup_{r \geq 0}|e^{-\mu_s(\omega''_0) r}T_0(r,\omega)x|: \exists s''_0 \geq 0 \text{ such that } s''_0\omega''_0 = \omega \}\\
	\leq & e^{\mu_s(\omega_0) t} |x|_{\omega}.
	\end{align*}
	That is $ \widehat{T}(t,\omega): (X_{\omega}, |\cdot|_{\omega}) \to (X_{t\omega}, |\cdot|_{t\omega}) $, $ \omega \in M $, satisfy \eqref{spectra}.  Similar norm $ |\cdot|_{c,\omega} $ can be defined on $ Y_{\omega} $. Let $ |x|^{\sim}_{\omega} = \max\{ |P_{\omega}x|_{\omega}, |P^c_{\omega}x|_{c,\omega} \} $, then $ |x| \leq 2|x|^{\sim}_{\omega} \leq 2C_1M_0(\omega) |x| $, which shows that $ |\cdot|^{\sim}_{\omega} $ is an equivalent norm on $ \overline{D(A)} $ for each $ \omega \in M $. Consider $ M \times \overline{D(A)} $ as a bundle over $ M $ with metric fibers $ (X_\omega \oplus Y_{\omega}, |\cdot|^{\sim}_{\omega}) $, $ \omega \in M $. Now $ T_0(\cdot,\cdot) $ also satisfies \textnormal{(b) (c)} in the bundle $ M \times \overline{D(A)} $. That $ M_0 = 1 $ is important in our argument. Also, note that in the new norm $ |\cdot|^{\sim}_{\omega} $, $ |P_\omega| = 1 $ and $ |P^c_\omega| = 1 $.
\end{rmk}

\begin{rmk}
	The notion of uniform dichotomy on $ \mathbb{R}_+ $ about a $ C_0 $ cocycle is essentially taken from \cite[Definition 6.14]{CL99} which states that $ \widehat{T}(t,r\omega) \leq M_0(\omega) e^{\mu_s(\omega)t} $ for all $ t,r \geq 0 $, similar for $ \widehat{S} $.
	In Multiplicative Ergodic Theorem, the functions $ \mu_s $, $ \mu_u $ are usually characterized through the Lyapunov numbers, so in general they are invariant under $ t $ (usually for the case that $ t $ is invertible). We refer the readers to see \cite[Section 7.1]{CL99} and the references therein for more characterizations about $ \mu_{s}, \mu_{u} $, i.e. the point-wise version of uniform dichotomy. Note that Sacker-Sell spectral theory (see e.g. \cite{SS94, CL99}) only characterizes the absolutely uniform dichotomy, which has a little restriction for us to apply.
	The readers should notice that there is \emph{no} any information about the decomposition of $ Z $, nor the extensions of $ \widehat{T}, \widehat{S} $.

\end{rmk}

By acting $ P_{\omega}, P^c_{\omega} $ on both side of \eqref{Uvcf}, we know that $ z(t) = ( x(t), y(t) ) \in X_{t\omega} \oplus Y_{t\omega} = \overline{D(A)} $ ($ t_1 \leq t \leq t_2 $) is a mild solution of \eqref{*3}, if and only if it satisfies
\begin{equation}\label{**0}
\begin{cases}
x(t) = \widehat{T}(t - t_1, t_1 \omega)x(t_1) + P_{t\omega} ( S_0 \Diamond ( f_{t_1}(\cdot\omega) z_{t_1}(\cdot)  ) ) (t_1 \omega) (t - t_1), \\
y(t) = \widehat{S}(t - t_2, t_2 \omega)y(t_2) - \widehat{S}(t - t_2, t_2 \omega) P^c_{t_2 \omega} ( S_0 \Diamond ( f_{t}(\cdot\omega) z_{t}(\cdot)  ) ) (t \omega) (t_2 - t),
\end{cases}
\end{equation}
for all $ t_1 \leq t \leq t_2 $, where $ f_t(s\omega) = f((t+s)\omega) $.

Through \eqref{**0}, one can define a cocycle correspondence $ H(s, \omega): X_{\omega} \oplus Y_{\omega} \to X_{s \omega} \oplus Y_{s \omega} $ as follows. Let $ t_1 = 0, t_2 = s $. Given $ (x_1, y_2) \in X_{\omega} \times Y_{s\omega} $, then there is a continuous $ z(t) = (x(t), y(t)) \in X_{t\omega} \oplus Y_{t\omega} $ ($ 0 \leq t \leq s $) satisfying \eqref{**0} with $ x(0) = x_1 $, $ y(s) = y_2 $. Define $ F_{s,\omega}(x_1, y_2) = x(s) $ and $ G_{s,\omega} (x_1, y_2) = y(0) $. Let $ H(s,\omega) \sim (F_{s,\omega}, G_{s,\omega}) $. Here note that $ \graph H(t, \omega) = \graph U(t, \omega) $, i.e. $ H $ can be regarded as a $ C_0 $ cocycle over $ t $. We also say the cocycle correspondence $ \{ H(s, \omega) \} $ (or $ H $) is induced by \eqref{**0}.

\begin{thm}\label{them:spec}
	Let \textnormal{(MR) (D1) (D2)} and \textnormal{(UD)} hold.
	Let $ \{H(t, \omega)\} $ be the cocycle correspondence induced by \eqref{**0}.
	\begin{enumerate}[(1)]
		\item Assume $ \delta(t) / t \to 1 $ as $ t \to 0^+ $ and $ \mu_u(\omega) - \mu_s(\omega) - 2 \varepsilon'(\omega)  > 0 $, where $ \varepsilon'(\omega) = 2C_1\varepsilon(\omega) $. Take $ \alpha, \beta, \lambda_u, \lambda_s $ such that
		\[
		\frac{\varepsilon'(\omega)}{\mu_u(\omega) - \mu_s(\omega) -  \varepsilon'(\omega)} \leq \alpha(\omega), \beta(\omega) < 1,~
		\lambda_u(\omega) = e^{-\mu_u(\omega) + \varepsilon'(\omega) },~
		\lambda_s(\omega) = e^{\mu_s(\omega) + \varepsilon'(\omega) }.
		\]
		Then $ H(t, s\omega) $ satisfies \textnormal{(A)}$(\alpha(\omega), \lambda^t_u(\omega))$ \textnormal{(B)}$(\beta(\omega), \lambda^t_s(\omega))$ condition for all $ t,s \geq 0 $ and $ \omega \in M $.
		In fact, if $ \alpha(\omega), \beta(\omega) \in (\frac{\varepsilon'(\omega)}{\sigma(\omega)}, 1) $ and $ t \geq \epsilon_1 > 0 $, then $ H(t,s\omega) $ satisfies (A)($ \alpha(\omega); k_{\alpha}(\omega, \epsilon_1)\alpha(\omega), \lambda^t_{u}(\omega) $) (B)($ \beta(\omega); k_{\beta}(\omega, \epsilon_1)\beta(\omega), \lambda^t_{s}(\omega) $) condition, where $ k_{h}(\omega, \epsilon_1) < 1 $, $ h = \alpha,\beta $, and $ \sigma(\omega) = \mu_{u}(\omega) - \mu_s(\omega) - \varepsilon'(\omega) $.
		In particular, if there is a constant $ c > 1 $ such that
		\[
		\inf_{\omega}\{ \mu_u(\omega) - \mu_s(\omega) - (1+c) \varepsilon'(\omega) \} > 0,
		\]
		then $ \alpha, \beta, k_{\alpha},k_{\beta} $ can be taken constants less than $ 1 $ and $ \sup_{\omega}\lambda_s(\omega) \lambda_u(\omega) < 1 $. In fact, $ \sup_{\omega} \{\alpha(\omega)\} $, $ \sup_{\omega}\{\beta(\omega)\} \to 0 $ if $ \sup_{\omega}\varepsilon'(\omega) \to 0 $.

		\item In general, assume there is a function $ \eta(\cdot): M \to \mathbb{R}_+ $ ($ \eta(\omega) > 0 $) such that $ \mu_u(\omega) - \mu_s(\omega) - 2\eta(\omega)  > 0 $ and $ \eta(\cdot) \in \mathcal{E}_1(t) $.
		For any given functions $ \varepsilon_1(\cdot): M \to \mathbb{R}_+ $ ($ \varepsilon_1(\omega) > 0 $), there are functions $ \vartheta(\cdot) $, $ \alpha(\cdot) $, $ \beta(\cdot) $, $ k(\cdot) $, $ \alpha_1(\cdot) $, $ \beta_1(\cdot) $ of $ M \to \mathbb{R}_+ $, and functions
		$ \tilde{\lambda}_s(\cdot) $, $ \tilde{\lambda}_u(\cdot) $ of $ M \to \mathbb{R} $, depending only on $ \delta_1(\cdot, \cdot) $, $ \mu_s(\cdot) $, $ \mu_u(\cdot), C_1 $, $ \varepsilon_1(\cdot) $, satisfying
		$ \vartheta(\cdot), \tilde{\lambda}_s(\cdot), \tilde{\lambda}_u(\cdot) \in \mathcal{E}_1(t) $, and $ \sup_{\omega} \{\alpha_1(\omega)- \alpha(\omega)\} > 0 $, $ \sup_{\omega} \{\beta_1(\omega)- \beta(\omega)\} > 0 $, $ \sup_{\omega}\{ \alpha_1(\omega), \beta_1(\omega) \} < 1 $,
		\[
		 \tilde{\lambda}_s(\omega) \leq e^{ \mu_s(\omega) + \eta(\omega)},~
		 \tilde{\lambda}_u(\omega) \leq e^{-\mu_u(\omega) +\eta(\omega)},
		 ~\tilde{\lambda}_s(\omega) \tilde{\lambda}_u(\omega) < 1,
		 ~ \omega \in M,
		\]
		such that if $ \varepsilon(\omega) < \vartheta(\omega) $, then the following hold.

		\begin{enumerate}[(a)]
			\item $ H(t, s\omega) $ (or $ U(t, s\omega) $) satisfies \textnormal{(A)}$(\alpha(\omega)$; $\alpha_1(\omega)$, $k(\omega) \tilde{\lambda}^t_u(\omega))$ \textnormal{(B)}$(\beta(\omega)$; $\beta_1(\omega)$, $k(\omega) \tilde{\lambda}^t_s(\omega))$ condition for all $ t,s \geq 0 $ and $ \omega \in M $.

			\item And for any $ \omega \in M $, if $ t \geq \varepsilon_1(\omega) $ then $ H(t, s\omega) $ (or $ U(t, s\omega) $) satisfies \textnormal{(A)}$(\alpha_1(\omega)$; $ \alpha(\omega) $, $k(\omega) \tilde{\lambda}^t_u(\omega))$ \textnormal{(B)}$(\beta_1(\omega)$; $ \beta(\omega) $, $k(\omega) \tilde{\lambda}^t_s(\omega))$ condition for all $ s \geq 0 $.

			\item $ \alpha(\cdot) $, $ \alpha_1(\cdot) $,  $ \beta(\cdot) $, $ \beta_1(\cdot) $, $ k(\cdot) $ can also be chosen as constants less than $ 1 $ under some appropriate choice of $ \vartheta(\cdot) $ (and so $ \varepsilon(\cdot) $). In fact, $ \sup_{\omega} \{\alpha_1(\omega), \beta_1(\omega)\} \to 0 $ as $ \sup_{\omega}\varepsilon(\omega) \to 0 $.

			\item \label{sp:dd} Let $ M $ be a locally metrizable space (see \autoref{rmk:amplitude} below) and $ M_1 \subset M $. If $ \mu_s(\cdot), \mu_u(\cdot), \eta(\cdot) $ are bounded and \emph{$ \xi $-almost uniformly continuous} around $ M_1 $ (see \autoref{rmk:amplitude} below), then $ \tilde{\lambda}_s(\cdot) $, $ \tilde{\lambda}_u(\cdot)$ can be chosen to be bounded and $ C\xi $-almost uniformly continuous around $ M_1 $ for some constant $ C > 0 $.

			\item If $ \inf_{\omega}\{ \mu_u(\omega) - \mu_s(\omega) \} > 0 $,
			then $ \sup_{\omega}\tilde{\lambda}_s(\omega) \tilde{\lambda}_u(\omega) < 1 $; if $ \sup_{\omega} \mu_{s}(\omega) < 0 $ (resp. $ \inf_{\omega} \mu_{u}(\omega) > 0 $), then $ \sup_{\omega}\tilde{\lambda}_s(\omega) < 1 $ (resp. $ \sup_{\omega}\tilde{\lambda}_u(\omega) < 1 $).

			\item If $ \pm \tau(\cdot) , \pm \varepsilon(\cdot) \in \mathcal{E}_1(t) $, then $ \pm\alpha(\cdot) $, $ \pm\alpha_1(\cdot) $,  $ \pm\beta(\cdot) $, $ \pm\beta_1(\cdot) $, $ \pm k(\cdot) $, $ \pm \tilde{\lambda}_s(\cdot) $, $ \pm \tilde{\lambda}_u(\cdot)$ can be chosen such that they all belong to $ \mathcal{E}_1(t) $.
		\end{enumerate}
	\end{enumerate}
\end{thm}

\begin{proof}
	For any given $ t_1 < t_2 $, let $ z(t) = (x(t),y(t)) $, $ z'(t) = (x'(t),y'(t)) $, $ t_1 \leq t \leq t_2 $, satisfy \eqref{**0}. Set $ \hat{z}(t) \triangleq (\hat{x}(t), \hat{y}(t)) \triangleq z(t) - z'(t) $.
	Let $ t_2 - t_1 = n \varepsilon_1 $. Consider
	\begin{align*}
	\hat{K}_m & = \widehat{S}(-m \varepsilon_1, t_2 \omega) P^c_{t_2 \omega} ( S_0 \Diamond ( \tilde{z}_{t_2 - m \varepsilon_1} (\cdot, \omega) ) ) ((t_2 - m \varepsilon_1) \omega) (m \varepsilon_1), \\
	K_m & = P^c_{(t_1 + m\varepsilon_1) \omega} ( S_0 \Diamond ( \tilde{z}_{t_1 } (\cdot, \omega) ) ) (t_1  \omega) (m \varepsilon_1),
	\end{align*}
	$ 0 \leq m \leq n $, where $ \tilde{z} (t, \omega) = f(t\omega) z(t) - f(t\omega) z'(t)  $. Note that $ \varepsilon(t\omega) \leq \varepsilon(\omega) $, $ t \geq 0 $, so $ |\tilde{z}(t,s\omega)| \leq \varepsilon(\omega) |\hat{z}(t)|  $ for $ t,s \geq 0 $.

	\begin{slem}\label{slem:11}
		\begin{enumerate}[(1)]
			\item If $ |\tilde{z}(t, \omega)| \leq e^{-\hat{\mu}(\omega)(t_2 - t)} b $ for $ t_1 \leq t \leq t_2 $, then
			\[
			|\hat{K}_n| \leq C_1b \delta_1(\varepsilon_1, \omega) \max\{ 1, e^{-\hat{\mu}(\omega)\varepsilon_1} \} e^{ -\mu_u(\omega) n \varepsilon_1 }  \frac{ e^{ -(\hat{\mu}(\omega) - \mu_u(\omega)) n \varepsilon_1 }- 1  }{ e^{ -(\hat{\mu}(\omega) - \mu_u(\omega))  \varepsilon_1 }- 1  }.
			\]
			\item If $ |\tilde{z}(t, \omega)| \leq e^{\hat{\mu}_1(\omega)(t - t_1)} a $ for $ t_1 \leq t \leq t_2 $, then
			\[
			|K_n| \leq C_1a \delta_1(\varepsilon_1,\omega) \max\{ 1, e^{\hat{\mu}_1(\omega)\varepsilon_1} \} e^{ \mu_s(\omega) (n-1) \varepsilon_1 } \frac{ e^{ (\hat{\mu}_1(\omega) - \mu_s(\omega)) n \varepsilon_1 }- 1  }{ e^{ (\hat{\mu}_1(\omega) - \mu_s(\omega))  \varepsilon_1 }- 1  }.
			\]
		\end{enumerate}
	\end{slem}
	\begin{proof}
		The proof is the same as \autoref{lem:prelem}. We only give the proof of (1). By \eqref{sf1}, we see that
		\begin{align*}
		\hat{K}_m = & \widehat{S}(-\varepsilon_1, (m-1)\varepsilon_1 \omega ) \hat{K}_{m-1} + \widehat{S}(-\varepsilon_1, (t_2 - (m-1)\varepsilon_1) \omega )  \\
		&\qquad\qquad \cdot P^c_{(t_2 - (m-1)\varepsilon_1) \omega} ( S_0 \Diamond ( \tilde{z}_{t_2 - m\varepsilon_1}(\cdot, \omega) ) ((t_2 - m \varepsilon_1) \omega) ) (\varepsilon_1).
		\end{align*}
		Thus, we have (by \eqref{spectra} and \eqref{s01})
		\begin{align*}
		|\hat{K}_m| \leq & e^{-\mu_u(\omega) \varepsilon_1} |\hat{K}_{m-1}| + e^{-\mu_u(\omega) \varepsilon_1} C_1 \delta_1(\varepsilon_1, (t_2 - m \varepsilon_1) \omega) \sup_{t \in [0,\varepsilon_1]} |\tilde{z}(t + t_2 - m \varepsilon_1,\omega)| \\
		\leq & e^{-\mu_u(\omega)\varepsilon_1} |\hat{K}_{m-1}| + e^{-\mu_u(\omega)\varepsilon_1}  C_1b \delta_1(\varepsilon_1,\omega) \max\{ 1, e^{\hat{\mu}(\omega)\varepsilon_1} \} e^{-\hat{\mu}(\omega) m\varepsilon_1},
		\end{align*}
		by induction, giving the proof.
	\end{proof}

	\textbf{Proof of (1).} For this case, all the main steps are the same as in the proof of \autoref{lem:simpleAB}.
	\begin{slem}
		Take any $ \alpha'(\cdot) $ such that $ \alpha(\omega) < \alpha'(\omega) < 1 $.
		Let $ \omega $ be fixed and $ t_1 <  t'_2 \leq t_2  $. If $ |\hat{x}(t)| \leq \alpha'(\omega) |\hat{y}(t)| $ for all $ t \in [t_1, t'_2] $, then $ |\hat{x}(t'_2)| < \alpha'(\omega) |\hat{y}(t'_2)| $ and $ |\hat{y}(t_1)| \leq \lambda^{t'_2 - t_1}_u(\omega) |\hat{y}(t'_2)| $. Moreover, if $ t'_2 - t_1 \geq \epsilon_1 $, then $ |\hat{x}(t'_2)| \leq k_{\alpha}(\omega,\epsilon_1)\alpha'(\omega) |\hat{y}(t'_2)| $ for some $ k_{\alpha}(\omega,\epsilon_1) < 1 $.
	\end{slem}
	\begin{proof}
		By the same argument in the proof of \autoref{lem:keylem} (using the above sublemma \textnormal{(1)} instead of \autoref{lem:prelem}), we have if $ t \in [t_1, t'_2] $, then
		\[
		|\hat{y}(t)| \leq e^{(-\mu_u(\omega) + \varepsilon'(\omega)) (t'_2 -t)}|\hat{y}(t'_2)|.
		\]
		Also, noting that in this case,
		\[
		|\tilde{z}(t, \omega)| \leq 2\varepsilon(\omega)e^{(-\mu_u(\omega) + \varepsilon'(\omega))(t'_2 - t_1)} e^{(\mu_u(\omega) - \varepsilon'(\omega))(t -t_1)}|\hat{y}(t'_2)|,
		\]
		by the above sublemma, when $ t'_2 - t_1 = n\varepsilon_1 $, $ n \geq 1 $, we see that
		\begin{align*}
		|\hat{x}(t'_2)| \leq &  e^{\mu_s(\omega) (t'_2 -t_1)} |\hat{x}(t_1)| + |K_n| \\
		\leq & \alpha'(\omega) e^{ -\sigma(\omega)  n\varepsilon_1} |\hat{y}(t'_2)| +  \hat{\delta} e^{-\sigma(\omega) n \varepsilon_1} e^{ -\mu_s(\omega)  \varepsilon_1 } \frac{ e^{ \sigma(\omega) n \varepsilon_1 }- 1  }{ e^{\sigma(\omega) \varepsilon_1 }- 1  }|\hat{y}(t'_2)|,
		\end{align*}
		where $ \hat{\delta} = \delta_1(\varepsilon_1,\omega) \varepsilon'(\omega) \max\{ 1, e^{(\mu_u(\omega)-\varepsilon'(\omega))\varepsilon_1} \} $ and $ \sigma(\omega) = \mu_u(\omega) - \mu_s(\omega) - \varepsilon'(\omega) > 0 $. Then $ |\hat{x}(t'_2)| < \alpha'(\omega) |\hat{y}(t'_2)| $, if
		\[
		\max\{ e^{-\mu_s(\omega)\varepsilon_1}, e^{\sigma(\omega) \varepsilon_1 } \} \delta_1(\varepsilon_1,\omega) \varepsilon'(\omega) \leq \alpha'(\omega) \sigma(\omega)\varepsilon_1 < \alpha'(\omega) (e^{\sigma(\omega) \varepsilon_1 } -1 ).
		\]
		Since $ \delta(t) / t \to 1 $ as $ t \to 0^+ $ and $ \frac{\varepsilon'(\omega)}{\sigma(\omega)} < \alpha'(\omega) < 1 $, the above inequality can be satisfied if $ \varepsilon_1 $ is taken sufficiently small (depending on $ \omega $ and $ t_2' $), i.e. $ n $ is sufficiently large.

		Let any $ \epsilon_1 > 0 $ and $ t'_2 - t_1 \geq \epsilon_1 $. In fact, choose $ k_0(\omega) < 1 $ such that $ \frac{\varepsilon'(\omega)}{\sigma(\omega)} < k_0(\omega)\alpha'(\omega) $, then for some small $ \varepsilon_1 > 0 $ such that $ \hat{\delta} e^{ -\mu_s(\omega) \varepsilon_1} \leq k_0(\omega)\alpha'(\omega) \sigma(\omega) \varepsilon_1 $ and $ t'_2 - t_1 = n \varepsilon_1 \geq \epsilon_1 $, we see that $ |\hat{x}(t'_2)| \leq \{(1 - k_0(\omega))e^{-\sigma(\omega)n\varepsilon_1} + k_0(\omega)\} \alpha'(\omega) |\hat{y}(t'_2)| $ and so $ |\hat{x}(t'_2)| \leq k_{\alpha}(\omega,\epsilon_1) \alpha'(\omega) |\hat{y}(t'_2)| $ where $ k_{\alpha}(\omega,\epsilon_1) = (1 - k_0(\omega))e^{-\sigma(\omega)\epsilon_1} + k_0(\omega) < 1 $.
	\end{proof}

	Now with the help of above sublemma, we can use the argument in the proof of \autoref{slem111} to conclude the proof of (1).

	\textbf{Proof of (2).}
	Let $ \varepsilon_1(\cdot): M \to \mathbb{R}_+ $ be any given function such that $ \varepsilon_1(\omega) > 0 $. Let $ \omega $ be fixed and $ t_1 \leq t'_1 < t_2  $.
	Choose $ \hat{\epsilon} = \hat{\epsilon}(\omega) > 0 $ such that
	\[
	1/2 > e^{1/2} \delta_1(2\hat{\epsilon},\omega)\max\{ e^{-\mu_s(\omega) 2\hat{\epsilon}}, 1 \},
	\]
	and
	\begin{equation}\label{aa1}
	2 (\mu_u(\omega) - \mu_s(\omega) - \eta(\omega)) > |\mu_s(\omega)| \hat{\epsilon}.
	\end{equation}
	Let
	\[
	K_1(\omega)(\epsilon) \triangleq e^{1/2} \delta_1(\epsilon,\omega)\max\{ e^{-\mu_s(\omega) \epsilon}, 1 \}.
	\]
	Let $ \varepsilon(\omega) $ satisfy $ 2C_1 \varepsilon(\omega)\delta_1(\hat{\epsilon}, \omega) < 1 $, and set
	\[
	\lambda(\omega) \triangleq 2C_1\varepsilon(\omega) \sup_{\epsilon \in [\hat{\epsilon}, 2\hat{\epsilon}]} \{K_1(\omega)(\epsilon)/\epsilon\},
	~~
	k(\omega) \triangleq (1 - 2C_1 \varepsilon(\omega)\delta_1(\hat{\epsilon}, \omega))^{-1} \max\{1,e^{-\mu_s(\omega) \hat{\epsilon}}\}.
	\]
	Also, we can assume $ \lambda(\omega) < \eta(\omega) $ by taking $ \varepsilon(\omega) $ further smaller. Set $ \tilde{\lambda}_s(\omega) = e^{\mu_s(\omega) + \lambda(\omega)} $.
	\begin{slem}
		If $ |\hat{y}(t)| \leq |\hat{x}(t)| $ for all $ t \in [t'_1, t_2] $, then
		\[
		|\hat{x}(t)| \leq e^{(\mu_s(\omega) + \lambda(\omega))(t-t'_1) } k(\omega) |\hat{x}(t'_1)|, ~t \in [t'_1, t_2].
		\]
	\end{slem}
	\begin{proof}
		The same argument in the proof of \autoref{lem:keylem} can be applied, so we omit the proof.
	\end{proof}
	\begin{slem}\label{slem:small}
		Let $ \varepsilon_1 = \varepsilon_1(\omega) > 0 $. Then there are $ \beta(\omega), \beta_1(\omega) \geq 0 $, $ \vartheta(\omega) > 0 $ satisfying $ \beta(\omega) < \beta_1(\omega) < 1 $ such that if $ \varepsilon(\omega) < \vartheta(\omega) $, then the following hold.
		\begin{enumerate}[(1)]
			\item If $ t_2 - t'_1 \leq \varepsilon_1 $ and $ |\hat{y}(t_2)| \leq \beta(\omega) |\hat{x}(t_2)| $, then $ |\hat{y}(t'_1)| \leq \beta_1(\omega) |\hat{x}(t'_1)| $. \label{small}
			\item If $ t_2 - t'_1 \geq \varepsilon_1 $ and $ |\hat{y}(t)| \leq \beta_1(\omega) |\hat{x}(t)| $ for all $ t \in [t'_1, t_2] $, then $ |\hat{y}(t'_1)| < \beta(\omega) |\hat{x}(t'_1)| $ and $ |\hat{x}(t_2)| \leq \tilde{\lambda}^{t_2 - t'_1}_s(\omega) k(\omega) |\hat{x}(t'_1)| $.
		\end{enumerate}
	\end{slem}
	\begin{proof}
		Set $ \hat{c}_1 \triangleq \max\{ 1, e^{\mu_s(\omega) \varepsilon_1}, e^{-\mu_u(\omega) \varepsilon_1} \} $, $ \hat{\delta}_1 \triangleq 2C_1 \delta_1(\varepsilon_1,\omega) \varepsilon(\omega) $.
		Assume
		\[
		\hat{c}_1 \hat{\delta}_1 = \hat{c}_1 2C_1 \delta_1(\varepsilon_1,\omega) \varepsilon(\omega) < 1,~
		\hat{c}_1 \hat{\delta}_1 / (  1 - \hat{c}_1 \hat{\delta}_1  ) < 1,
		\]
		by taking $ \varepsilon(\omega) $ further smaller.

		First assume $ t_2 - t'_1 \leq \varepsilon_1 $, then by \eqref{**0} \eqref{spectra} \eqref{s01}, we have
		\begin{align*}
		|\hat{x}|_{[t'_1,t_2]} & \leq \hat{c}_1 |\hat{x}(t'_1)| + 2C_1 \varepsilon(\omega)\delta_1(\varepsilon_1, \omega)|\hat{z}|^\sim_{[t'_1,t_2]}, \\
		|\hat{y}|_{[t'_1,t_2]} & \leq \hat{c}_1 |\hat{y}(t_2)| + \hat{c}_1 2C_1 \varepsilon(\omega)\delta_1(\varepsilon_1, \omega)|\hat{z}|^\sim_{[t'_1,t_2]},
		\end{align*}
		where $ |\hat{z}|^\sim_{[t'_1,t_2]} \triangleq \max\{ |\hat{x}|_{[t'_1,t_2]}, |\hat{y}|_{[t'_1,t_2]} \} $.
		Thus,
		\[
		|\hat{z}|^\sim_{[t'_1,t_2]} \leq \frac{\hat{c}_1}{1 - \hat{c}_1  \hat{\delta}_1} \max\{ |\hat{x}(t'_1)|, |\hat{y}(t_2)| \}.
		\]

		Assume $ |\hat{y}(t_2)| \leq \beta(\omega) |\hat{x}(t_2)| $, $ \beta(\omega) \leq 1 $.
		Then
		\begin{gather*}
		|\hat{x}(t_2)|  \leq \frac{e^{\mu_s(\omega)(t_2 - t'_1)} + \rho \hat{\delta}_1}{1-\beta(\omega) \rho \hat{\delta}_1 } |\hat{x}(t'_1)|, \\
		|\hat{y}(t'_1)| \leq  e^{-\mu_u(\omega)(t_2 - t'_1)} |\hat{y}(t_2)| + \hat{c}_1 \rho \hat{\delta}_1 \max\{ |\hat{x}(t'_1)|, |\hat{y}(t_2)| \}
		\leq  \beta_1(\omega) |\hat{x}(t'_1)|,
		\end{gather*}
		where
		\[
		\rho \triangleq \hat{c}_1  / (  1 - \hat{c}_1 \hat{\delta}_1  ),~
		\beta_1(\omega) \triangleq  ( k_1 + \hat{c}_1 \rho \hat{\delta}_1/\beta(\omega) ) \beta(\omega),~
		k_1 \triangleq  \frac{ 1  + \hat{c}_1 \rho \hat{\delta}_1 } {1- \rho \hat{\delta}_1 } > 1.
		\]
		In the above computation we use the fact that $ \mu_u(\omega) - \mu_s(\omega) > 0 $ and $ \rho \hat{\delta}_1 < 1 $.

		Suppose $ t_2 - t'_1 = n \varepsilon_0 $, $ n \geq 1 $, $ \varepsilon_0 \in [\varepsilon_1, 2\varepsilon_1] $, and $ |\hat{y}(t)| \leq \beta_1(\omega) |\hat{x}(t)| \leq |\hat{x}(t)| $ for all $ t \in [t'_1, t_2] $. Then by above sublemma and \autoref{slem:11} \textnormal{(1)} with
		\[
		|\tilde{z}(t, \omega)| \leq 2\varepsilon(\omega)e^{(\mu_s(\omega) + \lambda(\omega))(t_2 - t'_1)} e^{-(\mu_s(\omega) + \lambda(\omega))(t_2 -t)}k(\omega)|\hat{x}(t'_1)|,
		\]
		we have
		\begin{align*}
		|\hat{y}(t'_1)| & \leq e^{-\mu_u(\omega)(t_2 - t'_1)}|\hat{y}(t_2)| + |\hat{K}_n| \quad \text{(by the second equation of \eqref{**0})} \\
		& \leq \beta_1(\omega) k(\omega) e^{-\sigma_1(\omega) n \varepsilon_0} |\hat{x}(t_1')| + \hat{c}_2 \hat{\delta}_1 e^{-\sigma_1(\omega)n \varepsilon_0}  \frac{ e^{ \sigma_1(\omega) n \varepsilon_0 }- 1  }{ e^{\sigma_1(\omega) \varepsilon_0 }- 1  } k(\omega)|\hat{x}(t'_1)| \\
		& < \beta(\omega) |\hat{x}(t'_1)|,
		\end{align*}
		where $ \sigma_1(\omega) \triangleq \mu_u(\omega) - \mu_s(\omega) - \lambda(\omega) > 0 $ (since $ \lambda(\omega) < \eta(\omega) $) and $ \hat{c}_2 \triangleq \max\{ 1, e^{-2(\mu_s(\omega)+\lambda(\omega))\varepsilon_1} \} $.
		The last inequality can be satisfied if
		\begin{equation*}
		k_1 + ( \hat{c}_1 \rho + \hat{c}_2 ) \hat{\delta}_1 / \beta(\omega) < e^{\sigma_1(\omega) \varepsilon_0}/ k(\omega) < e^{2\sigma_1(\omega) \varepsilon_1}/ k(\omega).
		\end{equation*}

		Since $ e^{2\sigma_1(\omega) \varepsilon_1} > 1 $ and \eqref{aa1}, $ \varepsilon(\omega) $ can be made small enough such that the following two inequalities hold; first
		\[
		e^{2\sigma_1(\omega) \varepsilon_1} / k(\omega)  - k_1 > 0,
		\]
		and then
		\[
		\frac{( \hat{c}_1 \rho + \hat{c}_2 ) \hat{\delta}_1}{e^{2\sigma_1(\omega) \varepsilon_1}/ k(\omega) - k_1}  < \frac{1 - \hat{c}_1\rho \hat{\delta}_1}{k_1}  \leq \varsigma < 1,
		\]
		for some fixed constant $ \varsigma $ sufficiently close to $ 1 $.
		Now take $ \beta(\omega) $ such that
		\[
		\frac{( \hat{c}_1 \rho + \hat{c}_2 ) \hat{\delta}_1}{e^{2\sigma_1(\omega) \varepsilon_1}/ k(\omega) - k_1} < \beta(\omega) < \frac{1 - \hat{c}_1\rho \hat{\delta}_1}{k_1} ~(< 1).
		\]
		So $ \beta_1(\omega) < 1 $. The proof is complete.
	\end{proof}

	By the above sublemmas and using the argument in the proof of \autoref{slem111}, we obtain that $ H(t_2-t_1, t_1\omega) $ satisfies (a) and (b). (B) condition can be proved similarly. The other conclusions are obvious by our construction. So we complete the proof.
\end{proof}

\begin{rmk}[almost uniform continuity property]\label{rmk:amplitude}
	Let $ M $ be a Hausdorff topology space with an open cover $ \{ U_m: m \in M \} $. Assume (a) $ U_m $ is open and $ m \in U_m $, $ m \in M $; (b) every $ U_m $ is a metric space with metric $ d_m $ (might be incomplete) and in $ U_m $ the metric topology is the same as the subspace topology induced from $ M $.
	Then we say $ M $ is a \textbf{locally metrizable space}. Let $ U_m(\epsilon) \triangleq \{ m'\in U_m: d_m(m',m) < \epsilon \} $.
	Let $ g: M \to \mathbb{R} $, $ M_1 \subset M $. Define the \emph{amplitude} of $ g $ around $ M_1 $ as
	\[
	\mathfrak{A}_{M_1}(\sigma) = \sup \{ |g(m) - g(m_0)|: m \in U_{m_0}(\sigma), m_0 \in M_1 \}.
	\]
	We say $ g $ is \emph{$ \xi $-almost uniformly continuous around $ M_1 $} if $ \mathfrak{A}_{M_1}(\sigma) \leq \xi $ as $ \sigma \to 0 $. The notions are taken from \cite{Che18a}.
	Now \autoref{them:spec} \eqref{sp:dd} is obvious since we can choose $ \tilde{\lambda}_s(\omega) = e^{ \mu_s(\omega) + \eta(\omega)} $, $ \tilde{\lambda}_u(\omega) = e^{-\mu_u(\omega) +\eta(\omega)} $.
\end{rmk}

\begin{rmk}\label{rmk:sharp1}
	Similar as \autoref{thm:gencocycleAB} (2), the Lipschitz constant of $ f(\omega)(\cdot) $ can be computed in $ X_{\omega} \oplus Y_{\omega} $ with max norm $ |(x, y)| = \max\{|x|, |y|\} $, i.e.
	\[
	\sup\{ \lip f(t\omega)\mathcal{P}_{t\omega}(\cdot): \mathcal{P}_{t\omega} \in \{ P_{t\omega}, P^c_{t\omega} \}, t \geq 0 \} = \varepsilon_{m}(\omega);
	\]
	in this case, the result can be applied to the general case in \autoref{rmk:asymptotic} (by using the norm $ |\cdot|^{\sim}_{\omega} $).
	Also, we can take $ \varepsilon'(\omega) = C_1\varepsilon_{m}(\omega) $ in \autoref{them:spec} (1). The `spectral gap condition' $ \mu_u(\omega) - \mu_s(\omega) - 2 \varepsilon'(\omega) > 0 $ in \autoref{them:spec} (1) now seems new in the cocycle setting even for the case that $ A $ is a generator of a $ C_0 $ semigroup (so for the general $ C_0 $ cocycle case; note that what we really need is \eqref{**0}). The result might be also new even for the special case $ M = \{ \omega_0 \} $ and $ A $ is a Hille-Yosida operator.
	By some clever argument, one can give a very optimal estimate on `spectral gap condition' for the case when $ A_{\overline{D(A)}} $, the part of $ A $ in $ \overline{D(A)} $ (see \autoref{operator} for a definition), is a densely-defined sectorial operator and $ L: M \to L(\overline{D(A)},D(A^{-\alpha})) $ for some suitable $ \alpha > 0 $, which we do not give here; see also \cite[Theorem 2.15]{Zel14}. For the general case of \autoref{them:spec} (2), there is no optimal estimate.
\end{rmk}

\begin{rmk}
	Assume that there is a semigroup $ \widetilde{T}: \mathbb{R}_+ \times Z \to Z $ such that $ S \Diamond f = \widetilde{T} * f $ for all $ f \in C(\mathbb{R}, Z) $. For example, $ A $ is a generator of a $ C_0 $ semigroup or an analytic semigroup. For the latter case, $ \widetilde{T} $ might \emph{not} be a $ \bm{C_0} $ semigroup. Assume $ X = \widetilde{X}_\omega \oplus \widetilde{Y}_\omega $ with associated projections $ \widetilde{P}_\omega $, $ \widetilde{P}^c_\omega $, which are continuous extensions of $ P_\omega $, $ P^c_\omega $ and invariant about $ A(\omega) $.
	At this time, \eqref{**0} can be written in a more symmetry form as
	\[
	\begin{cases}
	x(t) = \widetilde{T}_1(t-t_1, t_1 \omega)x(t_1) + \int_{t_1}^{t} \widetilde{T}_1(t-s,s\omega)\widetilde{P}_{s\omega}f(s\omega)z(s\omega) ~\mathrm{d} s , \\
	y(t) = \widetilde{S}_1(t-t_2, t_2 \omega)y(t_2) - \int_{t}^{t_2} \widetilde{S}_1(t-s,s\omega)\widetilde{P}^c_{s\omega}f(s\omega)z(s\omega) ~\mathrm{d} s,
	\end{cases}
	\]
	where $ \widetilde{T}_1(t,\omega) : \widetilde{X}_\omega \to \widetilde{X}_{t\omega} $, $ \widetilde{S}_1(-t,t\omega) : \widetilde{Y}_{t\omega} \to \widetilde{Y}_{\omega} $ are continuous extensions of $ \widehat{T}(t,\omega), \widehat{S}(-t,t\omega) $ respectively. See also \cite{DPL88}. However, this form might fail when $ X $ can not be decomposed as two \textbf{closed} invariant subspaces about $ A(\omega) $.

	If $ S \Diamond f $ can not be written as $ \widetilde{T} * f $, then the continuous extensions of $ \widehat{T}, \widehat{S} $ might also fail. There are some authors who tried to give an extension of $ T_0(\cdot,\cdot) $ in some particular form, see e.g. \cite{MR09a}. The idea in that paper is that if $ \overline{D(A)} = X_\omega \oplus Y_\omega $, $ X = \hat{X}_\omega \oplus Y_\omega $, and $ X_\omega \subset \hat{X}_\omega $ with $ \hat{X}_\omega $ invariant under $ A(\omega) $, then use some method (see \cite{MR09a}, where the authors owned this method to H. R. Thieme) to give an extension of $ \int_{0}^{t} T_0(s, \omega)|_{X_{\omega}} ~\mathrm{d} s $ to $ \hat{X}_\omega $ under some context, which is denoted by $ \widetilde{S}_1(t,\omega) $. For this case, one has
	\[
	S_0(t,\omega) = (\widetilde{S}_1(t, \omega), \int_{0}^{t}T_0(s, \omega)|_{Y_{\omega}} ~\mathrm{d} s ):  \hat{X}_{\omega} \oplus Y_{\omega} \to \hat{X}_{t\omega} \oplus Y_{t\omega},
	\]
	and \eqref{**0} can be read as
	\[
	\begin{cases}
	x(t) = \widehat{T}(t-t_1, t_1 \omega)x(t_1) + ( \widetilde{S}_1 \Diamond ( \widetilde{P}_{(\cdot + t_1)\omega}f_{t_1}(\cdot\omega) z_{t_1}(\cdot)  ) ) (t_1 \omega) (t - t_1) , \\
	y(t) = \widehat{S}(t-t_2, t_2 \omega)y(t_2) - \int_{t}^{t_2} \widehat{S}(t-s,s\omega)\widetilde{P}^c_{s\omega}f(s\omega)z(s\omega) ~\mathrm{d} s,
	\end{cases}
	\]
	where $ \widetilde{P}_\omega $, $ \widetilde{P}^c_\omega $ are projections associated with $ X = \hat{X}_\omega \oplus Y_\omega $, which are continuous extensions of $ P_\omega $, $ P^c_\omega $. See \cite{MR09a} for details when $ M $ is a fixed point $ \{ \omega_0 \} $ and $ Y_{\omega_0} $ is finite dimensional. The existence of $ Y_{\omega_0} $ sometimes can be guaranteed under some `compact' assumption on $ T_0(\cdot, \omega_0) $, e.g. the \emph{essential growth bound} of $ T_0(\cdot, \omega_0) $ is strict less than the \emph{exponential growth bound} of $ T_0(\cdot, \omega_0) $, or especially $ T_0(t_0, \omega_0) $ is compact for some $ t_0 >0 $; see e.g. \cite[Section IV.1.20]{EN00} for details.
	The readers should notice that the uniform dichotomy assumption on \eqref{equ:linear} for our results in \autoref{HYccc} is only in $ \overline{D(A)} $ and there is no information in $ Z $.
\end{rmk}

\subsection{general $ C_0 $ case}\label{genccc}

Consider the following integral equation (i.e. equation \eqref{*06}),
\begin{equation} \label{equ:IE}
\begin{cases}
x(t) = T_1(t - t_1, t_1\omega)x(t_1) + \int_{t_1}^{t} T_1(t-s, s\omega) P_{s\omega} f(s\omega)z(s) ~\mathrm{d} s, \\
y(t) = S_1(t - t_2, t_2\omega)y(t_2) - \int_{t}^{t_2} S_1(t-s, s\omega) P^c_{s\omega} f(s\omega)z(s) ~\mathrm{d} s,
\end{cases}
t_1 \leq t \leq t_2,
\end{equation}
where $ \{(T_1(t,\omega), S_1(-t,t\omega)) \} $ is a $ C_0 $ cocycle correspondence satisfying uniform dichotomy on $ \mathbb{R}_+ $ (see \autoref{def:ud+}) and $ f $ satisfies assumption (D2) (in \autopageref{d2fff}). 

Using the solutions of above integral equation (i.e. \eqref{*06}), one can define a unique cocycle correspondence $ H(s, \omega) \sim (F_{s, \omega}, G_{s, \omega}) : X_{\omega} \oplus Y_{\omega} \to X_{s\omega} \oplus Y_{s\omega} $, i.e. $ (x_2, y_2) \in H(s,\omega)(x_1, y_1) $ if and only if there is a continuous $ z(t) = (x(t), y(t)) \in X_{t\omega} \times Y_{t\omega} $ ($ 0 \leq t \leq s $) satisfying \eqref{*06} with $ (x(0), y(0)) = (x_1, y_1) $ and $ (x(s), y(s)) = (x_2, y_2) $, where $ t_1 = 0, t_2 = s $. We say \emph{the cocycle correspondence $ H $ is induced by \eqref{equ:IE}}.

\begin{thm}\label{thm:gencocycleAB}
	Let the cocycle correspondence $ H $ be induced by \eqref{equ:IE}, where $ \{ (T_1(t,\omega), S_1(-t,t\omega)) \} $ is a $ C_0 $ cocycle correspondence satisfying uniform dichotomy on $ \mathbb{R}_+ $ (see \autoref{def:ud+}) and $ f $ satisfies assumption (D2) (in \autopageref{d2fff})
	\begin{enumerate}[(1)]
		\item The same results in \autoref{lem:00} hold for $ H $.
		\item Moreover, suppose $ { \mu_u(\omega) - \mu_s(\omega) - 2 \varepsilon_{m}(\omega) }  > 0 $ where
		\[
		\sup\{ \lip \mathcal{P}_{t\omega}f(t\omega)\mathcal{Q}_{t\omega}(\cdot): \mathcal{P}_{t\omega},\mathcal{Q}_{t\omega} \in \{ P_{t\omega}, P^c_{t\omega} \}, t \geq 0 \} = \varepsilon_{m}(\omega),
		\]
		i.e. the Lipschitz constant of $ f(\omega)(\cdot) $ is computed in $ X_{\omega} \oplus Y_{\omega} $ with max norm $ |(x, y)| = \max\{|x|, |y|\} $, $ (x, y) \in X_{\omega} \oplus Y_{\omega} $. Take $ \alpha, \beta, \lambda_u, \lambda_s $ such that
		\[
		\frac{\varepsilon_{m}(\omega)}{\mu_u(\omega) - \mu_s(\omega) -  \varepsilon_{m}(\omega)} \leq \alpha(\omega), \beta(\omega) < 1,~
		\lambda_u(\omega) = e^{-\mu_u(\omega) + \varepsilon_{m}(\omega) },~
		\lambda_s(\omega) = e^{\mu_s(\omega) + \varepsilon_{m}(\omega) }.
		\]
		Then $ H(t, s\omega) $ satisfies \textnormal{(A)}$(\alpha(\omega), \lambda^t_u(\omega))$ \textnormal{(B)}$(\beta(\omega), \lambda^t_s(\omega))$ condition for all $ t,s \geq 0 $ and $ \omega \in M $.
		In fact, if $ \alpha(\omega), \beta(\omega) \in (\frac{\varepsilon_{m}(\omega)}{\sigma(m)}, 1) $ and $ t \geq \epsilon_1 > 0 $, then $ H(t,s\omega) $ satisfies (A)($ \alpha(\omega); k_{\alpha}(\omega)\alpha(\omega), \lambda^t_{u}(\omega) $) (B)($ \beta(\omega); k_{\beta}(\omega)\beta(\omega), \lambda^t_{s}(\omega) $) condition, where
		\[
		\frac{( \sigma(\omega)- \frac{\varepsilon_{m}(\omega)}{h(\omega)})e^{-\sigma(\omega)\epsilon_1} + \frac{\varepsilon_{m}(\omega)}{h(\omega)}}{\sigma(\omega)} \leq k_{h}(\omega) < 1,~\sigma(\omega) = \mu_{u}(\omega) - \mu_s(\omega) - \varepsilon_{m}(\omega),
		~ h = \alpha,\beta.
		\]
		In particular, if there is a constant $ c > 1 $ such that
		\[
		\inf_{\omega}\{ \mu_u(\omega) - \mu_s(\omega) - (1+c) \varepsilon_{m}(\omega) \} > 0,
		\]
		then $ \alpha, \beta, k_{\alpha}, k_{\beta} $ can be taken constants less than $ 1 $ (for example $ \alpha = \beta = 1/c $ and $ k_{\alpha} = k_{\beta} = (1 - 1/c)r + 1/c  $ where $ r = e^{-\delta_0\epsilon_1} < 1 $ and $ \sigma(\omega) > \delta_0 > 0 $) and $ \sup_{\omega}\lambda_s(\omega) \lambda_u(\omega) < 1 $. In fact, $ \sup_{\omega} \alpha(\omega) $, $ \sup_{\omega}\beta(\omega) \to 0 $ if $ \sup_{\omega}\varepsilon_{m}(\omega) \to 0 $.
	\end{enumerate}
\end{thm}
\begin{proof}
	The proof of (1) is quite standard.
	The proof of (2) is essentially the same as \autoref{lem:simpleAB}; see also \autoref{thm:biAB}.
\end{proof}

Here, $ k_{h}(\omega) h(\omega) \in (\frac{\varepsilon_{m}(\omega)}{\sigma(m)}, 1) $, $ h = \alpha, \beta $. Note that \autoref{thm:biAB}, as well as \autoref{lem:simpleAB}, are consequences of \autoref{thm:gencocycleAB} and $ \varepsilon_{m}(\omega) \leq 2C_1\varepsilon(\omega) = \varepsilon'(\omega) $. 

What's the relation between the solutions of \eqref{equ:IE} and \eqref{equ:main}? We have dealt with this in \autoref{bi-semigroup} and \autoref{HYccc} for two special cases. However, in general, it's very complicated to talk about the solutions of \eqref{equ:main}; see also \cite{Sch02} for a survey. Here, we continue to consider other case which the solutions of \eqref{equ:main} can be described as mild solutions like \autoref{def:mild}. For brevity (in order to let our idea be presented clearly), let
\[\label{equ:cc1}\tag{$ \circledast $}
\mathcal{C}(\omega) = AB(\omega) + L(\omega), ~\omega \in M,
\]
where $ A: D(A) \subset Z \to Z $ is a closed linear operator, $ B: M \to L(Z, Z) $ and $ L: M \to L(Z, Z) $; here $ L $ satisfies (D1) (in \autopageref{d1L}) and similar for $ B $. In this case, A function $ u \in C([a,b], Z) $ is called a \textbf{(mild) solution} of \eqref{equ:main} if it satisfies (i) $ \int_{a}^{t} B(s\omega) u(s) ~\mathrm{ d } s \in D(A) $ for all $ t \in [a,b] $ and (ii) the following
\[
u(t) = u(a) + A \int_{a}^{t} B(s\omega) u(s) ~\mathrm{ d } s + \int_{a}^{t} (L(s\omega)u(s) + f(s\omega)u(s)) ~\mathrm{ d } s, ~t \in [a,b].
\]
(In fact, we can take $ B(\omega) $ only as closed linear operators but in this situation we need $ s \mapsto B(s\omega) u(s) $ is also continuous.) For a concrete example of $ \{\mathcal{C}(\omega)\} $ like \eqref{equ:cc1}, see \autoref{exa:non-auto}.

\begin{defi}\label{def:diffAndCocycle}
	Let $ \{\mathcal{C}(\omega)\} $ be as \eqref{equ:cc1}. 
	\begin{enumerate}[(a)]
		\item We say a $ C_0 $ cocycle $ \{U(t,\omega)\} $ (or $ U $) on $ M \times Z $ (or $ Z $) is generated by \eqref{equ:main} if for every $ x \in Z $, $ z(t) = U(t,\omega)x $, $ t \geq 0 $, is the \emph{unique} mild solution of \eqref{equ:main} with $ z(0) = x $.
		We say a $ C_0 $ (linear) cocycle $ \{T_0(t,\omega)\} $ (or $ T_0 $) on $ M \times Z $ (or $ Z $) is generated by $ \{\mathcal{C}(\omega)\} $ or \eqref{equ:linear} if \eqref{equ:linear} generates a $ C_0 $ cocycle $ \{T_0(t,\omega)\} $.
		
		\item A $ C_0 $ cocycle correspondence $ H_1 $ (or $ \{H_1(t,\omega)\} $) on $ M \times Z $ is generated (or induced) by \eqref{equ:main} or $ \{\mathcal{C}(\omega)\} $ if the assumptions (a) (b) in \autoref{def:ud+} hold (i.e. $ H_1(t,\omega) \sim (T_1(t,\omega), S_1(-t,t\omega)): X_{\omega} \oplus Y_{\omega} \to X_{t\omega} \oplus Y_{t\omega} $); in addition for $ z(t) \triangleq ( T_1(t-t_1, t_1 \omega)x_1, S_1(t-t_2,t_2 \omega)y_2 ) \in X_{t\omega} \oplus Y_{t\omega} $ ($ t_1 \leq t \leq t_2 $), $ z(\cdot) $ is the \emph{unique} mild solution of \eqref{equ:main} with $ P_{t_1\omega} z(t_1) = x_1 $ and $ P^c_{t_2\omega} z(t_2) = y_2 $.
	\end{enumerate}
\end{defi}

\begin{lem}\label{lem:mild}
	Let $ \{\mathcal{C}(\omega)\} $ be as \eqref{equ:cc1}. 
	\begin{enumerate}[(1)]
		\item Assume \eqref{equ:linear} generates a $ C_0 $ cocycle $ \{T_0(t,\omega)\} $. Then for any continuous $ g \in C(\mathbb{R}, Z) $, the mild solution $ z \in C([0,a], X) $ of the following
		\begin{equation*}\tag{$ \S $} \label{gg00}
		\dot{z}(t) = \mathcal{C}(t\omega)z(t) + g(t),
		\end{equation*}
		has the form
		\[\tag{$ \S\S $} \label{gg11}
		z(t) = T_0(t,\omega)z(0) + \int_{0}^{t} T_0(t-s,s\omega)g(s) ~\mathrm{ d } s, ~ t\in [0,a]. 
		\]
		\item Similarly, if \eqref{equ:linear} generates a $ C_0 $ cocycle correspondence $ \{H_1(t,\omega) \sim (T_1(t,\omega), S_1(-t,t\omega)) \} $ (see \autoref{def:diffAndCocycle}), then the mild solution $ z \in C([0,a], Z) $ of \eqref{gg00} has the form
		\begin{equation*}
		\begin{cases}
		x(t) = T_1(t, \omega)x(0) + \int_{0}^{t} T_1(t-s, s\omega) P_{s\omega} g(s) ~\mathrm{d} s, \\
		y(t) = S_1(t - a, a\omega)y(a) - \int_{t}^{a} S_1(t-s, s\omega) P^c_{s\omega} g(s) ~\mathrm{d} s,
		\end{cases}
		t \in [0,a],
		\end{equation*}
		where $ x(t) = P_{t\omega}z(t) $, $ y(t) = P^c_{t\omega}z(t) $.
	\end{enumerate}
\end{lem}

\begin{proof}
	(1) By the uniqueness of the solution of \eqref{equ:linear}, we only need to show \eqref{gg11} indeed is a solution of \eqref{gg00}. Let
	\[
	z_1(t) = T_0(t,\omega)z(0), ~ z_2(t) = \int_{0}^{t} T_0(t-s,s\omega)g(s) ~\mathrm{ d } s.
	\]
	Then $ A \int_{0}^{t} B(s\omega) z_1(s) ~\mathrm{ d } s = z_1(t) - z(0) - \int_{0}^{t} L(s\omega)z_1(s) ~\mathrm{ d } s $ and
	\begin{align*}
	A \int_{0}^{t} B(s\omega) z_2(s) ~\mathrm{ d } s & = A \int_{0}^{t} \int_{r}^{t} B(s\omega) T_0(s-r,r\omega) g(r) ~\mathrm{ d } s ~\mathrm{ d } r \\
	& = \int_{0}^{t} A \int_{r}^{t} B(s\omega) T_0(s-r,r\omega) g(r) ~\mathrm{ d } s ~\mathrm{ d } r \\
	& = \int_{0}^{t} \left\{T_0(t-r,r\omega)g(r) - g(r) - \int_{r}^{t} L(s\omega) T_0(s-r,r\omega) g(r) ~\mathrm{ d } s \right\} ~\mathrm{ d }r   \\
	& = z_2(t) - \int_{0}^{t} g(r) ~\mathrm{ d } r - \int_{0}^{t} \int_{0}^{s} L(s\omega) T_0(s-r,r\omega) g(r) ~\mathrm{ d } r ~\mathrm{ d }s \\
	& = z_2(t) - \int_{0}^{t} g(r) ~\mathrm{ d } r - \int_{0}^{t} L(s\omega)z_2(s) ~\mathrm{ d } s,
	\end{align*}
	where we exchange the order of integration and use the closeness of $ A $ (see e.g. \cite[Proposition 1.1.7]{ABHN11}), which yields $ z = z_1 + z_2 $ is the mild solution of \eqref{gg00}.
	The proof of (2) is almost identical with \autoref{lem:dich00}.
\end{proof}

Note that if \eqref{equ:linear} generates a $ C_0 $ cocycle $ T_0 $, then that \eqref{equ:linear} satisfies uniform dichotomy on $ \mathbb{R}_+ $ means that $ T_0 $ satisfies uniform dichotomy on $ \mathbb{R}_+ $ in the classical sense of \cite{CL99} (i.e. assumption (UD) in \autopageref{UD}).
Consider nonlinear differential equation \eqref{equ:main} with $ f $ satisfying (D2) in \autopageref{d2fff}, i.e.
\begin{equation*}
\dot{z}(t) = \mathcal{C}(t\omega)z(t) + f(t\omega)z(t).
\end{equation*}

\begin{lem}\label{lem:gencocycle}
	Let $ \{\mathcal{C}(\omega)\} $ be as \eqref{equ:cc1} and assumption (D2) (in \autopageref{d2fff}) hold.
	\begin{enumerate}[(1)]
		\item If $ \{ \mathcal{C}(\omega) \} $ generates a $ C_0 $ linear cocycle $ T_0 $ over $ t $, then \eqref{equ:main} also generates a $ C_0 $ (nonlinear) cocycle $ U $ over $ t $, i.e. for $ x \in X $, $ U(\cdot,\omega)x $ is the mild solution of \eqref{equ:main}.
		In this case, if $ f(\omega)(\cdot) \in C^r $, then $ U(t, \omega)(\cdot) \in C^r $, and in addition if $ (\omega,x) \mapsto D^r_xf(\omega)(x) $ is continuous, so is $ (t,\omega, x) \mapsto D^r_xU(t,\omega)(x) $.
		Moreover, $ U $ satisfies the following,
		\begin{equation}\label{ggg*1}
		U(t,\omega)x = T_0(t,\omega)x + \int_{0}^{t} T_0(t-s,s\omega) f(s\omega)U(s,\omega)x ~\mathrm{ d }s.
		\end{equation}
		\item Assume \eqref{equ:linear} generates a $ C_0 $ cocycle correspondence $ \{H_1(t,\omega) \sim (T_1(t,\omega), S_1(-t,t\omega)) \} $ (see \autoref{def:diffAndCocycle}).
		Then a continuous function $ z \in C([t_1, t_2], Z) $ is a mild solution of \eqref{equ:main} if and only if $ z(t) = (x(t), y(t)) \in X_{t\omega} \oplus Y_{t\omega} $ ($ t_1 \leq t \leq t_2 $) satisfies \eqref{equ:IE}, where $ x(t) = P_{t\omega}z(t) $, $ y(t) = P^c_{t\omega}z(t) $.
	\end{enumerate}
\end{lem}
\begin{proof}
	To prove (1), one can use \eqref{ggg*1} to define the unique $ U $ by applying standard argument using Banach Fixed Point Theorem. (Also the other statements follow the same way.) That $ U(\cdot,\omega)x $ is a mild solution of \eqref{equ:main} follows from \autoref{lem:mild} (1).
	To prove (2), first note that the solutions of \eqref{*06} are unique and then the conclusion follows from \autoref{lem:mild} (2).
\end{proof}

\begin{rmk}[parabolic type]
	Instead of using mild solutions to characterize the solutions of \eqref{equ:main}, one can use the so-called classical solutions especially for the \emph{parabolic} type (in the sense of Tanabe). We say a function $ u \in C([a,b], Z) \cap C^1((a,b], Z) $ is a \textbf{(singular) classical solution} of \eqref{equ:main} if $ u $ point-wisely satisfies \eqref{equ:main} for all $ t \in (a, b] $. In analogy with \autoref{def:diffAndCocycle}, a $ C_0 $ cocycle $ \{U(t,\omega)\} $ (resp. a $ C_0 $ cocycle correspondence $ \{H_1(t,\omega) \sim (T_1(t,\omega), S_1(-t,t\omega))\} $) on $ M \times Z $ (or $ Z $) is generated by \eqref{equ:main} if in the \autoref{def:diffAndCocycle}, `mild solution' is replaced by `(singular) classical solution'. In order to obtain the similar results like \autoref{lem:mild} where `mild solution' is replaced by `(singular) classical solution', additional regularity assumptions on $ \{T_0(t,\omega)\} $ (resp. $ T_1(t,\omega), S_1(-t,t\omega) $) and $ g(\cdot) $ are needed; see e.g. \cite[Section 5.6 and 5.7]{Paz83} for the non-autonomous case in $ [0, T] $ or \cite[Section VI.9.6]{EN00}. These regularity assumptions with other regularity assumption on $ (t, z) \mapsto f(t\omega)z $ will make \autoref{lem:gencocycle} hold; in particular, the solutions of \eqref{ggg*1} or \eqref{equ:IE} are the (singular) classical solutions of \eqref{equ:main}. The details will be given elsewhere.
\end{rmk}

\begin{rmk}[hyperbolic type]
	The description of the solutions of \eqref{equ:main} for the \emph{hyperbolic} type (in sense of Kato (see \cite[Section 5.3-5.5]{Paz83})) is intricate. The following is a natural but restricted way. We say a function $ u \in C^1([a,b], Z) $ is a \textbf{classical solution} of \eqref{equ:main} if $ u $ point-wisely satisfies \eqref{equ:main} for all $ t \in [a, b] $ (so particularly $ z(t) \in D(\mathcal{C}(t\omega)) $). A $ C_0 $ (Lipschitz) cocycle $ \{U(t,\omega)\} $ on $ M \times Z $ (or $ Z $) is generated by \eqref{equ:main} (i) if for every $ \omega \in M $ there is a dense linear space $ Y(\omega) \subset D(\mathcal{C}(\omega)) $ such that for every $ x \in Y(\omega) $, $ z(t) = U(t,\omega)x $, $ t \geq 0 $, is the classical solution of \eqref{equ:main} with $ z(0) = x $, and in addition, (ii) $ x \mapsto U(t, \omega)x $ is Lipschitz. Note that by (ii), $ U(t, \omega)(\cdot) $ is well defined in all $ Z $. Now assume the $ C_0 $ linear cocycle $ \{U(t,\omega)\} $ on $ M \times Z $ is generated by \eqref{equ:linear} (see also \cite[Section 5.3]{Paz83} and \cite[Section VI.9]{EN00}). In \cite{Paz83}, Pazy called the equation \eqref{gg11} as mild solutions of \eqref{gg00}. But this is not clear with the classical solutions of \eqref{gg00}. So one needs to give more regularity of \eqref{gg11}; there exist at least two ways: higher range-regularity or time-regularity of $ g $ (see e.g. \cite[Section 5.5]{Paz83}), which make equation \eqref{gg11} be the classical solutions of \eqref{gg00}. Using this way, \autoref{lem:gencocycle} also holds. Similar remark can be made for the $ C_0 $ cocycle correspondence. More precise statements will appear in our future work.
\end{rmk}

\begin{rmk}[Problem I: well-posed case]\label{rmk:defic}
	We have no result about \eqref{equ:main} in the general context when \eqref{equ:linear} generates a $ C_0 $ cocycle only on $ X_0 $, a closed subspace of $ X $. This will happen when we consider the delay equations or the age structured models in $ L^1 $ with nonlinear \emph{discrete} delay term which can be rewritten as an abstract equation \eqref{equ:main} with $ \mathcal{C}(\omega) = A + L(\omega) $, $ L: M \to L(\overline{D(A)},Z_{-1}) $, $ f: M \times \overline{D(A)} \to Z_{-1} $, where $ \overline{D(A)} \hookrightarrow Z_{-1} $ and $ A: D(A) \subset Z \to Z $ is a Hille-Yosida operator (but $ A_{\overline{D(A)}} $ is \emph{not} a sectorial operator); usually $ Z_{-1} \subset D(A^{-1}) $ with graph norm.
	What we can deal with for this situation is only the (type II) listed in \autoref{overview} (see \autoref{HYccc}).
\end{rmk}

\begin{rmk}[Problem II: ill-posed case]\label{rmk:laL}
	In \cite{dlLla09}, the author considered the case that $ \mathcal{C}(\omega) = A $, $ M = \{\omega\} $, where $ A $ is densely-defined \emph{bi-sectorial operator} (see \cite{vdMee08,dlLla09}) under the exponential trichotomy condition and the nonlinear map $ f $ is also allowed to be unbounded. One needs to explain more about the unbounded perturbation $ f $ in \cite{dlLla09}. $ f $ is a $ C^1 $ map of $ Z \to Z_{-1} $ where $ Z \hookrightarrow Z_{-1}  $ (for instance $ Z_{-1} = D(A^{-\alpha}) $, $ 0 < \alpha < 1 $), $ Z = X_s \oplus X_c \oplus X_u $ and $ Z_{-1} = Y_s \oplus Y_c \oplus Y_u $. Let $ A_h = A_{X_h} $, $ h = s,c,u $. A technical restriction is that $ \mathrm{ e }^{tA_h} : Y_h \to X_h ~ (\subset Y_h) $, $ h = s,c,u $. Although for some applications the additional restriction can be satisfied (see e.g. the examples in \cite{dlLla09}), in abstract setting this is not always true. See the operator on a Hilbert space $ \widetilde{H} $ given in \cite[Theorem 9]{MY90} which will be denoted by $ \widetilde{A} $. If we choose $ Z = \overline{D(\widetilde{A})} $, $ Z_{-1} = \widetilde{H} $, $ A = \widetilde{A}_{X} $, due to that $ \widetilde{A} $ is not decomposable, the assumption on $ Z_{-1} $ in the previous paper can not be satisfied. The readers might consider further about \autoref{ex:cylinder} \eqref{cya11} for the case $ X = C^1_0(\Omega) \times C(\Omega) $ and \autoref{ex:nondiss} \eqref{mixed}. However, if all the settings in \cite{dlLla09} are satisfied, the same argument in the proof of \autoref{thm:biAB} can be applied to show the cocycle correspondence generated by \eqref{equ:main} satisfies (A) (B) condition, so our results about invariant manifolds (see \autoref{continuous case} and \autoref{normalH}) can be applied to this situation.
	In general, we have no idea how to study \eqref{equ:main} when \eqref{equ:linear} satisfies uniform dichotomy on $ \mathbb{R}_+ $ only in a closed subspace of $ Z $, unlike the case (type III) listed in \autoref{overview} (see \autoref{genccc}).
\end{rmk}

%% file: sect3.tex
\section{Invariant manifolds and Applications}\label{ManifoldApp}

In this section, we give some applications of our abstract results about (discrete) `dynamical systems' in \cite{Che18a, Che18b}, i.e. the existence and regularity of invariant graphs for cocycle correspondences with (relatively) uniform hyperbolicity and approximately normal hyperbolicity theory, to differential equations with the help of the results obtained in \autoref{diffCocycle}. The restatements of the main results in \cite{Che18a, Che18b} are given in \autoref{continuous case} and \autoref{normalH} but in a version of continuous `dynamical systems', included for the reader's convenience. The detailed application to (abstract) differential equations is contained in \autoref{generalA} and \autoref{generalB} from an abstract viewpoint.

\subsection{existence and regularity of invariant graphs: the continuous case}\label{continuous case}

In \cite{Che18a}, we gave a detailed study of existence and regularity of invariant graphs for bundle correspondences. A number of applications were given to illustrate the power of this theory. In the following, we restate the corresponding results for cocycle correspondences which can be applied to some classes of differential equations introduced in \autoref{diffCocycle} (see also \autoref{examples} for some concrete examples). Verification of the following hyperbolicity condition (or (A$ '_1 $) (B$ _1 $) condition) has been already given in \autoref{diffCocycle}.

\emph{For the convenience of writing, we write the metrics $d(x,y) \triangleq |x-y|$}. A bundle with metric fibers means each fiber is a complete metric space.

\begin{enumerate}[$ \bullet $]
	\item  In \autoref{thmAc} and \autoref{thmBc}, we take functions $ \eta: \mathbb{R}_+ \times M \rightarrow \mathbb{R}_+ $ and $ \varepsilon_1(\cdot), \varepsilon(\cdot): M \to \mathbb{R}_+ $, such that they satisfy $ \eta(t, r\omega) \leq \varepsilon^{r}(\omega) \eta(t,\omega) $, and $ 0 \leq \varepsilon(\omega) \leq \varepsilon_1(\omega) $, for all $ \omega \in M $ and $ t , r \geq 0 $.
\end{enumerate}

For a cocycle correspondence $ H $ over $ t $, we say $ H $ satisfies \textbf{($ \bm{\mathrm{A'_1}} $)($ \bm{\alpha, \lambda_u; c} $) ($\bm{ \mathrm{B_1}} $)($\bm{ \beta; \beta', \lambda_s; c }$) condition}, if every $ H(t, \omega) \sim (F_{t,\omega}, G_{t,\omega}) $ satisfies (i) (A$ ' $)$(\alpha(\omega), c(\omega)\lambda^{t}_u(\omega))$ (B)$(\beta'(\omega); \beta(\omega), c(\omega)\lambda^{t}_s(\omega))$ condition (see \autoref{defAB}), and (ii) if $ t \geq \hat{\varepsilon}_1 > 0 $, (B)$(\beta(\omega); \beta'(\omega), c(\omega)\lambda^{t}_s(\omega))$ condition (see \autoref{defAB}). This class of hyperbolic condition is motivated by \autoref{them:spec} (2). But this is also satisfied by the (A) (B) condition studied in \autoref{thm:biAB}, \autoref{thm:gencocycleAB} and \autoref{them:spec} (1); for this case $ c(\omega) \equiv 1 $ and $ \beta \equiv \beta' $.

\begin{thm}\label{thmAc}
	Let $(X, M, \pi_1), (Y, M, \pi_2)$ be two bundles with metric fibers and $t: M \rightarrow M$ a semiflow.
	Let $H: \mathbb{R}_+ \times X \times Y \rightarrow X \times Y$ be a cocycle correspondence over $t$ with a generating cocycle $(F,G)$.
	Assume that the following hold.
	\begin{enumerate}[(i)]
		\item ($\varepsilon$-pseudo-stable section) $ i = (i_X, i_Y): M \to X \times Y $ is an $\varepsilon$-\textbf{pseudo-stable section} of $H$, i.e.
		\[
		| i_X(t\omega) - F_{t,\omega}( i_X(\omega), i_Y(t\omega) ) | \leq \eta(t,t\omega), ~ | i_Y(\omega) - G_{t, \omega}( i_X(\omega), i_Y(t \omega) ) | \leq \eta(t, \omega).
		\]
		\item $H$ satisfies $(\mathrm{A'_1})$$(\alpha, \lambda_u; c)$ $( \mathrm{B_1}) $$(\beta; \beta', \lambda_s; c)$ condition, where $\alpha$, $\beta$, $\beta'$, $\lambda_s$, $ \lambda_u $, $c$ are functions of $M \rightarrow \mathbb{R}_+$. In addition,
		\begin{enumerate}[(a)]
			\item (angle condition) $\sup_\omega \alpha(\omega) \beta'(\omega) < 1 $, $ \sup_\omega \{\alpha(\omega), \beta(\omega)\} < \infty $, $ \beta'(t \omega) \leq \beta(\omega)$, $ \forall \omega \in M  $ and $ t \geq 0 $,
			\item (spectral condition) $\sup_\omega \lambda_u(\omega) \lambda_s(\omega) < 1 $, $ \sup_\omega \lambda_u(\omega) \varepsilon_1(\omega) < 1 $. $ \sup_\omega c(\omega) < \infty $.
		\end{enumerate}
	\end{enumerate}
	Then there is a unique bundle map $ f: X \to Y $ over $\id$ satisfying the following \textnormal{(1) (2)}.
	\begin{enumerate}[(1)]
		\item $\lip f_\omega \leq \beta'(\omega) $, $| f_\omega(i_X(\omega))  - i_Y(\omega)| \leq K \inf_{r \geq t_0} \eta(r,\omega)$, where $K \geq 0$ is a constant and $ t_0 > 0 $ is large.
		\item $\graph f \subset H^{-1}\graph f$. More precisely, $ (x_{t,\omega}(x), f_{t \omega}(x_{t,\omega}(x)) ) \in H(t, \omega)(x, f_\omega(x)) , ~ \forall x \in X_\omega $, where $x_{t,\omega}(\cdot): X_\omega \to X_{t\omega} $ such that $\lip x_{t,\omega}(\cdot) \leq c(\omega)\lambda_s^t(\omega)$, $| x_{t,\omega}(i_X(\omega))  - i_X(t \omega)| \leq K_0(\eta(t,t\omega)+\inf_{r \geq t_0} \eta(r,t\omega)) $, for all $ \omega \in M $, $ t \geq 0 $ and some constant $ K_0 $.
	\end{enumerate}
\end{thm}

Before we prove the result, the following elementary fact will be needed; the proof is easy, see also \cite[Section 4.2]{Che18a}.
We use the symbols: $\sum_\lambda (X, Y) \triangleq \{ \varphi: X \rightarrow Y~ : \lip \varphi \leq \lambda \}$, if $X, Y$ are metric spaces, and
$\graph f \triangleq \{ (x, f(x)): x \in X \}$, if $f: X \rightarrow Y$ is a map.

\begin{lem}\label{lem:gr} \label{lem:bas1}
	Assume that $ X_i, Y_i $, $ i=1,2 $, are complete metric spaces, and $ H: X_1 \times Y_1 \rightarrow X_2 \times Y_2 $ is a correspondence with a generating map $ (F,G) $ satisfying (A$'$)$(\alpha, \lambda_u)$ (B)$(\beta; \beta', \lambda_s)$ condition. Take positive $ \hat{\beta} $ such that $ \alpha \hat{\beta} < 1 $ and $ \hat{\beta} \leq \beta $.

	\begin{enumerate}[(1)]
		\item Let $f_2 \in \sum_{\hat{\beta}} (X_2, Y_2)$. Then there exist unique $f_1 \in \sum_{\beta'} (X_1, Y_1)$ and $x_1(\cdot) \in \sum_{\lambda_s} (X_1, X_2)$, such that $ (x_1(x), f_2(x_1(x))) \in H(x, f_1(x)) $, $ x \in X_1$, i.e.
		\[
		F(x, f_2(x_1(x))) = x_1(x), ~ G(x, f_2(x_1(x))) = f_1(x).
		\]
		\item Under (1), take $ (\hat{x}_i, \hat{y}_i) \in X_i \times Y_i $, $ i =1,2 $, such that
		\[
		|F(\hat{x}_1, \hat{y}_2) - \hat{x}_2| \leq \eta_1, |G(\hat{x}_1, \hat{y}_2) - \hat{y}_1| \leq \eta_2,~|f_2(\hat{x}_2) - \hat{y}_2| \leq C_2.
		\]
		Then we have the following estimates:
		\[
		|f_1(\hat{x}_1) - \hat{y}_1| \leq \lambda_u \frac{\hat{\beta}\eta_1 + C_2}{1-\alpha \hat{\beta}} + \eta_2, ~|x_1(\hat{x}_1) - \hat{x}_2| \leq  \frac{\alpha C_2 + \eta_1
		}{1-\alpha \hat{\beta}}.
		\]
		Particularly, if $(x_2, y_2) \in H(x_1, y_1)$, and $y_2 = f_2(x_2)$, then $y_1 = f_1(x_1), ~ x_2 = x_1(x_1)$.
	\end{enumerate}

\end{lem}

\begin{proof}[Proof of \autoref{thmAc}]
	We use a standard method by reducing it to the discrete case which was obtained in \cite{Che18a}.
	By condition (ii) (a)(b), there is a sufficiently large $ t_0 > \max\{1,\hat{\varepsilon}_1\} $ such that
	\begin{equation*}
	\sup_\omega \frac{ c^2(\omega)\lambda_u^{t_0}(\omega) \lambda_s^{t_0}(\omega) + c(\omega)\lambda_u^{t_0}(\omega) \varepsilon_1^{t_0}(\omega)}{1 - \alpha(\omega) \beta'(t_0\omega)}   < 1.
	\end{equation*}
	Set
	\begin{equation}\label{discrete}
	\widehat{H}(\omega) = H(t_0,\omega), ~
	\hat{u}(\omega) = t_0 \omega, ~
	\hat{\lambda}_s(\omega) = c(\omega)\lambda_s^{t_0}(\omega), ~ \hat{\lambda}_u(\omega) = c(\omega)\lambda_u^{t_0}(\omega).
	\end{equation}
	Then $ \widehat{H} $ is a bundle correspondence over $ \hat{u} $ satisfying (A$ ' $)$(\alpha, \hat{\lambda}_u)$ (B)$(\beta; \beta', \hat{\lambda}_s)$ condition and $ i $ now is an $\varepsilon^{t_0}$-pseudo-stable section of $ \widehat{H} $ (by letting $ \eta(\omega) = \eta(t_0,\omega) $), i.e.
	\[
	| i_X(\hat{u}(\omega)) - F_{t_0,\omega}( i_X(\omega), i_Y(\hat{u}(\omega)) ) | \leq \eta(\hat{u}(\omega)), ~ | i_Y(\omega) - G_{t_0,\omega}( i_X(\omega), i_Y(\hat{u}(\omega)) ) | \leq \eta(\omega),
	\]
	with $ \eta(\hat{u}(\omega)) \leq \varepsilon^{t_0}(\omega) \eta(\omega) $.
	By the first existence results in \cite[Section 4.1]{Che18a}, there is a unique bundle map $ f: X \to Y $ over $ \id $ such that
	\begin{enumerate}[(a)]
		\item $\lip f_\omega \leq \beta'(\omega) $, $| f_\omega(i_X(\omega))  - i_Y(\omega)| \leq K \eta(\omega)$;
		\item $ \graph f_\omega \subset \widehat{H}(\omega)^{-1} \graph f_{t_0\omega} $ for all $ \omega \in M $, where $K \geq 0$ is a constant independent of for large $ t_0 $; in fact, $ K $ can be taken as $ K = \frac{\overline{\lambda}_1\sup_{\omega}\beta(\omega) + 1}{1 - \overline{\lambda}_1} $, where $ \overline{\lambda}_1 = \sup_{\omega} \frac{ c(\omega)\lambda_u^{t_*}(\omega) \varepsilon_1^{t_*}(\omega)}{1 - \alpha(\omega) \beta'(\omega)} < 1 $, if $ t_0 \geq t_* $.
		\item $ f $ does not depend on the choice of $ \eta^0(\cdot) = \eta(\cdot) $ as long as it satisfies $ \eta^0(\hat{u}(\omega)) \leq \varepsilon^{t_0}_1(\omega) \eta^0(\omega) $, $ \eta(\omega) \leq \eta^0(\omega) $, for all $ \omega \in M $.
	\end{enumerate}

	Next, we need to show
	\[
	\graph f_\omega \subset H(t,\omega)^{-1} \graph f_{t \omega},~\text{if}~ 0 < t < t_0 . \tag{$ \star $}
	\]
	Fix $ t $. Since $\sup_\omega \alpha(\omega) \beta'(\omega) < 1 $, by \autoref{lem:gr} (1), there is a unique bundle map $ f': X \to Y $ over $ \id $ such that $ \graph f'_\omega \subset H(t,\omega)^{-1} \graph f_{t\omega} $, $ \lip f'_\omega \leq \beta(\omega) $. Also, by \autoref{lem:bas1} (2), it holds that $| f'_\omega(i_X(\omega))  - i_Y(\omega)| \leq K' (\eta(t,\omega) + \eta(t_0,\omega))$ for some fixed constant $ K' $ independent of $ t $. If we show
	\[
	\graph f'_{\omega} \subset \widehat{H}(\omega)^{-1} \graph f'_{t_0 \omega}, ~\forall \omega \in M, \tag{$ \star \star $}
	\]
	then by (B$ _1 $) condition, it also holds $ \lip f'_{\omega} \leq \beta'(\omega) $; and so by the uniqueness of $ f $, we get ($ \star $) holds. Observe that, by the cocycle property of $ H $,
	\[
	\graph f'_{\omega} \subset H(t,\omega)^{-1} \graph f_{t\omega} \subset H(t,\omega)^{-1} \widehat{H}(t\omega)^{-1} \graph f_{t_0(t\omega)} \subset H(t+t_0,\omega)^{-1} \graph f_{(t+t_0)\omega}.
	\]
	So $ \graph f'_\omega \subset \widehat{H}(\omega)^{-1} H(t, t_0 \omega)^{-1} \graph f_{t(t_0\omega)} $, which yields that ($ \star \star $) holds; i.e. if for any $ x \in X_{\omega} $, there are $ (x_1, y_1) \in H(t, t_0 \omega)^{-1} \graph f_{t(t_0\omega)} $ and $ x_2 \in X_{t(t_0\omega)} $ such that $ (x_1, y_1) \in \widehat{H}(\omega) (x, f'_{\omega}(x)) $ and $ (x_2, f_{t(t_0\omega)}(x_2)) \in H(t, t_0 \omega)(x_1, y_1) $, so by \autoref{lem:gr} (2), $ y_1 = f'_{t_0\omega}(x_1) $, and thus $ (x, f'_{\omega}(x)) \in \widehat{H}(\omega)^{-1} \graph f'_{t_0 \omega} $.

	It follows from (B$ _1 $) condition that $\lip x_{t,\omega}(\cdot) \leq c(\omega)\lambda_s^t(\omega)$. By \autoref{lem:bas1} (2), one gets $| x_{t,\omega}(i_X(\omega))  - i_X(t \omega)| \leq K_0(\eta(t,t\omega)+\inf_{r \geq t_0} \eta(r,t\omega)) $, where $ K_0 = \frac{\max\{K, K'\}\sup_{\omega}\alpha(\omega) + 1}{1 - \alpha(\omega) \beta'(\omega)} $. The proof is complete.
\end{proof}

The following theorem can be proved as the same way as proving \autoref{thmAc} by using the second existence theorem in \cite[Section 4.1]{Che18a}. We use the \emph{notation} $\widetilde{d} (A, z) \triangleq \sup_{\tilde{z} \in A} d(\tilde{z}, z) $, if $A$ is a subset of a metric space. \label{notationDD}

\begin{thm}\label{thmBc}
	Let $(X, M, \pi_1), (Y, M, \pi_2)$ be two bundles with metric fibers and $t: M \rightarrow M$ a semiflow.
	Let $H: \mathbb{R}_+ \times X \times Y \rightarrow X \times Y$ be a cocycle correspondence over $t$ with a generating cocycle $(F,G)$.
	Assume that the following hold.
	\begin{enumerate}[(i)]
		\item ($\varepsilon$-$Y$-bounded-section) $ i = (i_X, i_Y): M \to X \times Y $ is an $\varepsilon$-$Y$-\textbf{bounded-section} of $H$, i.e.
		\[
		\widetilde{d}(G_{t,\omega}( X_\omega, i_Y(t\omega) ) , i_Y(\omega)) \leq \eta(t,\omega).
		\]
		\item $H$ satisfies $(\mathrm{A'_1})$$(\alpha, \lambda_u; c)$ $( \mathrm{B_1}) $$(\beta; \beta', \lambda_s; c)$ condition, where $\alpha, \beta, \beta', \lambda_u, \lambda_s, c$ are functions of $M \rightarrow \mathbb{R}_+$. In addition,
		\begin{enumerate}[(a)]
			\item (angle condition) $\sup_\omega \alpha(\omega) \beta'(\omega) < 1 $, $ \sup_\omega \{\alpha(\omega), \beta(\omega)\} < \infty $, $ \beta'(t\omega) \leq \beta(\omega)$, for all $ \omega \in M  $, $ t \geq 0 $.
			\item (spectral condition) $\sup_\omega \lambda_u(\omega) \varepsilon_1(\omega) < 1$. $ \sup_\omega c(\omega) < \infty $.
		\end{enumerate}
	\end{enumerate}
	Then there is a unique bundle map $ f: X \to Y $ over $\id$ satisfying the following \textnormal{(1) (2)}.
	\begin{enumerate}[(1)]
		\item $\lip f_\omega \leq \beta'(\omega) $, $\widetilde{d}( f_\omega(X_\omega), i_Y(\omega)) \leq K \eta(0,\omega)$, where $K \geq 0$ is a constant.
		\item $ \graph f \subset H^{-1} \graph f$. More precisely, $ (x_{t,\omega}(x), f_{t \omega}(x_{t,\omega}(x)) ) \in H(t, \omega)(x, f_\omega(x)) $, $ \forall x \in X_\omega $, where $ x_{t,\omega}(\cdot): X_\omega \to X_{t\omega} $ such that $\lip x_{t,\omega}(\cdot) \leq c(\omega)\lambda_s^t(\omega)$, for all $ \omega \in M $, $ t \geq 0 $. Moreover, if
		\[
		\widetilde{d}(F_{t,\omega}(X_{\omega}, i_Y(t \omega)) , i_X(t\omega)) \leq \eta(t, t\omega),~\forall (t,\omega) \in \mathbb{R}_+ \times M,
		\]
		then $\widetilde{d}( x_{t,\omega}(X_\omega), i_X(t \omega) ) \leq K_0 \eta(t,t\omega)$ for some constant $ K_0 $.
	\end{enumerate}
\end{thm}

We state a local version of the existence result for the strong stable case. The thresholds in angle condition and spectral condition are a little different from \autoref{thmAc} and \autoref{thmBc}.
\begin{thm}\label{thm:local}
	Let $(X, M, \pi_1), (Y, M, \pi_2)$ be two bundles with fibers being Banach spaces and $t: M \rightarrow M$ a semiflow. Take functions $ \varepsilon_1(\cdot), \varepsilon(\cdot): M \to \mathbb{R}_+ $. Let $ H: \mathbb{R}_+ \times X \times Y \to X \times Y $ be a cocycle correspondence over $ t $.
	For every $ (t,\omega) \in \mathbb{R}_+ \times M $, suppose that
	\[
	H(t,\omega) \sim (F_{t,\omega},G_{t,\omega}): X_{t,\omega}(r_{t,1}) \times Y_{t,\omega}(r'_{t,1}) \rightarrow X_{t\omega}(r_{t,1}) \times Y_{t\omega}(r'_{t,2}),
	\]
	satisfies (A$ ' $)($ \alpha(\omega) $, $ c(\omega)\lambda^{t}_u(\omega) $) (B)($ \beta'(\omega) $; $ \beta(\omega) $, $ c(\omega)\lambda^{t}_s(\omega) $) condition, and if $ t \geq \hat{\varepsilon}_1 > 0 $, (B)($ \beta(\omega) $; $ \beta'(\omega) $, $ c(\omega)\lambda^{t}_s(\omega) $) condition,
	where $ \sup\{r_{t,i}, r'_{t,i}: i = 1, 2, t\in[0,b] \} > 0 $ for any $ b > 0 $.
	\begin{enumerate}[(i)]
		\item
		Assume for a fixed $ t_0 > 0 $,
		\[
		\sup_{t\in [0,t_0]}| F_{t,\omega}( 0, 0 ) | \leq \eta(t_0\omega), ~ \sup_{t\in [0,t_0]}| G_{t, \omega}( 0, 0 ) | \leq \eta(\omega),
		\]
		where $ \eta: M \rightarrow \mathbb{R}_+ $, with $ \eta(t_0\omega) \leq \varepsilon^{t_0}(\omega) \eta(\omega) $ and $ 0 \leq \varepsilon(\omega) \leq \varepsilon_1(\omega) $, $ \forall \omega \in M $.
		\item
		Assume
		\begin{enumerate}[(a)]
			\item (angle condition) $ \sup_{\omega} \alpha(\omega) \beta'(\omega) < 1 / 2 $, $ \sup_{\omega} \{ \alpha(\omega), \beta(\omega) \} < \infty $, $ \beta'(t\omega) \leq \beta(\omega)$, $ \forall \omega \in M $, $ t \geq 0 $,
			\item (spectral condition) $\sup_\omega \lambda_u(\omega) \lambda_s(\omega) < 1 $, $ \sup_\omega \lambda_u(\omega) \varepsilon_1(\omega) < 1 $, $ \sup_\omega \lambda_s(\omega) < 1 $, and $ \sup_{\omega} c(\omega) < \infty $.
		\end{enumerate}
	\end{enumerate}
	If $ \eta_0 > 0 $ is small and $ \sup_{\omega}\eta(\omega) \leq \eta_0 $, then there is a small $ \sigma_0 > 0 $ such that there are maps $ f_\omega: X_\omega(\sigma_0) \to Y_\omega $, $ \omega \in M $, uniquely satisfying the following (1) (2).
	\begin{enumerate}[(1)]
		\item $\lip f_{\omega} \leq \beta'(\omega) $, $| f_{\omega}(0) | \leq K \eta(\omega)$, for some constant $ K \geq 0 $.
		\item $ \graph f_{\omega} \subset H(t,\omega)^{-1} \graph f_{t\omega} $. More precisely, $ (x_{t,\omega}(x), f_{t\omega}(x_{t,\omega}(x)) ) \in H(t,\omega)(x, f_{\omega}(x)) $, $ \forall x \in X_{\omega}(\sigma_0) $, where $ x_{t,\omega}(\cdot): X_{\omega}(\sigma_0) \to X_{t\omega}(\sigma_0) $ such that $\lip x_{t,\omega}(\cdot) \leq c(\omega)\lambda^{t}_s(\omega)$. Moreover, $ \sup_{t\in[0,t_0]}| x_{t,\omega}(0) | \leq K_0\eta(\omega) $, $\forall \omega \in M$, where $ K_0 > 0 $.
	\end{enumerate}
\end{thm}

\begin{proof}
	First note that we can assume, without loss of generality, $ \sup_{\omega} \varepsilon(\omega) \leq 1 $. This can be argued as follows. Take $ \eta_1(\omega) = \sup_{t \geq 0} \eta(t\omega) < \infty $, then $ \eta_1(t_0\omega) \leq \min\{1,\varepsilon^{t_0}(\omega)\} \eta_1(\omega) $; so by using $ \eta_1(\omega), \min\{1,\varepsilon(\omega)\} $ instead of $ \eta(\omega), \varepsilon(\omega) $ respectively, all the assumptions also hold.

	In the following, we will use an argument in the proof of existence results in \cite{Che18a}. Set
	\[
	\hat{\alpha} = \sup_{\omega} \alpha(\omega), ~\hat{\beta}= \sup_{\omega} \beta(\omega), ~\gamma = \sup_{\omega} \alpha(\omega) \beta(\omega) < 1/2.
	\]
	Choose a large $ n_0 \in \mathbb{N} $ (independent of $ \omega \in M $) such that $ nt_0 > \hat{\varepsilon}_1 $ and
	\[
	\sup_\omega  \{c^2(\omega)\lambda_u^{n_0t_0}(\omega) \lambda_s^{n_0t_0}(\omega) + c(\omega)\lambda_u^{n_0t_0}(\omega) \varepsilon_1^{n_0t_0}(\omega)  + c(\omega) \lambda^{n_0t_0}_{s}(\omega)\} < 1 - 2 \gamma.
	\]
	Set $ u = t_0: M \to M $.
	Fix any $ \omega_0 \in M $. Let $ \widehat{M}_{\omega_0} = \{ (nt_0)(\omega_0): n \in \mathbb{N} \} $. $ H(kt_0, \cdot)(\cdot) $ can be regarded as a bundle correspondence $ X|_{\widehat{M}_{\omega_0}} \times Y|_{\widehat{M}_{\omega_0}} \to X|_{\widehat{M}_{\omega_0}} \times Y|_{\widehat{M}_{\omega_0}} $ over $ u^{k} $, denoted by $ \widehat{H}^{kt_0}|_{\widehat{M}_{\omega_0}} $. Define a function $ \hat{\varepsilon}_1(\cdot) $ over $ u $ as
	\[
	\hat{\varepsilon}_2(\omega) = \max\{ \varepsilon^{t_0}(\omega), \lambda^{t_0}_{s}(\omega) \} (\leq 1), ~\omega \in M,
	\]
	and let
	\[
	\hat{\varepsilon}^{(k)}_2(\omega) = \hat{\varepsilon}_2(\omega)\cdot\hat{\varepsilon}_2(u(\omega))\cdots\hat{\varepsilon}_2(u^{k-1}(\omega)).
	\]

	\begin{slem}
		The section $ i $ now is an $ \hat{\varepsilon}_2^{(k)} $-pseudo-stable section of $ \widehat{H}^{kt_0}|_{\widehat{M}_{m_0}} $, i.e.
		\[
		| F_{kt_0,\omega}( 0, 0 ) | \leq \eta^{(\omega_0)}_{k}(u^{k}(\omega)), ~
		| G_{kt_0,\omega}( 0, 0 ) | \leq \eta^{(\omega_0)}_{k}(\omega),
		\]
		$ \omega \in \widehat{M}_{\omega_0} $, where $ \eta^{(\omega_0)}_{k}(u^i(\omega_0)) = c_k \hat{\varepsilon}^{(i)}_2(\omega_0) \eta(\omega_0) $, $ i \geq 0 $, and $ c_k \geq 1 $ is a constant independent of $ \omega_0 $.
	\end{slem}
	\begin{proof}
		This is a direct consequence of \autoref{lem:bas1} (2). We only consider the case $ k=2 $. Let $ \hat{y} \in Y_{t_0\omega} $ be the unique point satisfying $ \hat{y} = G_{t_0,t_0\omega}( F_{t_0,\omega}(0, \hat{y}), 0 ) $. Then (noting that $ \beta'(t\omega) \leq \beta(\omega) $)
		\[
		|\hat{y}| \leq \frac{ (\beta'(\omega)) + 1) \eta(u(\omega)) }{1 - \alpha(\omega)\beta'(\omega)},
		\]
		and
		\begin{gather*}
		|G_{2t_0,\omega}(0, 0) | = |G_{t_0,\omega}(0, \hat{y}) | \leq c(\omega)\lambda_u(\omega) | \hat{y} | + \eta(\omega) \leq c(\omega)\lambda_u(\omega) \frac{ (\beta'(\omega) + 1) \eta(u(\omega)) }{1 - \alpha(\omega)\beta'(\omega)} + \eta(\omega),\\
		|F_{2t_0,\omega}(0,0)| = | F_{t_0,t_0\omega} ( F_{t_0,\omega}(0, \hat{y}), 0 ) |
		\leq c(u(\omega))\lambda_s(u(\omega)) \frac{ (\alpha'(\omega) + 1)\eta(u(\omega))}{1 -  \alpha(\omega)\beta'(\omega)} + \eta(u^2(\omega)).
		\end{gather*}
		If $ \omega = u^{j}(\omega_0) $, then $ \eta(\omega) \leq \hat{\varepsilon}^{(j)}_{2}(\omega_0)\eta(\omega_0) $.
		So we can choose
		\[
		c_2 = \max\left\{ \frac{\hat{c}\lambda_1 (\hat{\beta} + 1)}{1 - \gamma} + 1, \frac{\hat{c}(\hat{\alpha} + 1)}{1 - \gamma} + 1 \right\},
		\]
		where $ \lambda_1 \triangleq \sup_\omega \varepsilon(\omega) \lambda_u(\omega) $ and $ \hat{c} = \sup_{\omega}c(\omega) < \infty $, completing the proof.
	\end{proof}
	We have shown $ \widehat{H}^{n_0t_0}|_{\widehat{M}_{\omega_0}} $ satisfies the third existence theorem in \cite[Section 4.1]{Che18a}, so when $ \eta_0 > 0 $ is small such that $ \sup_{\omega}\eta(\omega) \leq \eta_0 $, there are a small constant $ \sigma_{0} > 0 $ (independent of $ \omega_0 \in M $ since $ \eta^{(\omega_0)}_{n_0}(u^i(\omega_0)) \leq c_{n_0} \eta_0 $) and a unique bundle map $ f^{n_0, (\omega_0)}: X|_{\widehat{M}_{\omega_0}}(\sigma_{0}) \to Y|_{\widehat{M}_{\omega_0}} $ such that
	\begin{enumerate}[(a{0})]
		\item $ \lip f^{n_0, (\omega_0)}_{\omega} \leq \beta'(\omega) $, $ | f^{n_0, (\omega_0)}_{\omega}(0)| \leq K'_1 \eta^{(\omega_0)}_{n_0}(\omega) \leq K'_1c_{n_0} \eta_0 $, where $ K'_1 \geq 0 $ (independent of $ \omega_0 $),
		\item $ \graph f^{n_0, (\omega_0)}_{\omega} \subset H(n_0t_0,\omega)^{-1} \graph f^{n_0, (\omega_0)}_{u^{n_0}(\omega)} $, $ \omega \in \widehat{M}_{\omega_0} $.
		\item Also, $ f^{n_0, (\omega_0)} $ does not depend on the choice of $ \eta^0(\cdot) = \eta^{(\omega_0)}_{n_0}(\cdot) $ as long as it satisfies $ \eta^0(u^{n_0}(\omega)) \leq \hat{\varepsilon}^{(n_0)}_2(\omega) \eta^0(\omega) $, $ \eta^{(\omega_0)}_{n_0}(\omega) \leq \eta^0(\omega) \leq \eta_0 $, for all $ \omega \in \widehat{M}_{\omega_0} $.
	\end{enumerate}
	Using above property (c0), we find that $ f^{n_0, (\omega_0)}_{u^{n_0}(\omega)} = f^{n_0, (u^{n_0}(\omega_0))}_{u^{n_0}(\omega_0)} $. Indeed, $ \{f^{n_0, (\omega_0)}_{\omega'}: \omega' \in \widehat{M}_{u^{n_0}(\omega_0)} \} $ also fulfills (a0) (b0) for the case that $ \omega_0 $ is replaced by $ u^{n_0}(\omega_0) $, with $ \eta^0(\cdot) = \eta^{(\omega_0)}_{n_0}(\cdot) |_{\widehat{M}_{u^{n_0}(\omega)}} $ instead of $ \eta^{(u^{n_0}(\omega_0))}_{n_0}(\cdot) $. Set $ f_{\omega_0} = f^{n_0, (\omega_0)}_{\omega_0}: X_{\omega_0}(\sigma_{0}) \to Y_{\omega_0} $. Then the above argument shows that $ f $ is the unique bundle map of $ X(\sigma_{0}) \to Y $ satisfies
	\begin{enumerate}[(a{1})]
		\item $ \lip f_{\omega} \leq \beta'(\omega) $, $ | f_{\omega}(0)| \leq K'_1 c_{n_0}\eta(\omega) \leq K'_1c_{n_0} \eta_0 $,
		\item $ \graph f_{\omega} \subset H(n_0t_0,\omega)^{-1} \graph f_{(n_0t_0)(\omega)} $, $ \omega \in M $.
	\end{enumerate}
	The uniqueness of $ f $ also shows that $ \graph f_{\omega} \subset H(t,\omega)^{-1} \graph f_{t\omega} $ for all $ t \geq 0 $, if we choose further smaller $ \eta_0 $ and $ \sigma_{0} $; see the proof of \autoref{thmAc}. Other conclusions are obvious and this gives the proof.
\end{proof}

\begin{rmk}
	In \autoref{thmAc}, \autoref{thmBc} and \autoref{thm:local}, there are two special cases which are useful for application: (i) $ \varepsilon_1 \equiv 0 $ (specially $ \eta \equiv 0 $) and $ \varepsilon_1 \equiv 1 $. \autoref{thm:local} is a basic tool to construct strong (un)stable foliations (laminations) for differential equations; for this case $ \eta \equiv 0 $. Similar as \autoref{thm:local}, in \autoref{thmAc} and \autoref{thmBc}, the condition (i) can be stated only for $ 0 \leq t \leq t_0 $; we give the information for all $ t \geq 0 $ only for giving the estimate $| x_{t,\omega}(i_X(\omega))  - i_X(t \omega)| \leq K_0(\eta(t,t\omega)+\inf_{r \geq t_0} \eta(r,t\omega)) $ for all $ t \geq 0 $.

\end{rmk}
See more characterizations and corollaries in \cite{Che18a}.

We are going to discuss the regularity of the bundle map $ f $ given in \autoref{thmAc}. Since this map is constructed through the bundle correspondence $ \widehat{H} $ (see \eqref{discrete}), so the regularity results in \cite{Che18a} all hold.
The assumptions on $ X \times Y $, $ i $, and the (almost) continuity of the functions in (A$ '_1 $) (B$ _1 $) condition are the same as the discrete case. (The function $ c $ does not need any condition; just take $ \hat{c} \geq \sup_{\omega}c(\omega) $ instead of $ c $.)
The regularity properties of maps $ u, F, G $ in that paper \cite{Che18a} now become for the maps $ \hat{u} = t_0 $, $ \overline{F}_\omega = F_{t_0,\omega} $, $ \overline{G}_\omega = G_{t_0,\omega} $. Finally, the spectral gap conditions are the same as the discrete case but the abbreviation of them (see \cite[Remark 6.6]{Che18a}) needs to be little modified, i.e. $ \vartheta \equiv 1 $.

In order to more easily verify the conditions on $ \hat{u} $, $ \overline{F}, \overline{G} $, one can give a more restriction on $ t $, $ F, G $; that is, some conditions do not depend on the choice of $ t_0 $. Let $ b > a \geq 0 $. The idea is given in the following: the $ C^1 $ and $ C^{k,1} $ smoothness of $ g_{t_0}(\cdot) $ can be replaced by the $ C^1 $ and $ C^{k,1} $ smoothness of $ g_{t}(\cdot) $ for $ \forall t \in [a,b] $, where $ g_{t}(\cdot) $ can be taken as $ t(\cdot) $, $ F_{t,\omega}(\cdot), G_{t, \omega}(\cdot) $, and $ F_{t,\cdot}(\cdot), G_{t, \cdot}(\cdot) $; the $ C^{k,1} $ constants of $ g_{t}(\cdot) $ can depend on $ t $; however, for $ C^{k,\alpha} $ ($ 0 < \alpha < 1 $) case, we need the $ C^{k,\alpha} $ information of $ F_{t,\omega}(\cdot), G_{t, \omega}(\cdot) $ or $ F_{t,\cdot}(\cdot), G_{t, \cdot}(\cdot) $ for all sufficiently large $ t $.

Here are simple facts about the fiber-regularity of $ f $ (i.e. $ x \mapsto f_{\omega}(x) $), which are direct consequences of \cite[Lemma 6.7 and Lemma 6.12]{Che18a}.
\begin{thm}[$ C^k $ leaves]\label{lem:leaf1}
	Let $ f $ be obtained in \autoref{thmAc}.
	Assume (i) $\alpha,\beta,\beta'$, $\lambda_s, \lambda_u$, $c$ are all bounded functions; (ii) the fibers of the bundles $ X, Y $ are Banach spaces; (iii) for every $ (t,\omega) \in [a,b] \times M $ ($ b > a \geq 0 $), $ F_{t,\omega}(\cdot), G_{t,\omega}(\cdot) $ are $ C^1 $. Then for every $ \omega \in M $, $ f_\omega(\cdot) \in C^1 $.

	In addition, suppose (i) for $ (t,\omega) \in [a,b] \times M $ ($ b > a \geq 0 $), $ F_{t,\omega}(\cdot), G_{t,\omega}(\cdot) $ are $ C^{k} $, (ii)
	\[
	\sup_{\omega} \{|F_{t,\omega}(\cdot)|_{k-1,1}, |G_{t,\omega}(\cdot)|_{k-1,1}\} \leq C_{t} < \infty,
	\]
	and (iii) $ \sup_{\omega} \lambda^{k}_{s}(\omega) \lambda_{u}(\omega) < 1 $, then for every $ \omega \in M $, $ f_\omega(\cdot) \in C^{k} $ and $ \sup_{\omega}|f_\omega(\cdot)|_{k-1,1} < \infty $. Here $ |u|_{k,1} = \max\{ \sup_{x}|D^{i}u(x)|, \lip D^{k}u(\cdot): i = 1,2,\ldots,k \} $.

\end{thm}
Similar results for $ f $ obtained in \autoref{thmBc} and \autoref{thm:local} hold.
If $ H $ is induced by equation \eqref{equ:main} (see \autoref{diffCocycle}), then the $ C^{k,\alpha} $ assumptions on $ F_{t,\omega}(\cdot), G_{t,\omega}(\cdot) $ can change to on the tangent field (i.e. $ z \mapsto f(\omega)z $), which are the consequences of parameter-dependent fixed point theorem (cf e.g. \cite[Appendix D.1]{Che18a}), see e.g. \autoref{lem:00} and \autoref{lem:vcf1} in $ C^{k} $ case as well.

The statement of the base-regularity of $ f $ (i.e. $ \omega \mapsto f_{\omega}(\cdot) $) is more complicated. There is an issue of how we describe the uniform properties of the bundle map $ f $ respecting the base points in approximate bundles, such as the uniformly $ C^{k,\alpha} $ ($ k = 0,1 $, $ 0 \leq \alpha \leq 1 $) continuity. In \cite{Che18a}, this was done by a natural way based on the works e.g. \cite{HPS77, PSW12, BLZ99, BLZ08, Cha04, Eld13, Ama15}; that is, the bundle map is represented in local bundle charts belonging to preferred bundle atlases, called the local representations, and the uniform property of this bundle map means the uniform property of local representations. Some prerequisites are needed making it impossible for us to fully present the whole regularity results in \cite{Che18a}.
Let's take a quick glimpse of the regularity results about $ f $ in a \textbf{\emph{not}} very sharp and general setting.
For the precise meaning of uniformly (locally) $ C^{k,\alpha} $ regularity of $ f $ with respect to the bundle charts $ \mathcal{M}, \mathcal{A}, \mathcal{B} $ given in the following assumptions, see \cite{Che18a}.

\begin{enumerate}[({E}1)]
	\item (about $ M $) Let $ M $ be a $ C^1 $ Finsler manifold and $ M^0_1 \subset M $. Let $ M_1 $ be the $ \varepsilon $-neighborhood of $ M^0_1 $ ($ \varepsilon > 0 $). Take a $ C^1 $ atlas $ \mathcal{N} $ of $ M $. Let $ \mathcal{M} $ be the canonical bundle atlas of $ TM $ induced by $ \mathcal{N} $. Assume $ M $ is \emph{$ C^{1,1} $-uniform} around $ M_1 $ with respect $ \mathcal{M} $ (see \cite{Che18a}).

	\item (about $ X \times Y $) $ (X, M, \pi_1) $, $ (Y, M, \pi_2) $ are $ C^1 $ vector bundles endowed with $ C^0 $ \emph{connections} $ \mathcal{C}^X, \mathcal{C}^Y $ which are \emph{uniformly (locally) Lipschitz} around $ M_1 $ (see \cite{Che18a}). Take $ C^1 $ \emph{normal} (vector) bundle atlases $ \mathcal{A} $, $ \mathcal{B} $ of $ X, Y $ respectively. Assume for sufficiently small $ \xi > 0 $, $ (X, M, \pi_1) $, $ (Y, M, \pi_2) $ both have \emph{($ \xi $-almost) $ C^{1,1} $-uniform trivializations} at $ M_1 $ with respect to $ \mathcal{A}, \mathcal{B} $ (and $ \mathcal{M} $), respectively (see \cite{Che18a}).

	\item (about $ i $) Take the section $ i $ as the $ 0 $-section of $ X \times Y $.

	\item (about $ t $) $ t : M \to M $ is a $ C^1 $ semiflow. Denote the linear cocycle $ U $ on $ TM $ by $ U(t,\omega) = Dt(\omega) : T_\omega M \to T_{t\omega}M $. Assume $ U $ satisfies the following. Let $ \zeta > 0 $ be small (depending on the spectral gap condition).
	\begin{enumerate}[(i)]
		\item There are $ t'_0 > 0 $ and a positive function $ \mu $ such that $ | U(nt'_0,\omega) | \leq \mu^{nt'_0}(\omega) $ for all large integral $ n $ and $ \omega \in M $;
		\item $ \omega \mapsto |U(t'_0,\omega)| $ is $ \zeta $-almost uniformly continuous around $ M_1 $ (see \autoref{rmk:amplitude});
		\item $ \sup_{\omega \in M_1} \mu (\omega) < \infty $.
	\end{enumerate}
	\noindent Assume there is a $ t_1 > 0 $ such that $ t(M) \subset M^0_1 $ for all $ t \geq t_1 $.
	\item [(E5)] The functions in (A$ '_1 $) (B$ _1 $) condition are $ \zeta $-almost uniformly continuous around $ M_1 $ and $ \zeta $-almost continuous; see \autoref{rmk:amplitude}.

\end{enumerate}

(E1) will be satisfied, for example, (i) $ M $ is a uniformly regular Riemannian manifold on $ M_1 $ (see \cite{Ama15}), (ii) $ M $ has bounded geometry with $ M_1 $ far away from the boundary of $ M $, including the class of smooth compact Riemannian manifolds (see \cite{Ama15, Eld13}), (iii) $ M $ is the class of immersed manifolds in Banach spaces discussed \cite{BLZ99, BLZ08} (or see the hypothesis (H1) $ \sim $ (H4) in \autoref{normalH}), or (iv) see the class of immersed manifolds introduced in \cite[Chapter 6]{HPS77}; see \cite{Che18a} for details.

(E2) will be fulfilled, for example, $ X, Y $ have $ 2 $-th order bounded geometry (see \cite[Page 45]{Eld13} and \cite[Page 65]{Shu92}) or $ X, Y $ are the $ C^{2} $-uniform Banach bundles discussed in \cite[Chapter 6]{HPS77}; see also \cite{Che18a} for more examples.

The following abbreviation of spectral gap condition will be used.
Let $ \lambda, \theta : M \to \mathbb{R}_+ $.
$ \lambda\theta $, $ \max\{ \lambda, \theta \} $ are defined by
\[
(\lambda\theta) (\omega) = \lambda(\omega)\theta(\omega), ~\max\{ \lambda, \theta \}(\omega) = \max\{ \lambda(\omega), \theta(\omega) \}.
\]
The notation $ \bm {\lambda < 1} $ means that $ \sup_\omega \lambda(\omega) < 1 $.
If $ \theta < 1 $, then the meaning of the notation $ \bm{\lambda^{*\alpha} \theta} < 1 $ is different in different settings in \cite{Che18a}; for simplicity, here it means
\begin{enumerate}[(i)]
	\item $ \sup_{\omega} \lambda^\alpha(\omega) \theta(\omega) < 1 $ if $ \alpha = 1 $,
	\item $ \sup_\omega  \sup_{t \geq 0}\lambda(t\omega) \sup_{t \geq 0} \theta(t\omega) < 1 $ otherwise. Particularly if $ \lambda(t\omega) \leq \lambda(\omega) $ and $ \theta(t\omega) \leq \theta(\omega) $ for all $ (t,\omega) \in \mathbb{R}_+ \times M $, then this also means case (i); the functions in (A) (B) condition of $ H $ induced by differential equations usually are characterized in this case, see \autoref{diffCocycle}.
\end{enumerate}

\begin{thm}\label{thm:Regc}
	Assume \textnormal{(E1) $ \sim $ (E5)} hold with $ \xi,\zeta $ small (depending on the spectral gap condition). Let $ H: \mathbb{R}_+ \times X \times Y \rightarrow X \times Y $ be a cocycle correspondence over $ t $ with a generating cocycle $ (F,G) $. Let $ f $ be given in \autoref{thmAc} when $ i $ is an \textbf{invariant section} of $ H $ (i.e. $ \eta(\cdot,\cdot) \equiv 0 $) and assume $\alpha,\beta,\beta'$, $\lambda_s, \lambda_u$, $c$ are all bounded functions. Then we have the following results about the regularity of $ f $. (In the following, $ 0 < \alpha, \beta \leq 1 $ and for \textbf{all} $ \bm{t \in [a,b]} $.)
	\begin{enumerate}[(1)]
		\item $ f_\omega(0) = 0 $. If $ F_{t,(\cdot)}(\cdot), G_{t,(\cdot)}(\cdot) $ are continuous, so is $ f $. If $ (\omega, z) \mapsto D_zF_{t,\omega}(z), D_zG_{t, \omega}(z) $ are continuous, so is $ (\omega,x) \mapsto D_xf_\omega (x) $. Moreover, if $ DF_{t,\omega}(\cdot), DG_{t, \omega}(\cdot) $ are $ C^{0,\gamma} $ uniform for $ \omega $, and $ \lambda^{*\gamma\alpha}_s \lambda_s \lambda_u < 1 $,
		then $ Df_\omega(\cdot) $ is $ C^{0,\gamma\alpha} $ uniform for $ \omega $.

		\item Suppose
		(i) $ F_{t,(\cdot)}(\cdot), G_{t,(\cdot)}(\cdot) $ are uniformly (locally) $ C^{0,1} $ around $ M_1 $,
		(ii) $  (\max\{ \lambda^{-1}_s, 1 \} \mu)^{*\alpha} \lambda_s \lambda_u < 1 $.
		\textbf{Or} suppose
		(i$ ' $) $ F_{t,(\cdot)}(\cdot), G_{t,(\cdot)}(\cdot) $ are uniformly (locally) $ C^{1,1} $ around $ M_1 $,
		(ii$ ' $) $ (\frac{\mu}{\lambda_s})^{*\alpha} \lambda_s \lambda_u < 1 $.
		Then $ \omega \mapsto f_\omega(\cdot) $ is uniformly (locally) $ \alpha $-H\"older around $ M_1 $.

		\item Suppose
		(i) $ \omega \mapsto DF_{t,\omega}(0, 0), DG_{t,\omega}(0, 0) $ are uniformly (locally) $ C^{0,\gamma} $ around $ M_1 $,
		(ii) $  \mu^{* \gamma \alpha} \lambda_s \lambda_u < 1 $.
		Then $ \omega \mapsto Df_\omega(0) $ is uniformly (locally) $ C^{0,\gamma \alpha} $ around $ M_1 $.

		\item Suppose
		(i) $ D_zF_{t,(\cdot)}(\cdot), D_zG_{t, (\cdot)}(\cdot) $ are uniformly (locally) $ C^{0,1} $ around $ M_1 $,
		(ii) $ \lambda^2_s \lambda_u \mu^\alpha < 1 $, $ \lambda^{*\beta}_s \lambda_s \lambda_u < 1 $, $ \mu^{*\alpha} \lambda_s \lambda_u < 1 $.
		Then $ \omega \mapsto Df_\omega(\cdot) $ is locally $ \alpha \beta $-H\"older around $ M_1 $.

		\item\label{zz1} Suppose
		(i) $ F_{t,(\cdot)}(\cdot), G_{t,(\cdot)}(\cdot) $ are $ C^{1,1} $ around $ M_1 $ (see \cite[Remark 6.6]{Che18a}) and $ C^1 $ in $ X \times Y $,
		(ii) $ \lambda_s \lambda_u \mu < 1 $.
		Then $ f $ is $ C^1 $, $ \nabla_\omega f_\omega(0) = 0 $ for all $ \omega \in M_1 $ and there is a constant $ C $ such that $ |\nabla_\omega f_\omega(x)| \leq C |x| $ for all $ x \in X_\omega $, $ \omega \in M_1 $. Here $ \nabla f $ means the covariant derivative of $ f $ with respect to $ \mathcal{C}^X, \mathcal{C}^{Y} $; see e.g. \cite{Che18a}.
		Moreover, if an additional spectral gap condition holds: $ \lambda_s^{*\beta} \lambda_s \lambda_u < 1 $ and $ \max\{ 1, \lambda_s \}^{*\alpha} \lambda_s \lambda_u \mu < 1 $, then $ \nabla_\omega f_\omega(\cdot) $ is locally $ \alpha \beta $-H\"older uniform for $ \omega \in M $.

		\item Under \eqref{zz1}, assume for every $ t \in [a,b] $, $ U(t,\cdot) $ is uniformly (locally) $ C^{1,1} $ around $ M_1 $. Suppose
		$ \lambda^2_s \lambda_u < 1 $, $ \lambda^2_s \lambda_u \mu < 1 $, $ \max\{ \frac{\mu}{\lambda_s}, \mu \}^{*\alpha} \lambda_s \lambda_u \mu < 1 $.
		Then $ \omega \mapsto \nabla_\omega f_\omega(\cdot) $ is uniformly (locally) $ \alpha $-H\"older around $ M_1 $; this gives $ (\omega,x) \mapsto f_{\omega}(x) $ is uniformly locally $ C^{1,\alpha} $ around $ M_1 $.
	\end{enumerate}
\end{thm}

\begin{proof}
	Here we mention a fact, i.e. from (E1) and (E4), for any given $ \varsigma > 0 $ and any sufficiently large $ t_0 = n t'_0 > t_1 $ (depending on $ \varsigma, t'_0 $), there is an $ \epsilon_1 > 0 $ such that if $ \omega' \in U_{\omega}(\epsilon_1) = \{ \omega' \in M: d(\omega', \omega_0) < \epsilon_1 \} $, $ \omega \in M_1 $, where $ d $ is the Finsler metric in each component of $ M $, then
	\[
	d(t_0(\omega'), t_0(\omega)) \leq (\mu(\omega) + \varsigma)^{t_0} d(\omega', \omega).
	\]
	Also, the above conditions on $ F_{t,(\cdot)}(\cdot) $, $ G_{t,(\cdot)}(\cdot) $, $ t \in [a,b] $, will imply the same conditions on $ F_{t_0,(\cdot)}(\cdot) $, $ G_{t_0,(\cdot)}(\cdot) $. So apply the regularity results in \cite{Che18a} (see also \cite[Theorem 6.2]{Che18a}) to $ \hat{u} = t_0 $, $ \overline{F}_\omega = F_{t_0,\omega} $, $ \overline{G}_\omega = G_{t_0,\omega} $ to give desired results.
\end{proof}

If $ H $ is induced by the class of differential equations studied in \autoref{diffCocycle} (i.e. \eqref{equ:main}), then the uniformly (locally) $ C^{0,1} $ or $ C^{1,1} $ properties of $ F_{t,(\cdot)}(\cdot), G_{t,(\cdot)}(\cdot) $ can be satisfied if we give the uniformly (locally) $ C^{0,1} $ or $ C^{1,1} $ smoothness assumptions on the tangent field (i.e. $ (\omega,z) \mapsto f(\omega)(z) $ in \eqref{equ:main}), the semiflow $ t(\cdot) $, and also the spectral projections and $ T_{1}, S_1 $ in the uniform dichotomy condition (i.e. \autoref{def:ud+}) satisfied by \eqref{equ:main}.
See \cite{Che18a} for general regularity results about $ f $.

\subsection{approximately normal hyperbolicity theory}\label{normalH}

Normally hyperbolic invariant manifold theory was studied comprehensively in \cite{Fen72, HPS77} with further development in \cite{BLZ98, BLZ99, BLZ08} for abstract infinite-dimensional dynamical systems in Banach spaces and in \cite{LW97, PS01} for partial differential equations in Banach spaces. Roughly, this theory gives results of (i) persistence and existence of center and center-(un)stable manifolds which maintain some smoothness, (ii) decoupling and linearization of the system along the normally hyperbolic invariant manifold, and (iii) asymptotic behaviors of the orbits around this invariant manifold. The result (iii) can be characterized by the so called strong (un)stable foliations in center-(un)stable manifolds.
This theory was expanded and extended to more general settings in our paper \cite{Che18b} which we will state briefly below in a version of continuous dynamical systems towards making it applicable to both well-posed and ill-posed differential equations. The reader can find more illustrations and remarks about this theory in \cite{Che18b} and see \autoref{generalB} for a more intuitive application to some classes of autonomous different equations in Banach spaces.

\textbf{(I)}. The setting for a $ C^{0,1} $ (immersed) submanifold $ \Sigma $ of a Banach space $ Z $ is the following.
Take a representation $ (\widehat{\Sigma}, \phi) $ of $ \Sigma $.
Let $ \widehat{\Sigma} $ be a $ C^0 $ manifold and $ \phi: \widehat{\Sigma} \to Z $ with $ \phi(\widehat{\Sigma}) = \Sigma $. For any $ \widehat{m} \in \widehat{\Sigma} $, let $ \widehat{U}_{\widehat{m}} (\epsilon) $ be the component of $ \phi^{-1} (\Sigma \cap \mathbb{B}_{\phi(\widehat{m})} (\epsilon) ) $ containing $ \widehat{m} $, where $ \mathbb{B}_{\phi(\widehat{m})} (\epsilon) = \{ m': |m' - m| < \epsilon \} $. Let $ \phi ( \widehat{U}_{\widehat{m}} (\epsilon) ) = U_{m,\gamma} (\epsilon) $, $ \widehat{U}_{\widehat{m}} = \bigcup_{\epsilon > 0} \widehat{U}_{\widehat{m}} (\epsilon) $, $ U_{m,\gamma} = \phi ( \widehat{U}_{\widehat{m}} ) $, where $ m = \phi(\widehat{m}) $, $ \gamma \in \Lambda(m) \triangleq \phi^{-1}(\phi(\widehat{m})) $. $ \phi_{m,\gamma} \triangleq \phi: \widehat{U}_{\widehat{m}} \to U_{m,\gamma} $ is homeomorphic with $ \widehat{U}_{\widehat{m}} $ open.
There are a family of projections $ \{ \Pi^{\kappa}_{m}: m \in \Sigma \} $, $ \kappa = s, c, u $, such that $ \Pi^{s}_{m} + \Pi^{c}_{m} + \Pi^{u}_{m} = \id $. Set $ \Pi^{h}_{m} = \id - \Pi^{c}_m $, and $ X^{\kappa}_{m} = R(\Pi^{\kappa}_{m}) $, $ \kappa = s, c, u, h $; also $ X^{\kappa_1\kappa_2}_{m} = X^{\kappa_1}_{m} \oplus X^{\kappa_1}_{m} $, where $ \kappa_1 \neq \kappa_2 \in \{ s,c, u \} $. For $ K \subset \Sigma $, $ \widehat{K}' \subset \widehat{\Sigma} $ and $ \kappa = s, c, u $, set
\begin{gather*}
\widehat{K} = \phi^{-1}(K) \subset \widehat{\Sigma}, ~\widehat{K}_{\epsilon} = \bigcup_{\widehat{m} \in \widehat{K}} \widehat{U}_{\widehat{m}}(\epsilon), ~ K_{\epsilon} = \phi(\widehat{K}_{\epsilon}), \\
{X}^{\kappa}_{m}(r) = \{ x \in {X}^{\kappa}_{m}: |x| < r \}, ~
{X}^{\kappa}_{\widehat{K}'} (r) = \{ (\widehat{m}',x): x \in {X}^{\kappa}_{\phi(\widehat{m}')}(r), \widehat{m}' \in \widehat{K}' \}.
\end{gather*}

We make the following assumption which is essentially due to \cite{BLZ99, BLZ08}; see also \cite[Appendix C]{Che18a} and \cite[Section 4.3]{Che18b} for more explanations about this class of (uniformly) immersed submanifolds.
\begin{enumerate}[(H1)]
	\item ($ C^{0,1} $ immersed submanifold). For any $ m \in \Sigma $, $ \exists \epsilon_m > 0 $, $ \exists \delta_0(m) > 0 $ such that
	\[
	\sup\left\{ \frac{|m_1 - m_2 - \Pi^c_{m}(m_1 - m_2)|}{|m_1 - m_2|} : m_1 \neq m_2 \in U_{m,\gamma} (\epsilon) \right\} \leq \chi_{m}(\epsilon) < 1 / 4,
	\]
	where $ \gamma \in \Lambda(m) $, $ 0 < \epsilon \leq \epsilon_{m} $ and $ \chi_{m}(\cdot) > 0 $ is increased, and
	\[
	X^c_{m}(\delta_0(m)) \subset \Pi^c_{m} (U_{m,\gamma} (\epsilon_{m}) - m),
	\]
	that is, $ \Sigma $ is a $ C^{0,1} $ immersed submanifold of $ Z $. Let $ K \subset \Sigma $. We further assume $ \Sigma $ has some local uniform properties around $ K $. Suppose $ \bm{\inf_{m \in K} \epsilon_{m} > \epsilon_1} > 0 $.

	\item (about $ \{\Pi^{\kappa}_m\} $). There are constants $ L > 0, L_0 > 0 $, such that (i) $ \sup_{m \in K}|\Pi^{\kappa}_m| \leq L_0 $ and (ii)
	\[
	|\Pi^{\kappa}_{m_1} - \Pi^{\kappa}_{m_2}| \leq L|m_1 - m_2|,
	\]
	where $ m_1, m_2 \in U_{m,\gamma} (\epsilon_1) $, $\gamma \in \Lambda(m), m \in K $, $ \kappa = s, c, u $.

	\item (almost uniformly differentiable at $ K $). $ \sup_{m \in K} \chi_{m}(\epsilon) \leq \chi(\epsilon) < 1/ 4 $ if $ 0 < \epsilon \leq \epsilon_1 $, where $ \chi(\cdot) $ is an increased function.

	\item (uniform size neighborhood at $ K $). There is a constant $ \delta_0 > 0 $ such that $ \inf_{m \in K} \delta_0(m) > \delta_0 $.
\end{enumerate}
Some examples are the following. 
\begin{enumerate}[$ \bullet $]
	\item $ \Sigma $ is an open set of a complemented closed subspace of $ Z $ and $ d(K, \partial \Sigma) > 0 $.
	\item Any $ C^1 $ compact embedding submanifold $ \Sigma $ of $ Z $ with $ d(K, \partial \Sigma) > 0 $ and $ K \subset \Sigma $, where $ \partial \Sigma $ is the boundary of $ \Sigma $.
	\item If $ \phi: \widehat{\Sigma} \to Z $ is a $ C^{1} $ \emph{leaf immersion} (see \cite[Section 6]{HPS77}) with $ \widehat{\Sigma} $ boundaryless, then $ \phi(\widehat{\Sigma}) = \Sigma $ with itself satisfies (H1) $ \sim $ (H4) (by smooth approximation on $ T\Sigma $); in this case $ \widehat{\Sigma} $ is finite-dimensional. A special case is a leaf of a compact foliation.
	\item The finite-dimensional cylinders (or $ \mathbb{T}^{n}, \mathbb{S}^n $) in $ Z $. See also \cite{BLZ99} for an example.
\end{enumerate}

For the smooth result of invariant manifolds, we need, in some sense, the $ C^0 $ continuity of $ m \mapsto \Pi^{\kappa}_{m} $ in all $ \Sigma $ in the immersed topology.
\begin{enumerate}[(H0)]
	\item For small $ \xi_0 > 0 $ and every $ \widehat{m} \in \widehat{\Sigma} $, assume $ \limsup_{\widehat{m}' \to \widehat{m}} |\Pi^{\kappa}_{\phi(\widehat{m}')} - \Pi^{\kappa}_{\phi(\widehat{m})}| \leq \xi_0 $, $ \kappa = s, c, u $. That is, $ \Pi^{\kappa}_{(\cdot)} $ is $ \xi_0 $-almost continuous in the following sense.
\end{enumerate}

We say a function $ g : \Sigma \to \mathfrak{M} $ is \emph{$ \xi $-almost uniformly continuous} around $ K $ (in the immersed topology), if the \emph{amplitude} of $ g $ around $ K \subset \Sigma $, defined by
\begin{equation*}
\mathfrak{A}_{K}(\epsilon) \triangleq \mathfrak{A}_{K,g}(\epsilon) \triangleq \sup \{ d(g(m), g(m_0)): m \in U_{m_0,\gamma}(\epsilon), m_0 \in K, \gamma \in \Lambda(m_0) \},
\end{equation*}
where $ d $ is the metric in $ \mathfrak{M} $, satisfies $ \mathfrak{A}_{K}(\epsilon) \leq \xi $ as $ \epsilon \to 0 $. We say $ g $ is \emph{$ \xi $-almost continuous} (in the immersed topology) if for each $ m \in \Sigma $, one has $ \mathfrak{A}_{\{m\}}(\epsilon) \leq \xi $ as $ \epsilon \to 0 $. A family of functions $ g^{\lambda} : \Sigma \to \mathfrak{M} $, $ \lambda \in \Theta $, are said to be \emph{$ \xi $-almost equicontinuous} around $ K $ (in the immersed topology), if $ \sup_{\lambda \in \Theta}\mathfrak{A}_{K,g^{\lambda}}(\epsilon) \leq \xi $ as $ \epsilon \to 0 $. \label{almost}

\textbf{(II)}. Let $ t : m \mapsto tm \triangleq t(m) $ be a $ C_0 $ semiflow on $ \Sigma $ in the immersed topology; this means that there is a $ C_0 $ semiflow $ \widehat{t} $ on $ \widehat{\Sigma} $ such that $ \phi \circ \widehat{t} = t \circ \phi $.

\textbf{(III)}. Let $ \widehat{\Pi}^\kappa_{m} $, $ \kappa = s, c, u $, be projections such that $ \widehat{\Pi}^s_{m} + \widehat{\Pi}^c_{m} + \widehat{\Pi}^u_{m} = \id $, $ m \in \Sigma $. Set $ \widehat{X}^{\kappa}_m = R(\widehat{\Pi}^{\kappa}_m) $ and
$ \widehat{X}^{\kappa_1 \kappa_2}_m = \widehat{X}^{\kappa_1}_m \oplus \widehat{X}^{\kappa_2}_m $, $ m \in \Sigma $, $ \kappa_1, \kappa_2 \in \{ s, c, u \} $, $ \kappa_1 \neq \kappa_2 $. $ \widehat{X}^{\kappa}_{m}(r) = \{ x \in \widehat{X}^{\kappa}_{m}: |x| < r \} $.

\textbf{(IV)}. Let $ H: \mathbb{R}_+ \times Z \to Z $ be a continuous correspondence.
Let $ \widehat{H}(t,m) = H(t)(\cdot + m) - t(m) $, i.e. $ \graph \widehat{H}(t,m) = \graph H(t) - (m, t(m)) \subset Z \times Z $; so $ \widehat{H} $ is a continuous cocycle correspondence over $ t $.

\label{defi:ABk}
For brevity, we adopt the following notions. Suppose for all $ t,s \geq 0 $ and $ m \in \Sigma $,
\[
\widehat{H}(t,s(m)) \sim (\widehat{F}^{\kappa}_{t,s(m)}, \widehat{G}^{\kappa}_{t,s(m)}): \widehat{X}^{\kappa}_{s(m)}(r_{t,1}) \oplus \widehat{X}^{\kappa_1}_{s(m)}(r_{t,2}) \to \widehat{X}^{\kappa}_{(t+s)(m)}(r'_{t,1}) \oplus \widehat{X}^{\kappa_1}_{(t+s)(m)} (r'_{t,2}),
\]
where $ \kappa_1 = csu - \kappa $, and where $ r_{t,i}, r'_{t,i} > 0 $, $ i = 1, 2 $, are constants independent of $ m $ but might depend on $ t $ and $ \inf \{r_{i,t}, r'_{i,t}: t \in [a,b] \} > 0 $ for any $ b > a \geq 0 $.
\begin{enumerate}[$ \bullet $]
	\item If $ (\widehat{F}^{\kappa}_{t,s(m)}, \widehat{G}^{\kappa}_{t,s(m)}) $ satisfies \textnormal{(A$ ' $)}$(\alpha_1(m)$, $c(m) {\lambda}^t_1(m))$ (see \autoref{defAB}), then we say $ \widehat{H} \approx (\widehat{F}^{\kappa}, \widehat{G}^{\kappa}) $ satisfies \textbf{(A$ '_1 $) ($ \bm{\alpha_1, \lambda_1; c} $) condition in $ \kappa $-direction};
	
	\item if $ (\widehat{F}^{\kappa}_{t,s(m)}, \widehat{G}^{\kappa}_{t,s(m)}) $ satisfies \textnormal{(A$ ' $)}$(\alpha(m)$, $c(m) {\lambda}^t_1(m))$ condition (see \autoref{defAB}) and if $ t \geq \hat{\varepsilon}_1 $, \textnormal{(A)}($\alpha(m)$; $ \alpha_1(m) $, $c(m) {\lambda}^t_{1}(m)$) condition (see \autoref{defAB}), where $  \hat{\varepsilon}_1 > 0 $ is a constant, then we say $ \widehat{H} \approx (\widehat{F}^{\kappa}, \widehat{G}^{\kappa}) $ satisfies \textbf{(A$ _0 $) ($ \bm{\alpha; \alpha_1, \lambda_1; c} $) condition in $ \kappa $-direction};
	
	\item similarly, \textbf{(B$ '_1 $) ($ \bm{\beta_1, \lambda_2; c} $) condition in $ \kappa $-direction} and \textbf{(B$ _0 $) ($ \bm{\beta; \beta_1, \lambda_2; c} $) condition in $ \kappa $-direction}.
\end{enumerate}

\begin{enumerate}[({A}1)]
	\item (submanifold condition) Let $ \Sigma $ with $ K \subset \Sigma $ satisfy (H1) $ \sim $ (H4).

	\item ((strictly) inflowing condition) Let $ \bm{t_0} > 0 $ be fixed and $ \xi > 0 $ small. The semiflow in \textbf{(II)} satisfies the following. \textbf{(i)} $ t_0 (\Sigma) \subset K $;
	\textbf{(ii)} there are functions $ v^{t} : \Sigma \to Z $, $ t \in [0,t_0] $, which are {$ \xi $-almost equicontinuous} around $ K $ (in the immersed topology), and a (small) $ \xi_1 > 0 $ such that $ \sup_{t\in [0,t_0]}\sup_{m \in \Sigma}|v^{t}(m) - t(m)| \leq \xi_1 $.

	\item (approximately normal hyperbolicity condition) $ \Sigma $ satisfies the following `\emph{approximately $ cs $-normal hyperbolicity}' condition with respect to $ H $.
	Assume there is a small $ \xi_2 > 0 $ such that $ \sup_{m \in K}|\widehat{\Pi}^\kappa_m - \Pi^\kappa_m| \leq \xi_2 $, $ \kappa = s, c, u $.

	\begin{enumerate}[(a)]
		\item ((A$ ' $)(B) condition)
		Suppose $ \widehat{H} \approx (\widehat{F}^{cs}, \widehat{G}^{cs}) $ satisfies (A$ '_1 $)($ {\alpha, \lambda_{u}; c} $) (B$ _0 $)($ {\beta; \beta', \lambda_{cs}; c} $) condition in $ cs $-direction.
		Moreover,

		\textbf{(i)} (angle condition) $ \sup_m \alpha(m) \beta'(m) < 1/2 $, $ \inf_m \{\beta(m) - \beta'(m)\} > 0 $;

		\textbf{(ii)} (spectral condition) $ \sup_m \lambda_{u}(m) < 1 $;

		\textbf{(iii)} the functions $ \alpha(\cdot), \beta(\cdot), \beta'(\cdot) $, $ \lambda_{u}(\cdot), \lambda_{cs}(\cdot) $ are \emph{bounded}, $ \xi $-almost continuous and $ \xi $-almost uniform continuous around $ K $ (in the immersed topology) and also $ c(\cdot) $ is bounded.

		\item (approximation) There is a small $ \eta > 0 $ such that
		\[
		\sup_{t\in[0,t_0]}\sup_{m \in \Sigma} |\widehat{F}^{cs}_{t,m}(0,0)| \leq \eta, ~
		\sup_{t\in[0,t_0]}\sup_{m \in \Sigma} |\widehat{G}^{cs}_{t,m}(0,0)| \leq \eta.
		\]

		\item ($ s $-contraction)
		If $ (0, \hat{x}^s_i, \hat{x}^{u}_i) \times (\tilde{x}^{c}_i, \tilde{x}^s_i, \tilde{x}^{u}_{i}) \in \graph H(t_0, m) \cap \{\widehat{X}^{cs}_{m}(r_{t_0,1}) \oplus \widehat{X}^{u}_{m}(r_{t_0,2}) \times \widehat{X}^{cs}_{t_0(m)}(r'_{t_0,1}) \oplus \widehat{X}^{u}_{t_0(m)} (r'_{t_0,2})\} $, $ i = 1,2 $, $ m \in \Sigma $, and $ |\tilde{x}^{u}_{1} - \tilde{x}^{u}_{2}| \leq B |\hat{x}^s_1 - \hat{x}^s_2| $, then
		\[
		|\tilde{x}^{s}_{1} - \tilde{x}^{s}_{2}| \leq \lambda^*_s  |\hat{x}^s_1 - \hat{x}^s_2|,
		\]
		where $ B > \sup_{m \in M}c(\omega)\lambda^{t_0}_{cs}(m) \beta(m) $ is some constant and $ \lambda^*_s < 1 $.
	\end{enumerate}

	\item (smooth condition) \textbf{(i)} Assume for every $ m \in \Sigma $, $ \widehat{F}^{cs}_{t_0,m}(\cdot) $, $ \widehat{G}^{cs}_{t_0,m}(\cdot) $ are $ C^1 $, and (H0) holds with $ \xi_0 $ sufficiently small.

	\noindent \textbf{(ii)} (spectral gap condition) $ \sup_m \lambda_{cs}(m) \lambda_u(m) < 1 $.

\end{enumerate}

We say a submanifold $ \Sigma $ of $ X $ satisfying assumption (A1) is \emph{approximately (strictly) inflowing} and \emph{approximately $ cs $-normally hyperbolic} with respect to $ H $ if the assumptions (A2) (A3) hold; sometimes we also say $ \Sigma $ is \emph{(strictly) inflowing} with respect to the semiflow $ t $.

\label{defi:orbit}
\begin{enumerate}[$ \bullet $]
	\item Under (A1) (A2) (i), we say $ \{z_{t}\}_{t \geq 0} $ in $ X^s_{\widehat{\Sigma}}(\sigma) \oplus X^u_{\widehat{\Sigma}} (\varrho) $ is a ($ \sigma, \varrho, \varepsilon $)-\emph{forward orbit} (or \emph{forward orbit} for short) of $ H $,
	if $ z_t = (\widehat{m}_t, x^s_t, x^u_t) \in X^s_{\widehat{\Sigma}}(\sigma) \oplus X^u_{\widehat{\Sigma}} (\varrho) $, $ \phi(\widehat{m}_t) + x^s_t + x^u_t \in H(t-s)(\phi(\widehat{m}_{s}) + x^s_{s} + x^u_{s}) $ ($ t \geq s $), $ \widehat{m}_{t_0} \in \widehat{U}_{\widehat{t_0}(\widehat{m}_{0})}(\varepsilon) $ and $ \widehat{m}_t \in \widehat{U}_{\widehat{(t - t_1)}(\widehat{m}_{t_1})}(\varepsilon) $ where $ 0 \leq t - t_1 \leq t_0, t \geq t_0, t_1 \geq 0 $.
	
	\item Similarly, under (A1) (A2) (i), we say $ \{z_{n}\}_{n \in \mathbb{N}} $ in $ X^s_{\widehat{\Sigma}}(\sigma) \oplus X^u_{\widehat{\Sigma}} (\varrho) $ is a ($ \sigma, \varrho, \varepsilon $)-\emph{forward orbit} (or \emph{forward orbit} for short) of $ H(t_0) $,
	if $ z_n = (\widehat{m}_n, x^s_n, x^u_n) \in X^s_{\widehat{\Sigma}}(\sigma) \oplus X^u_{\widehat{\Sigma}} (\varrho) $, $ \phi(\widehat{m}_n) + x^s_n + x^u_n \in H(t_0)(\phi(\widehat{m}_{n-1}) + x^s_{n-1} + x^u_{n-1}) $ and $ \widehat{m}_{n} \in \widehat{U}_{\widehat{t_0}(\widehat{m}_{n-1})}(\varepsilon) $, $ n = 1, 2, \ldots $.
	
	\item Similar notion of ($ \sigma, \varrho, \varepsilon $)-\emph{backward orbit} of $ H(t_0) $ (resp. $ H $) can be defined, if $ t_0: \Sigma_1 (\subset \Sigma) \to \Sigma $ is invertible where $ \Sigma_1 \subset K $ (resp. $ t $ is a flow).
	For instance, $ \{z_{-t}\}_{t \geq 0} $ in $ X^s_{\widehat{\Sigma}}(\sigma) \oplus X^u_{\widehat{\Sigma}} (\varrho) $ is a ($ \sigma, \varrho, \varepsilon $)-\emph{backward orbit} (or \emph{backward orbit} for short) of $ H $,
	if $ z_{-t} = (\widehat{m}_{-t}, x^s_{-t}, x^u_{-t}) \in X^s_{\widehat{\Sigma}}(\sigma) \oplus X^u_{\widehat{\Sigma}} (\varrho) $, $ \phi(\widehat{m}_{-t}) + x^s_{-t} + x^u_{-t} \in H(s-t)(\phi(\widehat{m}_{-s}) + x^s_{-s} + x^u_{-s}) $ ($ 0 \leq t \leq s $), $ \widehat{m}_{-t_0} \in \widehat{U}_{\widehat{-t_0}(\widehat{m}_{0})}(\varepsilon) $ and $ \widehat{m}_t \in \widehat{U}_{\widehat{(t - t_1)}(\widehat{m}_{t_1})}(\varepsilon) $ where $ -t_0 \leq t - t_1 \leq 0, t \leq -t_0, t_1 \leq 0 $.
	$ \{z_{t}\}_{t \in \mathbb{F}} $ is a ($ \sigma, \varrho, \varepsilon $)-\emph{orbit} if $ \{z_{t}\}_{t \in \mathbb{F}_+} $ is a ($ \sigma, \varrho, \varepsilon $)-{forward orbit} and $ \{z_{-t}\}_{t \in \mathbb{F}_+} $ is a ($ \sigma, \varrho, \varepsilon $)-{backward orbit}, where $ \mathbb{F} = \mathbb{R} \text{ or } \mathbb{Z} $.
\end{enumerate}

\begin{thm}[center-stable manifold]\label{thm:A}
	Let (A1) (A2) (A3) hold. If $ \xi, \xi_1, \xi_2, \eta $ are small, and $ \chi(\epsilon) $ is small when $ \epsilon $ is small, (or more precisely, there are $ \zeta_1, \zeta_2 > 0 $ and a function $ \eta(\cdot, \cdot) > 0 $, if $ \xi, \xi_1, \xi_2, \chi(\epsilon) < \zeta_1 $ and $ \eta \leq \eta(\epsilon, \chi(\epsilon)) < \zeta_1 $, provided $ \epsilon < \zeta_2 $,) then there are positive small $ \varepsilon, \sigma, \varrho $ (depending on $ \chi(\epsilon), \epsilon $) such that the following hold.

	\begin{enumerate}[(1)]
		\item In $ X^s_{\widehat{\Sigma}}(\sigma) \oplus X^u_{\widehat{\Sigma}} (\varrho) $, there is a set $ W^{cs}_{loc}(\Sigma) $ called the \textbf{local center-stable manifold} of $ \Sigma $, which is defined by
		\begin{align*}
		W^{cs}_{loc}(\Sigma) = \{ z \in X^s_{\widehat{\Sigma}}(\sigma) \oplus X^u_{\widehat{\Sigma}} (\varrho): \exists & \{ z_n \}_{n\geq 0} \text{ such that } z_0 = z,  \\
		& \text{ and it is a ($ \sigma,\varrho,\varepsilon $)-forward orbit of } H(t_0) \}.
		\end{align*}
		Moreover, it has the following properties.
		\begin{enumerate}[(i)]
			\item \label{it:h0} $ W^{cs}_{loc}(\Sigma) $ can be represented as a graph of a map. That is there is a map $ h_0 $ such that $ h_0(\widehat{m}, \cdot): X^s_{\phi(\widehat{m})} (\sigma) \to X^u_{\phi(\widehat{m})}(\varrho) $, $ \widehat{m} \in \widehat{\Sigma} $, and
			\[
			W^{cs}_{loc}(\Sigma) = \graph h_0 \triangleq \{ (\widehat{m}, x^s, h_0(\widehat{m}, x^s)): x^s \in X^{s}_{\phi(\widehat{m})}(\sigma), \widehat{m} \in \widehat{\Sigma} \}.
			\]
			Moreover, $ h_0 $ in some sense is (uniformly) Lipschitz around $ K $, that is, there is a function $ \mu(\cdot) $, such that $ \mu(m) = (1+\chi_{*}) \beta'(m) + \chi_{*} $ with $ \chi_{*} \to 0 $ as $ \epsilon,\chi(\epsilon),\eta \to 0 $,
			and for every $ m \in K $, it holds
			\[
			|\Pi^u_{m}( h_0(\widehat{m}_1, x^s_1) - h_0(\widehat{m}_2, x^s_2)  )| \leq \mu(m) \max\{| \Pi^c_{m}(  \phi(\widehat{m}_1) - \phi(\widehat{m}_2)  ) |, | \Pi^s_{m}( x^s_1 - x^s_2 ) |\},
			\]
			where $ \widehat{m}_i \in \widehat{U}_{\widehat{m}}(\varepsilon) $, $ x^{s}_i \in X^s_{\phi(\widehat{m}_i)} (\sigma) $, $ \widehat{m} \in \phi^{-1}(m) $. 
			\item Furthermore, the map $ h_0 $ is unique in the sense that if $ h'_0 $ satisfies (i) and $ \graph h'_0 \subset H(t_0)^{-1} \graph h'_0 $, then $ h'_0 = h_0 $.
			\item We use the following notations: for any $ \widehat{K}' \subset \widehat{\Sigma} $ and $ 0 < \sigma' < \sigma $, set
			\[
			W^{cs}_{loc}(\widehat{K}') \triangleq \graph h_0|_{X^s_{ \widehat{K}' } (\sigma) }, ~W^{cs}_{\sigma'}(\widehat{K}') \triangleq \graph h_0|_{X^s_{ \widehat{K}' } (\sigma') }.
			\]
			Then (a) $ W^{cs}_{loc}(\Sigma) \subset H(t)^{-1} W^{cs}_{loc}(\widehat{K}_{\varepsilon}) $ if $ t \geq t_0 $; (b) $ W^{cs}_{\sigma_1}(\widehat{K}_{\varepsilon}) \subset H(t)^{-1} W^{cs}_{loc}(\Sigma) $ for some $ \sigma_1 < \sigma $ if $ t \geq 0 $; (c) $ \Sigma \subset W^{cs}_{loc}(\Sigma) $, if $ \eta = 0 $.
		\end{enumerate}

		\item $ H(t_0) $ in $ W^{cs}_{loc}(\Sigma) $ induces a map meaning that for any $ z_0 = (\widehat{m}_0, x^s_0, x^u_0) \in W^{cs}_{loc}(\Sigma) $, there is only one $ z_1 = (\widehat{m}_1, x^s_1, x^u_1) \triangleq \phi(\widehat{m}_1) + x^s_1 + x^u_1 \in W^{cs}_{loc}(\Sigma) $ such that $ z_1 \in H(t_0)(z_0) $, where $ \widehat{m}_1 \in \widehat{U}_{\widehat{u}(\widehat{m}_0)} (\varepsilon) $, $ x^\kappa_i \in X^{\kappa}_{\phi(\widehat{m}_i)} $, $ \widehat{m}_i \in \widehat{\Sigma} $, $ i = 0,1 $, $ \kappa = s, u $. Denote the map by $ \mathcal{H}: z_0 \mapsto z_1 $. Similarly, $ H $ in $ W^{cs}_{\sigma_1}(\widehat{K}_{\varepsilon}) $ induces a semiflow $ \varphi $, i.e. for every $ z = (\widehat{m}_0, x^s_0, x^u_0) \in W^{cs}_{\sigma_1}(\widehat{K}_{\varepsilon}) $, there is a unique ($ \sigma, \varrho, \varepsilon $)-forward orbit $ \{z_{t}\}_{t\geq 0} \subset W^{cs}_{loc}(\Sigma) $ of $ H $ such that $ z_0 = z $; define $ \varphi^t(z_0) = z_{t} $.

		\item In addition, let (A4) hold with $ \xi_0 $ (in (H0)) sufficiently small, then $ W^{cs}_{loc}(\Sigma) $ is a $ C^1 $ immersed submanifold of $ X $.
		Furthermore, suppose $ \eta = 0 $ and $ D_{x^{cs}}\widehat{G}^{cs}_{t_0,m}( 0, 0, 0 ) = 0 $ for all $ m \in \Sigma $. Then $ T_m W^{cs}_{loc}(\Sigma) = \widehat{X}^{cs}_{m} $ for all $ m \in \Sigma $.
	\end{enumerate}
\end{thm}

\begin{proof}
	The results associated with $ H(t_0) $, in fact, had been proved in our paper \cite{Che18b}; so we refer the reader to see that paper for details. What we need to prove is the center-stable manifold is also invariant under $ H(t) $ (i.e. (1) (iii)) and $ H $ in $ W^{cs}_{loc}(\widehat{K}_{\varepsilon}) $ induces a semiflow (i.e. (2)).

	Fix $ t $ such that $ 0 < t < t_0 $. First note that the constant $ \varepsilon $ is chosen such that (H2) $ \sim $ (H4) also hold and so the (uniformly) Lipschitz continuity of $ h_0 $ in (1) (i) when $ K $ is replaced by $ K_{2\varepsilon} $; see \cite{Che18b} for details. Take any $ \widehat{m} \in \widehat{K}_{\varepsilon} $ and set $ m = \phi(\widehat{m}) $ ($ \in K_{\varepsilon} $), $ m_t = t(m) $ and $ \widehat{m}_t = \widehat{t}(\widehat{m}) $. By assumption (A2) (ii), we can assume $ m_t \in K_{2\varepsilon} $, i.e. $ \widehat{m}_t \in \widehat{K}_{2\varepsilon} $, if $ \xi, \xi_1 $ are small. The local representation $ f_{\widehat{m}_t} $ of $ W^{cs}_{loc}(\widehat{K}_{2\varepsilon}) $ at $ \widehat{m}_t $ is given by
	\[
	m_{t} + x^{s} + x^{c} + f_{\widehat{m}_t}(x^{s}, x^{c}) = \phi(\widehat{m}') + \overline{x}^s + h_0(\widehat{m}', \overline{x}^s), \tag{$ \blacklozenge $}
	\]
	where $ x^{\kappa} \in X^{\kappa}_{m_t}(c_1 \sigma_{*}) $, $ \kappa = s, c $, ($ c_1 $ is smaller than $ 1 $ but sufficiently close to $ 1 $), $ \widehat{m}' \in \widehat{U}_{\widehat{m}_t}(\varepsilon) $, $ \overline{x}^s \in X^{s}_{\phi(\widehat{m}')}(\sigma_{*}) $; here $ \sigma_{*} $ can be taken in $ [\chi_{*} \varepsilon, \sigma] $ for some small $ \chi_{*} > 0 $ depending on $ \varepsilon, \chi(\varepsilon), \eta $ and $ \chi_{*} \to $ as $ \varepsilon, \chi(\varepsilon), \eta \to 0 $ (see \cite{Che18b}). Also, note that $ \lip f_{\widehat{m}_t}(\cdot) \approx \beta'(\widehat{m}_t) $. By \autoref{lem:bas1}, for any $ x^s_0 \in X^{s}_{m} (\sigma_1) $,
	there are unique $ x^u_0 \in X^{u}_{m} $ and $ x^{\kappa} \in X^{\kappa}_{m_t}(c_1 \sigma) $, $ \kappa = s, u $, such that
	\[
	m_{t} + x^{s} + x^{c} + f_{\widehat{m}_t}(x^{s}, x^{c}) \in H(t)(m + x^s_0 + x^u_0), \tag{$ \blacklozenge \blacklozenge $}
	\]
	where $ \sigma_1 = O(\chi_{*}\varepsilon) \in (\chi_{*} \varepsilon, \sigma) $; also we have $ x^u_0 \in X^{u}_{m} (\varrho_1) $, where $ \varrho_1 = O(\chi_{*}\varepsilon) < \varrho $; write
	$ m_{t} + x^{s} + x^{c} + f_{\widehat{m}_t}(x^{s}, x^{c}) $ as ($ \blacklozenge $).
	Next, if we show there is $ (\widehat{m}_1, \overline{x}^s_1, \overline{x}^u_1) \in X^s_{\widehat{K}_{\varepsilon}}(\sigma_1) \oplus X^u_{\widehat{K}_{\varepsilon}} (\varrho_1) $, where $ \widehat{m}_1 \in \widehat{U}_{\widehat{t}_0(\widehat{m})} (\varepsilon) $, such that
	\[
	\phi(\widehat{m}_1) + \overline{x}^s_1 + \overline{x}^{u}_1 \in H(t_0)(m + x^s_0 + x^u_0) \cap H(t)^{-1} (\graph h_0),
	\]
	where $ H(t)^{-1} (\graph h_0) $ is constructed as the above way (i.e. ($ \blacklozenge \blacklozenge $)), then from the characterization of $ W^{cs}_{loc}(\Sigma) $ (i.e. conclusion (1)), we see $ x^u_0 = h_0(\widehat{m}, x^s_0) $, which yields $ W^{cs}_{\sigma_1}(\widehat{K}_{\varepsilon}) \subset H(t)^{-1} W^{cs}_{loc}(\Sigma) $. By the cocycle property of $ H $ and $ \graph h_0 \subset H(t_0)^{-1} \graph h_0 $, there is a unique $ (\widehat{m}_2, \overline{x}^s_2, \overline{x}^u_2) \in \graph h_0 $, where $ \widehat{m}_2 \in \widehat{U}_{\widehat{t}_0(\widehat{m}')} (\varepsilon) $, such that $ \phi(\widehat{m}_2) + \overline{x}^s_2 + \overline{x}^u_2 \in H(t_0)( \phi(\widehat{m}') + \overline{x}^s + h_0(\widehat{m}', \overline{x}^s)) $; that is
	\[
	m + x^s_0 + x^u_0 \in H(t)^{-1} \circ H(t_0)^{-1} (\phi(\widehat{m}_2) + \overline{x}^s_2 + \overline{x}^u_2) = H(t_0)^{-1} \circ H(t)^{-1} (\phi(\widehat{m}_2) + \overline{x}^s_2 + \overline{x}^u_2),
	\]
	and so there are $ \widetilde{x}^{\kappa} \in X^{\kappa}_{t_0(m)} $, $ \kappa = s, c, u $, such that
	\[
	\begin{cases}
	t_0(m) + \widetilde{x}^{s} + \widetilde{x}^{c} + \widetilde{x}^{u} \in H(t_0) (m + x^s_0 + x^u_0), \\
	t_0(m) + \widetilde{x}^{s} + \widetilde{x}^{c} + \widetilde{x}^{u} \in  H(t)^{-1} (\phi(\widehat{m}_2) + \overline{x}^s_2 + \overline{x}^u_2).
	\end{cases}
	\]
	By assumption (A3) (c) about $ H(t_0) $ and $ x^s_0 \in X^{s}_{m} (\sigma_1) $, we have $ \widetilde{x}^{\kappa} \in X^{\kappa}_{t_0(m)}(c_1 \sigma_1) $, $ \kappa = s, c $, and $ \widetilde{x}^{u} \in X^{u}_{t_0(m)}(c_1 \varrho_1) $. From the detailed construction of tubular neighborhoods around $ K $ of $ \Sigma $ in $ X $ (see \cite{Che18b}), we can write
	\[
	t_0(m) + \widetilde{x}^{s} + \widetilde{x}^{c} + \widetilde{x}^{u} = \phi(\widehat{m}_1) + \overline{x}^s_1 + \overline{x}^{u}_1,
	\]
	where $ \overline{x}^s_1 \in X^s_{\phi(\widehat{m}_1)}(\sigma_1) $, $ \overline{x}^u_1 \in X^u_{\phi(\widehat{m}_1)} (\varrho_1) $, and $ \widehat{m}_1 \in \widehat{U}_{\widehat{t}_0(\widehat{m})} (\varepsilon) $. Note that from assumption (A2) (ii) and
	\[
	\begin{split}
	|\phi(\widehat{m}_2) - t(t_0(m))| = |\phi(\widehat{m}_2) - t_0(\phi(\widehat{m}')) + t_0(\phi(\widehat{m}')) - t_0(m_{t}) |, \\
	|t(\phi(\widehat{m}_1)) - \phi(\widehat{m}_2)| = |t(\phi(\widehat{m}_1)) - t(t_0(m)) + t(t_0(m)) - \phi(\widehat{m}_2)|,
	\end{split}
	\]
	we know $ |t(\phi(\widehat{m}_1)) - \phi(\widehat{m}_2)| $ can be made small if $ \varepsilon $ and $ \xi, \xi_1 $ are small, and so
	\[
	\phi(\widehat{m}_2) + \overline{x}^s_2 + \overline{x}^u_2 = t(\phi(\widehat{m}_1)) + \hat{x}^s + \hat{x}^c + f_{\widehat{t}(\widehat{m}_1)}(\hat{x}^s, \hat{x}^c) \in H(t)(\phi(\widehat{m}_1) + \overline{x}^s_1 + \overline{x}^{u}_1).
	\]
	Since $ |\overline{x}^{\kappa}_{i}| \leq O(\chi_{*} \varepsilon) $, $ i = 1,2 $, $ \kappa = s, u $, we can further get $ \widehat{m}_2 \in \widehat{U}_{\widehat{t}(\widehat{m}_1)} (\varepsilon) $ when $ \varepsilon, \chi(\varepsilon), \eta $ is small (in order to make $ \chi_{*} $ small). Therefore, $ \phi(\widehat{m}_2) + \overline{x}^s_2 + \overline{x}^u_2 \in H(t)^{-1} (\graph h_0) $ is constructed as ($ \blacklozenge \blacklozenge $). This completes the proof of $ W^{cs}_{\sigma_1}(\widehat{K}_{\varepsilon}) \subset H(t)^{-1} W^{cs}_{loc}(\widehat{K}_{2\varepsilon}) \subset H(t)^{-1} W^{cs}_{loc}(\Sigma) $ for $ 0 < t < t_0 $.

	As $ W^{cs}_{loc}(\Sigma) \subset H(t_0)^{-1} W^{cs}_{loc}(\widehat{K}_{\varepsilon}) $, one has (a) $ W^{cs}_{\sigma_1}(\Sigma) \subset H(t)^{-1} W^{cs}_{loc}(\widehat{K}_{\varepsilon}) $ for all $ t \geq t_0 $ and (b) $ W^{cs}_{\sigma_1}(\widehat{K}_{\varepsilon}) \subset H(t)^{-1} W^{cs}_{loc}(\Sigma) $ for all $ t \geq 0 $.
	Now for any $ z_0 = (\widehat{m}_0, x^s_0, x^u_0) \in W^{cs}_{\sigma_1}(\widehat{K}_{\varepsilon}) $, there are unique $ z_{nt_0} = (\widehat{m}_{nt_0}, x^s_{nt_0}, x^u_{nt_0}) \triangleq \phi(\widehat{m}_{nt_0}) + x^s_{nt_0} + x^u_{nt_0} \in W^{cs}_{\sigma_1}(\widehat{K}_{\varepsilon}) $ such that $ z_{nt_0} \in H(t_0)(z_{(n-1)t_0}) $, $ n = 1,2,\ldots $, where $ \widehat{m}_{nt_0} \in \widehat{U}_{\widehat{t}_0(\widehat{m}_{(n-1)t_0})} (\varepsilon) $, $ x^\kappa_{nt_0} \in X^{\kappa}_{\phi(\widehat{m}_{nt_0})} $, $ \widehat{m}_{nt_0} \in \widehat{K}_{\varepsilon} $, $ \kappa = s, u $; moreover, for $ (n-1)t_0 < t < nt_0 $, there are unique $ z_{t} = (\widehat{m}_{t}, x^s_{t}, x^u_{t}) \triangleq \phi(\widehat{m}_{t}) + x^s_{t} + x^u_{t} \in W^{cs}_{loc}(\widehat{K}_{2\varepsilon}) $ such that $ z_{t} \in H(t - (n-1)t_0)(z_{(n-1)t_0}) $, where $ \widehat{m}_{t} \in \widehat{U}_{\widehat{s}(\widehat{m}_{(n-1)t_0})} (\varepsilon) $ and $ s = t - (n-1)t_0 $. So $ \{z_{t}\}_{t \geq 0} $ is a ($ \sigma, \varrho, \varepsilon $)-\emph{forward orbit} of $ H $; we can let $ \varphi^t(z_0) = z_{t} $. The proof is complete.
\end{proof}

For the existence of strong stable foliation, we need the following assumption.
\begin{enumerate}[(A5)]
	\item (strong stable foliation condition) \textbf{(a)}
	Suppose $ \widehat{H} \approx (\widehat{F}^{s}, \widehat{G}^{s}) $ satisfies (A$ '_1 $)($ \alpha_{cu} $, $ \lambda_{cu} $; $ c $) (B$ _0 $)($ \beta_s $; $ \beta'_s $, $ \lambda_{s} $; $ c $) condition in $ s $-direction.
	Moreover,

	\noindent\textbf{(i)} (angle condition) $ \sup_m \alpha_{cu}(m) \beta'_s(m) < 1/2 $, $ \inf_m \{\beta_s(m) - \beta'_s(m)\} > 0 $;

	\noindent\textbf{(i$ ' $)} furthermore, $ \alpha_{cu}(m) \leq 1 $, $ \beta'(m) \leq 1 $, $ \beta'(m) \leq \beta'_s(m) $, $ m \in \Sigma $;

	\noindent\textbf{(ii)} (spectral gap condition) $ \sup_{m} \lambda_s(m) < 1 $, $ \sup_m \lambda_{s}(m)\lambda_{cu}(m) < 1 $;

	\noindent\textbf{(iii)} $ \alpha_{cu}(\cdot) $, $ \beta_{s}(\cdot) $, $ \beta'_{s}(\cdot) $, $ \lambda_s(\cdot) $, $ \lambda_{cu}(\cdot) $ are \emph{bounded}, $ \xi $-almost continuous and $ \xi $-almost uniform continuous around $ K $ (in the immersed topology).

	\noindent\textbf{(b)} (smooth condition)  Assume for every $ m \in \Sigma $, $ \widehat{F}^{s}_{t_0,m}(\cdot) $, $ \widehat{G}^{s}_{t_0,m}(\cdot) $ are $ C^1 $, and (A4) holds.
\end{enumerate}

In the following, we use the {notation}: $ a_n \lesssim b_n $, $ n \to \infty $ ($ a_n \geq 0 $, $ b_n > 0 $), meaning that $ \sup_{n\geq 0}b^{-1}_n a_n <\infty $. \label{notation1}

The following theorem is a corresponding result of \cite[Theorem II]{Che18b} for the  discrete version. The invariance of the strong stable foliation about $ H(t) $ (i.e. the following conclusion (ii) (d) (e)) can be proved as the same way as \autoref{thm:A}, so the details are omitted.

\begin{thm}[strong stable foliation]\label{thm:B}
	In the context of \autoref{thm:A}, let (A5) (a) hold.
	\begin{enumerate}[(1)]
		\item For possibly further smaller $ \eta, \epsilon, \chi(\epsilon) $, there are $ \sigma_0, \sigma_1 > 0 $ ($ \sigma_0 < \sigma_1 < \sigma $), an open $ V \subset W^{cs}_{loc}(\Sigma) $ (in the immersed topology) and a foliation $ \mathcal{W}^{ss} $ of $ V $ (called the \textbf{strong stable foliation}) such that $ W^{cs}_{\sigma_1}(\Sigma) \subset V $, and each leaf $ W^{ss}(z_0) $ and each small plaque $ W^{ss}_{\sigma_{z_0}}(z_0) $ of $ \mathcal{W}^{ss} $ (with diameter $ \sigma_{z_0} $) through $ z_0 \in V $ have the following properties:
		\begin{enumerate}[(i)]
			\item (Lipschitz representation)
			$ W^{ss}_{\sigma_{z_0}} (z_0) $ can be represented as a Lipschitz graph of $ X^{s}_{\phi(\widehat{m})} (\sigma_0) \to X^{cu}_{\phi(\widehat{m})} $ with a Lipschitz constant less than $ \approx \beta_{s}(\phi(\widehat{m})) $ and $ \sigma_{z_0} = \sigma_0 $, where $ z_0 = (\widehat{m}, x^s, x^u) \in W^{cs}_{\sigma_1}(\widehat{K}_{\varepsilon}) $.

			\item (invariance)
			(a) $ z_0 \in W^{ss}_{\sigma_{z_0}} (z_0) $, $
			\mathcal{H} (W^{ss}_{\sigma_{z_0}} (z_0)) \subset W^{ss}_{\sigma_0} (\mathcal{H}(z_0)) $, $ \mathcal{H} (W^{ss}(z_0)) \subset W^{ss}(\mathcal{H}(z_0)) $, $ z_0 \in V $; (b) $ \mathcal{H}(W^{cs}_{loc}(\Sigma)) \subset W^{cs}_{\sigma_1} (\widehat{K}_\varepsilon) $; (c) $ W^{ss} (z_0) = \bigcup_{z \in W^{ss} (z_0) } W^{ss}_{\sigma_{z}} (z) $;
			(d) $ \varphi^{t}(W^{ss}(z_0)) \subset W^{ss}(\varphi^{t}(z_0)) $ for $ t \geq t_0 $, $ z_0 \in V $; (e) $ \varphi^{t} (W^{ss}_{\sigma_0}(z)) \subset W^{ss}_{\sigma_0}(\varphi^{t}(z)) $ for $ t \geq 0 $, $ z \in W^{cs}_{\sigma_1} (\widehat{K}_\varepsilon) $.

			\item (characterization) if $ z \in W^{ss} (z_0) $, then $ |\mathcal{H}^n(z) - \mathcal{H}^n(z_0)| \lesssim \varepsilon^{(n)}_{s} (z_0) $, $ n \to \infty $;
			if $ z \in V $ and $ |\mathcal{H}^n(z) - \mathcal{H}^n(z_0)| \lesssim \varepsilon^{(n)}_{s} (z_0) $, $ n \to \infty $, then $ z \in W^{ss}(z_0) $, where $ z_0 = (\widehat{m}, x^s, x^u) \in V $, the function $ \varepsilon_s(\cdot) $ satisfies $ \lambda_s(\phi(\widehat{m})) + \varsigma < \varepsilon_s(z_0) < \lambda^{-1}_{cu}(\phi(\widehat{m})) - \varsigma $ for small $ \varsigma > 0 $ depending on $ \epsilon, \chi(\epsilon), \eta $ such that $ \varsigma \to 0 $ as $ \epsilon, \chi(\epsilon), \eta \to 0 $ and $ \sup \varepsilon_{s}(\cdot) < 1 $,
			\begin{gather*}
			\varepsilon^{(n)}_{s} (z_0) = \varepsilon_{\kappa} ( (\mathcal{H})^{n-1} (z_0)) \cdots \varepsilon_{\kappa} ( \mathcal{H} (z_0)) \varepsilon_{\kappa} ( z_0 ).
			\end{gather*}
			In fact, if $ z, z_0 \in W^{cs}_{\sigma_1} (\widehat{K}_\varepsilon) $, then $ z \in W^{ss}(z_0) $ if and only if $ |\varphi^t(z) - \varphi^t(z_0)| \lesssim \varepsilon^{t-r}_{s} (\varphi^r(z_0)) $ as $ t \to \infty $ for a large $ r > 0 $.

			\item $ \mathcal{W}^{ss} $ is a $ C^0 $ foliation; in fact, $ W^{cs}_{\sigma_1}(\widehat{K}_{\varepsilon}) \ni z_0 \mapsto W^{ss}_{\sigma_0} (z_0) $ is uniformly (locally) H\"older.

			\item \label{q11} In addition, let (A5) (b) hold. Then each small plaque $ W^{ss}_{\sigma_{z_0}}(z_0) $ is $ C^1 $.

			\item Under (v) with additional assumption that $ D\widehat{F}^{\kappa}_{t_0,m}(\cdot) $, $ D\widehat{G}^{\kappa}_{t_0,m}(\cdot) $, $ m \in \Sigma $, $ \kappa = cs, s $, are equicontinuous, it holds that $ W^{cs}_{\sigma_1}(\widehat{K}_{\varepsilon}) \ni z_0 \mapsto W^{ss}_{\sigma_{0}}(z_0) $ is uniformly $ C_0 $ in $ C^1 $-topology in bounded subsets (in the immersed topology).
		\end{enumerate}

		\item (invariant case) Assume $ \eta = 0 $.
		\begin{enumerate}[(i)]
			\item Then for $ z_0 \in W^{cs}_{loc}(\Sigma) $, $ \mathcal{H}^n(z_0) \to \Sigma $. In fact, the convergence is uniform for $ z_0 $; that is $ \mathcal{H}^n(W^{cs}_{loc}(\Sigma)) \to \Sigma $ as $ n \to \infty $. Also, $ \varphi^t(W^{cs}_{\sigma_1}(\widehat{K}_{\varepsilon})) \to \Sigma $ as $ t \to \infty $.
			\item Let (A5) (b) hold. Further, assume for all $ m \in \Sigma $,
			\[
			D_{x^{cs}}\widehat{G}^{cs}_{t_0,m}( 0, 0, 0 ) = 0,~
			D_{x^{s}}\widehat{G}^{s}_{t_0,m}( 0, 0, 0 ) = 0,
			\]
			then $ T_{m}W^{ss}_{\sigma_{m}}(\widehat{m}) = \widehat{X}^s_{m} $, where $ m = \phi(\widehat{m}) \in \Sigma $.

			\item In addition, assume that (a) $ z \mapsto \widehat{F}^s_{t_0,m}(z), \widehat{G}^s_{t_0,m}(z) $ are $ C^{1,1} $ uniform for $ m \in \Sigma $, that (b) $ \Sigma \in C^{1} $, and $ m \mapsto \Pi^{\kappa}_m $, $ \kappa = s, c, u $ are $ C^1 $ in the immersed topology, and that (c) $ \sup_m\lambda_{cs}(m) \lambda_{s}(m) \lambda_{cu}(m) < 1 $. Then $ \mathcal{W}^{ss} $ is a $ C^1 $ foliation (or more precisely, a $ C^1 $ bundle over $ \widehat{\Sigma} $).
		\end{enumerate}
	\end{enumerate}
\end{thm}

See \cite{Che18b} for more results about the regularity of the strong stable foliation.

Next we consider the manifold being both approximately (strictly) inflowing and overflowing and approximately (full) normal hyperbolicity. We need the following assumptions.
\begin{enumerate}[({B}1)]
	\item Let (A1) hold but $ K = \Sigma $.

	\item \textbf{(i)} Assume $ t: \Sigma \to \Sigma $ is a $ C_0 $ flow (in the immersed topology) and let $ \xi > 0 $ be small and $ \bm{t_0 > 0} $.
	\textbf{(ii)} There are functions $ v^{t} : \Sigma \to X $, $ t \in [-t_0,t_0] $, which are {$ \xi $-almost equicontinuous} (in the immersed topology), and a (small) $ \xi_1 > 0 $ such that $ \sup_{t\in [-t_0,t_0]}\sup_{m \in \Sigma}|v^{t}(m) - t(m)| \leq \xi_1 $.

	\item $ \Sigma $ satisfies the following `\emph{approximately (full) normal hyperbolicity}' condition with respect to $ H $.
	Assume there is a small $ \xi_2 > 0 $ such that $ \sup_{m \in \Sigma}|\widehat{\Pi}^\kappa_m - \Pi^\kappa_m| \leq \xi_2 $, $ \kappa = s, c, u $.
	\begin{enumerate}[(a)]
		\item ((A) (B) condition)
		Let $ \kappa_1 = cs $, $ \kappa_2 = u $, $ \kappa = cs $, or $ \kappa_1 = s $, $ \kappa_2 = cu $, $ \kappa = cu $.
		Suppose $ \widehat{H} \approx (\widehat{F}^{\kappa}, \widehat{G}^{\kappa}) $ satisfies (A$ _0 $) ($ {\alpha}_{\kappa_2} $; ${\alpha}_{\kappa_2}' $, $ {\lambda}_{\kappa_2} $; $ c $) (B$ _0 $) ($ {\beta}_{\kappa_1} $; $ {\beta}_{\kappa_1}' $, $ {\lambda}_{\kappa_1} $; $ c $) condition in $ \kappa_1 $-direction.
		Moreover,

		\textbf{(i)} (angle condition) $ \sup_m \alpha'_{\kappa_2}(m) \beta'_{\kappa_1}(m) < 1/2 $,
		$ \inf_m\{ \alpha_{\kappa_2}(m) - \alpha'_{\kappa_2}(m) \} > 0 $, $ \inf_m\{ \beta_{\kappa_1}(m) - \beta'_{\kappa_1}(m) \} > 0 $;

		\textbf{(i$ ' $)} $ \sup_{m}\alpha'_{u}(m)\beta'_{s}(m) < 1 $, $ \alpha'_{cu}(m) \leq 1 $, $ \beta'_{cs}(m) \leq 1 $, $  \beta'_{cs}(m) \leq \beta'_{s}(m) $, $ \alpha'_{cu}(m) \leq \alpha'_{u}(m) $, $ m \in \Sigma $;

		\textbf{(ii)} (spectral condition) $ \sup_m \lambda_{s}(m) < 1 $, $ \sup_m \lambda_{u}(m) < 1 $, $ \sup_m \lambda_{\kappa_1}(m) \lambda_{\kappa_2}(m) < 1 $;

		\textbf{(iii)} $ \alpha_{\kappa_2}(\cdot) $, $ \alpha'_{\kappa_2}(\cdot) $, $ \beta_{\kappa_1}(\cdot) $, $ \beta'_{\kappa_1}(\cdot) $, $ \lambda_{\kappa_2}(\cdot) $, $ \lambda_{\kappa_1}(\cdot) $ are \emph{bounded} and $ \xi $-almost uniform continuous (in the immersed topology).

		\item (approximation) There is a small $ \eta > 0 $ such that for $ \kappa = cs, cu $,
		\begin{gather*}
		\sup_{t\in[0,t_0]}\sup_{m \in \Sigma} |\widehat{F}^{\kappa}_{t,m}(0,0)| \leq \eta, ~
		\sup_{t\in[0,t_0]}\sup_{m \in \Sigma} |\widehat{G}^{\kappa}_{t,m}(0,0)| \leq \eta.
		\end{gather*}

		\item
		($ s $-contraction and $ u $-expansion)
		For some small $ r_0 > 0 $ and any $ (\hat{x}^s_i, \hat{x}^c_i, \hat{x}^{u}_i) \times (\tilde{x}^s_i, \tilde{x}^{c}_i, \tilde{x}^{u}_{i}) \in \graph H(t_0, m) \cap \{\widehat{X}^{s}_{m}(r_{0}) \oplus \widehat{X}^{c}_{m}(r_{0}) \oplus \widehat{X}^{u}_{m}(r_{0}) \times \widehat{X}^{s}_{t_0(m)}(r_{0}) \oplus \widehat{X}^{c}_{t_0(m)}(r_{0}) \oplus \widehat{X}^{u}_{t_0(m)} (r_{0})\} $, $ i = 1,2 $, $ m \in \Sigma $,
		\begin{enumerate}[(i)]
			\item if $ \tilde{x}^s_1 = \tilde{x}^s_2 = 0 $ and $ |\tilde{x}^{u}_{1} - \tilde{x}^{u}_{2}| \leq B |\hat{x}^s_1 - \hat{x}^s_2| $, then $ |\tilde{x}^{s}_{1} - \tilde{x}^{s}_{2}| \leq \lambda^*_s  |\hat{x}^s_1 - \hat{x}^s_2| $;
			\item if $ \tilde{x}^{u}_{1} = \tilde{x}^{u}_{2} = 0 $ and $ |\hat{x}^{s}_{1} - \hat{x}^{s}_{2}| \leq B |\tilde{x}^u_1 - \tilde{x}^u_2| $, then $ |\hat{x}^{u}_{1} - \hat{x}^{u}_{2}| \leq \lambda^*_u  |\tilde{x}^u_1 - \tilde{x}^u_2| $,
		\end{enumerate}
		where $ B > \max \{\sup_{m \in M}c(\omega)\lambda^{t_0}_{cs}(m) \beta_{cs}(m), \sup_{m \in M}c(\omega)\lambda^{t_0}_{cu}(m) \alpha_{cu}(m) \} $ is some (large) constant and $ \lambda^*_s < 1, \lambda^*_u < 1 $.

	\end{enumerate}

	\item (smooth condition) Assume for every $ m \in \Sigma $, $ \kappa = cs, cu $, $ \widehat{F}^{\kappa}_{t_0,m}(\cdot) $, $ \widehat{G}^{\kappa}_{t_0,m}(\cdot) $ are $ C^1 $.
\end{enumerate}

We say a submanifold $ \Sigma $ of $ X $ satisfying assumption (B1) is \emph{approximately invariant} and \emph{approximately (full) normally hyperbolic} with respect to $ H $ if the assumptions (B2) (B3) hold. The following result is a corollary of \autoref{thm:A} and \autoref{thm:B} by using the notion of \emph{dual correspondence} (see \autoref{dual}); see also \cite[Corollary III]{Che18b}

\begin{cor}[trichotomy case]\label{cor:tri}
	Let (B1) (B2) (B3) hold. If $ \xi, \xi_1, \xi_2, \eta $ are small, and $ \chi(\epsilon) $ is small when $ \epsilon $ is small, then there are positive $ \varepsilon $, $ \sigma $, $ \varrho $ small such that the following hold.

	\begin{enumerate}[(1)]
		\item In $ X^s_{\widehat{\Sigma}} \oplus X^u_{\widehat{\Sigma}} $, there are three sets $ W^{cs}_{loc}(\Sigma) $, $ W^{cu}_{loc}(\Sigma) $, $ \Sigma^c $, called the \textbf{local center-stable, local center-unstable, local center manifold} of $ \Sigma $, respectively, which are defined and characterized by
		\begin{alignat*}{2}
		W^{cs}_{loc}(\Sigma) & =
		\{ z \in X^s_{\widehat{\Sigma}}(\sigma) \oplus X^u_{\widehat{\Sigma}} (\varrho): \exists && \{ z_n \}_{n\geq 0} \text{ such that } z_0 = z,  \\
		& && \text{ and it is a ($ \sigma,\varrho,\varepsilon $)-forward orbit of } H(t_0) \}
		\\
		& = \{ z \in X^s_{\widehat{\Sigma}}(\sigma) \oplus X^u_{\widehat{\Sigma}} (\varrho): \exists && \text{ a ($ \sigma,\varrho,\varepsilon $)-forward orbit $ \{ z_n \}_{n\geq 0} $ of } H(t_0) \\
		& && \text{ such that } z_0 = z, d(z_n, \Sigma^c) \to 0, n \to \infty \},\\
		W^{cu}_{loc}(\Sigma) & =
		\{ z \in X^s_{\widehat{\Sigma}}(\varrho) \oplus X^u_{\widehat{\Sigma}} (\sigma): \exists && \{ z_{-n} \}_{n\geq 0} \text{ such that } z_0 = z,  \\
		& && \text{ and it is a ($ \varrho,\sigma,\varepsilon $)-backword orbit of } H(t_0) \}
		\\
		& = \{ z \in X^s_{\widehat{\Sigma}}(\varrho) \oplus X^u_{\widehat{\Sigma}} (\sigma): \exists && \text{ a ($ \varrho,\sigma,\varepsilon $)-backword orbit $ \{ z_{-n} \}_{n\geq 0} $ of } H(t_0) \\
		& && \text{ such that } z_0 = z, d(z_{-n}, \Sigma^c) \to 0, n \to \infty \}, \\
		\Sigma^c & = W^{cs}_{loc}(\Sigma) \cap W^{cu}_{loc}(\Sigma) \\
		& = \{ z \in X^s_{\widehat{\Sigma}}(\sigma) \oplus X^u_{\widehat{\Sigma}} (\sigma): \exists && \{ z_{-n} \}_{n \in \mathbb{Z}} \text{ such that } z_0 = z,  \\
		& && \text{ and it is a ($ \sigma,\sigma,\varepsilon $)-orbit of } H(t_0) \}.
		\end{alignat*}

		Moreover, they have the following properties.
		\begin{enumerate}[(i)]
			\item $ W^{cs}_{loc}(\Sigma) \subset H(t_0)^{-1} W^{cs}_{loc}(\Sigma) $, $ W^{cu}_{loc}(\Sigma) \subset H(t_0) W^{cu}_{loc}(\Sigma) $, $ \Sigma^{c} \subset H(t_0)^{\pm1} \Sigma^c $.
			\item $ W^{cs}_{loc}(\Sigma) $, $ W^{cu}_{loc}(\Sigma) $, $ \Sigma^c $ can be represented as graphs of maps respectively. That is there are maps $ h^{\kappa}_0 $, $ \kappa = cs, cu, c $, such that for $ \widehat{m} \in \widehat{\Sigma} $,
			\begin{gather*}
			h^{cs}_0(\widehat{m}, \cdot): X^s_{\phi(\widehat{m})} (\sigma) \to X^u_{\phi(\widehat{m})}(\varrho), \\
			h^{cu}_0(\widehat{m}, \cdot): X^u_{\phi(\widehat{m})} (\sigma) \to X^s_{\phi(\widehat{m})}(\varrho), \\
			h^{c}_0 (\widehat{m}) \in X^s_{\phi(\widehat{m})}(\sigma) \oplus X^u_{\phi(\widehat{m})}(\sigma),
			\end{gather*}
			and $ W^{cs}_{loc}(\Sigma) = \graph h^{cs}_0 $, $ W^{cu}_{loc}(\Sigma) = \graph h^{cu}_0 $, $ \Sigma^{c} = \graph h^c_0 $.
			$ h^{\kappa}_0 $, $ \kappa = cs, cu, c $, are Lipschitz in the following sense. There are functions $ \mu_{\kappa}(\cdot) $, $ \kappa = cs, cu, c $,
			such that $ \mu_{cs}(m) = (1+\chi_{*}) \beta'_{cs}(m) + \chi_{*} $, $ \mu_{cu}(m) = (1+\chi_{*}) \alpha'_{cu}(m) + \chi_{*} $ and $ \mu_{c} = \max\{ \mu_{cs}, \mu_{cu} \} $, with $ \chi_{*} \to 0_+ $ as $ \epsilon,\chi(\epsilon),\eta \to 0 $, and for every $ m \in \Sigma $, it holds for $ \kappa = s, u $,
			\begin{gather*}
			|\Pi^{su-\kappa}_{m}( h^{c\kappa}_0(\widehat{m}_1, x^\kappa_1) - h^{c\kappa}_0(\widehat{m}_2, x^{\kappa}_2)  )| \leq \mu_{c\kappa}(m) \max \{| \Pi^c_{m}(  \phi(\widehat{m}_1) - \phi(\widehat{m}_2)  ) |, | \Pi^\kappa_{m}( x^\kappa_1 - x^\kappa_2 ) |\},\\
			\max\{|\Pi^u_{m}( h^{c}_0(\widehat{m}_1) - h^{c}_0(\widehat{m}_2)  )|, |\Pi^s_{m}( h^{c}_0(\widehat{m}_1) - h^{c}_0(\widehat{m}_2) )|\} \leq \mu_{c}(m) | \Pi^c_{m}(  \phi(\widehat{m}_1) - \phi(\widehat{m}_2)  ) |,
			\end{gather*}
			where $ \widehat{m}_i \in \widehat{U}_{\widehat{m}}(\varepsilon) $, $ x^{\kappa}_i \in X^{\kappa}_{\phi(\widehat{m}_i)} (\sigma) $, $ \widehat{m} \in \phi^{-1}(m) $.
			\item Moreover, the maps $ h^{\kappa}_0 $, $ \kappa = cs, cu, c $, satisfying (i) (ii), are unique. Also, if $ \eta = 0 $, then $ \Sigma^c = \Sigma = W^{cs}_{loc}(\Sigma) \cap W^{cu}_{loc}(\Sigma) $.
			\item There is a positive constant $ \sigma_1 $ less than $ \sigma $ such that $ W^{cs}_{\sigma_1}(\Sigma) \subset H(t)^{-1} W^{cs}_{loc}(\Sigma) $, $ W^{cu}_{\sigma_1}(\Sigma) \subset H(t)W^{cu}_{loc}(\Sigma) $, $ \Sigma^{c} \subset H(t)^{\pm1} \Sigma^c $ for all $ t \in \mathbb{R}_+ $.
		\end{enumerate}

		\item $ H: W^{cs}_{\sigma_1}(\Sigma) \to W^{cs}_{loc}(\Sigma) $, $ H^{-1}: W^{cu}_{\sigma_1}(\Sigma) \to W^{cu}_{loc}(\Sigma) $, and $ H: \Sigma^{c} \to \Sigma^{c} $ induce semiflows $ \varphi^t_{cs}, \varphi^{t}_{cu}, \varphi^t_{c} $ respectively, with $ \varphi^t_{c} $ being a flow; that is, for any $ z^{cs}, z^{cu}, z^{c} $ belonging to $ W^{cs}_{\sigma_1}(\Sigma) $, $ W^{cu}_{\sigma_1}(\Sigma) $, $ \Sigma^{c} $, respectively, there are a forward orbit $ \{z^{cs}_{t}\}_{t \geq 0} \subset W^{cs}_{loc}(\Sigma) $, a backward orbit $ \{z^{cu}_{t}\}_{t \leq 0} \subset W^{cu}_{loc}(\Sigma) $, an orbit $ \{z^{c}_{t}\}_{t \in \mathbb{R}} \subset \Sigma^c $ of $ H $, respectively, such that $ z^{\kappa}_0 = z^{\kappa} $, $ \kappa = cs, cu, c $; now let $ \varphi^t_{\kappa} = z^{\kappa}_{t} $, $ \kappa = cs, cu, c $.
		Furthermore, $ \varphi^t_{cs} (W^{cs}_{\sigma_1}(\Sigma)) \to \Sigma^c $, $ \varphi^t_{cu} (W^{cu}_{\sigma_1}(\Sigma)) \to \Sigma^c $, as $ t \to \infty $. Set $ \varphi^{t_0}_{cs} = \mathcal{H}^{cs}_0 $, $ \varphi^{t_0}_{cu} = \mathcal{H}^{-cu}_0 $, $ \varphi^{t_0}_{c} = \mathcal{H}^c_0 $.

		\item In addition, let (B4) hold. Then $ W^{cs}_{loc}(\Sigma) $, $ W^{cu}_{loc}(\Sigma) $, $ \Sigma^c $ are $ C^1 $ immersed submanifolds of $ X $. Let $ \widetilde{X}^{\kappa}_{m_c} \triangleq T_{m_c}W^{\kappa}_{loc}(\Sigma) $, $ \kappa = cs, cu $, $ \widetilde{X}^{c}_{m_c} \triangleq T_{m_c} \Sigma^c $, $ m_c \in \Sigma^c $, then $ \widetilde{X}^{cs}_{m_c} \cap \widetilde{X}^{cu}_{m_c} = \widetilde{X}^{c}_{m_c} $. Moreover, $ D\varphi^{t}_{cs}(m_c) \widetilde{X}^{cs}_{m_c} \subset \widetilde{X}^{cs}_{\varphi^{t}_{cs}(m_c)} $, $ D\varphi^{t}_{cu}(m_c) \widetilde{X}^{cu}_{m_c} \subset \widetilde{X}^{cu}_{\varphi^{t}_{cu}(m_c)} $, $ D\varphi^{t}_{c}(m_c) \widetilde{X}^{c}_{m_c} = \widetilde{X}^{c}_{\varphi^{t}_{c}(m_c)} $, $ m_c \in \Sigma^c $.

		\item There is a $ C^0 $ foliation $ \mathcal{W}^{ss} $ (resp. $ \mathcal{W}^{uu} $) of $ W^{cs}_{loc}(\Sigma) $ (resp. $ W^{cu}_{loc}(\Sigma) $), with leaves $ W^{ss}(m_c) $ (resp. $ W^{uu}(m_c) $), $ m_c \in \Sigma^c $, called \textbf{strong stable foliation} (resp. \textbf{strong unstable foliation}) such that the following properties hold. Set $ m_c = (\widehat{m}, x^s, x^u) \in \Sigma^c $.
		\begin{enumerate}[(i)]
			\item $ \mathcal{W}^{ss} $ (resp. $ \mathcal{W}^{uu} $) is invariant under $ \varphi^t_{cs} $ (resp. $ \varphi^{t}_{cu} $), meaning that $ \varphi^t_{cs} (W^{ss}(m_c) \cap W^{cs}_{\sigma_1}(\Sigma)) \subset W^{ss}(\varphi^t_{c}(m_c)) $, $ t \geq 0 $ (resp. $ \varphi^t_{cu} (W^{uu}(m_c) \cap W^{cu}_{\sigma_1}(\Sigma)) \subset W^{uu}(\varphi^t_{c}(m_c)) $, $ t \leq 0 $).
			\item The leaves are characterized by the following. $ z \in W^{ss}(m_c) $ (resp. $ z \in W^{uu}(m_c) $) if and only if $ z \in W^{cs}_{loc}(\Sigma) $ (resp. $ z \in W^{cu}_{loc}(\Sigma) $) and $ |(\mathcal{H}^{cs}_0)^{n}(z) - (\mathcal{H}^{cs}_0)^{n}(m_c)| \lesssim \varepsilon^{(n)}_{s}(m_c) $ (resp. $ |(\mathcal{H}^{-cu}_0)^{n}(z) - (\mathcal{H}^{-cu}_0)^{n}(m_c)| \lesssim \varepsilon^{(n)}_{u}(m_c) $), $ n \to \infty $, where $ \lambda_s(\phi(\widehat{m})) + \varsigma < \varepsilon_s(m_c) < \lambda^{-1}_{cu}(\phi(\widehat{m})) - \varsigma $ (resp. $ \lambda_u(\phi(\widehat{m})) + \varsigma < \varepsilon_u(m_c) < \lambda^{-1}_{cs}(\phi(\widehat{m})) - \varsigma $) for small $ \varsigma > 0 $ depending on $ \epsilon, \chi(\epsilon), \eta $ such that $ \varsigma \to 0 $ as $ \epsilon, \chi(\epsilon), \eta \to 0 $. Here $ \sup \{\varepsilon_{s}(\cdot), \varepsilon_{u}(\cdot) \} < 1 $,
			\begin{gather*}
			\varepsilon^{(n)}_{s} (m_c) = \varepsilon_{\kappa} ( (\mathcal{H}^{c}_0)^{n-1} (m_c)) \cdots \varepsilon_{\kappa} ( \mathcal{H}^{c}_0 (m_c)) \varepsilon_{\kappa} ( m_c ),\\
			\varepsilon^{(n)}_{u} (m_c) = \varepsilon_{\kappa} ( (\mathcal{H}^{c}_0)^{-n+1} (m_c)) \cdots \varepsilon_{\kappa} ( (\mathcal{H}^{c}_0)^{-1} (m_c)) \varepsilon_{\kappa} ( m_c ).
			\end{gather*}
			In fact, for $ \kappa = s, u $, if $ z, z_0 \in W^{c\kappa}_{\sigma_1} (\Sigma) $, then $ z \in W^{\kappa\kappa}(z_0) $ if and only if $ |\varphi^t_{c\kappa}(z) - \varphi^t_{c\kappa}(z_0)| \lesssim \varepsilon^{t-r}_{\kappa} (\varphi^r_{c\kappa}(z_0)) $ as $ t \to \infty $ for a large $ r > 0 $.

			\item There is a small $ \sigma_0 > 0 $ such that each small plaque $ W^{ss}_{\sigma_0}(m_c) $ (resp. $ W^{uu}_{\sigma_0}(m_c) $) of $ \mathcal{W}^{ss} $ (resp. $ \mathcal{W}^{uu} $) (with diameter $ \sigma_0 $) through $ m_c \in \Sigma^c $ is a Lipschitz graph of $ X^{s}_{\phi(\widehat{m})} (\sigma_0) \to X^{cu}_{\phi(\widehat{m})} $ (resp. $ X^{u}_{\phi(\widehat{m})} (\sigma_0) \to X^{cs}_{\phi(\widehat{m})} $) with Lipschitz constant less than approximately $  \beta'_{s} (\phi(\widehat{m})) $ (resp. $ \alpha'_{u} (\phi(\widehat{m})) $).

			\item The holonomy maps for $ \mathcal{W}^{ss} $ and $ \mathcal{W}^{uu} $ are uniformly (locally) H\"older, or equivalently $ \mathcal{W}^{ss} $ and $ \mathcal{W}^{uu} $ are uniformly (locally) H\"older foliations in the immersed topology.

			\item In addition, let (B4) hold. Then each leaf of $ \mathcal{W}^{ss} $ and $ \mathcal{W}^{uu} $ is a $ C^1 $ immersed submanifold of $ X $. Let $ T_{m_c} W^{ss}(m_c) = \widetilde{X}^{s}_{m_c} $ and $ T_{m_c} W^{uu}(m_c) = \widetilde{X}^{u}_{m_c} $. Then $ \widetilde{X}^{\kappa_1\kappa_2}_{m_c} = \widetilde{X}^{\kappa_1}_{m_c} \oplus \widetilde{X}^{\kappa_2}_{m_c} $, $ \kappa_1 = c $, $ \kappa_2 = s, u $, $ X = \widetilde{X}^{s}_{m_c} \oplus \widetilde{X}^{c}_{m_c} \oplus \widetilde{X}^{u}_{m_c} $ and, $ m_c \mapsto \widetilde{X}^{\kappa}_{m_c} $ is continuous (in the immersed topology), $ \kappa = s, c, u $. Moreover, under additional assumption that $ D\widehat{F}^{\kappa}_{t_0,m}(\cdot) $, $ D\widehat{G}^{\kappa}_{t_0,m}(\cdot) $, $ m \in \Sigma $, $ \kappa = cs, cu $, are equicontinuous, it holds that $ \Sigma^{c} \ni m_c \mapsto W^{c\kappa}_{\sigma_0}(m_c) $ is uniformly $ C^0 $ in $ C^1 $-topology in bounded sets (in the immersed topology), $ \kappa = s, u $.

			\item In fact, $ W^{c\kappa}_{loc}(\Sigma) $ is a $ C^0 $ bundle over $ \Sigma^{c} $ with fibers $ W^{\kappa\kappa}(m_c) $, $ m_c \in \Sigma^c $, $ \kappa = s, u $.

			\item \label{regF} Assume that (a) $ z \mapsto \widehat{F}^{\kappa}_{t_0, m}(z), \widehat{G}^{\kappa}_{t_0, m}(z) $ are $ C^{1,1} $ uniform for $ m \in \Sigma $, $ \kappa = cs, cu $, that (b) $ \Sigma \in C^{1} $, and $ m \mapsto \Pi^{\kappa}_m $, $ \kappa = s, c, u $, are $ C^1 $ in the immersed topology, and that (c) $ \sup_m\lambda_{cs}(m) \lambda_{s}(m) \lambda_{cu}(m) < 1 $, $ \sup_m \lambda_{cu}(m) \lambda_{u}(m) \lambda_{cs}(m) < 1 $. Then $ \mathcal{W}^{ss} $ and $ \mathcal{W}^{uu} $ are $ C^1 $ foliations (or more precisely, $ C^1 $ bundles over $ {\Sigma}^c $).
		\end{enumerate}
		\item The following exponential tracking property of $ W^{cs}_{loc}(\Sigma) $ and $ W^{cu}_{loc}(\Sigma) $ holds. If $ \{ z_{-t} \}_{t \geq 0} $ is a $ (\sigma, \varrho, \varepsilon) $-backward orbit of $ H $, then there is a unique $ (\sigma, \varrho, \varepsilon) $-orbit $ \{ \overline{z}_{t} \}_{t \in \mathbb{R}} \subset \Sigma^c $ of $ H $ such that $ |z_{-t} - \overline{z}_{-t}| \lesssim \varepsilon^{t-r}_{u}(\varphi^{r}_{c}(\overline{z}_0)) $ as $ t \to \infty $ for a large $ r > 0 $; similarly, if $ \{ z'_{t} \}_{t \geq 0} $ is a $ (\sigma, \varrho, \varepsilon) $-forward orbit of $ H $, then there is a unique $ (\sigma, \varrho, \varepsilon) $-orbit $ \{ \overline{z}'_{t} \}_{t \in \mathbb{R}} \subset \Sigma^c $ of $ H $, such that $ |z'_{t} - \overline{z}'_{t}| \lesssim \varepsilon^{t-r}_{s}(\varphi^{r}_{c}(\overline{z}'_0)) $ as $ t \to \infty $ for a large $ r > 0 $.
	\end{enumerate}
\end{cor}

\subsection{application I: general $ C_0 $ cocycle case}\label{generalA}
\begin{enumerate}[$ \bullet $]
	\item Let $ M $ be a topology space. $ t: M \to M $ is a $ C_0 $ semiflow, i.e. $ \mathbb{R}_+ \times M \to M, (t, \omega) \mapsto t\omega $ is continuous and $ 0\omega = \omega $, $ (t+s)\omega = t(s\omega) $ for all $ t, s \in \mathbb{R}_+ $, $ \omega \in M $.
\end{enumerate}

Many concrete well-posed and ill-posed differential equations can be reformulated as the abstract differential equation \eqref{equ:main}, i.e.
\begin{equation*}
\dot{z}(t) = \mathcal{C}(t\omega)z(t) + f(t\omega)z(t), \\
\end{equation*}
where $ \mathcal{C}(\omega): D(\mathcal{C}(\omega)) \subset Z \to Z $, $ \omega \in M $, are a family of closed linear operators, $ Z $ is a Banach space and $ f $ is a nonlinear map.

\begin{exa}\label{exa:classic}
	We first give two classical results to show how our results can be applied to differential equations.
	\begin{asparaenum}[(a)]
		\item In \cite{Yi93, CY94}, the authors studied the (global) invariant manifolds of \eqref{equ:main} in the setting that $ \mathcal{C}(\omega) \in L(Z, Z) $, $ Z = \mathbb{R}^n $, $ M $ is a smooth compact manifold, $ t: M \to M $ is a smooth flow and $ f \in C^k_b(M \times Z, Z) $ with sufficient smallness of $ \sup_\omega\sup_x\|D_xf(\omega)(x)\| $ by using the Lyapunov-Perron method. The uniform dichotomy on $ \mathbb{R} $ of the smooth cocycle $ \{ T_0(t,\omega) \} $ generated by $ \{A(\omega)\} $ is described simply in the absolute sense, i.e. $ \lambda_s $, $ \lambda_u $ are constant in the assumption (UD) in \autopageref{UD}. The existence of the invariant manifolds of \eqref{equ:main} for this case is the direct consequence of \autoref{thmBc} or \autoref{thmAc} (for the case $ \varepsilon_1(\cdot) \equiv 1 $), and the $ C^{1,1} $ smoothness of those invariant manifolds follows from the corresponding regularity results for the bounded section case, i.e. $ \sup_{\omega}\eta(t,\omega) < \infty $ (see \cite[Section 6.10]{Che18a} for details). Here note that under $ C^{1,1} $ smoothness of $ \omega \mapsto \mathcal{C}(\omega) $ with corresponding spectral gap condition, the (spectral) projections appeared in uniform dichotomy condition (i.e. \autoref{def:ud+}) also depend on $ \omega \in M $ in a $ C^{1,1} $ fashion and so are $ T_{1}, S_1 $; this fact combining with $ C^{1,1} $ smoothness of $ f $ gives the desired regularity condition (needed in \cite[Section 6.10]{Che18a}) on the cocycle generated by the nonlinear equation \eqref{equ:main}. (Higher regularity can be proved in the same way.) The existence of the invariant foliations for the cocycle given in \cite{CY94} is the corollary of \cite[Theorem 7.6]{Che18a} or \autoref{thmAc} for the case $ \eta(\cdot,\cdot) \equiv 0 $, and the $ C^1 $ smoothness of those invariant foliations now follows from \autoref{thm:Regc}; see also \cite[Theorem 7.6]{Che18a}.

		\item In \cite{CL97}, the authors also developed the theory of invariant manifolds for cocycles in Banach spaces. In the first part of their paper, the authors studied the (global) center manifolds of the non-autonomous differential equation \eqref{equ:main} in the setting that $ M = \mathbb{R} $, $ t(s) = t+s $, $ Z $ is a Banach space, $ \{A(t)\} $ generates a $ C_0 $ evolution family $ \{ T_0(t,s) \}_{t \geq s} $ in $ Z $, and $ f \in C^{k,1} = C^{k,1}(\mathbb{R} \times Z, Z) $ with sufficient smallness of $ \sup_{t \in \mathbb{R}}\sup_x\|D_{x}f(t,x)\| $. In \cite{CL97}, the uniform dichotomy on $ \mathbb{R} $ of $ \{T_0(t,s)\} $ is also described simply in the absolute sense, i.e. $ \lambda_s $, $ \lambda_u $ are constant in the assumption (UD) in \autopageref{UD}.
		$ f \in C^{k,1} $ here means that $ (t,x) \mapsto D^i_xf(t,x) $, $ i = 0, 1, \ldots, k $, are continuous and
		\[
		|f|_{C^{k,1}} = \max\{ \sup_t\lip_x D^k_x f(t,x), \max \{\sup_{(t,x)}\|D^i_xf(t,x)\|: i = 0,1,\ldots, k \}  \} < \infty.
		\]
		It seems that in their definition of $ C^{k,1} $ (see \cite[Page 363]{CL97}), $ \sup_{(t,x)} |f(t,x)| < \infty $ is missing. If this additional assumption removes, $ C^{k,1} $ is no longer a Banach space. In fact, the authors also technically assumed $ \sup_{(t,x)} |f(t,x)| < \infty $ in their proof (see e.g. \cite[Theorem 2.7]{CL97}).
		Moreover, the assumption $ \sup_{(t,x)} |f(t,x)| < \infty $ can be replaced by $ f(t,0) = 0 $ for all $ t \in \mathbb{R} $. No matter in what situation, the existence of the center manifold is a consequence of \autoref{thmAc} (if $ \sup_{(t,x)} |f(t,x)| < \infty $ using the case $ \varepsilon_1(\cdot) \equiv 1 $ and if $ f(t,0) = 0 $ for $ t \in \mathbb{R} $ the case $ \varepsilon_1(\cdot) \equiv 0 $). The $ C^{k-1,1} \cap C^{k} $ regularity of this center manifold follows from \autoref{lem:leaf1}.

		In the second part of their paper, the authors further studied the existence of center manifolds for cocycles (or skew-product flows) generated by \eqref{ggg*1} (see \autoref{genccc}). The authors reduced this case to the non-autonomous differential equation by a `\emph{lifting}' method; for details consult their paper. However, no smoothness result about center manifolds was given. This can be done from our general results; see \autoref{continuous case} and \cite{Che18a}. However, unlike case (a), the regularity condition on spectral projections as well as $ T_{1}, S_1 $ in uniform dichotomy condition (i.e. \autoref{def:ud+}) should be assumed directly. Using \autoref{thm:gencocycleAB}, our results in \autoref{continuous case} as well as \cite[Section 7.2]{Che18a} can be applied to \eqref{equ:main} for this case.
	\end{asparaenum}
\end{exa}

Let us consider the existence and regularity of the invariant manifolds of \eqref{equ:main} in more general settings where $ \mathcal{C} $, $ L $, $ f $ are given by one of (type I) $ \sim $ (type III) listed in \autoref{overview}; for (type III), in fact we study the integral equation \eqref{equ:IE}, and as a matter of convenience, \emph{we identify integral equation \eqref{equ:IE} and differential equation \eqref{equ:main}}. We focus on two situations below:
\begin{enumerate}[($ \circ $1)]
	\item $ f(\omega)(0) = 0 $, $ \omega \in M $;
	\item $ \sup_{t \geq 0} \sup_{z} |f(t\omega)(z)| < \infty $, $ \omega \in M $.
\end{enumerate}

\begin{thm}\label{thm:equI}
	(Case I). Let ($ \circ $1) and the conditions in one of the following cases hold with $ c > 1 $: (i) \autoref{thm:gencocycleAB} (2) with $ \inf_{\omega}\{ \mu_u(\omega) - \mu_s(\omega) - (1+c) \varepsilon_{m}(\omega) \} > 0 $; (ii) \autoref{thm:biAB} or \autoref{them:spec} (1) with $ \inf_{\omega}\{ \mu_u(\omega) - \mu_s(\omega) - (1+c) \varepsilon'(\omega) \} > 0 $; (iii) \autoref{them:spec} (2) with $ \inf_{\omega}\{ \mu_u(\omega) - \mu_s(\omega) - (1+c) \eta(\omega) \} > 0 $. Then there is a unique set $ \mathcal{M} = \bigcup_{\omega \in M} (\omega, \mathcal{M}_{\omega}) \subset M \times Z $ such that the following (1) $ \sim $ (3) hold.
	\begin{enumerate}[(1)]
		\item $ 0 \in \mathcal{M}_{\omega} $;
		\item $ \mathcal{M}_{\omega} = \graph \Psi_{\omega} $, a Lipschitz graph of $ \Psi_{\omega}: X_{\omega} \to Y_{\omega} $ with $ \lip \Psi_{\omega} \leq \beta(\omega) $;
		\item $ \mathcal{M} $ is positively invariant under \eqref{equ:main}, meaning for each $ (\omega,z) \in \mathcal{M} $, there is a (mild) solution $ u(t) $ ($ t\geq 0 $) of \eqref{equ:main} with $ u(0) = z $ such that $ u(t) \in \mathcal{M}_{t\omega} $ for all $ t \geq 0 $.
		\item If $ z \mapsto f(\omega)z $ is $ C^1 $ for each $ \omega \in M $, so is $ x \mapsto \Psi_{\omega}(x) $.
	\end{enumerate}

	(Case II). Let ($ \circ $2) hold. Assume one of the above cases (i) $ \sim $ (iii) holds with an additional spectral condition in each corresponding case: (i) $ \inf_{\omega}\{\mu_u(\omega) - \varepsilon_{m}(\omega)\} > 0 $; (ii) $ \inf_{\omega}\{\mu_u(\omega) - \varepsilon'(\omega)\} > 0 $; (iii) $ \inf_{\omega}\{\mu_u(\omega) - \eta(\omega)\} > 0 $.
	Then there is a unique set $ \mathcal{M} = \bigcup_{\omega \in M} (\omega, \mathcal{M}_{\omega}) \subset M \times Z $ such that (2) (3) hold with $ \sup_{t \geq 0} |\Psi_{t\omega}(0)| < \infty $ for each $ \omega \in M $; moreover this set also satisfies (4). In addition, if all the functions $ \mu_{s}, \mu_{u} $ and $ \varepsilon_{m}(\cdot) $ (or $ \varepsilon'(\cdot) $, $ \eta(\cdot) $) are bounded and $ \sup_{\omega} \sup_{z} |f(\omega)(z)| < \infty $, then $ \sup_{\omega} |\Psi_{\omega}(0)| < \infty $.
\end{thm}
Note that for the case of \autoref{them:spec}, $ \mathcal{M}_{\omega} \subset \overline{D(A)} $. If $ t $ is a flow and \eqref{equ:linear} satisfies \emph{uniform trichotomy} condition (see e.g. \autopageref{def:ut} below) and similar spectral gap condition as \autoref{thm:equI}, then from \autoref{thm:equI}, we see there are three sets which are \emph{invariant}, positively invariant and negatively invariant about \eqref{equ:main} under case ($ \circ $1), and there is an invariant set of \eqref{equ:main} if $ \sup_{t \in \mathbb{R}} \sup_{z} |f(t\omega)(z)| < \infty $, $ \omega \in M $; the precise statement is left to the reader (see also \autoref{thm:triH}).

\begin{proof}
	(Case I). The unique existence of $ \mathcal{M} $ such that (1) $ \sim $ (3) hold directly follows from \autoref{thmAc} with the section $ i = 0 $ and the function $ \eta(\cdot,\cdot) $ thereof satisfying $ \eta(\cdot,\cdot) \equiv 0 $ (i.e. $ \varepsilon_1(\cdot) \equiv 0 $) as ($ \circ $1) holds. The $ C^1 $ smoothness of $ \Psi_{\omega}(\cdot) $ is a consequence of \autoref{lem:leaf1}.

	(Case II). Note that under ($ \circ $2), the cocycle correspondence induced by \eqref{equ:main} fulfills \autoref{thmAc} with the function $ \eta(\cdot,\cdot) $ thereof satisfying $ \sup_{s\geq 0}\eta(t,s\omega) < \infty $ for each $ (t,\omega) \in \mathbb{R}_+ \times M $ and the section $ i = 0 $. Indeed, for any $ 0 \leq t_1 < t_2 $, let $ z(t) = (x(t), y(t)) \in X_{t\omega} \oplus Y_{t\omega} $ ($ t_1 \leq t \leq t_2 $) satisfy \eqref{**0} or \eqref{*06} with $ x(t_1) = 0 $, $ y(t_2) = 0 $, and then $ |x(t_2)| \leq C_{r}(\omega) $ and $ |y(t_1)| \leq C_{r}(\omega) $, where $ t_2 - t_1 = r $; so we can take $ \sup_{s\geq0}\eta(r,s\omega) \leq C_{r}(\omega) $. This can be seen as follows. If $ z(\cdot) $ satisfies \eqref{*06}, then obviously there are $ \lambda(\omega) > 0 $ and $ K(\omega) > 0 $ such that $ |x(t_2)| \leq K(\omega) e^{\lambda(\omega)(t_2 - t_1)} $ and $ |y(t_1)| \leq K(\omega) e^{\lambda(\omega)(t_2 - t_1)} $. If $ z(\cdot) $ satisfies \eqref{**0}, then from \autoref{lem:prelem}, we also have $ \lambda(\omega) > 0 $ and $ K(\omega) > 0 $ such that $ |x(t_2)| \leq K(\omega) e^{\lambda(\omega)(t_2 - t_1)} $ and $ |y(t_1)| \leq K(\omega) e^{\lambda(\omega)(t_2 - t_1)} $. Moreover, if all the functions $ \mu_{s}, \mu_{u} $ and $ \varepsilon_{m}(\cdot) $ (or $ \varepsilon'(\cdot) $, $ \eta(\cdot) $) are bounded and $ \sup_{\omega} \sup_{z} |f(\omega)(z)| < \infty $, then we can choose $ C_{r}(\omega) $ independent of $ \omega \in M $. The proof is complete.
\end{proof}

Recall \autoref{def:ud+}, where we have assumed \eqref{equ:linear} satisfies uniform dichotomy on $ \mathbb{R}_+ $. So we have $ Z_0 = X_{\omega} \oplus Y_{\omega} $ associated with projections $ P_{\omega}, P^{c}_{\omega} = \id - P_{\omega} $, and functions $ \mu_{s}, \mu_{u}: M \to \mathbb{R} $.  Here $ Z_0 = \overline{D(A)} $ if \autoref{them:spec} holds and $ Z_0 = Z $ otherwise. Let $ \hat{\varepsilon}(\cdot) $ be equal to $ \varepsilon_{m}(\cdot) $, $ \varepsilon'(\cdot) $, or $ \eta(\cdot) $ according to case (i) $ \sim $ (iii) in \autoref{thm:equI}.

\begin{thm}\label{thm:equII}
	Assume $ \mu_{s}, \mu_{u}, \hat{\varepsilon} $ are bounded and $ \xi $-almost continuous; the latter means that
	\[
	\limsup_{\omega' \to \omega}\max\{|\mu_{\kappa}(\omega') - \mu_{\kappa}(\omega)|, |\hat{\varepsilon}(\omega') - \hat{\varepsilon}(\omega)|\} \leq \xi, \text{ for each } \omega \in M,
	\]
	where $ \kappa = s, u $.
	Then for small $ \xi > 0 $, the set $ \mathcal{M} $ obtained in \autoref{thm:equI} further satisfies the following:
	(1) $ (\omega,x) \mapsto \Psi_{\omega}(P_{\omega}x) $ is $ C^0 $;
	(2) if, in addition, $ (\omega,z) \mapsto D_zf(\omega)z $ is $ C^0 $, so is $ (\omega,x) \mapsto D_x\Psi_{\omega}(P_{\omega}x) $.
\end{thm}

\begin{proof}
	We mention that, the $ C^0 $ continuity of $ \Psi_{(\cdot)}(\cdot) $ and $ D_{x}\Psi_{(\cdot)}(\cdot) $, which is not the case considered in \autoref{thm:Regc} or even \cite{Che18a} as maybe the spectral subbundles $ \bigsqcup_{\omega} X_{\omega}, \bigsqcup_{\omega} Y_{\omega} $ (in \autoref{def:ud+}) don't have $ C^0 $ topology compatible with the bundle structure, should prove directly, but the strategy is almost the same as \cite{Che18a}. Let $ H \sim (F, G) $ be the cocycle correspondence induced by \eqref{equ:main}. Note that under the spectral condition given in \autoref{thm:equI}, the functions $ \alpha, \beta $ in (A$ '_1 $)($ \alpha, \lambda_u; k $) (B$ _1 $)($ \beta; \beta', \lambda_s; c $) condition can be chosen as constants; also, the functions $ \lambda_s, \lambda_u $ can be chosen as $ C\xi $-almost continuous for certain constant $ C > 0 $.
	Take
	$ \hat{c} = \sup_{\omega} k(\omega) $ if case (iii) holds and $ \hat{c} = 1 $ otherwise, and $ \lambda''_{\kappa}(\omega) = \lambda''_{\kappa}(\omega) + C\xi $ where $ \kappa = s, u $, $ \beta''(\omega) = \beta $, $ \alpha''(\omega) = \alpha $. Then we can let $ t_0 > 0 $ be large and $ \xi $ small such that
	\[
	\theta'' \triangleq \sup_{\omega} \frac{ \hat{c}^2 {\lambda''_{s}}^{t_0}(\omega) {\lambda''_{u}}^{t_0}(\omega) }{1 - \alpha''(\omega) \beta''(\omega)} < 1.
	\]
	The map $ \Psi_{\omega} $ is constructed through the following equation
	\[
	\begin{cases}
	F_{t_0,\omega}(x, \Psi_{t_0\omega}(x_{t_0,\omega}(x))) = x_{t_0,\omega}(x),\\
	G_{t_0,\omega}(x, \Psi_{t_0\omega}(x_{t_0,\omega}(x))) = \Psi_{\omega}(x).
	\end{cases}
	\]
	Consider for fixed $ x \in Z_0 $,
	\begin{align*}
	& \limsup_{\omega' \to \omega} |\Psi_{\omega'}(P_{\omega'}x) - \Psi_{\omega}(P_{\omega}x)| \\
	\leq & \limsup_{\omega' \to \omega}| G_{t_0,\omega'}(P_{\omega'}x, \Psi_{t_0\omega'}(x_{t_0,\omega'}(P_{\omega'}x))) - G_{t_0,\omega}(P_{\omega}x, \Psi_{t_0\omega}(x_{t_0,\omega}(P_{\omega}x))) | \\
	\leq & \limsup_{\omega' \to \omega} | G_{t_0,\omega'}(P_{\omega'}x, \Psi_{t_0\omega'}(x_{t_0,\omega'}(P_{\omega'}x))) - G_{t_0,\omega'}(  P_{\omega'}x, P^c_{t_0\omega'} \Psi_{t_0\omega}(  x_{t_0,\omega}(  P_{\omega}x  )  )  ) | \\
	& + \limsup_{\omega' \to \omega} |G_{t_0,\omega'}( P^c_{\omega'}x, P_{t_0\omega'}\Psi_{t_0\omega}(x_{t_0,\omega}(x)) ) - G_{t_0,\omega}(P_{\omega}x, \Psi_{t_0\omega}(x_{t_0,\omega}(P_{\omega}x)))| \\
	\leq & \hat{c}{\lambda''_{u}}^{t_0}(\omega) \limsup_{\omega' \to \omega} | \Psi_{t_0\omega'}(x_{t_0,\omega'}(P_{\omega'}x)) - \Psi_{t_0\omega}(x_{t_0,\omega}(P_{\omega}x)) | + 0\\
	\leq & \hat{c}{\lambda''_{u}}^{t_0}(\omega) \beta''(\omega) \limsup_{\omega' \to \omega}  | x_{t_0,\omega}(P_{\omega'}x) - x_{t_0,\omega}(P_{\omega}x) | \\
	& + \hat{c}{\lambda''_{u}}^{t_0}(\omega)\limsup_{\omega' \to \omega} | \Psi_{t_0\omega'}(P_{t_0\omega'}\hat{x}) - \Psi_{t_0\omega}(P_{t_0\omega}\hat{x}) |,
	\end{align*}
	where $ \hat{x} = x_{t_0,\omega}(P_{\omega}x) $, and similarly,
	\begin{multline*}
	\limsup_{\omega' \to \omega}  | x_{t_0,\omega}(P_{\omega'}x) - x_{t_0,\omega}(P_{\omega}x) |
	\leq \alpha''(\omega) \beta''(\omega) \limsup_{\omega' \to \omega}  | x_{t_0,\omega}(P_{\omega'}x) - x_{t_0,\omega}(P_{\omega}x) | \\
	+ \alpha''(\omega)\limsup_{\omega' \to \omega} | \Psi_{t_0\omega'}(P_{t_0\omega'}\hat{x}) - \Psi_{t_0\omega}(P_{t_0\omega}\hat{x}) |,
	\end{multline*}
	which yields
	\[
	\limsup_{\omega' \to \omega} |\Psi_{\omega'}(P_{\omega'}x) - \Psi_{\omega}(P_{\omega}x)| \leq \frac{\hat{c}{\lambda''_{u}}^{t_0}(\omega)}{1 - \alpha''(\omega) \beta''(\omega)} \limsup_{\omega' \to \omega} | \Psi_{t_0\omega'}(P_{t_0\omega'}\hat{x}) - \Psi_{t_0\omega}(P_{t_0\omega}\hat{x}) |.
	\]
	Therefore,
	\begin{multline*}
	\sup_{x}\limsup_{\omega' \to \omega} \frac{|\Psi_{\omega'}(P_{\omega'}x) - \Psi_{\omega}(P_{\omega}x)|}{|P_{\omega}x|} \\
	\leq \frac{\hat{c}{\lambda''_{u}}^{t_0}(\omega)}{1 - \alpha''(\omega) \beta''(\omega)} \sup_{x}\limsup_{\omega' \to \omega} \frac{| \Psi_{t_0\omega'}(P_{t_0\omega'}\hat{x}) - \Psi_{t_0\omega}(P_{t_0\omega}\hat{x}) |}{|x_{t_0,\omega}(P_{\omega}x)|} \frac{|x_{t_0,\omega}(P_{\omega}x)|}{|P_{\omega}x|},
	\end{multline*}
	and so,
	\begin{align*}
	\sup_{\omega}\sup_{x}\limsup_{\omega' \to \omega} \frac{|\Psi_{\omega'}(P_{\omega'}x) - \Psi_{\omega}(P_{\omega}x)|}{|P_{\omega}x|}
	\leq & \theta'' \sup_{\omega} \sup_{x}\limsup_{\omega' \to \omega} \frac{| \Psi_{t_0\omega'}(P_{t_0\omega'}\hat{x}) - \Psi_{t_0\omega}(P_{t_0\omega}\hat{x}) |}{|x_{t_0,\omega}(P_{\omega}x)|} \\
	\leq & \theta'' \sup_{\omega} \sup_{x}\limsup_{\omega' \to \omega} \frac{|\Psi_{\omega'}(P_{\omega'}x) - \Psi_{\omega}(P_{\omega}x)|}{|P_{\omega}x|},
	\end{align*}
	Note that
	\[
	\sup_{\omega}\sup_{x}\limsup_{\omega' \to \omega} \frac{|\Psi_{\omega'}(P_{\omega'}x) - \Psi_{\omega}(P_{\omega}x)|}{|P_{\omega}x|} \leq 2\hat{c}\sup_{\omega}\lambda_{s}(\omega) < \infty,
	\]
	and this shows that $ \limsup_{\omega' \to \omega} |\Psi_{\omega'}(P_{\omega'}x) - \Psi_{\omega}(P_{\omega}x)| = 0 $. And because of $ \sup_{\omega}\lip \Psi_{\omega}(\cdot) < \infty $, we see $ (\omega,x) \mapsto \Psi_{\omega}(P_{\omega}x) $ is $ C^0 $.

	For the proof of $ C^0 $ continuity of $ D_{x}\Psi_{(\cdot)}(\cdot) $, first observe that $ K^1_{\omega}(x) = D_x\Psi_{\omega}(x) $ satisfies the following `variant' equations:
	\[
	\begin{cases}
	DF_{t_0,\omega}(x, \Psi_{t_0\omega}(x_{t_0,\omega}(x))) ( \id, K^1_{t_0\omega}(x_{t_0,\omega}(x)) D_{x}x_{t_0,\omega}(x) ) = D_{x}x_{t_0,\omega}(x),\\
	DG_{t_0,\omega}(x, \Psi_{t_0\omega}(x_{t_0,\omega}(x))) ( \id, K^1_{t_0\omega}(x_{t_0,\omega}(x)) D_{x}x_{t_0,\omega}(x) ) = K^1_{\omega}(x).
	\end{cases}
	\]
	Now the same argument in the proof of $ C^0 $ continuity of $ \Psi_{(\cdot)}(\cdot) $ can be applied, which is omitted here.
\end{proof}

Here, we do not give more results about the regularity of $ \Psi_{\omega} $ with respect to $ \omega \in M $; this can be done, for example, if the regularity condition respecting base points on spectral projections and $ T_{1}, S_1 $ in the uniform dichotomy condition (i.e. \autoref{def:ud+}) satisfied by \eqref{equ:main} is assumed directly, which could be induced from the regularity of $ \omega \mapsto L(\omega) $ if $ \{\mathcal{C}(\omega)\} $ has the special form $ \mathcal{C}(\omega) = A + L(\omega) $, $ \omega \in M $.

Finally, let's introduce the following class of non-autonomous differential equations which can be reformulated as \eqref{equ:main} and so our results in \autoref{continuous case} can be applied; see also \cite{CL99}.

\begin{exa}\label{exa:non-auto}
	Consider the following non-autonomous linear differential equation:
	\begin{equation}\label{non*1}
	\dot{z}(t) = a(t)Az(t),
	\end{equation}
	where $ A $ is a generator of a $ C_0 $ semigroup $ T $ and $ a \in L^{\infty}(\mathbb{R}, \mathbb{R}_+\backslash\{0\}) $ (the all measurable bounded functions such that $ \inf a (\cdot) \geq r > 0 $). Let $ M $ be the closure of $ \{ a_t = a(t+\cdot): t \in \mathbb{R}_+ \} $ in $ L^{\infty}(\mathbb{R}, \mathbb{R}) $.
	Set
	\[
	T_0(t,\omega)x = T\left( \int_{0}^{t} \omega(s) ~\mathrm{ d } s \right) x, ~x \in Z,~\omega \in M,
	\]
	Then \eqref{non*1} generates the cocycle $ \{T_0(t,\omega)\} $ on $ Z $ over $ t: M \to M $, where $ (t\omega)(s) \triangleq \omega(t+s) $. Let $ A(\omega) = \omega(0)A $ and $ f: Z \to Z $. One can study the invariant manifolds of \eqref{equ:main} for this type of differential equations.
	If $ A $ is a generator of $ C_0 $ bi-semigroup, then similarly \eqref{non*1} generates cocycle correspondence; see also \autoref{genccc}. Here, note that since we do not assume $ a \in C(\mathbb{R}, \mathbb{R}_+) $, in general, $ \frac{\mathrm{d}}{\mathrm{d}t}|_{t=0}T_0(t,\omega)x $ may not exist except for $ x = 0 $.
\end{exa}

\begin{exa}
	Consider the following non-autonomous linear differential equation:
	\begin{equation}\label{non*2}
	\dot{z}(t) = Az(t) + a_1(t) L_1z(t) + a_1(t) L_2z(t) + \cdots + a_n(t) L_n z(t),
	\end{equation}
	where (i) $ A : D(A) \subset Z \to Z $ is one class of (type $ \bullet $a) $ \sim $ (type $ \bullet $c) listed in \autoref{overview}, and $ Z_0, Z_{-1} $ as well, (ii) $ L_i \in L(Z_0, Z_{-1}) $, and (iii) $ a_i \in L^{\infty}(\mathbb{R}, \mathbb{R}) $, $ i = 1,2,\ldots,n $. Let $ M $ be the closure of
	\[
	\{ (a_1, a_2, \cdots, a_n)_t = (a_1(t+\cdot), a_2(t+\cdot), \cdots, a_n(t+\cdot)): t \in \mathbb{R}_+ \}
	\]
	in $ L^{\infty}(\mathbb{R}, \mathbb{R}^n) $. Define
	\[
	L(\omega)x = \sum_{i=1}^{n}\omega_i(0)L_ix: M \times Z_0 \to Z_{-1}, ~\omega = (\omega_1, \omega_2, \cdots, \omega_n) \in M.
	\]
	The semiflow $ t $ on $ M $ is defined by $ (t\omega)(s) = \omega(t+s) $. Let $ \mathcal{C}(\omega) = A + L(\omega) $ and $ f: M \times Z_0 \to Z_{-1} $. Then one can study the invariant manifolds of \eqref{equ:main} for this type of differential equations. Some concrete examples are the following.
	\begin{enumerate}[(i)]
		\item $ \partial_tz = \partial_{ss}z - a(t)z + f(z) $, $ s \in (0,1) $. Take $ X = L^2(0,1) $, $ Ax = \ddot{x} $, $ x \in D(A) = H^2(0,1) \cap H^1_0(0,1) $.
		\item $ \partial_tz = \partial_{s}z + a(t)z + f(z) $, $ s \in (0,1) $. Take $ X = C[0,1] $, $ Ax = \dot{x} $, $ x \in D(A) = C^1_0[0,1] $. (Note that $ \overline{D(A)} \neq X $.)
		\item Let $ f_i \in \lip(Z,Z) $, $ i = 1,2 $.
		\[
		\dot{z}(t) =
		\begin{cases}
		f_1(z), ~n \leq t < n+1, \\
		f_2(z), ~n+1 \leq t < n+2,
		\end{cases}
		~n=\pm1, \pm3, \ldots.
		\]
		In this case $ f(\omega,z) = \omega_1(0) f_1(z) + \omega_2(0) f_2(z) $, $ A(\omega) = 0 $, where $ \omega = (\omega_1, \omega_2) = (\omega_1, 1-\omega_1) \in M $. $ n = 2 $. $ a_1(t) = 1 $ if $ 2k-1 \leq t < 2k $, $ a_1(t) = 0 $ otherwise. $ a_2 = 1 - a_1 $.
	\end{enumerate}
\end{exa}

\subsection{application II: autonomous different equations around some invariant sets} \label{generalB}

In applications, equation \eqref{equ:main} will arise naturally when we study the following autonomous differential equation around some invariant set. \textbf{($ \bullet $I)} Consider
\begin{equation}\label{2**1}
\begin{cases}
\dot{z}(t) = Az(t) + g(z(t)), \\
z(0) = z_0 \in Z_0,
\end{cases}
\end{equation}
where $ g \in C^{1}(Z_0, Z_{-1}) $, and $ A: D(A) \subset Z \to Z $, $ Z_0, Z_{-1} $ are one of (type $ \bullet $a) $ \sim $ (type $ \bullet $c) listed in \autoref{overview}. Note that $ Z_0 \hookrightarrow D(A^{-\alpha}) $. For some concrete examples of \eqref{2**1}, see \autoref{examples}.

We say a set $ M $ is \emph{positively invariant} under \eqref{2**1} (resp. for time $ t > t_0 $), if for every $ z_0 \in M $, there is a mild solution $ u(t) $ ($ t \geq 0 $) of \eqref{2**1} with $ u(0) = z_0 $ satisfying $ u(t) \in M $ for all $ t \geq 0 $ (resp. $ t > t_0 $). Similar notion of negatively invariant (or invariant) set can be defined as well.

\textbf{($ \bullet $II)} Assume there is a set $ M $ positively invariant under \eqref{2**1} such that it induces a natural semiflow $ t $ in $ M $ by $ t\omega = u(t) $ where $ u(t) $ ($ t \geq 0 $) is  the unique mild solution of \eqref{2**1} in $ M $ with $ u(0) = \omega $.
The invariant $ M $ usually can be taken as equilibriums, (a-)periodic orbit, several orbits or with their closure (including e.g. homoclinic orbits, heteroclinic orbits), or the global compact attractor, etc.

\textbf{($ \bullet $III)} The linearized equation \eqref{2**1} along $ M $ is given by
\begin{equation}\label{2**2}
\dot{z}(t) = Az(t) + L(t\omega)z(t),
\end{equation}
where $ L(\omega) = Dg(\omega): Z_0 \to Z_{-1} $, $ \omega \in M $; in addition, assume that for every $ \omega \in M $, $ \sup_{t \geq 0}|L(t\omega)| = \tau(\omega) < \infty $, and $ \omega \mapsto \tau(\omega) $ is locally bounded (i.e. (D1) in \autopageref{d1L} holds) if $ Z_0 = Z $, and $ \sup_{\omega \in M} |L(\omega)| < \infty $ otherwise.
Now by studying equation \eqref{equ:main} with
\[
f(\omega)z = g(z + \omega) - L(\omega)z - g(\omega): ~Z_0 \to Z_{-1},
\]
i.e.
\begin{equation}\label{2**3}
\dot{z}(t) = Az(t) + L(t\omega)z(t) + f(t\omega)z(t), \\
\end{equation}
one can give some dynamical results about \eqref{2**1} around $ M $, e.g. stability, persistence or bifurcation; see \cite{Hen81,DPL88,Wig94,Tem97,BLZ98,BLZ99,BLZ08,MR09a,ElB12,Zel14}.

The results about abstract dynamical systems in \autoref{continuous case} (or see \cite{Che18a} in detail) with \autoref{them:spec} and \autoref{thm:gencocycleAB} can be applied to \eqref{2**3} directly to obtain different types of invariant manifolds and foliations such as the (un)stable, center-(un)stable and pseudo-(un)unstable manifolds for an equilibrium and strong (un)stable foliations; see \autoref{generalA}. The reader can find more results about invariant foliations in \cite{Che18a} for the discrete case. In the following, let us apply the results in \autoref{normalH} to \eqref{2**1} under ($ \bullet $I) $ \sim $ ($ \bullet $III) and the normal hyperbolicity of $ M $.

\begin{lem}
	$ u(t) $ is a mild solution of \eqref{2**1} if and only if $ z(t) = u(t) - t\omega $ is a mild solution of \eqref{2**3}.
\end{lem}
Set
\[
\mathbb{B}_{\epsilon}(M) = \{ x \in Z_0: d(x, M) < \epsilon \}.
\]
For a map $ h: \mathbb{B}_{r}(M) \to \mathfrak{M} $, where $ \mathfrak{M} $ is a metric space with metric $ d $, the \emph{amplitude} of $ h $ in $ \mathbb{B}_{r}(M) $, as usual, is defined by
\begin{equation*}
\mathfrak{A}_{h|_{\mathbb{B}_{r}(M)}} \triangleq \limsup_{\epsilon \to 0} \{ d(h(m), h(m_0)): |m - m_0| \leq \epsilon, m, m_0 \in \mathbb{B}_{r}(M) \}.
\end{equation*}
For example, if $ h|_{\mathbb{B}_{r}(M)} $ is uniformly continuous if and only if $ \mathfrak{A}_{h|_{\mathbb{B}_{r}(M)}} = 0 $; if $ M $ is precompact in $ Z_0 $, then $ \mathfrak{A}_{h|_{\mathbb{B}_{r}(M)}} $ can be sufficiently small when $ r $ is small. For a $ C^0 $ (resp. $ C^1 $) map $ h: \mathbb{B}_{r}(M) \to B $ where $ B $ is a Banach space, we use the notation $ |h|_{C^0(\mathbb{B}_{r}(M))} = \sup_{z \in \mathbb{B}_{r}(M)}|h(z)| $ (resp. $ |h|_{C^1(\mathbb{B}_{r}(M))} = \max\{|h|_{C^0(\mathbb{B}_{r}(M))}, |Dh(\cdot)|_{C^0(\mathbb{B}_{r}(M))}\} $).

\vspace{.5em}
\noindent \textbf{Settings A}.
\textbf{(AI)} (submanifold). Let $ \Sigma = M \subset Z_0 $ with $ K \subset \Sigma $ satisfy assumption (A1) in \autoref{normalH} with $ \chi(\epsilon) $ sufficiently small as $ \epsilon \to 0 $; so we have projections $ \Pi^{\kappa}_{\omega} $, $ \omega \in M $, $ \kappa = s, c, u $.

\noindent\textbf{(AII)} (semiflow). Assume the semiflow $ t $ is $ C^0 $ in the immersed topology and satisfies the following. \textbf{(i)} $ \exists t_0 > 0 $ such that $ t_0 (M) \subset K $;
\textbf{(ii)} $ t : M \to M $, $ 0 \leq t \leq t_0 $, considered as maps of $ M \to Z_0 $, are $ \xi $-almost equicontinuous around $ K $ (in the immersed topology); see \autopageref{almost}.

Condition \textbf{(i)}, roughly, means the semiflow $ t $ (uniformly) crosses the `boundary' of $ M $ transversally;
condition \textbf{(ii)} is redundant if $ K $ is precompact in $ Z_0 $; if $ t_0 $ is small and $ \sup_{t \in [0,t_0]} \sup_{m\in \Sigma} |t(m) - m| \leq \xi $, then condition \textbf{(ii)} is also satisfied (see \cite{BLZ99, BLZ08} in this case as well).

\noindent\textbf{(AIII)}
Suppose the linear equation \eqref{2**2} satisfies uniform dichotomy on $ \mathbb{R}_+ $ (see \autoref{def:ud+} or \autopageref{UD+}); so we have $ Z_0 = X_{\omega} \oplus Y_{\omega} $ associated with projections $ P_{\omega}, P^{c}_{\omega} = \id - P_{\omega} $, and functions $ \mu_{u}, \mu_{s}: M \to \mathbb{R} $. Take $ C_1 > 1 $ such that $ |P_{\omega}|, |P^{c}_{\omega}| \leq C_1 $. We assume $ \mu_{u}, \mu_{s} $ are bounded, $ \xi $-almost continuous and $ \xi $-almost uniformly continuous around $ K $ (in the immersed topology).

\noindent\textbf{(AIV)} (a) Suppose $ \mathfrak{A}_{Dg|_{\mathbb{B}_{r}(M)}} \leq \chi $ when $ r $ is small. (b) Let $ \widehat{\Pi}^s_{\omega} = 0 $, $ \widehat{\Pi}^{c}_{\omega} = P_{\omega} $, $ \widehat{\Pi}^{u}_{\omega} = P^{c}_{\omega} $, $ \omega \in M $. Assume there is a small $ \xi_2 > 0 $ such that $ \sup_{\omega \in K}|\widehat{\Pi}^\kappa_\omega - \Pi^\kappa_\omega| \leq \xi_2 $, $ \kappa = s, c, u $.

By \autoref{thm:A}, we can give a persistence result about $ M $.
\begin{thm}\label{thm:app1}
	Under above \textbf{Settings A}, let $ \widetilde{g} \in C^1(\mathbb{B}_{r}(M), Z_0) $ such that $ \sup_{z \in \mathbb{B}_{r}(M)}\{ |\widetilde{g}(z) - g(z)| \} \leq \eta $.
	\begin{enumerate}[(1)]
		\item If $ A $ is a generator of a $ C_0 $ semigroup or a $ C_0 $ bi-semigroup, or a Hille-Yosida operator, assume $ \mathfrak{A}_{D\widetilde{g}|_{\mathbb{B}_{r}(M)}} \leq \chi $,
		\[
		\sup_{t \geq 0}\{|D\widetilde{g}(t\omega) - Dg(t\omega)|\} \leq \varepsilon(\omega), ~\omega \in M,
		\]
		and $ \omega \mapsto \varepsilon(\omega) $ is bounded, $ \xi $-almost continuous and $ \xi $-almost uniformly continuous around $ K $ (in the immersed topology).
		If $ \xi,\xi_2, \chi, \eta, r > 0 $ are small, and there is a constant $ c > \sqrt{2} $ such that
		\[
		\inf_{\omega} \{ \mu_{u}(\omega) - \mu_{s}(\omega) - (1 + c) 2C_1\varepsilon(\omega) \} > 0, ~\inf_{\omega} \{ \mu_{u}(\omega) - 2C_1\varepsilon(\omega) \} > 0,
		\]
		then for the following perturbed equation about \eqref{2**1},
		\begin{equation*}
		\begin{cases}
		\dot{z}(t) = Az(t) + \widetilde{g}(z(t)), \\
		z(0) = z_0 \in Z_0,
		\end{cases} \tag{$ \divideontimes $}
		\end{equation*}
		there is a $ C^1 $ immersed submanifold $ \widetilde{M} $ in $ \mathbb{B}_{r}(M) $ such that it is homeomorphic to $ M $ and positively invariant under above equation for time $ t > t_0 $. Also, $ M $ is a $ C^1 $ immersed submanifold. For more properties of $ \widetilde{M} $, see \autoref{thm:A}.

		\item If $ A $ is an MR operator (see the assumption (MR) in \autopageref{MR}), assume $ \sup_{z \in \mathbb{B}_{r}(M)}\{ |D\widetilde{g}(z) - Dg(z)| \} \leq \varepsilon_1 $. If $ \xi,\xi_2, \chi, \eta, r, \varepsilon_1 > 0 $ are small and
		\[
		\inf_{\omega} \mu_{u}(\omega) > 0, ~\inf_{\omega} \{ \mu_{u}(\omega) - \mu_{s}(\omega) \} > 0, \tag{$ \circledast $}
		\]
		then ($ \divideontimes $) also has a $ C^1 $ immersed submanifold $ \widetilde{M} $ in $ \mathbb{B}_{r}(M) $ such that it is homeomorphic to $ M $ and positively invariant under above equation for time $ t > t_0 $; furthermore $ M $ is a $ C^1 $ immersed submanifold, and $ \widetilde{M} \to M $, $ T\widetilde{M} \to TM $ as $ \chi, \eta, \varepsilon_1 \to 0 $.
	\end{enumerate}
\end{thm}
\begin{proof}
	We will apply \autoref{thm:A}.
	First consider (2). Let $ |\widetilde{g} - g|_{ C^1(\mathbb{B}_r(M)) } \leq \eta + \varepsilon_1 \triangleq \varepsilon_0 $. Since $ \mathfrak{A}_{Dg|_{\mathbb{B}_{r}(M)}} \leq \chi $, we have $ \mathfrak{A}_{D\widetilde{g}|_{\mathbb{B}_{r}(M)}} \leq \chi + 2\varepsilon_0 $. In particular, there is a small $ \epsilon > 0 $ ($ \epsilon < r $) such that for $ \widetilde{f}(\omega)(z) = \widetilde{g}(z+\omega) - L(\omega)z - g(\omega) $, it holds $ \lip\widetilde{f}(\omega)|_{\mathbb{B}_{\epsilon}} < 2(\chi + 2\varepsilon_0) \triangleq \delta_{\epsilon} $, where $ \mathbb{B}_{\epsilon} = \{ x \in Z_0: |x| < \epsilon \} $. So by using the radial retraction, i.e.
	\begin{equation*}
	r_\varepsilon(x) =
	\begin{cases}
	x, &~ \text{if}~ |x| \leq \varepsilon, \\
	\varepsilon x / |x|, & ~ \text{if}~ |x| \geq \varepsilon,
	\end{cases}
	\end{equation*}
	we assume there is a map $ \widehat{f} $ such that $ \lip\widehat{f}(\omega) < 2\delta_{\epsilon} $ and $ \widehat{f}(\omega)|_{\mathbb{B}_{\epsilon}} = \widetilde{f}(\omega) $. Applying \autoref{them:spec} (2) to the following equation if $ \delta_{\epsilon} $ is small:
	\begin{equation*}
	\dot{z}(t) = Az(t) + L(\omega)z(t) + \widehat{f}(\omega)z(t),
	\end{equation*}
	we see that the cocycle correspondence $ \widehat{H} \sim (\widehat{F}, \widehat{G}) $ induced by this equation satisfies (A$ _0 $)($ \alpha; \alpha_1, \lambda_{cs};k $) (B$ _0 $)($ \alpha; \alpha_1, \lambda_{u};k $) condition (see \autopageref{defi:ABk}), where $ \alpha, \alpha_1, k $ are constant functions such that $ 0 < \alpha_1 < \alpha < 1/2 $, $ k > 1 $, and $ \lambda_{cs}(\cdot), \lambda_{u}(\cdot) $ are bounded and $ C\xi $-almost continuous and $ C\xi $-almost uniformly continuous around $ K $ (in the immersed topology) for some constant $ C > 0 $ such that $ \sup_{\omega} \lambda_{cs}(\omega) \lambda_{u}(\omega) < 1 $, $ \sup_{\omega} \lambda_{u}(\omega) < 1 $; also note that we can take $ \alpha \to 0 $ as $ \delta_{\epsilon} \to 0 $. Note that when $ \eta $ is sufficiently smaller than $ \epsilon $, it holds
	\begin{equation}\label{FG00}
	\sup_{t\in [0,t_0]}\sup_{\omega \in M}\{|\widehat{F}_{t,\omega}(0,0)|, |\widehat{G}_{t,\omega}(0, 0)|\} \leq C_2\eta,
	\end{equation}
	where $ C_2 > 0 $ only depends on small $ t_0 > 0 $, as $ |\widetilde{g} - g|_{ C^0(\mathbb{B}_r(M)) } \leq \eta $.
	Thus, for the semiflow $ H $ generated by equation ($ \divideontimes $), we have
	\begin{multline} \label{HHAB}
	\widetilde{H}(t,\omega)(\cdot) \triangleq H(t)(\cdot + \omega) - t(\omega) \sim (\widehat{F}_{t,\omega}, \widehat{G}_{t,\omega}): \\
	X_{\omega}(r_{t,1}) \oplus Y_{\omega}(r_{t,2}) \to X_{t(\omega)}(r'_{t,1}) \oplus Y_{t(\omega)} (r'_{t,2}),~ \omega \in M,
	\end{multline}
	where $ r_{t, i}, r'_{t,i} $, $ i = 1,2 $, are taken such that they satisfy
	\[
	k\hat{\lambda}^{t}_{cs} r_{t,1} + r'_{t,2} + C_2\eta \leq r'_{t,1} \leq \epsilon / 2, ~
	k r'_{t,2} + r_{t,1} + C_2\eta \leq r_{t,2} \leq \epsilon / 2,
	\]
	and $ \hat{\lambda}_{cs} = \max \{\sup_{\omega}\lambda_{cs}(\omega), 1\} $; for instance, if $ t \in [0,b] $, then $ r_{t,1} \leq (c\hat{\lambda}^{b}_{cs} )^{-1}  \epsilon / 8 $, $ r'_{t,2} \leq \epsilon / 8 $, and $ \eta $ is assumed to be small such that $ C_2\eta \leq \epsilon / 8 $. As $ \widetilde{g} \in C^1(\mathbb{B}_{r}(M), Z_0) $ and $ \widehat{f}(\omega)|_{\mathbb{B}_{\epsilon}} = \widetilde{f}(\omega) $, we know $ \widehat{F}_{t,\omega}(\cdot,\cdot), \widehat{G}_{t,\omega}(\cdot,\cdot) $ are $ C^1 $ in $ X_{\omega}(r_{t,1}) \times Y_{t(\omega)} (r'_{t,2}) $. Therefore, all the assumptions in \autoref{thm:A} are fulfilled if $ \xi,\xi_2, \chi, \eta, r, \varepsilon_1 $ are small, and then the conclusion (2) follows. (Here note that $ \widetilde{M} $ is constructed by the graph of $ h_0 $, where $ h_0(m) \in X^{u}(\varrho) $, $ m \in M $; in addition, $ T_{m}\widetilde{M} = \graph Df^0_{m}(0,0) $, where $ h_0(m') + m' = m + f^0_{m}(x^c) + x^c $, $ x^c \in X^{c}_{m}(\sigma) $, $ f^0_{m}(x^c) \in X^{u}_{m} $ and $ m' $ belongs to a component of $ M \cap \mathbb{B}_{\epsilon_{m}}(m) $ for small $ \epsilon_{m} > 0 $; $ |Df^0_{m}(x^c)| < \alpha $; $ \varrho \to 0 $ as $ \eta \to 0 $. This induces $ \widetilde{M} \to M $, $ T\widetilde{M} \to TM $ as $ \eta, \varepsilon_1 \to 0 $.)

	Proof of (1). This is essentially the same as (2), where the difference is $ \widetilde{g} $ might be a `large' perturbation. Define $ \widetilde{f} $ as in (1). Since $ \mathfrak{A}_{D\widetilde{g}|_{\mathbb{B}_{r}(M)}} \leq \chi $ (for this case, in fact $ \mathfrak{A}_{D{g}|_{\mathbb{B}_{r}(M)}} \leq \chi $ being useless), there is a small $ \epsilon > 0 $ ($ \epsilon < r $) such that $ \lip\widetilde{f}(\omega)|_{\mathbb{B}_{\epsilon}} \leq 2\chi + \varepsilon(\omega) $ and $ \lip f(\omega)|_{\mathbb{B}_{\epsilon}} \leq \chi $. In the proof of (2), the radial retraction is used to truncate $ \widetilde{f} $, but this is unnecessary, for we can consider the following equation directly:
	\[
	\dot{z}(t) = Az(t) + L(\omega)z(t) + \widetilde{f}(\omega)z(t), \tag{$ \circledcirc $}
	\]
	or \eqref{**0} (if $ A $ is a Hille-Yosida operator), or \eqref{*06} (if $ A $ is a generator of a $ C_0 $ semigroup or a $ C_0 $ bi-semigroup), where $ f(\omega)(\cdot) $ thereof is replaced by $ \widetilde{f}(\omega)(\cdot) $. Note that $ \widetilde{f}(\omega)(\cdot) = f(\omega)(\cdot) + \widetilde{g}(\cdot + \omega) - g(\cdot + \omega) $ and $ f(\omega)(0) = 0 $. Let $ z(t) = (x(t), y(t)) \in X_{t\omega}(\epsilon/2) \oplus Y_{t\omega}(\epsilon/2) $ satisfy \eqref{**0} or \eqref{*06} with $ x(t_1) = 0 $, $ y(t_2) = 0 $. Then $ |x(t)| \leq C_2 \eta $ and $ |y(t)| \leq C_2 \eta $ for $ t \in [t_1, t_2] $ where $ t_2 - t_1 = t_0 $ is small and $ C_2 $ depends on $ t_0 $ but not $ \eta $. Indeed, from \eqref{**0} or \eqref{*06}, we see
	\[
	\max\{|x|_{[t_1, t_2]}, |y|_{[t_1, t_2]}\} \leq\max\{ 1, e^{\mu_st_0}, e^{-\mu_ut_0} \}  C_1 \delta_1(t_0) \{(\chi + \varepsilon(m))|z|_{[t_1, t_2]} + \eta \},
	\]
	where $ |x|_{t_1, t_2} = \sup_{t\in [t_1, t_2]}|x(t)| $ (similar for $ |y|_{[t_1, t_2]}, |z|_{[t_1, t_2]} $) and $ \delta_1(t_0) \to 0 $ as $ t_0 \to 0 $; here if $ A $ is a generator of a $ C_0 $ semigroup or a $ C_0 $ bi-semigroup, then $ \delta_1(t_0) = t_0 $; for the case $ Z_0 \neq Z $, we need $ \sup_{\omega}|L(\omega)| < \infty $ (in \textbf{($ \bullet $III)}) to give a uniform estimate on $ \delta_1(t_0) $ (see \autoref{lem:coycle0} \eqref{integral}). In the proof of \autoref{thm:biAB} or \autoref{them:spec}, we certainly prove the fact that for any given $ t_1 < t_2 $ and $ \omega \in M $, if $ (x(t),y(t)), (x'(t),y'(t)) \in X_{t\omega}(\epsilon/2) \oplus Y_{t\omega}(\epsilon/2) $, $ t_1 \leq t \leq t_2 $, satisfy \eqref{**0} or \eqref{*06}, and $ |\hat{x}(t_1)| \leq \alpha(\omega) |\hat{y}(t_1)| $, then $ |\hat{x}(t_2)| \leq \alpha(\omega) |\hat{y}(t_2)| $ (or $ |\hat{x}(t_2)| \leq k_{\alpha}(\omega,\epsilon_1)\alpha(\omega) |\hat{y}(t_2)| $) and $ |\hat{y}(t_1)| \leq \lambda^{t_2 - t_1}_u(\omega) |\hat{y}(t_2)| $, where $ \hat{x}(t) = x(t) - x'(t) $ and $ \hat{y}(t) = y(t) - y'(t) $.

	Combining with two facts above, we obtain (i) the correspondence $ H $ induced by ($ \divideontimes $) satisfies \eqref{HHAB} with $ 0 < \sup\{t_{t, i}, t'_{t,i}: i = 1,2, t\in [0,b] \} < \epsilon / 2 $, where $ \{ \widetilde{H}(t,\omega) \} $ is the cocycle correspondence induced by equation ($ \circledcirc $); (ii) \eqref{FG00} holds; (iii) moreover, \eqref{HHAB} satisfies (A)($ \alpha_1, \lambda^t_{cs}(\omega) $) (B)($ \alpha_1, \lambda^t_{u}(\omega) $) condition and if $ t \geq \epsilon_1 > 0 $, (A)($ \alpha; \alpha_1, \lambda^t_{cs}(\omega) $) (B)($ \alpha; \alpha_1, \lambda^t_{u}(\omega) $) condition, where $ 0 < \alpha_1 < \alpha < 1/c $, and $ \lambda_{cs}(\omega) = e^{\mu_{s}(\omega)+2C_1\varepsilon(\omega)} $, $ \lambda_{u}(\omega) = e^{-\mu_{u}(\omega)+2C_1\varepsilon(\omega)} $, $ \sup_{\omega} \lambda_{cs}(\omega) \lambda_{u}(\omega) < 1 $, $ \sup_{\omega} \lambda_{u}(\omega) < 1 $. So all the assumptions in \autoref{thm:A} are fulfilled if $ \xi,\xi_2, \chi, \eta, r $ are small, then the conclusion (1) follows. The proof is complete.
\end{proof}

Instead of assuming $ \widehat{\Pi}^s_\omega = 0 $, we consider the following.
\begin{enumerate}[($ \star $)]
	\item Let $ X_\omega = \widehat{X}^s_{\omega} \oplus \widehat{X}^c_{\omega} $ associated with projections $ \widehat{\Pi}^s_\omega, \widehat{\Pi}^c_\omega $ such hat $ R(\widehat{\Pi}^s_\omega) = \widehat{X}^s_{\omega} $, $ R(\widehat{\Pi}^c_\omega) = \widehat{X}^c_{\omega} $. In addition, suppose for $ T_1(t,\omega) $ in \autoref{def:ud+} \eqref{udd} (in \autopageref{def:ud+}), it further has
	\[
	T_1(t,\omega) = (T_{1,s}(t,\omega), T_{1,1}(t,\omega)) : \widehat{X}^s_{\omega} \oplus \widehat{X}^c_{\omega} \to \widehat{X}^s_{t\omega} \oplus \widehat{X}^c_{t\omega},
	\]
	with $ |T_1(t,r\omega)| \leq e^{\mu_{ss}(\omega)t} $ for all $ t,r \geq 0 $ and $ \omega \in \Sigma $, where $ T_{1,s}(t,\omega): \widehat{X}^s_{\omega} \to \widehat{X}^s_{t\omega} $ and $ T_{1,1}(t,\omega): \widehat{X}^s_{\omega} \oplus \widehat{X}^c_{\omega} \to \widehat{X}^c_{t\omega} $. Also, $ \mu_{ss}(\cdot) $ is bounded, $ \xi $-almost continuous and $ \xi $-almost uniformly continuous around $ K $ (in the immersed topology) and $ \sup_{\omega}\mu_{ss}(\omega) < 0 $.
\end{enumerate}

\begin{thm}\label{thm:app2} 
	Under above \textbf{Settings A} and ($ \star $), if $ \xi, \xi_2, \chi, r $ are small and $ (\circledast) $ holds, then there is a local center-stable invariant manifold $ W^{cs}_{loc}(M) \subset \mathbb{B}_{r}(M) $ of $ M $. $ W^{cs}_{loc}(M) $ is a $ C^1 $ immersed submanifold of $ Z_0 $ and is positively invariant under \eqref{2**1} for time $ t > t_0 $. Moreover, if a mild solution $ u(t) $ ($ t \geq 0 $) of \eqref{2**1} always `stays' in $ \mathbb{B}_{r}(M) $, then it must belong to $ W^{cs}_{loc}(M) $ (i.e. $ u(t) \in W^{cs}_{loc}(M) $). For more properties of $ W^{cs}_{loc}(M) $, see \autoref{thm:A} and \autoref{thm:B}. The above results are persistent under small $ C^1 $ perturbation of \eqref{2**1}.
\end{thm}
A mild solution $ u(t) $ ($ t \geq 0 $) of \eqref{2**1} always `stays' in $ \mathbb{B}_{r}(M) $ meaning $ \{ u(t) \}_{t \geq 0} $ is a $ (\sigma, \varrho, \varepsilon) $-forward orbit of the correspondence $ H $ induced by \eqref{2**1}; see \autopageref{defi:orbit}.

\begin{proof}
	Since $ \widehat{\Pi}^s_\omega \neq 0 $, we need to verify the assumption (A3) (c) in \autoref{thm:A} (s-contraction). This is simple as shown in the following. For any given $ t_1 < t_2 $ and $ \omega \in M $, let $ (x(t),y(t)), (x'(t),y'(t)) \in X_{t\omega}(\epsilon/2) \oplus Y_{t\omega}(\epsilon/2) $, $ t_1 \leq t \leq t_2 $, satisfy \eqref{**0} or \eqref{*06} with $ x(t_1) = (x_s, 0), x'(t_1) = (x'_s, 0) $, $ y(t_2) = x_{u}, y'(t_2) = x'_{u} $, and $ \hat{x}(t) = x(t) - x'(t) $ and $ \hat{y}(t) = y(t) - y'(t) $. Then
	\[
	\max\{ |\hat{x}|_{[t_1, t_2]}, |\hat{y}|_{[t_1, t_2]} \} \leq \hat{\sigma}_1(t_0) \max\{|x_s - x'_s|, |x_{u} - x'_{u}|\},
	\]
	where $ t_0 = t_2 - t_1 $ and $ \hat{\sigma}_1(t_0) \to 1_+ $ as $ t_0 \to 0 $; see e.g. the proof of \autoref{slem:small} \eqref{small}. So for $ \hat{x}_s(t) = \widehat{\Pi}^s_{t\omega}\hat{x}(t) $, when $ |x_{u} - x'_{u}| \leq B |x_{s} - x'_{s}| $, we see
	\begin{align*}
	|\hat{x}_s(t_2)| & \leq e^{\mu_{ss}(\omega)t_0}|x_s - x'_s| + \widehat{C}\delta_1(t_0) \max\{ |\hat{x}|_{[t_1, t_2]}, |\hat{y}|_{[t_1, t_2]} \} \\
	& \leq \{e^{\mu_{ss}(\omega)t_0} + \widehat{C}\delta_1(t_0) \hat{\sigma}_1(t_0) (1+B)\} |x_s - x'_s|,
	\end{align*}
	where $ \widehat{C} > 0 $ and $ \delta_1(t_0) \to 0 $ as $ t_0 \to 0 $. Now the conclusion follows from above.
\end{proof}

\vspace{.5em}
\noindent\textbf{Settings B}. \textbf{(BI)} Under \textbf{(AI)}, let $ M $ be an invariant set of \eqref{2**1} and $ K = M $.

\noindent\textbf{(BII)} Let $ t: M \to M $ be a $ C_0 $ flow in the immersed topology and $ \exists t_ 0 >0 $ such that $ t : M \to M $, $ -t_0 \leq t \leq t_0 $, considered as maps of $ M \to Z_0 $, are $ \xi $-almost equicontinuous (in the immersed topology); see \autopageref{almost}.

\noindent\textbf{(BIII)} Suppose the linear equation \eqref{2**2} satisfies the following \textbf{uniform trichotomy} condition. \label{def:ut}
\begin{enumerate}[(a)]
	\item Assume $ Z_0 = \widehat{X}^{s}_{\omega} \oplus \widehat{X}^{c}_{\omega} \oplus \widehat{X}^{u}_{\omega} $, $ \omega \in M $ associated with projections $ P^s_{\omega} $, $ P^c_{\omega} $, $ P^u_{\omega} = \id - P^{s}_{\omega} - P^{c}_{\omega} $ such that $ R(P^{\kappa}_{\omega}) = \widehat{X}^{\kappa}_{\omega} $, $ \kappa = s, c, u $. $ (\omega, z) \mapsto P^{\kappa}_{\omega}z $ is continuous.

	\item There are three $ C_0 $ linear cocycles $ T_s $, $ T_{c} $, $ T_{u} $ such that $ T_\kappa(t,\omega): \widehat{X}^{\kappa}_{\omega} \to \widehat{X}^{\kappa}_{t\omega} $ for all $ (t, \omega) \in \mathbb{R} \times M $ if $ \kappa = c $, $ (t, \omega) \in \mathbb{R}_+ \times M $ if $ \kappa = s $, and $ (t, \omega) \in -\mathbb{R}_+ \times M $ if $ \kappa = u $.  Write $ T_{\kappa_1\kappa_2} = T_{\kappa_1} \oplus T_{\kappa_2} $, $ \kappa_1 \neq \kappa_2 $.
	Let $ z_1(t) = (T_{s}(t-t_1, t_1\omega)x_s, T_{c}(t-t_1, t_1\omega)x_c, T_{u}(t-t_2, t_2\omega)x_u) \in \widehat{X}^{s}_{t\omega} \oplus \widehat{X}^{c}_{t\omega} \oplus \widehat{X}^{u}_{t\omega} $ ($ t_1 \leq t \leq t_2 $), then it is the \emph{unique} mild solution of \eqref{2**2} with
	$ P^{cs}_{\omega} z_{1}(t_1) = x_s + x_c $ and $ P^{u}_{\omega} z_{1}(t_2) = x_u $. Also let $ z_2(t) = (T_{s}(t-t_1, t_1\omega)x_s, T_{c}(t-t_2, t_2\omega)x_c, T_{u}(t-t_2, t_2\omega)x_u) \in \widehat{X}^{s}_{t\omega} \oplus \widehat{X}^{c}_{t\omega} \oplus \widehat{X}^{u}_{t\omega} $ ($ t_1 \leq t \leq t_2 $), then it is the \emph{unique} mild solution of \eqref{2**2} with
	$ P^{s}_{\omega} z_{2}(t_1) = x_s $ and $ P^{cu}_{\omega} z_{2}(t_2) = x_c + x_u $.

	\item There is a constant $ C_1 > 0 $ such that $ \sup_{\omega}|P^{\kappa}_{\omega}| \leq C_1 $, $ \kappa = s, c, u $.

	\item There are functions $ \mu_s, \mu_u, \mu_{cs}, \mu_{cu} $ of $ M \to \mathbb{R} $, such that
	\begin{equation*}
	|T_s(t,r\omega)| \leq e^{\mu_s(\omega) t},~
	|T_{cs}(t,r\omega)| \leq e^{\mu_{cs}(\omega) t},~|T_{cu}(-t,r\omega)| \leq e^{-\mu_{cu}(\omega) t},~
	|T_u(-t,r\omega)| \leq e^{-\mu_u(\omega) t},
	\end{equation*}
	for all $ t, r \geq 0 $ and $ \omega \in M $.
\end{enumerate}

Assume $ \mu_{\kappa} $, $ \kappa = s, cs, cu, u $, are bounded and $ \xi $-almost uniformly continuous (in the immersed topology). Moreover,
\[
\sup_{\omega} \mu_{s}(\omega) < 0, ~\inf_{\omega} \mu_{u}(\omega) > 0, ~
\inf_{\omega} \{ \mu_{u}(\omega) - \mu_{cs}(\omega) \} > 0, ~
\inf_{\omega} \{ \mu_{cu}(\omega) - \mu_{s}(\omega) \} > 0.
\]

\noindent\textbf{(BIV)} $ \mathfrak{A}_{Dg|_{\mathbb{B}_{r}(M)}} \leq \chi $ when $ r $ is small. Assume there is a small $ \xi_2 > 0 $ such that $ \sup_{\omega \in \Sigma}|P^{\kappa}_{\omega} - \Pi^\kappa_\omega| \leq \xi_2 $, $ \kappa = s, c, u $.

The following theorem is a consequence of \autoref{cor:tri} by using the same argument in the proof of \autoref{thm:app1} and \autoref{thm:app2}.

\begin{thm}\label{thm:triH}
	Under above \textbf{Settings B}, if $ \xi, \xi_2, \chi, r $ are small, then there are local center-stable and local center-unstable manifolds $ W^{cs}_{loc}(M), W^{cu}_{loc}(M) \subset \mathbb{B}_{r}(M) $ of $ M $, which are $ C^1 $ immersed submanifolds of $ Z_0 $. There is a positive constant $ r' < r $ such that for any $ z_0 \in W^{cs}_{loc}(M) \cap \mathbb{B}_{r'}(M) $, there is a mild solution $ \{ u(t) \}_{t \geq 0} \subset W^{cs}_{loc}(M) $ (resp. $ \{ u(t) \}_{t \leq 0} \subset W^{cu}_{loc}(M) $) of \eqref{2**1} with $ u(0) = z_0 $.
	Moreover, if a mild solution $ \{u(t)\}_{t \geq 0} $ (resp. $ \{u(t)\}_{t \leq 0} $) of \eqref{2**1} always `stays' in $ \mathbb{B}_{r}(M) $, then it must belong to $ W^{cs}_{loc}(M) $ (resp. $ W^{cu}_{loc}(M) $), and there is certain $ \omega \in M $ such that $ |u(t) - t\omega| \to 0 $ (resp. $ |u(-t) - (-t)\omega| \to 0 $) exponentially as $ t \to \infty $. For more properties of $ W^{cs}_{loc}(M), W^{cu}_{loc}(M) $, see \autoref{cor:tri}.

	The above results are persistent under small $ C^1 $ perturbation of \eqref{2**1}; i.e. if $ |\widetilde{g} - g|_{ C^1(\mathbb{B}_r(M)) } $ is small when $ r $ is small, then corresponding results also hold for the equation ($ \divideontimes $). Moreover, there is a local center manifold $ \widetilde{M} $ which is $ C^1 $ immersed in $ \mathbb{B}_{r}(M) $, homeomorphic (in fact $ C^1 $ diffeomorphic) to $ M $ and invariant under equation ($ \divideontimes $); also $ W^{cs}_{loc}(\widetilde{M}) \cap W^{cu}_{loc}(\widetilde{M}) = \widetilde{M} $, and $ \widetilde{M} \to M $, $ T\widetilde{M} \to TM $ as $ |\widetilde{g} - g|_{ C^1(\mathbb{B}_r(M)) } \to 0 $ and $ \chi \to 0 $.
\end{thm}
Here, a mild solution $ \{u(t)\}_{t \geq 0} $ (resp. $ \{u(t)\}_{t \leq 0} $) of \eqref{2**1} always `stays' in $ \mathbb{B}_{r}(M) $ meaning $ \{ u(t) \}_{t \geq 0} $ (resp. $ \{u(t)\}_{t \leq 0} $) is a $ (\sigma, \varrho, \varepsilon) $-forward orbit (resp. $ (\sigma, \varrho, \varepsilon) $-backward orbit) of the correspondence $ H $ induced by \eqref{2**1}; see \autopageref{defi:orbit}.

Note that by \autoref{cor:tri}, if $ z_0 \in W^{cs}_{loc}(M) $, then \eqref{2**1} has a (mild) solution $ u \in C([0,\infty), Z_0) $ (in $ W^{cs}_{loc}(M) $) such that $ u(0) = z_0 $, and this gives a semiflow in $ W^{cs}_{loc}(M) $. In other words, equation \eqref{2**1} is well-posed in $ W^{cs}_{loc}(M) $ which shows that this equation describes well the physical situation in $ W^{cs}_{loc}(M) $ although it might be ill-posed when $ A $ is a generator of a $ C_0 $ bi-semigroup. A similar result holds for $ W^{cu}_{loc}(M) $.

\begin{rmk}[periodic orbit case]
	When $ M $ consists of an isolated periodic orbit of \eqref{2**1} with period $ L > 0 $, intuitively, this shows if $ M $ is a normally hyperbolic (with respect to \eqref{2**1}), then this orbit is persistent under small $ C^1 $ perturbation of \eqref{2**1}; i.e. if $ |\widetilde{g} - g|_{ C^1(\mathbb{B}_r(M)) } $ is small when $ r $ is small, then equation ($ \divideontimes $) also exists a periodic orbit in $ \mathbb{B}_r(M) $. In the well-posed case, $ M $ is normally hyperbolic if the time-$ L $ solution map $ \mathcal{P} $ of \eqref{2**2} is compact (or quasi-compact meaning $ \sigma_{ess}(\mathcal{P}) < 1 $) and $ \sigma(\mathcal{P}) \cap \mathbb{S}^1 $ consists of only a simple point spectrum $ 1 $. For some characterizations about the hyperbolicity of $ M $ in the ill-posed case, see e.g. \cite{LP08, SS99, HVL08}.
\end{rmk}

\begin{rmk}['large' perturbation]\label{rmk:largeP}
	In some cases, we may not require $ |\widetilde{g} - g|_{ C^1(\mathbb{B}_r(M)) } $ is small as we do in \autoref{thm:app1}. For example, (i) when $ A $ is a generator of a $ C_0 $ semigroup or a $ C_0 $ bi-semigroup (and so $ Z_0 = Z_{-1} $), assume $ \sup_{z \in \mathbb{B}_{r}(M)}\{ |\widetilde{g}(z) - g(z)| \} \leq \eta $ and $ \mathfrak{A}_{D\widetilde{g}|_{\mathbb{B}_{r}(M)}} \leq \chi $,
	\begin{gather*}
	\sup\{|P^{\nu}_{t\omega}(D\widetilde{g}(t\omega) - Dg(t\omega))P^{\kappa}_{t\omega}(\cdot)|: t \geq 0, \kappa = s, c, u\} \leq \varepsilon^{\nu}_{m}(\omega), ~\omega \in M, ~\nu = s, c, u,
	\end{gather*}
	and $ \omega \mapsto \varepsilon^{\nu}_{m}(\omega) $ is bounded and $ \xi $-almost uniformly continuous (in the immersed topology), $ \nu = s, c, u $. Set
	$ \varepsilon_{m}(\omega) = \max \{ \varepsilon^{s}_{m}(\omega), \varepsilon^{c}_{m}(\omega), \varepsilon^{u}_{m}(\omega) \} $.
	There is a constant $ c > \sqrt{2} $ such that
	\begin{gather*}
	\inf_{\omega} \{ \mu_{cu}(\omega) - \mu_{s}(\omega) - (1+c) \varepsilon_{m}(\omega) \} > 0, ~
	\inf_{\omega} \{ \mu_{u}(\omega) - \mu_{cs}(\omega) - (1+c) \varepsilon_{m}(\omega) \} > 0, \\
	\sup_{\omega} \{ \mu_{s}(\omega) + \varepsilon^{s}_{m}(\omega) \} < 0, ~
	\inf_{\omega} \{ \mu_{u}(\omega) - \varepsilon^{u}_{m}(\omega) \} > 0.
	\end{gather*}
	If $ \xi,\xi_2, \chi, \eta, r > 0 $ are small, then the results in \autoref{thm:triH} also hold; see also \autoref{rmk:detail}. Also, $ \widetilde{M} \to M $, $ T\widetilde{M} \to TM $ as $ \eta, \chi, \sup_{\omega}\{ \varepsilon_{m}(\omega) \} \to 0 $.
	(ii) If $ A $ is a Hille-Yosida operator, let
	\[
	\sup\{|(D\widetilde{g}(t\omega) - Dg(t\omega))P^{\kappa}_{t\omega}(\cdot)|: t \geq 0, \kappa = s, c, u\} \leq \varepsilon'(\omega), ~\omega \in M,
	\]
	and use $ C_1\varepsilon'(\omega) $ instead of $ \varepsilon_{m}(\omega) $ and $ \varepsilon^{\nu}_{m}(\omega) $; the same result also holds for this case.
	Furthermore, if one only focuses on the existence result, $ \widetilde{g} $ can be non-smooth but satisfies
	\[
	\max \{ \sup_{z \in \mathbb{B}_{r}(M)}|\widetilde{g}(z) - g(z)|, \lip (\widetilde{g}(\cdot) - g(\cdot) )|_{\mathbb{B}_{r}(M)} \} \leq \eta.
	\]
\end{rmk}

%% file: app.tex
\section{Appendix. a little background from operator semigroup theory} \label{operator}

For readers' convenience, in this appendix, we collect some basic definitions and notations taken from operator semigroup theory. For more details, see \cite{EN00,ABHN11,vdMee08}.

Let $ X $ be a Banach space. A linear operator $ A : D(A) \subset X \to X $ with domain $ D(A) $ is densely-defined if $ \overline{D(A)} = X $. $ A $ is closed if $ \graph A $ is closed in $ X \times X $. Set $ R(\lambda, A) = (\lambda - A)^{-1} $ the resolvent of $ A $ at $ \lambda \in \rho(A) $. Let $ Y \hookrightarrow X $, i.e. $ Y \subset X $ is a Banach space such that it continuously embeds in $ X $. For example $ D(A) \hookrightarrow X $ if $ A $ is closed where $ D(A) $ is equipped with graph norm $ \|\cdot\|_A $, i.e. $ \|x\|_A = \|x\| + \|Ax\| $. The part of $ A $ in $ Y $ denoted by $ A_Y $, is defined by
\[
A_Y x = Ax,~ x \in D(A_Y) = \{ x \in Y \cap D(A): Ax \in Y \}.
\]
\begin{enumerate}[(a)]
	\item $ A $ is a \textbf{generator of a $ C_0 $ semigroup $ T $} if $ T: \mathbb{R}_+ \to L(X,X) $ is strongly continuous and there is a constant $ \omega \in \mathbb{R} $ such that $ (\omega, \infty) \subset \rho(A) $ and
	\[
	R(\lambda, A) = \int_{0}^{\infty} e^{-\lambda t}T(t) ~\mathrm{ d }t,~\lambda > \omega.
	\]
	Note that is this case $ T(0) = \id $, $ T(t+s) = T(t)T(s) $, $ t,s \geq 0 $ and $ \overline{D(A)} = X $. And we also say $ A $ generates a $ C_0 $ semigroup $ T $. For a more classical definition and the basic properties, see \cite{EN00}.
	\item $ A $ is a \textbf{generator of a once (exponentially bounded) integrated semigroup $ S $} if $ S: \mathbb{R}_+ \to L(X,X) $ is strongly continuous and exponentially bounded, and there is a constant $ \omega \in \mathbb{R} $ such that $ (\omega, \infty) \subset \rho(A) $ and
	\[
	R(\lambda, A) = \lambda \int_{0}^{\infty} e^{-\lambda t}S(t) ~\mathrm{ d }t,~\lambda > \omega.
	\]
	In this case, we also say $ A $ generates the (once) integrated semigroup $ S $. See \cite[Chapter 3]{ABHN11} for basic properties and some characterizations.
	\item $ A $ is a \textbf{Hille-Yosida (HY) operator} if $ A $ generates the (once) integrated semigroup $ S $ with $ S $ being locally Lipschitz. An equivalent definition is the following (see \cite{DPS87}). There are $ \omega > 0 $ and $ M_1 \geq 1 $ such that $ (\omega, \infty) \subset \rho(A) $ and
	\[
	\|R(\lambda, A)^n\| \leq \frac{M_1}{(\lambda - \omega)^n}, ~\forall \lambda > \omega,~ n \in \mathbb{N}.
	\]
	See \cite[Section 3.5]{ABHN11} for a proof of this equivalence.
	\item See \cite[Section 3.7]{ABHN11} for the different equivalent definitions of \textbf{sectorial operators}, and \cite[Section 3.7]{ABHN11} for the definition of \textbf{fractional powers}: $ A^{-\alpha} $.
	\item $ A $ is called a \textbf{generator of a $ C_0 $ bi-semigroup $ E $} if $ A = A_1 \oplus (-A_2) $ in the decomposition $ X = X_1 \oplus X_2 $ with $ X_i $ closed, where $ A_i : D(A_i) \subset X_i \to X_i $ is the generator of the $ C_0 $ semigroup $ T_i $, $ i = 1,2 $. Let $ T_1(t) = 0 $, $ T_2(t) = 0 $ if $ t < 0 $. For this case $ E(t) = T_1(t) \oplus T_2(-t) $ is called a $ C_0 $ \textbf{bi-semigroup}. Note that for this case $ \overline{D(A)} = X $. See \cite{vdMee08} and references therein for some characterizations of $ A $ such that $ T_i $, $ i = 1,2 $ are all exponentially stable, i.e. $ \|T_i(t)\| \leq C_0 e^{-\mu t}  $, $ \forall t \geq 0 $ for some constant $ \mu > 0 $, $ C_0 \geq 1 $, where the author called $ A $ an \textbf{exponentially dichotomous operator}. See \cite{LP08} (or \autoref{examples}) for some concrete examples about $ A $.
	\item See \cite[Section 1.4.1]{vdMee08} for a definition of a bi-sectorial (and  densely-defined) operator which is also a generator of a $ C_0 $ bi-semigroup.
\end{enumerate}

\section{Appendix. a fixed point equation and a smooth result}

Follow the notations in \autoref{HYccc}. Particularly, let \textnormal{(MR) (D1)} hold.
Let $ \mathcal{L} : [0, a] \to L(Z, Z) $ be strongly continuous. Set $ |\mathcal{L}| = \sup_{t \in [0,a]} |\mathcal{L}(t)|~ (< \infty) $.
Define $ \mathfrak{B} $ as follows,
\[
(\mathfrak{B} u)(t) = (S \Diamond( \mathcal{L}(\cdot)u(\cdot) ) )(t), ~u \in C([0,a], Z).
\]
Note that $ (\mathfrak{B} u)(t) \in \overline{D(A)} \subset Z $.
If $ \delta(a) |\mathcal{L}| < 1 $, then $ \mathfrak{B}: C([0,a], Z) \to C([0,a], Z)  $ is a contraction mapping with $ \lip \mathfrak{B} \leq \delta(a) |\mathcal{L}| $.
Consider the following fixed point equation which is frequently used in \autoref{HYccc},
\begin{equation}\label{fpe0}
u = v + \mathfrak{B} u.
\end{equation}

\begin{lem}\label{lem:dd}
	Suppose $ \mathcal{L}(\cdot) $ is strongly $ C^1 $, i.e. for every $ x \in Z $, $ t \mapsto \mathcal{L}(t)x: [0,a] \to Z $ is $ C^1 $.
	\begin{enumerate}[(1)]
		\item If $ v \in C^1([0,a], Z) $ and $ \mathcal{L}(0)v(0) \in \overline{D(A)} $, then $ \mathfrak{B}^k v \in C^1 $, $ k = 1,2,3,\cdots $.
		\item Let $ v $ satisfy the condition in \textnormal{(1)}, then the unique point $ u $ of \eqref{fpe0} also belongs to $  C^1([0,a], Z) $.
	\end{enumerate}
\end{lem}
\begin{proof}
	In the following, we will frequently use the fact that on any bounded subsets of $ L(Z, Z) $, the strong operator topology coincides with the topology of uniform convergence on any relatively compact subsets of $ Z $ (see e.g. \cite[Proposition A.3]{EN00}).
	Since $ v \in C^1([0,a], Z) $ and $ \mathcal{L}(0)v(0) \in \overline{D(A)} $, we see that $ t \mapsto \mathcal{L}(t)v(t) $ is $ C^1 $,
	\[
	(\mathfrak{B} v)(t) = S(t)\mathcal{L}(0)v(0) + \int_{0}^{t} S(s) (\mathcal{L}(t-s) v(t-s))' ~\mathrm{d} s
	\]
	is $ C^1 $, and $ \frac{\mathrm{d}}{\mathrm{d} t} (\mathfrak{B} v)(t) = T(t)\mathcal{L}(0)v(0) + (S \Diamond (\mathcal{L}v)') (t) $. Since $ (\mathfrak{B} v)(0) = 0 $, we get $ \mathfrak{B}  (\mathfrak{B} v) \in C^1 $. By induction, complete the proof of (1).
	We will use the notation $ \mathcal{L}(t)'x \triangleq \frac{\mathrm{d}}{\mathrm{d} t} (\mathcal{L}(t)x) $.
	For (2),
	note that $ u = \sum\limits_{k=0}^{\infty} \mathfrak{B}^k v $; let $ \hat{v} = v' + T(\cdot)\mathcal{L}(0)v(0) + S \Diamond (\mathcal{L}'(\cdot)u(\cdot)) $. Then we get
	\[
	\sum\limits_{k=0}^{\infty} \frac{\mathrm{d}}{\mathrm{d} t} \mathfrak{B}^k v = \sum\limits_{k=0}^{\infty} \mathfrak{B}^k \hat{v},
	\]
	where we use $ \mathcal{L}'(t)u(t) = \sum_{k=0}^{\infty} \mathcal{L}'(t) (\mathfrak{B}^k v)(t) $. Due to all the convergences are uniform we conclude (2).
\end{proof}

\begin{lem}\label{lem:diff1}
	Let \autoref{lem:coycle0} hold with additional assumption that for every $ \omega \in M $, $ t \mapsto L(t\omega) $ is strongly $ C^1 $. Then $ \frac{\mathrm{d}}{\mathrm{d} t}|_{t=0} T_0(t, \omega)x $ exists if and only if $ x \in D(A_0(\omega)) $.
\end{lem}

\begin{proof}
	The `only if' part is easy. Consider the `if' part.
	Let $ x \in D(A_0(\omega)) $, i.e. $ x \in D(A) $ and $ Ax + L(\omega)x \in \overline{D(A)} $. Set $ w(t) = T_0(t, \omega)x - x $. Then $ w(0) = 0 $ and it satisfies
	\[
	w(t) = S(t)(Ax + L(\omega)x) + (S \Diamond ( L(\cdot\omega)x - L(\omega)x ))(t) + (S \Diamond (L(\cdot\omega)w(\cdot)))(t),
	\]
	see \eqref{equ00} in the proof of \autoref{slem:diff}. Now by \autoref{lem:dd} we have $ w(\cdot) $ is $ C^1 $ in $ [0, a] $ for small $ a > 0 $ (and hence for all $ a > 0 $).
\end{proof}

\section{Appendix. some concrete examples}\label{examples}

We give some classical concrete examples of equation \eqref{2**1}. The \emph{ill-posed} differential equations are \autoref{ex:cylinder} (a) and (b) (ii), \autoref{ex:bou} (a), \autoref{ex:spatial}, and \autoref{ex:nondiss} (c).
\begin{exa}\label{ex:classical}
	Assume $ f $ is smooth in all cases. Consider the following nonlinear dissipative parabolic PDEs taken from \cite{Tem97}, where the readers can find more details thereof.
	\begin{enumerate}[(a)]
		\item \label{rde} Consider the following reaction-diffusion equation:
		\[
		\partial_t u - \Delta_x u + \alpha u = f(x,u), ~x \in \Omega = [0,1],
		\]
		endowed by e.g. the Dirichlet boundary condition, where $ \alpha > 0 $.
		(i) Take $ X = L^2(\Omega) $, $ D(A) = H^2(\Omega) \cap H^1_0(\Omega) $, $ Au = \Delta_x u - \alpha u $, $ (g(u))(x) = f(x, u(x)) $. (ii) Take $ X = C(\Omega) $, $ D(A) = C^2_0(\Omega) $, $ A $ and $ g $ the same as in (i). For the two cases, it is well known that $ A $ is a sectorial operator but in (i) $ \overline{D(A)} = X $ and in (ii) $ \overline{D(A)} \neq X $. The readers might consider the third case: $ X = C^{0,\gamma}(\Omega) $, $ D(A) = C^{2,\gamma}_0(\Omega) $, $ A $ and $ g $ the same as in (i). For this case $ A $ is not a sectorial operator even not a Hille-Yosida operator but is an MR operator (see \autopageref{MR}) with $ A_{C^{\gamma}_0(\Omega)} $ a sectorial operator. (The general $ \Omega $ can be taken as a bounded open set of $ \mathbb{R}^n $ with smooth boundary or smooth compact Riemannian manifold without boundary.)
		\item Consider the following Cahn-Hilliard equation endowed by the periodic boundary condition, e.g. $ \Omega = \mathbb{S}^1 $:
		\[
		\partial_t u = \Delta_x ( -\Delta_x u + f(x,u)) , ~x \in \Omega.
		\]
		$ f $ is taken as e.g. $ f(x,u) = -\alpha u + \beta u^3 $ with $ \alpha, \beta > 0 $. Take $ X = L^2(\Omega) $, $ D(A) = H^4(\Omega) $, $ A = -\Delta^2_x $, (then $ D(A^{-1/2}) = H^{-2}(\Omega) $), $ (g(u))(x) = \Delta_x f(x,u(x)): L^2(\Omega) \to H^{-2}(\Omega) $. $ A $ is a densely-defined sectorial operator.
		\item Consider the following Kuramoto-Sivashinksy equation in one dimension:
		\[
		\partial_t u + \partial^4_x u + 2a \partial^2_x u = \partial_x(u^2), ~x \in \Omega = \mathbb{S}^1, a \in \mathbb{R}.
		\]
		Take $ H = L^2(\Omega) $, $ D(A) = H^4(\Omega) $, $ A = -\partial^4_x - 2a \partial^2_x $, (then $ D(A^{-1/4}) = H^{-1}(\Omega) $), $ (g(u))(x) = \partial_x(u^2)(x) : L^2(\Omega) \to H^{-1}(\Omega) $. $ A $ is a densely-defined sectorial operator.
	\end{enumerate}
\end{exa}

\begin{exa}\label{ex:cylinder}
	\begin{enumerate}[(a)]
		\item \label{cya11} Consider the following elliptic problem on the cylinder $ \mathbb{R} \times \Omega $, for brevity assuming $ \Omega = [0,1] $:
		\[
		\partial^2_t u + \Delta_x u + f(x,u, \nabla_x u) = 0,~(t,x) \in \mathbb{R} \times \Omega,
		\]
		endowed by e.g. the Dirichlet boundary condition, where $ f $ is smooth. Take $ A_0 = -\Delta_x $, $ X = H^1_0(\Omega) \times L^2(\Omega) $, $ D(A) = (H^2(\Omega) \cap H^1_0(\Omega)) \times H^1_0(\Omega) $,
		\[
		A = \left(  \begin{matrix}
		0 & 1 \\
		A_0 & 0
		\end{matrix} \right),
		~g \left( \begin{matrix}
		u \\
		v
		\end{matrix} \right)(x) = \left(  \begin{matrix}
		0  \\
		- f(x,u(x), \nabla_x u(x))
		\end{matrix} \right): X \to X.
		\]
		$ A $ is a bi-sectorial operator (i.e. $ A $ generates an analytic bi-semigroup) as $ A_0 $ is self-adjoint and $ \sigma(A_0) \subset \mathbb{R}_+ $. See also \cite{CMS93,ElB12} for more details, where the dynamical method was applied to study the elliptic problem in the cylindrical domains.

		One can consider another case: $ X = C^1_0(\Omega) \times C(\Omega) $ and $ D(A) = C^2_0(\Omega) \times C^1_0(\Omega) $. For this case, $ A $ is also a bi-sectorial operator but $ \overline{D(A)} \neq X $, which makes our general results in \autoref{diffCocycle} can not apply. One needs to show directly that this equation generates continuous correspondence which satisfies (A) (B) condition. We believe this is true but until now we have no proof.

		\item Consider the following semi-linear wave equation in one dimension:
		\[
		\partial^2_t u - a\partial_t u - \partial^2_x u - b \partial_x u + g(u) + f(x,u) = 0, ~t \in \mathbb{R}, ~ x \in \Omega,
		\]
		endowed by the Dirichlet boundary condition or periodic boundary condition (e.g. $ \Omega = \mathbb{S}^1 $), where $ a = 0 $ or $ 1 $, $ b \in \mathbb{R} $. Rewrite it as an abstract differential equation. Take $ A_0 = \partial^2_x + b \partial_x $, $ X = \hat{H}^1 \times L^2(\Omega) $, $ D(A) = \hat{H}^2 \times \hat{H}^1 $,
		\[
		A = \left(  \begin{matrix}
		0 & 1 \\
		A_0 & -a
		\end{matrix} \right),
		~g \left( \begin{matrix}
		u \\
		v
		\end{matrix} \right)(x) = \left(  \begin{matrix}
		0  \\
		-g(u(x)) - f(x,u(x))
		\end{matrix} \right),
		\]
		where for Dirichlet boundary condition, take $ \hat{H}^1 = H^1_0(\Omega) $, $ \hat{H}^2 = H^2(\Omega) \cap H^1_0(\Omega) $, and for periodic boundary condition, take $ \hat{H}^1 = H^1(\Omega) $, $ \hat{H}^2 = H^2(\Omega) $. In most cases, $ g $ is taken as $ g(u) = -\gamma u +  \lambda|u|^\gamma u $ (Klein-Gordon equation), or $ g(u) = c\sin u $ (sin-Gordon equation), where $ \gamma, c > 0 $, $ \lambda \in \mathbb{R} $. See \cite{NS11,Tem97}.
		\begin{enumerate}[(i)]
			\item If $ b = 0 $, then $ A $ generates a $ C_0 $ group as $ A_0 $ is self-adjoint and $ \sigma(A_0) \subset \mathbb{R}_- $. Note that the Decoupling Theorem in \cite{Che18a} can be applied to this model nearby the equilibrium.
			\item If $ b \neq 0 $, then $ A $ generates a non-analytic bi-semigroup (but not a semigroup); see e.g. \cite{LP08}. So in this case the wave equation is not well-posed in general.
		\end{enumerate}
	\end{enumerate}
\end{exa}

\begin{exa}\label{ex:bou}
	\begin{enumerate}[(a)]
		\item Consider the following (good) Boussinesq equation in one dimension:
		\[
		\partial^2_t u  - \partial^2_x u - \alpha\partial^4_x u + \partial^2_x (u^2)  = 0, ~t \in \mathbb{R}, ~ x \in \Omega, \tag{gBou}
		\]
		endowed by the periodic boundary condition i.e. $ \Omega = \mathbb{S}^1 $, where $ \alpha > 0 $. Take $ A_0 = \partial^2_x + \alpha\partial^4_x $, $ X = H^2(\Omega) \times L^2(\Omega) $, $ D(A) = H^4(\Omega) \times H^2(\Omega) $,
		\[
		A = \left(  \begin{matrix}
		0 & 1 \\
		A_0 & 0
		\end{matrix} \right),
		~g \left( \begin{matrix}
		u \\
		v
		\end{matrix} \right)(x) = \left(  \begin{matrix}
		0  \\
		-\partial^2_x (u^2)
		\end{matrix} \right): X \to X.
		\]
		$ A $ is a bi-sectorial operator (i.e. $ A $ generates an analytic bi-semigroup), and hence (gBou) in general is not well-posed; see e.g. \cite{LP08, dlLla09} and the references therein for more details.

		\item \label{ex:bou2} Consider another situation. A periodic traveling wave of (gBou) is of the form $ u(x,t) = u_{c,a}(x-ct) \neq 0 $ satisfying
		\[
		\partial^2_x u_{c,a} + (1-c^2)u_{c,a} - u^2_{c,a} = a,~ x \in \mathbb{S}^1,
		\]
		for some (real) constant $ a $, where $ c^2 < 1 $. Consider the Boussinesq equation in the traveling frame $ (x-ct, t) $, i.e.
		\[
		(\partial_t - c\partial_x)^2 u  - \partial^2_x u - \alpha\partial^4_x u + \partial^2_x (u^2)  = 0. \tag{tBou}
		\]
		Linearizing the above equation at the equilibrium $ u_{c,a} $, one gets
		\[
		(\partial_t - c\partial_x)^2 u  - \partial^2_x u - \alpha\partial^4_x u + \partial^2_x (2u_{c,a} u)  = 0.
		\]
		Rewrite it as an abstract equation (let $ v_x = (\partial_t - c\partial_x) u $). Take $ X = H^1(\Omega) \times L^2(\Omega) $, $ \mathcal{L}_0 = \alpha\partial^2_x + 1 - c^2 - 2u_{c,a} $,
		\[
		J = \left(  \begin{matrix}
		0 & \partial_x \\
		\partial_x & 0
		\end{matrix} \right),
		~L = \left(  \begin{matrix}
		\mathcal{L}_0 + c^2 & c \\
		c & 1
		\end{matrix} \right),
		\]
		$ D(A) = H^3(\Omega) \times H^2(\Omega) $, $ A = JL $. Then the above linear equation is equivalent to $ \partial_t z = Az $, where $ z = (u,v)^\top $, which is a Hamiltonian system with $ J $ unbounded and $ \dim \ker J = 2 $ (due to the periodic boundary condition). Moreover, $ JL $ generates a $ C_0 $ group in $ X $ with \emph{finite} dimensional stable and unstable subspaces, and \emph{infinite} dimensional center subspace, and (tBou) is well-posed around $ u_{c,a} $. We refer the readers to see \cite{LZ17} for more details, where the readers can find very unified analysis of the instability, index theorem, and (exponential) trichotomy of $ A $ and other classes of more general wave equations admitting Hamiltonian structures with finite Morse index.
	\end{enumerate}
\end{exa}

\begin{exa}\label{ex:spatial}
	Consider the following reaction-diffusion equation in one dimension:
	\[
	\partial_t u = d \partial^2_x u + c \partial_x u + f(u), ~ t, x \in \mathbb{R},
	\]
	where $ d > 0 $. $ q: \mathbb{R} \times \mathbb{R} \to \mathbb{R} $ is called a \emph{modulated wave} if it satisfies above equation and $ q(x,t) = q(x,t+ T) $ for some $ T > 0 $ (see \cite{SS01}). Let us consider the dynamic generated by the above equation in the spatial variable $ x $, called the \emph{spatial dynamic} (which might be first introduced by K. Kirchg\"assner). Take $ X = H^{1/2}(\Omega) \times L^2(\Omega) $, $ D(A) = H^1(\Omega) \times H^{1/2}(\Omega) $,
	\[
	A = \left(  \begin{matrix}
	0 & 1 \\
	d^{-1}\partial_t & -c d^{-1}
	\end{matrix} \right),
	~g \left( \begin{matrix}
	u \\
	v
	\end{matrix} \right)(t) = \left(  \begin{matrix}
	0 \\
	d^{-1}f(u(t))
	\end{matrix} \right): X \to X.
	\]
	$ A $ is a bi-sectorial operator; see e.g. \cite{LP08}. So the spatial dynamic is not well-posed. This dynamic is interesting since the study of the stability of the traveling wave $ u_c $ (or the modulated wave $ q $) relies on the uniform dichotomy on $ \mathbb{R}_+ $ and $ \mathbb{R}_- $ of linearized \eqref{2**1} along $ u_c $ (see e.g. \cite{SS01}), and to find modulated wave $ q $, one can study the Hopf bifurcation of this system when restricting $ t \in \mathbb{S}^1 $ (see e.g. \cite{SS99}). (Note that the spectral theory for this class of the dynamic does not contain in \cite{LZ17}.)

	For this $ A $, it has \emph{infinite} dimensional stable and unstable subspaces, and \emph{finite} dimensional center subspace. By our existence results (or see \cite{dlLla09}), equation \eqref{2**1} in this case has a finite-dimensional $ C^k $ center manifold in the neighborhood of $ 0 $ if $ f \in C^k $. Furthermore, if $ f $ depends smoothly on a extra parameter, then our regularity results also say that the center manifold also depends smoothly on the extra parameter. One can apply Hopf bifurcation theorem in the finite-dimensional setting to get the existence of the non-trivial periodic orbits depending on parameter for the restricted dynamic on the center manifold. Now returning to the original system, one also obtains the non-trivial periodic orbits. See \cite{SS99} for more details. This idea was also used in \cite{MR09a} to deduce their Hopf bifurcation theorem in the age structured models in $ L^1 $. Note that this argument will fail if the center manifold is infinite-dimensional (see e.g. \autoref{ex:bou} \eqref{ex:bou2}).

	The reader can consider more spatial dynamics generated by equations in \autoref{ex:classical} and Swift-Hohenberg equation, Korteweg-de Vries equation, etc, which are all ill-posed in general.
\end{exa}

\begin{exa}\label{ex:nondiss}
	\begin{enumerate}[(a)]
		\item \label{asm} Consider the following age structured model:
		\[
		\begin{cases}
		\partial_t u - d \partial^2_x u + \partial_x u = M(x,u), \\
		u(t,0) = G(u(t,\cdot)),~
		u(0,x) = \phi(x),
		\end{cases}
		\]
		where $ d \geq 0 $, $ M: \mathbb{R}_+ \times \mathbb{R}_+ \to \mathbb{R}_+ $ is smooth, $ G: L^p(\mathbb{R}_+) \to \mathbb{R}_+ $ is smooth and $ \phi(\cdot) \in L^p(\mathbb{R}_+) $. $ 1 \leq p < \infty $. Rewrite it as an abstract equation. Take $ A_0 = d\partial^2_x - \partial_x $, $ X = \mathbb{R} \times L^p(\mathbb{R}_+) $, $ D(A) = \{0\} \times W_p $,
		\[
		A = \left(  \begin{matrix}
		0 & \delta_0 \\
		A_0  & 0
		\end{matrix} \right),
		~g \left( \begin{matrix}
		0 \\
		v
		\end{matrix} \right) = \left(  \begin{matrix}
		G(v(\cdot))  \\
		M(\cdot,v(\cdot))
		\end{matrix} \right): \{0\} \times L^p(\mathbb{R}_+) \to X,
		\]
		where $ W_p = W^{1,p}(\mathbb{R}_+) $ if $ d = 0 $, and $ W_p = W^{2,p}(\mathbb{R}_+) $ if $ d > 0 $. The Dirac measure $ \delta_0: W_p \to \mathbb{R} $ is defined by $ v \mapsto v(0) $. $ M(\cdot,v(\cdot))(x) = M(x,v(x)) $. The property of $ A $ is the following: (i) $ A $ is a Hille-Yosida operator if and only if $ p = 1 $; (ii) For $ p > 1 $, $ A $ is an MR operator (see \autopageref{MR}); (iii) $ A_{\overline{D(A)}} $ is a sectorial operator if and only if $ d > 0 $. See e.g. \cite{MR07, MR09a} and the references therein for more details. Note that $ \overline{D(A)} \neq X $.

		\item \label{ade} Consider the following (abstract) delay equation:
		\[
		\begin{cases}
		\dot{u}(t) = \mathcal{A} u(t) + f(u_t), ~ t \geq 0, \\
		u(\theta) = \phi(\theta), ~\theta \in [-r,0].
		\end{cases} \tag{ADE}
		\]
		The settings are (i) $ \mathcal{A}: D(\mathcal{A}) \to \mathcal{X} $ is an operator of the class (a) $ \sim $ (c) listed in \autoref{generalA} (for instance $ \mathcal{A} $ can be any operator in \autoref{ex:classical});
		(ii) $ f : \mathcal{C}_{\mathcal{A}} \to \mathcal{X} $ is smooth;
		(iii) $ 0 < r < \infty $,
		where $ \mathcal{X} $ is a Banach space, $ \mathcal{C} = C([-r,0],\mathcal{X}) $, and $ \mathcal{C}_{\mathcal{A}} = \{ \phi \in \mathcal{C}: \phi(0) \in \overline{D(\mathcal{A})} \} $. $ u_t(\theta) \triangleq u(t+\theta) $, $ \theta \in [-r,0] $, $ t \geq 0 $. Let us rewrite it as an abstract equation. Take $ X = \mathcal{X} \times \mathcal{C} $,
		\[
		D(A) = \{0\} \times \{ \phi \in C^1([-r,0], \mathcal{C}): \phi(0) \in D(\mathcal{A}) \},
		\]
		\[
		A = \left(  \begin{matrix}
		0 & \mathcal{A}\delta_0 - \delta'_0 \\
		0  & \partial_\theta
		\end{matrix} \right),
		~ g \left( \begin{matrix}
		0 \\
		\phi
		\end{matrix} \right) = \left(  \begin{matrix}
		f(\phi) \\
		0
		\end{matrix} \right) : \{0\} \times \mathcal{C}_{\mathcal{A}} \to X,
		\]
		where $ \mathcal{A}\delta_0 - \delta'_0 : \phi \mapsto \mathcal{A}\phi(0) - \dot{\phi}(0) $, $ (0, \phi) \in D(A) $. Note that $ \overline{D(A)} = \mathcal{C}_{\mathcal{A}} $ (see e.g. \cite[Section 3.4]{Che18f} for details). The property of $ A $ is the following: (i) If $ \mathcal{A} $ is a Hille-Yosida operator, then so is $ A $ (even $ \mathcal{A} = 0 $); (ii) If $ \mathcal{A} $ is an MR operator (see \autopageref{MR}), so is $ A $. The case (i) is well known, see e.g. \cite{EA06}. For the case (ii), we can not find the proof in the published papers, however case (ii) can be proved in the same way as case (i). The relation between \eqref{2**1} for this case and (ADE) is not so obvious as the previous examples, which we state in the following.
		\begin{enumerate}[$ \bullet $]
			\item  A continuous function $ u $ with $ u_0 = \phi \in \mathcal{C}_{\mathcal{A}} $ is a mild solution of (ADE) if and only if $ z(t) = (0,u_t)^\top $ is a mild solution of \eqref{2**1} for this case; all the mild solution $ z(t) = (0,z_1(t))^{\top} $ of \eqref{2**1} for this case has the form $ z_1(t) =  u_t $, $ t \in [0,b] $, $ b > 0 $, where $ u : [-r,b] \to \overline{D(\mathcal{A})} $ is continuous.
		\end{enumerate}
		For a proof, see e.g. \cite{EA06}. So (ADE) is well-posed and one only needs to study the dynamic generated by \eqref{2**1} which can reflect the property of (ADE). 

		\item \label{mixed} Consider the following mixed delay equation:
		\[
		\begin{cases}
		\dot{u}(t) = u(t-1) + u(t+1) + f(u_t), ~|t| \geq 1,\\
		u(\theta) = \phi(\theta), ~\theta \in [-1,1].
		\end{cases} \tag{MDE}
		\]
		Take $ X = C([-1,1], \mathbb{C}) $,
		\[
		Ax = \dot{x},~x \in D(A) = \{ x \in C^1([-1,1], \mathbb{C}): \dot{x}(0) = x(-1) + x(1) \}.
		\]
		It is known that (MDE) is not well-posed in general, and $ A : D(A) \subset X \to X $ generates a $ C_0 $ bi-semigroup on $ X $ which is induced by mild solutions of (MDE) (for $ f = 0 $). That is, there is a decomposition of $ X $ as $ X = X_+ \oplus X_- $ with projections $ P_{\pm} $ such that $ R(P_{\pm}) = X_{\pm} $ and $ P_+ + P_- = \id $, and $ \pm A_{\pm} \triangleq \pm A_{X_{\pm}} $ generate $ C_0 $ semigroups $ T_{\pm} $ respectively. For any mild solution $ u $ of (MDE) (for $ f = 0 $) has the form $ u_t = (T_+(t-t_1)P_{+}u_{t_1}, T_{-}(t_2-t)P_{-}u_{t_2}) $, for $ t_1 \leq t \leq t_2 $. See \cite{Mal99,vdMee08} for a proof. However, even the nonlinear (MDE) can rewrite as \eqref{2**1}, it is not the case we study in \autoref{bi-semigroup}, so we can not apply \autoref{thm:biAB}. The dynamical results obtained in \autoref{continuous case} and \autoref{normalH} can be applied to this model and a direct proof of the (A) (B) condition satisfied by this model is needed which we will give elsewhere.
	\end{enumerate}
\end{exa}